\documentclass[a4paper]{amsart}
\usepackage[english]{babel}
\usepackage[utf8]{inputenc}
\usepackage[T1]{fontenc}
\usepackage{csquotes} 
\usepackage{amssymb}
\usepackage{float} 
\usepackage{amsthm}
\usepackage[top=2cm, bottom=2cm, left=2cm, right=2cm, twoside=false]{geometry}

\usepackage{tabularray}

\usepackage{mathtools} 
\usepackage[usenames,x11names]{xcolor}
\usepackage[all]{xy} 
\usepackage[normalem]{ulem}
\usepackage{multirow} 

\usepackage[shortlabels]{enumitem}
\setlist[1]{leftmargin=*}
\setlist[enumerate,1]{label=(\alph*)}
\setlist[enumerate,2]{label=(\roman*), ref=(\alph{enumi}.\roman*)}
\newlist{enumerate-alt}{enumerate}{1}
\setlist[enumerate-alt,1]{label=(\roman*)}
\newlist{longlist}{enumerate}{2}
\setlist[longlist,1]{label=\small{(\arabic*)}, itemsep=0.2em}
\setlist[longlist,2]{label=\small{(\arabic{longlisti}.\arabic*)}}
\newlist{shortlist}{enumerate}{1}
\setlist[shortlist]{label=\small{(\Alph*)}}
\newlist{parlist}{enumerate}{1}
\setlist[parlist]{leftmargin=0cm, itemindent=2\parindent, label=(\alph*), itemsep=0.2em}

\newlist{casesp}{enumerate}{5} 
\setlist[casesp]{align=left, 
	listparindent=\parindent, 
	parsep=\parskip, 
	font=\normalfont\bfseries, 
	leftmargin=0pt, 
	labelwidth=0pt, 
	itemindent=.4em,labelsep=.4em, 
	topsep=.6em, 
	itemsep=.6em, 
}
\setlist[casesp,1]{label=Case~\arabic*:,ref=\arabic*}
\setlist[casesp,2]{label=Subcase~\thecasespi.\arabic*:,ref=\thecasespi.\arabic*}
\setlist[casesp,3]{label=Subcase~\thecasespii.\arabic*:,ref=\thecasespii.\arabic*}
\setlist[casesp,4]{label=Subcase~\thecasespiii.\arabic*:,ref=\thecasespiii.\arabic*}
\setlist[casesp,5]{label=Subcase~\thecasespiv.\arabic*:,ref=\thecasespiv.\arabic*}
\newcommand\litem[1]{\item{\bfseries #1.\enspace}}

\newlist{casesp*}{enumerate}{2}
\setlist[casesp*]{align=left, listparindent=\parindent, parsep=\parskip, font=\normalfont\bfseries, leftmargin=0pt, 	labelwidth=0pt, itemindent=.4em,labelsep=.4em, topsep=.6em, itemsep=.6em }
\setlist[casesp*,1]{label=Case,ref=\arabic*}
\setlist[casesp*,2]{label=Subcase,ref=\arabic*}
\newlist{types}{enumerate}{1}
\setlist[types]{align=left, listparindent=\parindent, parsep=\parskip, font=\normalfont\bfseries, leftmargin=0pt, 	labelwidth=0pt, itemindent=.4em,labelsep=.4em, topsep=.6em, itemsep=.6em }
\setlist[types,1]{label=Type,ref=\arabic*}

\usepackage{subcaption} 
\usepackage{tikz}
\usepackage{tikz-cd}
\usepackage{pgfplots}
\usetikzlibrary{calligraphy,snakes,calc}
\tikzset{
	add/.style args={#1 and #2}{
		to path={%
			($(\tikztostart)!-#1!(\tikztotarget)$)--($(\tikztotarget)!-#2!(\tikztostart)$)%
			\tikztonodes},add/.default={.2 and .2}}
}

\usepackage{xr-hyper} 
\usepackage{hyperref} 
\hypersetup{
	linkcolor=DodgerBlue3, 
	citecolor  = teal,
	urlcolor   = DodgerBlue4,
	colorlinks = true, 
	hyperfootnotes =false,
	bookmarksdepth=3,
	linktoc=all
}
\makeatletter
\@namedef{subjclassname@2020}{%
	\textup{2020} Mathematics Subject Classification}
\makeatother

\makeatletter 
\renewcommand\subsection{\@startsection{subsection}{3}
	\z@{.5\linespacing\@plus.7\linespacing}{.5\linespacing}
	{\bfseries\itshape}} 
\renewcommand\paragraph{\@startsection{paragraph}{4}%
	\z@{.5\linespacing\@plus.7\linespacing}{-.5\linespacing}%
	{\normalfont\bfseries}}
\makeatother

\usepackage{indentfirst}

\makeatletter \renewenvironment{proof}[1][\proofname]{
	\par\pushQED{\qed}\normalfont
	\topsep6\p@\@plus6\p@\relax
	\trivlist\item[\hskip\labelsep\bfseries#1\@addpunct{.}]
	\ignorespaces}{
	\popQED\endtrivlist\@endpefalse} \makeatother

\theoremstyle{plain}
\newtheorem{theorem}{Theorem}[section]
\newtheorem*{theorem*}{Theorem}

\theoremstyle{definition}

\newtheorem{corollary}[theorem]{Corollary}
\newtheorem{definition}[theorem]{Definition}
\newtheorem{lemma}[theorem]{Lemma}
\newtheorem{proposition}[theorem]{Proposition}
\newtheorem{example}[theorem]{Example}

\newtheorem{remark}[theorem]{Remark}
\newtheorem{claim}{Claim}
\newtheorem*{claim*}{Claim}

\newtheorem*{remark*}{Remark} 

\let\sec\S 

\newcommand{\A}{\mathbb{A}}
\newcommand{\C}{\mathbb{C}}
\newcommand{\F}{\mathbb{F}}
\newcommand{\G}{\mathbb{G}}

\renewcommand{\P}{\mathbb{P}}
\newcommand{\Q}{\mathbb{Q}}

\newcommand{\Z}{\mathbb{Z}}

\newcommand{\cA}{\mathcal{A}}

\newcommand{\cC}{\mathcal{C}}
\newcommand{\cD}{\mathcal{D}}
\newcommand{\cE}{\mathcal{E}}

\newcommand{\cL}{\mathcal{L}}

\newcommand{\cO}{\mathcal{O}}
\newcommand{\cP}{\mathcal{P}}
\newcommand{\cQ}{\mathcal{Q}}

\newcommand{\cS}{\mathcal{S}}
\newcommand{\cX}{\mathcal{X}}
\newcommand{\cY}{\mathcal{Y}}

\newcommand{\Pcusp}{\cP_{\textnormal{cusp}}}

\newcommand{\rA}{\mathrm{A}}
\newcommand{\rD}{\mathrm{D}}
\newcommand{\rE}{\mathrm{E}}
\newcommand{\rC}{\mathrm{C}}
\newcommand{\rN}{\mathrm{N}}

\usepackage[scr=boondoxo]{mathalfa}

\renewcommand{\aa}{\mathscr{a}}

\renewcommand{\ll}{\mathscr{l}}
\newcommand{\cc}{\mathscr{c}}
\newcommand{\qq}{\mathscr{q}}
\newcommand{\pp}{\mathscr{p}}
\newcommand{\hh}{\mathscr{h}}
\newcommand{\kk}{\mathscr{k}}

\newcommand{\ttt}{\mathscr{t}}

\renewcommand{\epsilon}{\varepsilon}
\renewcommand{\phi}{\varphi}
\renewcommand{\theta}{\vartheta}

\DeclareFontFamily{U}{mathx}{}
\DeclareFontShape{U}{mathx}{m}{n}{<-> mathx10}{}
\DeclareSymbolFont{mathx}{U}{mathx}{m}{n}
\DeclareMathAccent{\widehat}{0}{mathx}{"70}
\DeclareMathAccent{\widecheck}{0}{mathx}{"71}

\renewcommand{\tilde}{\widetilde}
\renewcommand{\check}{\widecheck}
\renewcommand{\hat}{\widehat}
\renewcommand{\bar}{\overline}

\renewcommand{\leq}{\leqslant}
\renewcommand{\geq}{\geqslant}

\renewcommand{\to}{\longrightarrow}
\newcommand{\map}{\dashrightarrow}

\newcommand{\into}{\hookrightarrow}

\newcommand{\cha}{\operatorname{char}} 

\newcommand{\Aut}{\operatorname{Aut}}
\newcommand{\Sing}{\operatorname{Sing}}

\newcommand{\NS}{\operatorname{NS}}
\newcommand{\Exc}{\operatorname{Exc}}
\newcommand{\Supp}{\operatorname{Supp}}

\newcommand{\redd}{_{\mathrm{red}}} 
\newcommand{\reg}{^{\mathrm{reg}}} 
\newcommand{\trp}{^{\scriptscriptstyle{\top}}} 

\newcommand{\PGL}{\mathrm{PGL}}
\newcommand{\im}{\operatorname{Im}}

\newcommand{\ld}{\operatorname{ld}} 
\newcommand{\cf}{\operatorname{cf}}

\newcommand{\lts}[2]{\mathcal{T}_{#1}(-\log #2)} 

\newcommand{\hor}{_{\mathrm{hor}}} 
\renewcommand{\vert}{_{\mathrm{vert}}} 

\newcommand{\cp}[1]{^{(#1)}} 

\newcommand{\am}{_{\textnormal{am}}}

\newcommand{\height}{\operatorname{ht}}

\newcommand{\id}{\mathrm{id}}

\newcommand{\Sec}{\Xi}
\newcommand{\contS}{\varsigma}
\newcommand{\bs}[1]{\boldsymbol{#1}} 
\newcommand{\ub}[1]{\uline{\bs{#1}}} 

\newcommand{\dec}[1]{^{|#1|}}

\newcommand{\de}{\coloneqq} 

\begin{document}
	\title[On the structure of open del Pezzo surfaces]{On the structure of open del Pezzo surfaces}
	
	\author{Karol Palka}
	\address{Institute of Mathematics, Polish Academy of Sciences, \'{S}niadeckich 8, 00-656 Warsaw, Poland}
	\email{palka@impan.pl}
	\author{Tomasz Pe{\l}ka}
	\address{University of Warsaw, Faculty of Mathematics, Informatics and Mechanics, Banacha 2, 02-097 Warsaw, Poland}
	\email{tpelka@mimuw.edu.pl}
	\thanks{This project was funded by the National Science Centre, and, grant number 2021/41/B/ST1/02062. For the purpose of Open Access, the authors have applied a CC-BY public copyright license to any Author Accepted Manuscript version arising from this submission.}
\begin{abstract}
	Let $(\bar{X}, \bar{D})$ be an open log del Pezzo surface of rank one, that is, $\bar{X}$ is a normal projective surface of Picard rank one, the boundary $\bar{D}$ is a reduced nonzero divisor on $\bar{X}$, and $-(K_{\bar{X}}+\bar{D})$ is ample. We show that, up to well described exceptions in characteristics $2,3$ and $5$, the smooth part of $\bar{X}\setminus \bar{D}$ admits an $\A^{1}$- or an $\A^{1}_{*}$-fibration, which extends to a $\P^1$-fibration of the minimal log resolution of $(\bar{X},\bar{D})$. In characteristic $0$, this improves a well-known structure theorem of Miyanishi--Tsunoda \cite{Miy_Tsu-opendP}.	Within the proof, we classify rational anti-canonical curves contained in smooth loci of canonical del Pezzo surfaces of rank one.	
\end{abstract}

\maketitle
\setcounter{tocdepth}{1}
\tableofcontents
\section{Introduction}

A  pair $(\bar{X},\bar{D})$ consisting of a normal projective surface $\bar{X}$ and a reduced Weil divisor $\bar{D}$ is called a \emph{log del Pezzo surface} if the anti-log canonical divisor $-(K_{\bar{X}}+\bar{D})$ is ample. It is \emph{open} if $\bar{D}\neq 0$. Its \emph{rank} is the Picard rank of $\bar{X}$. In this article, we assume that the base field $\kk$ is algebraically closed, of arbitrary characteristic; and, unless explicitly stated otherwise, consider only log surfaces with reduced boundaries. 

Being one of the possible outcomes of the minimal model program in dimension $2$, log del Pezzo surfaces of rank one are of great importance in birational geometry, and have been studied by many authors from various points of view, see  \cite{Miy_Tsu-opendP,GurZha_1,Keel-McKernan_rational_curves}. In case $\bar{D}\neq 0$, the structure of the open part $\bar{X}\setminus \bar{D}$ is described by theorems of Fujita, Miyanishi, Sugie, Russell and Tsunoda \cite{Fujita-Zariski,Sugie,Miy_Su,Russell-ruled, Miy_Tsu-nonconnected}. These theorems have numerous applications in affine geometry, cf.\ \cite{KR_contractible,Kojima_Sing-1,PK_QHP,Daigle_Russell} or \cite[\sec II.5]{Miyan-OpenSurf} for a discussion. They imply that if $\cha\kk=0$ then the open surface $\bar{X}\setminus \bar{D}$ admits an $\A^{1}$-fibration or a very specific (\emph{Platonic}) $\A^{1}_{*}$-fibration, see \cite{Miy_Tsu-opendP}.
\smallskip

The main goal of this article is to generalize the above theorems to algebraically closed fields of arbitrary characteristic, and at the same time to gain control on the base points of the corresponding $\A^1$- or $\A^1_{*}$- fibrations on the boundary of the minimal log resolution of $(\bar{X},\bar{D})$, call it $(X,D)$. More precisely, while \cite{Miy_Tsu-opendP} proves the existence of $\A^{1}$- or $\A^{1}_{*}$-fibrations of $X\setminus D$ -- which might have base points on $D$, see Remark \ref{rem:hidden_A1_fibrations} -- we construct $\P^1$-fibrations of $X$ whose fiber meets $D$ at most twice. We introduce the following definition. 
 
 \begin{definition}[Height]\label{def:height}
 	Let $\bar{D}$ be a reduced Weil divisor on a normal surface $\bar{X}$, and let $(X,D)\to (\bar{X},\bar{D})$ be the minimal log resolution, where $D$ is the reduced total transform of $\bar{D}$. The \emph{height} of the log surface $(\bar{X},\bar{D})$, denoted by $\height(\bar{X},\bar{D})$, is the infimum of the set of integers $\hh$ such that $X$ admits a $\P^1$-fibration whose fiber $F$ satisfies $F\cdot D=\hh$. 
 	The \emph{height} of a normal surface $\bar{X}$ is $\height(\bar{X})\de \height(\bar{X},0)$. 
 \end{definition}
 
Enriques classification of smooth projective surfaces implies that $\height(\bar{X},\bar{D})=\infty$ if and only if $\bar{X}\cong \P^2$ or $\kappa(\bar{X})\geq 0$. Any $\P^{1}$-fibration realizing the infimum in Definition \ref{def:height} is said to \emph{witness} the height. 

We can now state our first main result, see Proposition \ref{prop:MT_smooth} and Corollary \ref{cor:MT} for more detailed versions.

\begin{theorem}[Structure of open del Pezzo surfaces]\label{thm:MT-version}
	Let $(\bar{X},\bar{D})$ be a log del Pezzo surface of rank one such that $\bar{X}\not\cong \P^{2}$ and $\bar{D}$ is reduced and nonzero. Then $\height(\bar{X},\bar{D})\leq 2$, unless one of the following holds.
	\begin{enumerate}
		\item\label{item:thm-relative} $\cha\kk\in \{2,3,5\}$ and the minimal log resolution of $(\bar{X},\bar{D})$ is also the minimal log resolution of a log surface $(\bar{Y},\bar{T})$, where $\bar{Y}$ is a canonical surface of rank one and $\bar{T}\subseteq \bar{Y}\reg$ is a cuspidal curve with $p_{a}(\bar{T})=1$, see Proposition \ref{prop:canonical}. In particular, $\height(\bar{X})=\height(\bar{Y})+2\in \{4,6\}$, see Proposition \ref{prop:height}.
		\item\label{item:thm-exception} $\cha\kk=2$ and $(\bar{X},\bar{D})$ is isomorphic to one of the two log del Pezzo surfaces constructed in Example \ref{ex:exception}, see Figures \ref{fig:exception-1}, \ref{fig:exception-2}. In particular,  $\height(\bar{X},\bar{D})=3$.
	\end{enumerate}
\end{theorem} 

The assumption $\bar{D}\neq 0$ is crucial. Indeed, for a del Pezzo surface $\bar{X}$ of rank one, the height $\height(\bar{X})$ can be as high as $4$ (for arbitrary $\cha\kk$), or even $6$ if $\cha\kk=2$ or $3$, see \cite[Theorem E(f)]{PaPe_ht_2}. 

As a consequence of Proposition \ref{prop:MT_smooth}, we get the following Corollaries \ref{cor:Russell} and \ref{cor:uniruled}, generalizing the results of \cite{Russell-ruled} and \cite[\sec 8.2]{Miy_Tsu-opendP}. We recall that a surface is \emph{$\A^{1}$-ruled} if it contains an open subset which is isomorphic to a cylinder $\A^1\times C$ for some curve $C$; and is \emph{$\A^1$-uniruled} if it admits a dominant (not necessarily separable) morphism from an $\A^{1}$-ruled surface. 

\begin{corollary}[Log surfaces with $\kappa=-\infty$ and connected boundary, see \cite{Russell-ruled}]\label{cor:Russell}
	Let $X$ be a smooth projective surface and let $D$ be a connected, reduced divisor on $X$ such that $\kappa(K_{X}+D)=-\infty$. Then the following hold.
	\begin{enumerate}
		\item\label{item:Rus-ht} We have either $\height(X,D)\leq 2$, or $X=\P^2$ and $\deg D\leq 2$.
		\item\label{item:Rus-Cylinder} The open part $X\setminus D$ is $\A^{1}$-ruled.
	\end{enumerate}
\end{corollary}

Part \ref{item:Rus-Cylinder} is the original result of \cite{Russell-ruled}, obtained by generalizing to arbitrary characteristic previous results of Fujita and Miyanishi--Sugie \cite{Fujita-Zariski,Miy_Su}, which later gave the foundation for the Miyanishi--Tsunoda theorem \cite{Miy_Tsu-opendP}. 
In Section \ref{sec:MT} we deduce it from Proposition \ref{prop:MT_smooth}, essentially repeating some arguments of \cite{Russell-ruled}. 

	In case $\cha\kk=0$, the construction of \cite[\sec 8.1]{Miy_Tsu-opendP} shows that every Platonic $\A^{1}_{*}$-fiber space becomes $\A^{1}$-ruled after a finite base change, hence is $\A^{1}$-uniruled. We will see in Section \ref{sec:MT} that similar construction shows $\A^{1}$-uniruledness for the open parts of exceptional log del Pezzo surfaces in Theorem \ref{thm:MT-version}, too. This leads to the following Corollary \ref{cor:uniruled}. By a \emph{minimal model} we mean an outcome of the birational part of the minimal model program, see Section \ref{sec:singularities} for details.
	\begin{corollary}[{$\A^{1}$-uniruledness, cf.\ \cite[\sec 8.2]{Miy_Tsu-opendP}}]\label{cor:uniruled}
		Let $X$ be a smooth projective surface with a reduced divisor $D$ such that $\kappa(K_{X}+D)=-\infty$ and $(X,D)$ admits a minimal model with nonzero boundary (this holds for instance if the intersection matrix of $D$ is not negative definite). Then $X\setminus D$ is $\A^{1}$-uniruled. 
	\end{corollary}

To prove Theorem \ref{thm:MT-version} we need to show that log surfaces $(\bar{Y},\bar{T})$ as in \ref{thm:MT-version}\ref{item:thm-relative} occur only if $\cha\kk\in \{2,3,5\}$.  To do this, we classify all log surfaces $(\bar{Y},\bar{T})$  such that $\rho(\bar{Y})=1$ and $\bar{T}$ is a rational curve with $p_a(\bar{T})=1$ contained in the smooth locus of $\bar{Y}$. In this case $\bar{Y}$ is canonical and $\bar{T}\in |-K_{\bar{Y}}|$, see Lemma  \ref{lem:T-in-smooth-locus}\ref{item:elliptic}. The classification is summarized in Proposition \ref{prop:canonical} below, where by the \emph{singularity type} of a normal surface we mean the weighted graph of the exceptional divisor of its minimal resolution. 

We note that canonical surfaces of rank one and their anti-canonical systems are well known, various descriptions are given e.g.\ in  \cite{Zhang_canonical,Ye,Kawakami_Nagaoka_canonical_dP-in-char>0}. We give a self-contained argument based on a simpler approach, essentially independent of $\cha\kk$. Namely, using a $\P^{1}$-fibration of low height we construct a birational map $\bar{Y}\map \P^{2}$ such that the total transform of $\bar{T}$ is an arrangement of a rational cubic, possibly one conic, and few lines. This allows to avoid technical analysis of elliptic fibrations done in loc.\ cit.

\begin{proposition}[Rational anti-canonical curves on canonical del Pezzo surfaces]\label{prop:canonical}
	Let $\cS$ be a singularity type of a canonical del Pezzo surface of rank one other than $\P^2$ or the quadric cone (listed in the second column of Table \ref{table:canonical}). Then the following hold.
	\begin{parlist}
		\item\label{item:canonical-types} There exists a del Pezzo surface $\bar{Y}$ of rank one and type $\cS$ whose smooth part $\bar{Y}\reg$ contains a complete nodal (respectively, cuspidal) curve $\bar{T}$ of arithmetic genus one if and only if the corresponding row of Table \ref{table:canonical} contains $\rN$ (respectively, $\rC$ or $\rC_d$) in the last column. 
		\item\label{item:canonical-uniqueness} The  isomorphism class of the log surface $(\bar{Y},\bar{T})$ as in \ref{item:canonical-types} is uniquely determined by the singularity types of $\bar{Y}$ and $\bar{T}$, except for types denoted in Table \ref{table:canonical} by $\rC_{d}$. In these cases, the set $\Pcusp(\cS)$ of isomorphism classes of log surfaces $(\bar{Y},\bar{T})$ with $\bar{Y}$ of type $\cS$ and cuspidal $\bar{T}$ has moduli dimension $d$. 
		\item\label{item:canonical-aut} If $\bar{T}$ is nodal then $(\bar{Y},\bar{T})$ has an automorphism  interchanging the analytic branches of $\bar{T}$ at the node.
		\item \label{item:canonical-Noether} We have $\bar{T}\in |-K_{\bar{Y}}|$, and the Picard rank of the minimal resolution of $\bar{Y}$ equals $10-\bar{T}^2$. 
	\end{parlist} 
\end{proposition}

For the definition of moduli dimension, see \cite[Definitions 1.9, 1.10]{PaPe_ht_2} or Section \ref{sec:moduli}. By this definition, Proposition \ref{prop:canonical}\ref{item:canonical-uniqueness} implies that the minimal log resolutions of all log surfaces in $\Pcusp(\cS)$ are parametrized by a family with base of dimension $d$. Such a family blows down to a family parametrizing all surfaces in $\Pcusp(\cS)$, cf.\ \cite[Lemma 2.15]{PaPe_ht_2}. Additional properties of these families are summarized in  Table \ref{table:symmetries}.

\begin{remark}
It is important to note that even if $\#\Pcusp(\cS)=1$, it can happen that the singularity type $\cS$ is realized by more than one canonical del Pezzo surface of rank one, see \cite[Table 1.1]{Ye} or \cite[Table 15]{PaPe_ht_2}. 
Hence there do exist pairs of canonical del Pezzo surfaces of rank one and the same singularity type, such that one of them contains a rational anti-canonical curve in its smooth locus, and the other does not. For an explicit example of type $\rA_2+\rE_6$ see \cite[Example 5.6(a)]{PaPe_ht_2}. 
\end{remark}

The classification of canonical del Pezzo surfaces of rank one is known \cite{Furushima,Ye,Kawakami_Nagaoka_canonical_dP-in-char>0,Alexxev-Nikulin_delPezzo-index-2}, cf.\ the introduction to Section \ref{sec:canonical}. In Proposition \ref{prop:canonical_ht>2} we give an independent argument for surfaces of height at least three; cases of lower height being elementary and covered for instance in \cite{PaPe_ht_2}. It turns out that the only canonical del Pezzo surfaces of height at least $3$ are of type $4\rA_2$ or, if $\cha\kk=2$, of type $8\rA_1$. In Sections \ref{sec:4A2} and \ref{sec:8A1} we prove that those surfaces are in fact  of height $4$, and we describe their geometry, including their automorphism groups. In many cases this computation is known, see e.g.\  \cite{Kawakami_Nagaoka_canonical_dP-in-char>0}. Nonetheless, we give a self-contained argument which exploits the geometry of a witnessing $\P^{1}$-fibration. 

\begin{table}[ht]
	\begin{tabular}{c|c|cccc}
		{\multirow{2}{*}{$\height(\bar{Y})$}} & {\multirow{2}{*}{singularity type of $\bar{Y}$}}  & \multicolumn{4}{c}{rational anti-canonical curves if  $\cha \kk$}
		\\ 
		&& $\ \ \neq 2,3,5\ $ & $\ =5\ $ & $\ =3\ $ & $\ =2\ $  \\ \hline\hline
		1 & $\rA_{1}+\rA_{2}$, $\rA_{4}$, $\rD_5$, $\rE_6$, $\rE_7$, $\rE_8$ & C, N & C, N & C, N & C, N  \\
		& $2\rA_{1}+\rA_{3}$, $\rA_{1}+\rA_{5}$, $\rA_{7}$, $\rA_1+\rD_6$, $\rD_8$, $\rA_1+\rE_7$ & N & N & N & C\\
		& $3\rA_{1}+\rD_{4}$, $2\rA_{1}+\rD_{6}$ & - & - & - & $\rC_{1}$ \\
		\hline 
		2 & $3\rA_2$, $\rA_2+\rA_5$, $\rA_8$, $\rA_2+\rE_6$  & N & N & C & N\\
		&  $\rA_1+2\rA_3$,  $\rA_3+\rD_5$, $\rA_1+\rA_7$,  & N & N & N & - \\
		&  $2\rA_4$ & N & C & N & N \\
		&  $\rA_1+\rA_2+\rA_5$ & N & N & - & - \\
		& $2\rD_4$,  $2\rA_1+2\rA_3$ & - & - & - & - \\
		&  $7\rA_1$, $4\rA_1+\rD_4$ & - & - & - & $\rC_{2}$ \\ 		
		\hline
		4 & $4\rA_2$ & - & - & $\rC_{1}$ & - \\
		& $8\rA_1$ & - & - & - & $\rC_{3}$  
	\end{tabular}
	\caption{Canonical del Pezzo surfaces of rank one in Proposition \ref{prop:canonical}, cf.\ \cite[Table 15]{PaPe_ht_2}.}
	\label{table:canonical}
\end{table}

\section{Preliminaries}

We settle the notation and recall some known results which will be useful in this and the forthcoming articles. 

\subsection{Divisors on surfaces}\label{sec:log_surfaces}

Let $D$ be a divisor on a normal projective surface $X$.  We write $D\redd$ for $D$ with reduced structure. By a \emph{component} of $D$ we mean an \emph{irreducible} component of $D\redd$; we denote their number by $\#D$. We say that $D$ is \emph{connected} if $\Supp D$ is connected (in the Zariski topology). If divisors $D,T$ and $D-T$ are effective, we say that $T$ is  a \emph{subdivisor} of $D$. In this case $\beta_{D}(T)\de T\cdot (D-T)$ is the \emph{branching number} of $T$ in $D$. 
\smallskip

Let $D$ be a reduced divisor on $X$. A component $T$ of $D$ is a \emph{tip} of $D$ if $\beta_{D}(T)\leq 1$; it is \emph{branching} in $D$ if $\beta_{D}(T)\geq 3$. A point $p\in \Sing D$ is a \emph{normal crossing} (\emph{nc}) point of $D$ if it lies in the smooth locus $X\reg$ of $X$ and in some local analytic coordinates $(x,y)$ at $p$ we have  $D=\{xy=0\}$. If all components of $D$ are smooth and all singular points of $D$ are nc, we call $D$ a \emph{simple normal crossing} (\emph{snc}) divisor. A \emph{log smooth surface} is a pair $(X,D)$ consisting of a smooth projective surface $X$ and an snc divisor $D$ on $X$.

Given a birational morphism $\phi\colon X\to \bar{X}$ and a divisor $D$ on $X$ we write $\phi\colon (X,D)\to (\bar{X},\bar{D})$ if $\bar{D}=\phi_{*}D$. We denote by  $\Exc\phi$ the reduced exceptional divisor of $\phi$, and by $\phi^{-1}_{*}\bar{D}$ the proper transform of $\bar{D}$ on $X$. 
\smallskip

A \emph{curve} is an irreducible and reduced variety of dimension one. A curve $C$ on a smooth projective surface $X$ is called an \emph{$n$-curve} if $C\cong \P^1$ and $C^2=n$. A $(-1)$-curve $C$ in an snc divisor $D$ is \emph{superfluous} if $\beta_{D}(C)\in \{1,2\}$ and $C\cdot T\leq 1$ for every component $T$ of $D-C$. This way, the image of $D$ after a contraction of $C$ is still an snc divisor. An snc divisor is \emph{snc-minimal} if it contains no superfluous $(-1)$-curves.

Let $D$ be an snc divisor. We say that $D$ is \emph{rational} if all its components are. We say that $D$ is \emph{negative definite} if so is its intersection matrix $[T_{i}\cdot T_{j}]_{1\leq i,j\leq \#D}$, where $T_1,\dots, T_{\#D}$ are components of $D$. The \emph{discriminant} of $D$ is $d(D)\de \det[-T_{i}\cdot T_{j}]_{1\leq i,j\leq \#D}$, we put $d(0)=1$. For elementary properties of discriminants see \cite[\S 3]{Fujita-noncomplete_surfaces}.

A connected snc divisor  with no branching component is a \emph{chain} if it has a tip and \emph{circular} if it does not. A connected snc divisor is a \emph{tree} if it has no circular subdivisor. It is known that a connected reduced divisor has arithmetic genus zero if and only if it is a rational tree. As a consequence, we have the following observation.

\begin{lemma}[{\cite[Lemma II.2.2.2]{Miyan-OpenSurf}}]\label{lem:trick-trees}
	Let $X$ be a smooth projective rational surface, and let $D$ be a connected reduced divisor on $X$ such that $|K_{X}+D|=\emptyset$. Then $p_{a}(D)=0$, so $D$ is a rational tree.
\end{lemma}
\begin{proof}
	The exact sequence $0\to\cO_{X}(-D)\to \cO_{X}\to \cO_{D}\to 0$ yields an exact sequence $H^{1}(X,\cO_{X})\to H^{1}(D,\cO_{D})\to H^{2}(X,\cO_{X}(-D))$, whose first term vanishes because $X$ is rational, and the last one is dual to $H^{0}(X,\cO_{X}(K_{X}+D))$, which vanishes by assumption. Thus $H^{1}(D,\cO_{D})=0$, so $\chi(\cO_{D})=1$ since $D$ is connected. Taking Euler characteristics in the above exact sequence gives $1=\chi(\cO_{D})=\chi(\cO_{X})-\chi(\cO_{X}(-D))=1-p_{a}(D)$ by Riemann-Roch. Thus $p_{a}(D)=0$, as needed.
\end{proof}

Let $T$ be a chain with a chosen \emph{first} tip $T\cp{1}$. Then the components $T\cp{1},\dots,T\cp{\#T}$  of $T$ are ordered in a unique way such that $T\cp{i}\cdot T\cp{i+1}=1$, $1\leq i\leq \#T-1$. 
We write $T\trp$ for the same chain with an  opposite order.

A \emph{type} of an ordered rational chain $T$ is a sequence of integers $[a_1,\dots,a_{\#T}]$, where $a_{i}=-(T\cp{i})^{2}$. We  often abuse notation and identify a chain with its type. Moreover, for types $T_1=[a_1,\dots, a_k]$, $T_2=[b_1,\dots, b_l]$ we write $[T_1,T_2]=[T_1,b_1,\dots,b_l]=[a_1,\dots,a_k,b_1,\dots,b_l]$ etc. 
We write $(m)_{k}$ for an integer $m$ repeated $k$ times. 
\smallskip

Let again $D$ be an snc divisor. An ordered subchain $T$ of $D$ is a \emph{twig} of $D$ if its first tip is a tip of $D$, and no component of $T$ is branching in $D$. In this case, the last tip of $T$ meets $D-T$, or $T$ is a connected component of $D$. A \emph{$(-2)$-twig} (chain, fork) is a twig (chain, fork) whose all components are $(-2)$-curves. 
\smallskip

A tree $T$ whose unique branching component $B$ has  $\beta_{T}(B)=3$ is called a \emph{fork}. A \emph{type} of a rational fork with branching component $B$ consists of an integer $b\de -B^2$ and an (unordered) set of types $\{T_1,T_2,T_3\}$ of its twigs meeting $B$. We denote this type by 
	$\langle b;T_1,T_2,T_3 \rangle$. 
If $T_{1},T_{2},T_{3}$ are admissible chains, we put 
\begin{equation*}
	\delta_{T}\de \sum_{i=1}^{3}\frac{1}{d(T_i)}.
\end{equation*} 
We say that $T$ is \emph{admissible} if $b\geq 2$ and $\delta_{T}>1$. The latter condition means that $\{d(T_1),d(T_2),d(T_3)\}$ is one of the \emph{Platonic} triples $\{2,3,5\}$, $\{2,3,4\}$, $\{2,3,3\}$ or $\{2,2,k\}$ for some $k\geq 2$; in particular one of the twigs $T_1,T_2,T_3$ is of type $[2]$, cf.\ \cite[\sec I.5.3.4]{Miyan-OpenSurf}.

\subsection{Dlt log surfaces and minimal models}\label{sec:singularities}

We now recall the basic notions of dlt, log terminal and log canonical surfaces, see \cite[\sec 2.1]{Kollar_singularities_of_MMP}. Let $\bar{X}$ be a normal surface, and let $\bar{D}$ be a divisor on $\bar{X}$ with coefficients between $0$ and $1$ (in this article, we will only consider the case when $\bar{D}$ is reduced). A \emph{log resolution} of $(\bar{X},\bar{D})$ is a birational morphism $\phi\colon (X,D)\to (\bar{X},\bar{D})$ such that $X$ is smooth and $D\redd$ is snc, where $D=\phi^{-1}_{*}\bar{D}+\Exc\phi$. A log resolution is \emph{minimal} if it not birationally dominated by any other log resolution. 

Assume that the divisor $K_{\bar{X}}+\bar{D}$ is $\Q$-Cartier. Put $E\de \Exc\phi$. Since $E$ is negative definite, the formulas
\begin{equation}\label{eq:discrepancy}
	\phi^{*}(K_{\bar{X}}+\bar{D})=K_{X}+\phi^{-1}_{*}\bar{D}+\sum_{T\subseteq E}\cf_{E}(T)\, T\quad \mbox{and}\quad \ld_{E}(T)=1-\cf_{E}(T)
\end{equation}
uniquely define, for each component $T$ of $E$, rational numbers $\cf_{E}(T)$ and $\ld_{E}(T)$, called the \emph{coefficient} and \emph{log discrepancy} of $T$, respectively. Here the sum runs over all components of $\Exc\phi$. 
 They can be easily computed from the weighted graph of $E$ \cite[\sec 3.2]{Flips_and_abundance}, cf.\ \cite[\sec 4C and \sec 7]{Palka_almost_MMP} or Lemmas \ref{lem:ld_formulas}, \ref{lem:Alexeev} below. If $\phi$ is a minimal resolution of singularities of $\bar{X}$, i.e.\ $E$ contains no $(-1)$-curves, then $\cf_{E}(T)\geq 0$ for every component $T$ of $E$. We skip the subscript $E$ whenever it is clear from the context.

A log surface $(\bar{X},\bar{D})$ is \emph{log terminal}, (respectively, \emph{log canonical} or \emph{canonical}) if for every exceptional divisor $T$ of some (equivalently, any) log resolution $\phi$ we have $\cf(T)<1$ (respectively, $\cf(T)\leq 1$ or $\cf(T)\leq 0$). It is \emph{divisorially log terminal} (\emph{dlt}) if $\cf(T)<1$ for every exceptional divisor $T$ such that $(\bar{X},\bar{D})$ is not log smooth near $\phi(T)$. A normal surface $\bar{X}$ is log terminal (or (log) canonical) if so is $(\bar{X},0)$. 

Assume that $\bar{D}$ is reduced. Then the log surface $(\bar{X},\bar{D})$ is dlt if and only if the exceptional divisor of a minimal log resolution $(X,D)\to (\bar{X},\bar{D})$ is a sum of some admissible twigs of $D$ and connected components of $D$ which are admissible chains or forks \cite[3.31, 3.40]{Kollar_singularities_of_MMP}, cf.\ \cite[Lemma 4.5]{Palka_almost_MMP}. A surface singularity is canonical if and only if it is du Val, i.e.\ its minimal resolution is a $(-2)$-chain or fork \cite[3.26]{Kollar_singularities_of_MMP}.

\begin{lemma}[{Formulas for coefficients, see \cite[II.3.3]{Miyan-OpenSurf} or \cite[\sec 3.2]{Flips_and_abundance}}]\label{lem:ld_formulas}\ 
	\begin{enumerate}
		\item\label{item:ld_chain} Let $T$ be an ordered admissible chain. Put 
		\begin{equation*}
			T\cp{<j}=\sum_{i=0}^{j-1}T\cp{i},\quad T\cp{>j}=\sum_{i=j+1}^{\#T}T\cp{i}. \qquad \mbox{Then} \quad 
			\cf_{T}(T\cp{j})=1-\frac{d(T\cp{>j})+d(T\cp{<j})}{d(T)}.
		\end{equation*}
		\item\label{item:ld_fork} Let $T=\langle b;T_1,T_2,T_3\rangle$ be an admissible fork, with branching component $B$. Put
		\begin{equation*}
			\delta=\sum_{i=1}^{3}\frac{1}{d(T_{i})},\quad e=\sum_{i=1}^{3}\frac{d(T_i-T_{i}^{-})}{d(T_i)},
		\end{equation*}
		where $T_{i}^{-}\de T_{i}\cp{\#T_i}$ is the component of $T_{i}$ meeting $B$. Then $\delta>1$, $e<2\leq b$, and 
		\begin{equation*}
			\cf_{T}(B)=1-\frac{\delta-1}{b-e},\quad \cf_{T}(T_{i}\cp{j})=1-\frac{(1-\cf_{T}(B))\cdot d(T_i\cp{<j})+d(T_i\cp{>j})}{d(T_i)}.
		\end{equation*}
		\item\label{item:ld_twig} Let $T$ be an admissible twig of a reduced divisor $D\neq T$, and let $(X,D)\to (\bar{X},\bar{D})$ be the contraction of $T$. Then the coefficients with respect to the pair $(\bar{X},\bar{D})$ are given by
		\begin{equation*}
			\cf(T\cp{j})=1-\frac{d(T\cp{>j})}{d(T)}
		\end{equation*}
	\end{enumerate}
\end{lemma}

 Explicit formulas for log discrepancies in Lemma \ref{lem:ld_formulas} imply the following.

\begin{lemma}[{Coefficients decrease in subgraphs {\cite[L.1(2)]{Keel-McKernan_rational_curves}, cf.\ \cite[Lemma 7.6]{Palka_almost_MMP}}}] \label{lem:Alexeev}
	Let $D$ and $D'$ be rational chains or rational forks, such that $D$ is admissible, the graph of $D'$ is a subgraph of the graph of $D$, and if $C'$ is a component of $D'$ corresponding to a component $C$ of $D$ then $C^2\leq (C')^2\leq -2$. Then $D'$ is admissible, too, and we have 
	\begin{equation*}
		\cf_{D}(D)\leq \cf_{D'}(C').
	\end{equation*}
\end{lemma}

We now recall basic consequences of the minimal model program for dlt log surfaces, see  \cite[Theorem 3.47]{KollarMori-bir_geom}. We say that a dlt log surface $(\bar{X},\bar{D})$ is \emph{minimal} if it is an outcome of the birational part of the MMP, that is, there is no divisorial contraction which is negative with respect to $K_{\bar{X}}+\bar{D}$. Every log smooth (or, more generally, dlt) log surface $(X,D)$ admits a morphism onto a minimal dlt log surface $(X_{\min},D_{\min})$ of the same Kodaira dimension, say $\kappa$, such that either $\kappa\geq 0$ and $K_{X_{\min}}+D_{\min}$ is nef; or $\kappa=-\infty$ and there is a morphism $f\colon X_{\min}\to B$ of relative Picard rank one such that $\dim B\leq 1$ and $-(K_{X_{\min}}+D_{\min})$ is $f$-ample. We call $(X_{\min},D_{\min})$ a \emph{minimal model} of $(X,D)$, and refer to the morphism $(X,D)\to (X_{\min},D_{\min})$ as an \emph{MMP run}. 

We say that a dlt log surface is \emph{almost minimal} if it admits an MMP run with exceptional divisor contained in the boundary, see \cite[Lemma 3.16(2)]{Palka_almost_MMP}; cf. \cite[p.\ 107]{Miyan-OpenSurf} and the discussion in \cite[Remark 3.24]{Palka_almost_MMP}. If $(X,D)$ is log smooth, then an MMP run $(X,D)\to (X_{\min},D_{\min})$ factors through the minimal log resolution $(X\am,D\am)$ of $(X_{\min},D_{\min})$, which we call an \emph{almost minimal model} of $(X,D)$. It follows from \cite[Lemma 4.9]{Palka_almost_MMP} that the morphism $(X,D)\to (X\am,D\am)$ is in fact an snc-minimalization of $D+E$, where $E$ is a sum of certain $(-1)$-curves characterized in  loc.\ cit.

\subsection{\texorpdfstring{$\P^1$}{P1}-fibrations}\label{sec:P1-fibrations}

Let $(X,D)$ be a log smooth log surface. A \emph{$\P^{1}$-fibration} of $X$ is a morphism $p\colon X\to B$ onto a curve $B$ whose general fiber $F$ is isomorphic to $\P^1$. We call $F\cdot D$ the \emph{height} of $p$ (with respect to $D$). By Definition \ref{def:height}, the \emph{height} of $(X,D)$ is the minimal height of a $\P^1$-fibration of $X$; a \emph{witnessing} $\P^1$-fibration is any $\P^1$-fibration realizing this minimum. Height of a pair $(\bar{X},\bar{D})$ consisting of a normal surface $\bar{X}$ and a reduced Weil divisor $\bar{D}$ on $\bar{X}$ is defined as the height of its minimal log resolution. For a normal surface $\bar{X}$ we put $\height(\bar{X})=\height(\bar{X},0)$.

Note that by Tsen's theorem, cf.\ \cite[C.4]{Reid-chapters}, every $\P^1$-fibration of a smooth surface is trivial over some dense open subset of the base.

Fix a $\P^1$-fibration $p\colon X\to B$ of a smooth projective surface $X$.  A curve $C\subseteq X$ is \emph{vertical} if $p(C)$ is a point; otherwise $C$ is horizontal, it is an $n$-section if $C\cdot F=n$ for a fiber $F$. Given a reduced divisor $T$ we write $T\vert$ and $T\hor$ for the sum of its vertical and horizontal components, respectively. Every fiber of $p$ is obtained by inductively blowing up over a $0$-curve. This leads to the following description of degenerate fibers. 

\begin{lemma}[Degenerate fibers]\label{lem:degenerate_fibers}
	Let $F$ be a degenerate fiber of a $\P^{1}$-fibration of a smooth projective surface. Then $F\redd$ is a rational tree and its $(-1)$-curves are non-branching. Furthermore, the following hold.
	\begin{enumerate}
		\item\label{item:unique_-1-curve} If some $(-1)$-curve $L$ has multiplicity $1$ in $F$ then $L$ is a tip of $F\redd$  and $F\redd-L$ contains a $(-1)$-curve.
		\item\label{item:adjoint_chain}
		Assume that $F$ contains exactly one $(-1)$-curve $L$. Then $F$ has exactly two components of multiplicity $1$, they are tips of $F\redd$, and one of the following holds. 
		\begin{enumerate}
			\item \label{item:columnar} 
			$F\redd=[T,1,T^{*}]$ for some admissible chain $T$. Here $T^{*}$ is a type of an admissible chain which is \emph{adjoint} to $T$, it is uniquely determined by $T$, see \cite[\sec 3.9]{Fujita-noncomplete_surfaces}.
			\item \label{item:not_columnar} Both components of multiplicity one belong to the same connected component of $F\redd-L$. The other connected component, if exists, is a chain.
		\end{enumerate}
		\item \label{item:F-2}
		Assume that $F$ consists of $(-1)$- and $(-2)$-curves only. Then $F\redd$ is of one of the following types.
		\begin{enumerate}
			\item\label{item:fiber_[1,2_k,1]} $[1,(2)_{k},1]$ for some $k\geq 0$. In this case, $F=F\redd$.
			\item\label{item:fiber_[2,1,2]} $[2,1,2]$ or $\langle 2;[1,(2)_{k}],[2],[2]\rangle$ for some $k\geq 0$. In this case, the $(-2)$-tips of $F\redd$ have multiplicity $1$ in $F$, and all the remaining components of $F$ have multiplicity $2$.
		\end{enumerate}
	\end{enumerate}
\end{lemma}
Contracting all fibers to $0$-curves gives a morphism onto a Hirzebruch surface $\F_{m}=\P(\cO_{\P^{1}}(m)\oplus \cO_{\P^{1}})$, and since $\rho(\F_m)=2$ one gets the following formula proved in \cite[4.16]{Fujita-noncomplete_surfaces}, cf.\ {\cite[Lemma 2]{Palka-AMS_LZ}.

\begin{lemma}[Number of vertical curves off $D$]\label{lem:fibrations-Sigma-chi}
	Let $(X,D)$ be a log smooth log surface with a fixed $\P^1$-fibration. Let $\nu_{\infty}$ be the number of fibers contained in $D$. For a fiber $F$ let $\sigma(F)$ be the number of components of $F\redd$ not contained in $D$. Then 
	\begin{equation*}\#D\hor+\nu_{\infty}+\rho(X)=\#D+2+\textstyle\sum_{F}(\sigma(F)-1).\end{equation*}
	 where the sum runs over all fibers $F$ not contained in $D$.
\end{lemma}

Recall that a normal surface $\bar{X}$ is \emph{del Pezzo} if $-K_{\bar{X}}$ is ample. Lemma \ref{lem:fibrations-Sigma-chi} applied to the minimal resolution of a del Pezzo surface of rank one gives the following.

\begin{lemma}
	\label{lem:delPezzo_fibrations}
	Let $\bar{X}$ be a del Pezzo surface of rank one, let $\pi\colon X\to \bar{X}$ be its minimal resolution, and let $D=\Exc\pi$. Fix a $\P^1$-fibration of $X$. Then the following hold.
	\begin{enumerate}
		\item\label{item:-1_curves} A component of a degenerate fiber is a $(-1)$-curve if and only if it is not contained in $D$.
		\item\label{item:Sigma} For a degenerate fiber $F$ let $\sigma(F)$ be the number of $(-1)$-curves in $F\redd$, i.e.\ the number of components not contained in $D$. Then $\sum_{F}(\sigma(F)-1)=\#D\hor -1$.
	\end{enumerate}
\end{lemma}
\begin{proof}
	\ref{item:-1_curves}   
	Since $\pi$ is the minimal resolution, $D$ contains no $(-1)$-curves. 
	In turn, every component $L$ of a degenerate fiber such that $L\not\subseteq D$ satisfies $L^2<0$ and, since $-K_{\bar{X}}$ is ample, $0>\pi(L)\cdot K_{\bar{X}}\geq L\cdot K_{X}$ because  $\pi^{*}K_{\bar{X}}-K_{X}$ is effective by minimality of $\pi$. Thus $L$ is a $(-1)$-curve, as needed.
	
	\ref{item:Sigma} We have $\rho(X)-\#D=\rho(\bar{X})=1$, and $D$ contains no fibers since it is negative definite. Hence the claim follows from Lemma \ref{lem:fibrations-Sigma-chi}.
\end{proof}

We note that a del Pezzo surface of rank one is rational if and only if all its singularities are. The \enquote{only if} part is elementary, while the \enquote{if} part is proved in \cite[\sec 2]{Cheltsov_non-rational} and \cite[Corollary 1.9]{Fujisawa} for $\kk=\C$, and in general follows from \cite[Theorem 2.2(i)]{Schroer_non-rational}. We recall the argument for the readers' convenience.

\begin{lemma}
	\label{lem:rationality}
	Del Pezzo surface of rank one is rational if and only if it has only rational singularities.
\end{lemma}
\begin{proof}
	Let $\bar{X}$ be a del Pezzo surface of rank one, let $\pi\colon X\to \bar{X}$ be its minimal resolution and let $D=\Exc\pi$. Assume first that $X$ is rational, so $H^{1}(X,\cO_{X})=0$. From the second page of the Leray spectral sequence we get an exact sequence 
		$
		0=H^{1}(X,\cO_{X})\to H^{0}(\bar{X},R^{1}\pi_{*}\cO_{X})\to H^{2}(\bar{X},R^{0}\pi_{*}\cO_{X})
		$. 
	Since $\bar{X}$ is normal, we have $R^{0}\pi_{*}\cO_{X}=\cO_{\bar{X}}$. Moreover, by Serre criterion $\bar{X}$ is Cohen--Macaulay, so by Serre duality, see \cite[Theorem 5.71 and Proposition 5.75]{KollarMori-bir_geom}, we have $H^{2}(\bar{X},\cO_{\bar{X}})\cong H^{0}(\bar{X},\cO_{\bar{X}}(K_{\bar{X}}))^{*}=0$ since $-K_{\bar{X}}$ is ample. We conclude that $H^{0}(\bar{X},R^{1}\pi_{*}\cO_{X})=0$, so the singularities of $\bar{X}$ are rational, as needed.
	
	Conversely, assume that $\bar{X}$ has only rational singularities, so $D$ is a sum of rational trees, see \cite[7.1.2(iv)]{Nemethi_book}. Since $\bar{X}$ is del Pezzo, we have $\kappa(X)=-\infty$, so $X$ admits a $\P^{1}$-fibration over some curve $B$. Since $D$ is negative definite, it contains no fibers, so the equality $\#D=\rho(X)-1$ implies that $D$ has a horizontal component. Since the latter is rational, so is $B$ and therefore $X$, as needed.
\end{proof}

\subsection{$\A^{1}$-(uni)ruled surfaces} 
Recall that a surface is \emph{$\A^{1}$-ruled} if it contains an open subset isomorphic to $\A^{1}\times C$ for some curve $C$; and it is \emph{$\A^{1}$-uniruled} if it admits a dominant morphism from an $\A^1$-ruled one. The following equivalence is known.
\begin{lemma}[{\cite[Lemma II.2.9.1]{Miyan-OpenSurf}}]\label{lem:Tsen}
	A smooth surface is $\A^{1}$-ruled if and only if it admits a log smooth completion of height at most $1$.
\end{lemma}
\begin{proof}
	Let $V$ be a smooth surface and let $(X,D)$ be its log smooth completion. If a $\P^1$-fibration $p\colon X\to B$ has height at most one with respect to $D$ then by Tsen's theorem $p|_{V}$ restricts to a projection $\A^1\times B^{\circ}\to B^{\circ}$ over some dense open subset $B^{\circ}\subseteq B$. Conversely, if $V$ is $\A^1$-ruled then the open embedding $\A^1\times C\into V$ extends to a morphism $\P^1\times C\to X$, and blowing up over $D$ we can assume that this is an open embedding, too. The projection $\P^1\times C\to C$ extends to the required $\P^1$-fibration of height at most one. 
\end{proof}

	It follows from Lemma \ref{lem:Tsen} that every $\A^{1}$-ruled surface is $\A^{1}$-fibered. The converse holds if $\cha\kk=0$, see \cite[Lemma III.1.3.1]{Miyan-OpenSurf}, but fails if $\cha\kk>0$, as shown by the following example.
 
\begin{example}[$\A^{1}$-fibration which is non-trivial over all dense open subsets of the base]
	Assume $p\de \cha\kk>0$. Let $\pi\colon \A^{1}_{*}\times \P^1 \to \A^{1}_{*}$ be the projection, and let $H\de \{(x,[y_0:y_1])\in \A^{1}_{*}\times \P^1: xy_{0}^{p}=y_{1}^{p}\}$, $S\de (\A^{1}_{*}\times \P^1)\setminus H$. Then $\pi|_{S}\colon S\to \A^{1}_{*}$ is an $\A^{1}$-fibration. Since every morphism $\A^{1}\to \A^{1}_{*}$ is constant, every curve on $S$ which is isomorphic to $\A^{1}$ is vertical, in particular $\pi|_{S}$ is the unique $\A^{1}$-fibration of $S$. Since $H\cdot F=p$ for a general fiber $F$ of $\pi$, the proper transform of $H$ is a $p$-section of every birational model of $\A^{1}_{*}\times \P^1$ over $\A^{1}_{*}$. It follows that $\pi|_{S}$ cannot be completed to a $\P^1$-fibration of height $1$, hence is non-trivial over any open dense subset of the base $\A^{1}_{*}$. As a consequence, $S$ is not $\A^{1}$-ruled.
\end{example}

	\begin{lemma}[{\cite[Corollary II.2.11.1]{Miyan-OpenSurf}}]\label{lem:twig} Let $(X,D)$ be a log smooth surface such that $D$ contains a rational twig which is not negative definite. Then $X\setminus D$ is $\A^{1}$-ruled.
	\end{lemma}
	\begin{proof}
		Contracting superfluous $(-1)$-curves in $D$ and its images we can assume that $D$ has a rational twig $T$ whose $k$-th component $C$ has $C^2\geq 0$. Blowing up $C^2$ times over $C\cap (D-C)$, or over any point of $C$ if $D=C$, we can assume $C=[0]$. Assume $k$ is smallest possible. If $k=1$, i.e.\ if $C$ is a tip of $D$, then the linear system $|C|$ induces a $\P^1$-fibration of height 1 with respect to $D$, so the claim follows from Lemma \ref{lem:Tsen}. Otherwise, let $C_{-}$, $C_{+}$ be the $(k-1)$-th and $(k+1)$-th component of $T$ (if $C$ meets $D-T$ we let $C_{+}$ be the component of $D-T$ meeting $C$). Blow up $-C_{-}^{2}$ times at $C\cap C_{+}$, each time on the proper transform of $C_{+}$, and contract the total transform of $C$ minus the exceptional $(-1)$-curve. Then the reduced total transform of $D$ is a chain whose $(k-1)$-th component, the proper transform of $C_{-}$, is a $0$-curve. We conclude by induction on $k$.
	\end{proof}

\subsection{Rational curves of arithmetic genus $0$ and $1$ contained in the smooth locus}

\begin{lemma}[$\P^1$'s and rational elliptic curves in the smooth locus]\label{lem:T-in-smooth-locus}
	Let $\bar{Y}$ be a normal projective surface of rank one, and let $\bar{T}\subseteq \bar{Y}\reg$ be a projective curve contained its smooth locus. The following hold.
	\begin{enumerate}
		\item\label{item:P1} If $\bar{T}\cong \P^1$ then $\bar{Y}$ is isomorphic to a cone over a rational normal curve, which is the image of $\bar{T}$.
		\item\label{item:elliptic} Assume that $\bar{T}$ is a rational curve of arithmetic genus one. Then $\bar{Y}$ is a canonical del Pezzo surface and $\bar{T}\in |-K_{\bar{Y}}|$. Moreover, if $\bar{T}$ is reducible then $\bar{Y}$ is isomorphic to $\P^2$ or to the quadric cone.
	\end{enumerate}
\end{lemma}
\begin{proof}
	We can assume $\bar{Y}\not\cong \P^2$, otherwise the result is clear. Let $\pi\colon Y\to \bar{Y}$ be the minimal resolution, $B=\Exc\pi$ and $T=\pi^{-1}_{*}\bar{T}$. Since $\bar{T}\subseteq \bar{Y}\reg$, $\pi$ is an isomorphism near $T$, in particular $T\cong \bar{T}$.
	
	\ref{item:P1} Since $\bar{T}\subseteq\bar{Y}\reg$ and $\rho(\bar{Y})=1$, the number $m=\bar{T}^2$ is a  positive integer. It is enough to prove that $Y$ is isomorphic to the Hirzebruch surface $\F_m$, see \cite[Example V.2.11.4]{Hartshorne_AG}.
	
	Blow up $m$ times at some point $p\in T$ and its infinitely near points on the proper transforms of $T$. Denote the resulting morphism by $\phi$. Put $B'=\phi^{-1}_{*}B$, $T'=\phi^{-1}_{*}T$ and write $(\phi^{*}T)\redd=T'+H+V$, where $H=[1]$ and $V=[(2)_{m-1}]$. Note that $\phi$ is an isomorphism near $B'$. The linear system $|T'|$ induces a $\P^{1}$-fibration such that $H$ is a $1$-section and $B'+V$ is vertical. If there are no degenerate fibers then $B'+V=0$, so $m=1$, $Y=\P^2$ and $T$ is a line, as needed. Let $F$ be a degenerate fiber. By Lemma \ref{lem:fibrations-Sigma-chi} $F\redd$ has a unique component off $B'+V$, say $L_{F}$. Since $B'+V$ contains no $(-1)$-curves, $L_{F}$ is the unique $(-1)$-curve in $F$, so it has multiplicity at least $2$ in $F$, and thus $L_{F}\cdot H=0$. Since $B'\cdot H=0$,  $L_{F}$ meets $V$ and $B'$. Hence there is only one degenerate fiber, supported on $B'+L_{F}+V=[m,1,(2)_{m-1}]$. Thus $B=[m]$, so $Y=\F_{m}$, as claimed.
	
	\ref{item:elliptic} Since $\rho(\bar{Y})=1$, the condition $p_{a}(\bar{T})=0$ implies that $\bar{T}\in |-K_{\bar{Y}}|$, so  $\bar{Y}$ is del Pezzo. 
	
	Assume that $T$ is reducible. Then $T$ has a component $T_0$ with $p_a(T_0)=0$. Since $Y\not\cong \P^2$,  part \ref{item:P1} shows that $Y$ is a Hirzebruch surface $\F_{m}$ for some $m\geq 2$, and $T_0$ is disjoint from the $(-m)$-curve $\Sec_m$. Since $\bar{T}\subseteq \bar{Y}\reg$, $T$ is disjoint from $\Sec_m$, too, so  $T\in |kT_0|$ for some $k\geq 1$. Now $0=2p_{a}(T)-2=k(m(k-1)-2)$, so $k=m=2$. We conclude that $Y\cong \F_{2}$, so $\bar{Y}$ is a quadric cone, as needed.
	
	Assume now that $T$ is irreducible. Since $\bar{Y}$ is del Pezzo, we have $\kappa(Y)=-\infty$, so $Y$ admits a $\P^1$-fibration onto some curve $C$. Since the curve $T$ is singular, it is horizontal, and since $T$ is rational, so is the base $C$. It follows that $Y$ is a rational surface, so by Riemann--Roch $\dim |K_{Y}+T|\geq p_{a}(T)-1=0$. Thus $|K_{Y}+T|\neq \emptyset$.
	
	Because the resolution $\pi$ is minimal, we have  $\pi^{*}K_{\bar{Y}}=K_{Y}+B^{\#}$ for some effective $\Q$-divisor $B^{\#}$ supported on $B$. Since $T=\pi^{*}\bar{T}=-\pi^{*}K_{\bar{Y}}$, we get  $K_{Y}+T=-B^{\#}$. Thus  $B^{\#}=0$, so $\bar{Y}$ is canonical, as needed.
\end{proof}

\subsection{Representing families and moduli dimension}\label{sec:moduli} 
\noindent
We now recall the definition of \emph{moduli dimension} from \cite[Definitions 1.9, 1.10]{PaPe_ht_2} used in Proposition~\ref{prop:canonical}\ref{item:canonical-uniqueness}. 

Let $f\colon \cX\to B$  be a smooth surjective morphism between smooth connected varieties, and let $\cD$ be a reduced divisor on $\cX$ such that the restriction of $f$ to each component of $\cD$ is a smooth and surjective morphism with irreducible fibers. Let $\operatorname{Fib}(f)$ be the set of isomorphism classes of fibers $(X_{b},D_{b})\de(f^{-1}(b),\cD|_{f^{-1}(b)})$. We say that $f$ is \emph{almost faithful} if $f$ is equivariant with respect to the action of some finite group $G$ such that two fibers $(X_b,D_b)$ and $(X_{b'},D_{b'})$ are isomorphic if and only if $b$ and $b'$ lie in the same orbit. The image of $G$ in $\Aut(B)$ is called the \emph{symmetry group} of $f$. We say that $f$ is \emph{almost universal} if it is almost faithful and furthermore the formal germ of $f$ at each $b\in B$ is a semiuniversal deformation of $(X_b,D_b)$, see \cite[Definition 2.2.6]{Sernesi_deformations}. 

We say that a set $\cC$ consisting of isomorphism classes of log surfaces \emph{has moduli dimension $d$} if there is an almost universal family $f\colon (\cX,\cD)\to B$ with $\dim B=d$ which \emph{represents} $\cC$, i.e.\  has $\cC=\operatorname{Fib}(f)$. In this case $d=h^{1}(\lts{X_b}{D_b})$ for any $b\in B$, cf.\ \cite[Proposition 3.4.17]{Sernesi_deformations}. 
More generally, if $\bar{\cC}$ consists of some isomorphism classes of log surfaces with reduced boundaries, we say that $\bar{\cC}$ has moduli dimension $d$ if so does the set of isomorphism classes of minimal log  resolutions of log surfaces in $\bar{\cC}$. 
\smallskip

In this article, we construct almost faithful families representing the sets $\Pcusp(\cS)$ of isomorphism classes of log surfaces $(\bar{Y},\bar{T})$, where $\bar{Y}$ is a canonical surface of rank $1$ and singularity type $\cS$, and $\bar{T}\subseteq \bar{Y}\reg$ is a cuspidal member of $|-K_{\bar{Y}}|$. If $\#\Pcusp(\cS)=1$ then the base of such a family is a point, otherwise their bases and symmetry groups are listed in Table \ref{table:symmetries} below. The constructions of these families are based on the constructions of surfaces $\bar{Y}$ which have already appeared in the literature, e.g.\ in \cite{KN_Pathologies}.

\begin{table}[htbp]
{\renewcommand{\arraystretch}{1.3}
\begin{tabular}{r||c|c|c|c|c|c}
	$\cha\kk$ & 3 & \multicolumn{5}{c}{$2$} \\ \hline 
	type $\cS$ of $\bar{Y}$ & $4\rA_2$ & $3\rA_1+\rD_4$ & $2\rA_1+\rD_6$ & $7\rA_1$ & $4\rA_1+\rD_4$ & $8\rA_1$  \\
	 \hline 
	 base $B$ & $\P^1\setminus \P^1_{\F_3}$ & \multicolumn{2}{c|}{$\P^1\setminus \P^1_{\F_2}$} & \multicolumn{2}{c|}{$ \P^2\setminus \{\mbox{all } \F_2\mbox{-rational lines}\}$} & $3$-fold \eqref{eq:3fold}  \\ \hline 
	 symmetry group	& $\PGL_{2}(\F_{3})$ 
	& $S_3$  & $\Z/2$ & $\PGL_{3}(\F_2)$ & $S_4$ & $\PGL_{3}(\F_2)$ \\ \hline
	construction & Remark \ref{rem:4A2-alternative} & \multicolumn{2}{c|}{\hyperref[moduli]{Proof of \ref{prop:canonical} in \sec \ref{sec:relatives}}} & Lemma \ref{lem:7A1}\ref{item:7A1_cuspidal} & Lemma \ref{lem:4A1+D4}\ref{item:4A1+D4_cuspidal} & Lemma \ref{lem:8A1}\ref{item:8A1_cuspidal} 
\end{tabular}
}
\caption{Almost universal families representing the sets $\Pcusp(\cS)$ from Proposition \ref{prop:canonical}\ref{item:canonical-uniqueness}.}
\label{table:symmetries}
\end{table}


\section{Improvements of the theorem of Miyanishi--Tsunoda}\label{sec:MT}

We now state our first main result, Proposition \ref{prop:MT_smooth}, and derive its consequences, such as Theorem \ref{thm:MT-version} and Corollary \ref{cor:Russell}. The proof will be given in Section \ref{sec:MT_proof}. We study minimal models of log smooth surfaces of negative Kodaira dimension whose boundary does not contract to log terminal singularities, e.g.\ log smooth completions of open del Pezzo surfaces as in Theorem \ref{thm:MT-version}. We bound their height and precisely describe their geometry.

Recall that a log surface $(\bar{X},\bar{D})$ of negative Kodaira dimension is \emph{minimal} if and only if it is a log Mori fiber space, i.e.\ $-(K_{\bar{X}}+\bar{D})$ is $f$-ample for some morphism $f\colon \bar{X}\to B$ of relative Picard rank one with $\dim B\leq 1$. 

To formulate our result, it will be convenient to distinguish the following types of degenerate fibers.

\begin{definition}[Some particular types of degenerate fibers in case $\height=2$, see Figure \ref{fig:fibers}] \label{def:fibers}
	Let $(X,D)$ be a log smooth surface with a $\P^{1}$-fibration such that $D\hor$ consists of two $1$-sections. We say that 
	\begin{enumerate}
		\item \label{item:F-columnar} $F$ is \emph{columnar} if $F\redd$ is a chain meeting $D\hor$ in both tips, such that $F\redd-L\subseteq D$, where $L$ is the unique $(-1)$-curve in $F\redd$. In this case, $F\redd=[T,1,T^{*}]$ for some admissible chain $T$, see Lemma \ref{lem:degenerate_fibers}\ref{item:columnar}.
		\item \label{item:F-quasi-columnar} $F$ is \emph{quasi-columnar} if either $F=[0]$ and $F$ is a component of $D$; or $F\redd+D\hor$ is a fork with exactly one component off $D$, which is the last tip of a vertical twig of type $[(2)_{s},1]$ for some $s\geq 0$ which does not meet $D\hor$. In this case, $F\redd=[(2)_{s},1,s+1]$ or $\langle s+2;T,[(2)_{s},1],T^{*}\rangle$ for some admissible chain $T$.
	\end{enumerate}
\end{definition}
\begin{figure}[htbp]
	\vspace{-1em}
		\begin{tikzpicture}[scale=0.6]
\begin{scope}
			\draw (-1,0) -- (1,0);
			\draw (0,-0.2) -- (0.2,0.8);
			\node[left] at (0.3,1) {$\vdots$};
			\draw (0.2,1) -- (0,2);
			\draw [decorate, decoration = {calligraphic brace}, thick] (-0.2,-0.2) --  (-0.2,1.9);
			\node[rotate=90] at (-0.55,0.9) {\small{$T^{*}$}}; 
			\draw[dashed] (0,1.8) -- (0.2,3);
			\node[left] at (0.1,2.5) {\small{$-1$}};
			\draw (0.2,2.8) -- (0,3.8);
			\node[right] at (-0.1,4) {$\vdots$};
			\draw (0,4) -- (0.2,5);
			\draw [decorate, decoration = {calligraphic brace}, thick] (-0.2,2.9) --  (-0.2,5);
			\node[rotate=90] at (-0.55,3.9) {\small{$T$}};
			\draw (-1,4.8) -- (1,4.8);
			\node[left] at (-1,2.5) {columnar fiber:};
	\end{scope}
		\begin{scope}[shift={(10,0)}]
		\node[left] at (-1,2.5) {quasi-columnar fibers:};	
			\draw (-1,0) -- (1,0);
			\draw (0,-0.2) -- (0,5);
			\node[left] at (0.1,2.5) {\small{$0$}};
			\draw (-1,4.8) -- (1,4.8);
		\end{scope}
		\begin{scope}[shift={(15,0)}]
			\draw (-1,0) -- (1,0);
			\draw (0,-0.2) -- (0,5);
			\node[left] at (0.1,2.5) {\small{$-s-1$}};
			\draw[dashed] (-0.1,2.1) -- (1.8,1.9);
			\node at (0.8,2.3) {\small{$-1$}};
			\draw (1.6,1.7) -- (1.8,2.7);
			\node[left] at (1.9,2.9) {$\vdots$};
			\node[rotate=90] at (2.15,2.8) {\small{$[(2)_{s}]$}};
			\draw (1.8,2.9) -- (1.6,3.9);  
			\draw (-1,4.8) -- (1,4.8);
		\end{scope}
		\begin{scope}[shift={(20,0)}]
			\draw (-1,0) -- (1,0);
			\draw (0,-0.2) -- (0.2,0.8);
			\node[left] at (0.3,1) {$\vdots$};
			\draw (0.2,1) -- (0,2);
			\draw [decorate, decoration = {calligraphic brace}, thick] (-0.2,-0.2) --  (-0.2,1.9);
			\node[rotate=90] at (-0.55,0.9) {\small{$T^{*}$}}; 
			\draw (0,1.8) -- (0.2,3);
			\node[left] at (0.2,2.5) {\small{$-s-2$}};
			\draw (0.2,2.8) -- (0,3.8);
			\node[right] at (-0.1,4) {$\vdots$};
			\draw (0,4) -- (0.2,5);
			\draw [decorate, decoration = {calligraphic brace}, thick] (-0.2,2.9) --  (-0.2,5);
			\node[rotate=90] at (-0.55,3.9) {\small{$T$}};
			\draw[dashed] (-0.1,2.2) -- (1.8,1.9);
			\node at (0.9,2.3) {\small{$-1$}};
			\draw (1.6,1.7) -- (1.8,2.7);
			\node[left] at (1.9,2.9) {$\vdots$};
			\node[rotate=90] at (2.15,2.8) {\small{$[(2)_{s}]$}};
			\draw (1.8,2.9) -- (1.6,3.9);
			\draw (-1,4.8) -- (1,4.8);
		\end{scope}
		\end{tikzpicture}
\vspace{-0.5em}
	\caption{Degenerate fibers in Definition \ref{def:fibers}.}\vspace{-1em}
	\label{fig:fibers}
\end{figure}
\begin{proposition}[{Description of almost minimal log surfaces, generalizing \cite{Miy_Tsu-opendP}, see Figures \ref{fig:MT-ht=2-1}--\ref{fig:MT_exceptions}}]\label{prop:MT_smooth}
	Let $(\bar{X},\bar{D})$ be a minimal dlt log surface with reduced boundary, such that $\kappa(K_{\bar{X}}+\bar{D})=-\infty$ and $\bar{D}\neq 0$. Let $(X,D)$ be the minimal log resolution of $(\bar{X},\bar{D})$. Then one of the following holds.
	\begin{longlist}
		\item \label{item:MT-P2} $X=\P^2$ and $\deg D\leq 2$.  
		\item \label{item:MT-ht=1} $\height(\bar{X},\bar{D})=1$. 
		\item \label{item:MT-ht=2} $\height(\bar{X},\bar{D})=2$, for some witnessing $\P^{1}$-fibration the horizontal part of $D$ consists of two sections $H_{-},H_{+}$ 
		and one of the following holds, where $F_{1},\dots, F_{\nu}$ are all fibers which are degenerate or contained in $D$.
		\begin{longlist}
			\item \label{item:MT-columnar} All degenerate fibers are columnar, $H_{-}\cdot H_{+}\leq 1$ and  $\nu\leq 2-H_{-}\cdot H_{+}$, see Fig.\ \ref{fig:MT-columnar-nu=0}, \ref{fig:MT-columnar},  \ref{fig:MT-columnar-non-min-1}, \ref{fig:MT-columnar-non-min-2}.
			\item \label{item:MT-q-columnar} $\nu\leq 2$, $H_{-}\cdot H_{+}=0$, $F_{1}$ is quasi-columnar and, if $\nu=2$, $F_{2}$ is columnar, see Fig.\ \ref{fig:MT-q-columnar-nu=0}, \ref{fig:MT-q-columnar}, \ref{fig:MT-q-columnar-non-min-1}, \ref{fig:MT-q-columnar-non-min-2}.
			\item \label{item:MT-3-fibers} $\nu=3$, $H_{-}\cdot H_{+}=0$; $F_{1}=[1,(2)_{m+r-3},1]$ for some $r\geq 2$, $F_{2}$ is columnar, and $F_{3}=[0]$, see Fig.\ \ref{fig:open_dP_non-min-3}. 
			\item \label{item:MT-Platonic} $(X,D)$ is a \emph{Platonic $\A^{1}_{*}$-fiber space} as in \cite{Miy_Tsu-opendP}, that is, $\nu=3$, $H_{-}\cdot H_{+}=0$, each $F_{i}$ is columnar and we have  $(F_{i})\redd=[T_{i},1,T_{i}^{*}]$ for admissible chains $T_{i}$ such that   $\sum_{i=1}^{3}\frac{1}{d(T_i)}>1$, see Fig.\ \ref{fig:MT-Platonic}.
		\end{longlist}
Moreover, $H_{-}=[m]$ for some $m\geq 2$. In cases \ref{item:MT-columnar}, \ref{item:MT-q-columnar},   \ref{item:MT-Platonic} we have  $H_{+}=[\nu'-m-2H_{-}\cdot H_{+}]$, where $\nu'$ is the number of degenerate fibers meeting $H_{+}$, $H_{-}$ in different components. In case  \ref{item:MT-3-fibers} we have $H_{+}=[r]$. 
	\item \label{item:MT-ht=2-cha=2} $\cha\kk=2$, $\height(\bar{X},\bar{D})=2$, and $(X,D)$ is as in Example \ref{ex:cha=2_ht=2}. In particular, $X$ admits a $\P^1$-fibration $p$ such that the following hold, see Figure \ref{fig:MT-ht=2_cha=2} for an example.
	\begin{enumerate}
		\item The horizontal part of $D$ is a $2$-section and meets each fiber in exactly one point.
		\item There is exactly one fiber contained in $D$. It is supported on a chain $[2,1,2]$.
		\item There are at least three degenerate fibers, all but one supported on $[2,1,2]$ or $\langle 2;[2],[2],[1,2,\dots, 2]\rangle$.
		\item For the remaining fiber $F$, one of the following holds: either $F\redd=[2,2,1,3]$, and $F$ meets $D\hor$ in the second component; or $F\redd=T_1+T_2+B$ where $T_1,T_2=[2]$ are twigs, and $B$ is a chain with exactly one $(-1)$-curve such that $B$ meets $D\hor$ in one tip, and $T_1,T_2$ in the other.
	\end{enumerate}
	\item \label{item:MT-exception} $\cha\kk=2$, $\height(\bar{X},\bar{D})=3$, and $(X,D)$ is as in Example \ref{ex:exception}, so  
	one of the following holds:
\begin{longlist}
	\item\label{item:MT-exc-1} $D=\langle \bs{1};[2,2\dec{1},2],[2]\dec{2},[3]\dec{3}\rangle+[2,2\dec{3},\ub{3}\dec{1,2}]+[2]\dec{2}$, see Figure \ref{fig:exception-1},
	\item\label{item:MT-exc-2} $D=\langle \bs{1};[2\dec{1},3],[2]\dec{2},[3]\dec{3}\rangle+[2,3\dec{1},\ub{2}\dec{2},2\dec{3},2]+[2]\dec{2}$, see Figure \ref{fig:exception-2},
\end{longlist}
where the branching $(-1)$-curves are $1$-sections, underlined numbers refer to $2$-sections, and numbers decorated by $\dec{i}$ refer to components meeting the $i$-th vertical $(-1)$-curve.
	\item \label{item:MT-eliptic-relative} The log surface $(X,D)$ is a minimal log resolution of $(\bar{Y},\bar{T})$, where $\bar{Y}$ is a canonical surface of rank one, and $\bar{T}\subseteq \bar{Y}\reg$ is a cuspidal member of $|-K_{\bar{Y}}|$, see Proposition \ref{prop:canonical}. Moreover, one of the following holds.
	\begin{longlist}
		\item\label{item:MT-elipitic-relative-5} $\cha\kk=5$, $\height(\bar{X},\bar{D})=4$ and $\bar{Y}$ is of type $2\rA_4$,
		\item\label{item:MT-elipitic-relative-3} $\cha\kk=3$, $\height(\bar{X},\bar{D})=4$ or $6$ and $\bar{Y}$ is of type $3\rA_2$ or $4\rA_2$, respectively,
		\item\label{item:MT-elipitic-relative-2} $\cha\kk=2$, and either $\height(\bar{X},\bar{D})=3$ and $\bar{Y}$ is of type $\rA_1+\rA_5$, $\rA_7$, $3\rA_1+\rD_4$, $2\rA_1+\rD_6$, $\rA_1+\rD_6$, $\rD_8$, $\rA_1+\rE_7$; or $\height(\bar{X},\bar{D})=4$ and $\bar{Y}$ is of type $7\rA_1$, $4\rA_1+\rD_4$; or $\height(\bar{X},\bar{D})=6$ and $\bar{Y}$ is of type $8\rA_1$. 
	\end{longlist}
	\end{longlist}
\end{proposition}

\begin{figure}[htbp]
	\subcaptionbox{\ref{prop:MT_smooth}\ref{item:MT-columnar}, $\nu=0$
		\label{fig:MT-columnar-nu=0}}[0.6\linewidth]
	{
		\begin{tikzpicture}[scale=0.8]
			\begin{scope}
			\path[use as bounding box] (0,-0.5) rectangle (2.3,1.3);
			\draw (0,0) -- (2.2,0);
			\node at (1,-0.3) {\small{$m$}};
			\draw (0,1) -- (2.2,1);
			\node at (1,1.3) {\small{$-m$}};
		\end{scope}
	\begin{scope}[shift={(4,0)}]
			\path[use as bounding box] (0,-0.5) rectangle (2.3,1.3);
			\draw (0,0) -- (1.5,0) to[out=0,in=180] (2.2,0.6);
			\node at (1,-0.3) {\small{$m+2$}};
			\draw (0,1) -- (1.5,1) to[out=0,in=180] (2.2,0.4);
			\node at (1,1.3) {\small{$-m$}};
	\end{scope}
		\end{tikzpicture}
	}
	\subcaptionbox{\ref{prop:MT_smooth}\ref{item:MT-q-columnar}, $\nu=1$ 
		\label{fig:MT-q-columnar-nu=0}}[0.3\linewidth]
	{
		\begin{tikzpicture}[scale=0.8]
			\path[use as bounding box] (0,-0.5) rectangle (2.3,1.3);
			\draw (0,0) -- (2.2,0);
			\node at (1,-0.3) {\small{$m$}};
			\draw (0,1) -- (2.2,1);
			\node at (1,1.3) {\small{$-m$}};
			\draw (2,-0.2) -- (2,1.2);
			\node at (1.8,0.5) {\small{$0$}};
		\end{tikzpicture}
	}
\smallskip

\subcaptionbox{\ref{prop:MT_smooth}\ref{item:MT-columnar}, $\nu=1$
	\label{fig:MT-columnar}}
[.28\linewidth]
{
	\begin{tikzpicture}[scale=0.8]
		\path[use as bounding box] (-0.7,-0.3) rectangle (2.5,5.2);
		\draw (-0.1,0) -- (1.4,0) to[out=0,in=180] (2.5,2.5);
		\node at (1,-0.3) {\small{$m+1$}};
		\draw (0,-0.2) -- (0.2,0.8);
		\node at (0.1,1) {$\vdots$};
		\draw (0.2,1) -- (0,2);
		\draw [decorate, decoration = {calligraphic brace}, thick] (-0.2,-0.2) --  (-0.2,1.9);
		\node[rotate=90] at (-0.5,0.9) {\small{$T^{*}$}}; 
		\draw[dashed] (0,1.8) -- (0.2,3);
		\node at (-0.2,2.5) {\small{$-1$}};
		\draw (0.2,2.8) -- (0,3.8);
		\node at (0.1,4) {$\vdots$};
		\draw (0,4) -- (0.2,5);
		\draw [decorate, decoration = {calligraphic brace}, thick] (-0.2,2.9) --  (-0.2,5);
		\node[rotate=90] at (-0.5,3.9) {\small{$T$}};
		\draw (-0.1,4.8) -- (1.4,4.8) to[out=0,in=180] (2.5,2.3);
		\node at (1,5) {\small{$-m$}};
	\end{tikzpicture}
}
\subcaptionbox{\ref{prop:MT_smooth}\ref{item:MT-q-columnar}, $\nu=2$ 
	\label{fig:MT-q-columnar}}
[.28\linewidth]
{
	\begin{tikzpicture}[scale=0.8]
		\path[use as bounding box] (-0.7,-0.3) rectangle (2,5.2);
		\draw (-0.1,0) -- (2,0);
		\node at (1,-0.3) {\small{$m-1$}};
		\draw (0,-0.2) -- (0.2,0.8);
		\node at (0.1,1) {$\vdots$};
		\draw (0.2,1) -- (0,2);
		\draw [decorate, decoration = {calligraphic brace}, thick] (-0.2,-0.2) --  (-0.2,1.9);
		\node[rotate=90] at (-0.5,0.9) {\small{$T^{*}$}}; 
		\draw[dashed] (0,1.8) -- (0.2,3);
		\node at (-0.2,2.5) {\small{$-1$}};
		\draw (0.2,2.8) -- (0,3.8);
		\node at (0.1,4) {$\vdots$};
		\draw (0,4) -- (0.2,5);
		\draw [decorate, decoration = {calligraphic brace}, thick] (-0.2,2.9) --  (-0.2,5);
		\node[rotate=90] at (-0.5,3.9) {\small{$T$}};
		\draw (-0.1,4.8) -- (2,4.8);
		\node at (1,5) {\small{$-m$}};
		\draw (1.9,-0.2) -- (1.9,5);
		\node at (1.7,2.5) {\small{$0$}};
	\end{tikzpicture}
}
	\subcaptionbox{\ref{prop:MT_smooth}\ref{item:MT-Platonic}: Platonic $\A^{1}_{*}$-fiber space \cite{Miy_Tsu-opendP} \label{fig:MT-Platonic}}
	[.4\linewidth]
	{
		\begin{tikzpicture}[scale=0.8]
			\path[use as bounding box] (-1.4,-0.3) rectangle (4.5,5.2);
			\draw (-0.7,0) -- (2.8,0);
			\node at (0.3,-0.3) {\small{$m-3$}};
			\draw (-0.6,-0.2) -- (-0.4,0.8);
			\node at (-0.5,1) {$\vdots$};
			\draw (-0.4,1) -- (-0.6,2);
			\draw [decorate, decoration = {calligraphic brace}, thick] (-0.8,-0.2) --  (-0.8,1.9);
			\node[rotate=90] at (-1.15,0.9) {\small{$T_1^{*}$}}; 
			\draw[dashed] (-0.6,1.8) -- (-0.4,3);
			\node at (-0.8,2.5) {\small{$-1$}};
			\draw (-0.4,2.8) -- (-0.6,3.8);
			\node at (-0.5,4) {$\vdots$};
			\draw (-0.6,4) -- (-0.4,5);
			\draw [decorate, decoration = {calligraphic brace}, thick] (-0.8,2.9) --  (-0.8,5);
			\node[rotate=90] at (-1.15,3.9) {\small{$T_{1}$}};
			\draw (1,-0.2) -- (1.2,0.8);
			\node at (1.1,1) {$\vdots$};
			\draw (1.2,1) -- (1,2);
			\draw [decorate, decoration = {calligraphic brace}, thick] (0.8,0.1) --  (0.8,1.9);
			\node[rotate=90] at (0.45,0.9) {\small{$T_2^{*}$}}; 
			\draw[dashed] (1,1.8) -- (1.2,3);
			\node at (0.8,2.5) {\small{$-1$}};
			\draw (1.2,2.8) -- (1,3.8);
			\node at (1.1,4) {$\vdots$};
			\draw (1,4) -- (1.2,5);
			\draw [decorate, decoration = {calligraphic brace}, thick] (0.8,2.9) --  (0.8,4.7);
			\node[rotate=90] at (0.45,3.9) {\small{$T_{2}$}};
			\draw (2.4,-0.2) -- (2.6,0.8);
			\node at (2.5,1) {$\vdots$};
			\draw (2.6,1) -- (2.4,2);
			\draw [decorate, decoration = {calligraphic brace}, thick] (2.8,1.9) -- (2.8,-0.2);
			\node[rotate=90] at (3.15,0.9) {\small{$T_3^{*}$}}; 
			\draw[dashed] (2.4,1.8) -- (2.6,3);
			\node at (2.2,2.5) {\small{$-1$}};
			\draw (2.6,2.8) -- (2.4,3.8);
			\node at (2.5,4) {$\vdots$};
			\draw (2.4,4) -- (2.6,5);
			\draw [decorate, decoration = {calligraphic brace}, thick] (2.8,5) -- (2.8,2.9);
			\node[rotate=90] at (3.15,3.9) {\small{$T_{3}$}};
			\draw (-0.7,4.8) -- (2.8,4.8);
			\node at (0.3,5) {\small{$-m$}};
			\node[right] at (3,2.4) {\small{$\sum_{i}\frac{1}{d(T_i)}>1$}};
		\end{tikzpicture}
	}
\vspace{-2em}
\caption{Proposition \ref{prop:MT_smooth}\ref{item:MT-ht=2}: almost minimal surfaces occurring if $\cha\kk\not\in \{2,3,5\}$, part 1.}
\vspace{1em}
\label{fig:MT-ht=2-1}
\end{figure}

\begin{figure}[htbp]
		\subcaptionbox{\ref{prop:MT_smooth}\ref{item:MT-columnar}, $\nu=1$ \label{fig:MT-columnar-non-min-1}}
[.25\linewidth]
{
	\begin{tikzpicture}[scale=0.8]
		\path[use as bounding box] (-0.6,-0.3) rectangle (1.8,5);
		\draw (-0.1,0) -- (1.6,0);
		\node at (0.75,-0.3) {\small{$m-1$}};
		\draw (0,-0.2) -- (0.2,0.8);
		\node at (0.1,1) {$\vdots$};
		\draw (0.2,1) -- (0,2);
		\draw [decorate, decoration = {calligraphic brace}, thick] (-0.2,-0.2) --  (-0.2,1.9);
		\node[rotate=90] at (-0.55,0.9) {\small{$T^{*}$}}; 
		\draw[dashed] (0,1.8) -- (0.2,3);
		\node at (-0.2,2.5) {\small{$-1$}};
		\draw (0.2,2.8) -- (0,3.8);
		\node at (0.1,4) {$\vdots$};
		\draw (0,4) -- (0.2,5);
		\draw [decorate, decoration = {calligraphic brace}, thick] (-0.2,2.9) --  (-0.2,5);
		\node[rotate=90] at (-0.55,3.9) {\small{$T$}};
		\draw (-0.1,4.8) -- (1.7,4.8);
		\node at (0.75,5) {\small{$-m$}};
	\end{tikzpicture}
}	
	\subcaptionbox{\ref{prop:MT_smooth}\ref{item:MT-columnar}, $\nu=2$ \label{fig:MT-columnar-non-min-2}}
	[.25\linewidth]
	{
		\begin{tikzpicture}[scale=0.8]
			\path[use as bounding box] (-0.8,-0.3) rectangle (2.4,5);
			\draw (-0.1,0) -- (1.6,0);
			\node at (0.75,-0.3) {\small{$m-2$}};
			\draw (0,-0.2) -- (0.2,0.8);
			\node at (0.1,1) {$\vdots$};
			\draw (0.2,1) -- (0,2);
			\draw [decorate, decoration = {calligraphic brace}, thick] (-0.2,-0.2) --  (-0.2,1.9);
			\node[rotate=90] at (-0.55,0.9) {\small{$T_1^{*}$}}; 
			\draw[dashed] (0,1.8) -- (0.2,3);
			\node at (-0.2,2.5) {\small{$-1$}};
			\draw (0.2,2.8) -- (0,3.8);
			\node at (0.1,4) {$\vdots$};
			\draw (0,4) -- (0.2,5);
			\draw [decorate, decoration = {calligraphic brace}, thick] (-0.2,2.9) --  (-0.2,5);
			\node[rotate=90] at (-0.55,3.9) {\small{$T_{1}$}};
			\draw (1.4,-0.2) -- (1.6,0.8);
			\node at (1.5,1) {$\vdots$};
			\draw (1.6,1) -- (1.4,2);
			\draw [decorate, decoration = {calligraphic brace}, thick] (1.8,1.9) -- (1.8,-0.2);
			\node[rotate=90] at (2.15,0.9) {\small{$T_2^{*}$}}; 
			\draw[dashed] (1.4,1.8) -- (1.6,3);
			\node at (1.2,2.5) {\small{$-1$}};  
			\draw (1.6,2.8) -- (1.4,3.8);
			\node at (1.5,4) {$\vdots$};
			\draw (1.4,4) -- (1.6,5);
			\draw [decorate, decoration = {calligraphic brace}, thick] (1.8,5) -- (1.8,2.9);
			\node[rotate=90] at (2.15,3.9) {\small{$T_{2}$}};
			\draw (-0.1,4.8) -- (1.7,4.8);
			\node at (0.75,5) {\small{$-m$}};
		\end{tikzpicture}
	}
		\subcaptionbox{\ref{prop:MT_smooth}\ref{item:MT-q-columnar}, $\nu=1$ \label{fig:MT-q-columnar-non-min-1}}
[.4\linewidth]
{
	\begin{tikzpicture}[scale=0.8]
	\begin{scope}
		\path[use as bounding box] (-0.6,-0.3) rectangle (2,5);
		\draw (-0.1,0) -- (1.6,0);
		\node at (0.75,-0.3) {\small{$m$}};
		\draw (0,-0.2) -- (0,5);
		\node at (-0.4,2.5) {\small{$-r$}};
		\draw[dashed] (-0.1,2) -- (1.4,2);
		\node at (0.6,2.2) {\small{$-1$}};
		\draw (1.2,1.8) -- (1.4,2.8);
		\node at (1.3,3) {$\vdots$};
		\node[rotate=90] at (1.7,2.9) {\small{$[(2)_{r-1}]$}};
		\draw (1.4,3) -- (1.2,4);  
		\draw (-0.1,4.8) -- (1.7,4.8);
		\node at (0.75,5) {\small{$-m$}};
	\end{scope}
	\begin{scope}[shift={(5,0)}]
		\path[use as bounding box] (-0.8,-0.3) rectangle (1.8,5);
		\draw (-0.1,0) -- (1.6,0);
		\node at (0.75,-0.3) {\small{$m-1$}};
		\draw (0,-0.2) -- (0.2,0.8);
		\node at (0.1,1) {$\vdots$};
		\draw (0.2,1) -- (0,2);
		\draw [decorate, decoration = {calligraphic brace}, thick] (-0.2,-0.2) --  (-0.2,1.9);
		\node[rotate=90] at (-0.55,0.9) {\small{$T^{*}$}}; 
		\draw (0,1.8) -- (0.2,3);
		\node at (-0.2,2.5) {\small{$-r$}};
		\draw (0.2,2.8) -- (0,3.8);
		\node at (0.1,4) {$\vdots$};
		\draw (0,4) -- (0.2,5);
		\draw [decorate, decoration = {calligraphic brace}, thick] (-0.2,2.9) --  (-0.2,5);
		\node[rotate=90] at (-0.55,3.9) {\small{$T$}};
		\draw[dashed] (-0.1,2.2) -- (1.6,1.9);
		\node at (0.9,2.3) {\small{$-1$}};
		\draw (1.4,1.8) -- (1.6,2.8);
		\node at (1.5,3) {$\vdots$};
		\node[rotate=90] at (1.9,2.9) {\small{$[(2)_{r-2}]$}};
		\draw (1.6,3) -- (1.4,4);  
		\draw (-0.1,4.8) -- (1.7,4.8);
		\node at (0.75,5) {\small{$-m$}};
	\end{scope}
	\end{tikzpicture}
}

\subcaptionbox{\ref{prop:MT_smooth}\ref{item:MT-q-columnar}, $\nu=2$ \label{fig:MT-q-columnar-non-min-2}}
[.6\linewidth]
{
	\begin{tikzpicture}[scale=0.8]
		\begin{scope}
		\path[use as bounding box] (-0.7,-0.3) rectangle (3.3,5.2);
		\draw (-0.1,0) -- (1.6,0);
		\node at (0.75,-0.3) {\small{$m-1$}};
		\draw (0,-0.2) -- (0.2,0.8);
		\node at (0.1,1) {$\vdots$};
		\draw (0.2,1) -- (0,2);
		\draw [decorate, decoration = {calligraphic brace}, thick] (-0.2,-0.2) --  (-0.2,1.9);
		\node[rotate=90] at (-0.55,0.9) {\small{$T_1^{*}$}}; 
		\draw[dashed] (0,1.8) -- (0.2,3);
		\node at (-0.2,2.5) {\small{$-1$}};
		\draw (0.2,2.8) -- (0,3.8);
		\node at (0.1,4) {$\vdots$};
		\draw (0,4) -- (0.2,5);
		\draw [decorate, decoration = {calligraphic brace}, thick] (-0.2,2.9) --  (-0.2,5);
		\node[rotate=90] at (-0.55,3.9) {\small{$T_{1}$}};
		\draw (1.4,-0.2) -- (1.4,5);
		\node at (1,2.5) {\small{$-r$}};
		\draw[dashed] (1.3,2) -- (2.8,2);
		\node at (2,2.2) {\small{$-1$}};
		\draw (2.6,1.8) -- (2.8,2.8);
		\node at (2.7,3) {$\vdots$};
		\node[rotate=90] at (3.1,2.9) {\small{$[(2)_{r-1}]$}};
		\draw (2.8,3) -- (2.6,4);  
		\draw (-0.1,4.8) -- (1.7,4.8);
		\node at (0.75,5) {\small{$-m$}};
		\end{scope}
			\begin{scope}[shift={(6,0)}]
			\path[use as bounding box] (-0.8,-0.3) rectangle (3.5,5.2);
			\draw (-0.1,0) -- (1.6,0);
			\node at (0.75,-0.3) {\small{$m-2$}};
			\draw (0,-0.2) -- (0.2,0.8);
			\node at (0.1,1) {$\vdots$};
			\draw (0.2,1) -- (0,2);
			\draw [decorate, decoration = {calligraphic brace}, thick] (-0.2,-0.2) --  (-0.2,1.9);
			\node[rotate=90] at (-0.55,0.9) {\small{$T_1^{*}$}}; 
			\draw[dashed] (0,1.8) -- (0.2,3);
			\node at (-0.2,2.5) {\small{$-1$}};
			\draw (0.2,2.8) -- (0,3.8);
			\node at (0.1,4) {$\vdots$};
			\draw (0,4) -- (0.2,5);
			\draw [decorate, decoration = {calligraphic brace}, thick] (-0.2,2.9) --  (-0.2,5);
			\node[rotate=90] at (-0.55,3.9) {\small{$T_{1}$}};
			\draw (1.4,-0.2) -- (1.6,0.8);
			\node at (1.5,1) {$\vdots$};
			\draw (1.6,1) -- (1.4,2);
			\draw [decorate, decoration = {calligraphic brace}, thick] (1.8,1.9) -- (1.8,-0.2);
			\node[rotate=90] at (2.15,0.9) {\small{$T_2^{*}$}}; 
			\draw (1.4,1.8) -- (1.6,3);
			\node at (1.2,2.5) {\small{$-r$}};
			\draw[dashed] (1.3,2.2) -- (3,1.9);
			\node at (2.3,2.3) {\small{$-1$}};
			\draw (2.8,1.8) -- (3,2.8);
			\node at (2.9,3) {$\vdots$};
			\node[rotate=90] at (3.3,2.9) {\small{$[(2)_{r-2}]$}};
			\draw (3,3) -- (2.8,4);  
			\draw (1.6,2.8) -- (1.4,3.8);
			\node at (1.5,4) {$\vdots$};
			\draw (1.4,4) -- (1.6,5);
			\draw [decorate, decoration = {calligraphic brace}, thick] (1.8,5) -- (1.8,2.9);
			\node[rotate=90] at (2.15,3.9) {\small{$T_{2}$}};
			\draw (-0.1,4.8) -- (1.7,4.8);
			\node at (0.75,5) {\small{$-m$}};
		\end{scope}
		\end{tikzpicture}
	}
	\subcaptionbox{\ref{prop:MT_smooth}\ref{item:MT-3-fibers} \label{fig:open_dP_non-min-3}}
[.35\linewidth]
{
	\begin{tikzpicture}[scale=0.8]
		\path[use as bounding box] (-1.2,-0.3) rectangle (3.5,5.2);
		\draw (-0.5,0) -- (3,0);
		\node at (2,-0.3) {\small{$-r$}};
		\draw[dashed] (-0.3,-0.2) -- (0,1.3);
		\node at (-0.6,0.5) {\small{$-1$}};
		\draw (0,1.1) -- (-0.2,2.3);
		\node at (-0.1,2.45) {$\vdots$};
		\node[rotate=90] at (-0.5,2.3) {\small{$[(2)_{r+m-3}]$}};
		\draw (-0.2,2.5) -- (0,3.7);
		\draw[dashed] (0,3.5) -- (-0.3,5);
		\node at (-0.5,4.2) {\small{$-1$}};
		\draw (1.2,-0.2) -- (1.4,0.8);
		\node at (1.3,1) {$\vdots$};
		\draw (1.4,1) -- (1.2,2);
		\draw [decorate, decoration = {calligraphic brace}, thick] (1,0.1) --  (1,1.9);
		\node[rotate=90] at (0.65,0.9) {\small{$T^{*}$}}; 
		\draw[dashed] (1.2,1.8) -- (1.4,3);
		\node at (1,2.5) {\small{$-1$}};
		\draw (1.4,2.8) -- (1.2,3.8);
		\node at (1.3,4) {$\vdots$};
		\draw (1.2,4) -- (1.4,5);
		\draw [decorate, decoration = {calligraphic brace}, thick] (1,2.9) --  (1,4.7);
		\node[rotate=90] at (0.65,3.9) {\small{$T$}};
		\draw (2.8,-0.2) -- (2.8,5);
		\node at (2.6,2.5) {\small{$0$}};
		\draw (-0.5,4.8) -- (3,4.8);
		\node at (2,5) {\small{$-m$}};
	\end{tikzpicture}
}
\vspace{-0.5em}
\caption{Proposition \ref{prop:MT_smooth}\ref{item:MT-ht=2}: almost minimal surfaces occurring if $\cha\kk\not\in \{2,3,5\}$, part 2: log surfaces admitting morphisms onto ones in Figure \ref{fig:MT-ht=2-1}, see Remark \ref{rem:adjust}.}
\vspace{1em}
\label{fig:MT-ht=2-2}
\end{figure}
\begin{figure}[htbp]
	\subcaptionbox{\ref{prop:MT_smooth}\ref{item:MT-ht=2-cha=2}, $\height=2$
		\label{fig:MT-ht=2_cha=2}}
	[.3\linewidth]
	{
		\begin{tikzpicture}[scale=0.8]
			\path[use as bounding box] (-0.8,0.6) rectangle (3.8,5);
			\node at (1.5,4.7) {\small{admits a morphism onto}};
			\draw (0.2,0.4) -- (0,1.8);
			\node[left] at (0.15,0.8) {\small{$-2$}};
			\draw (0,1.6) -- (0.2,3.2);
			\node[left] at (0.15,2.5) {\small{$-1$}};
			\draw (0.2,3) -- (0,4.4);
			\node[left] at (0.15,3.7) {\small{$-2$}};
			\draw (1.4,0.4) -- (1.2,1.8);
			\node[left] at (1.35,0.8) {\small{$-2$}};
			\draw[dashed] (1.2,1.6) -- (1.4,3.2);
			\node[left] at (1.35,2.5) {\small{$-1$}};
			\draw (1.4,3) -- (1.2,4.4);
			\node[left] at (1.35,3.7) {\small{$-2$}};
			\draw (3.4,0.4) -- (3.2,1.8); 
			\node[left] at (3.35,0.8) {\small{$-2$}};
			\draw[dashed] (3.2,1.6) -- (3.4,3.2);
			\node[left] at (3.35,2.5) {\small{$-1$}};
			\draw (3.4,3) -- (3.2,4.4);
			\node[left] at (3.35,3.7) {\small{$-2$}};
			\draw[very thick] (-0.1,2.2) -- (3.6,2.2);
			\node at (2.2,1.9) {\small{$4-2\nu$}};
			\node at (2.3,3) {\LARGE{$\dots$}};
			%
		\end{tikzpicture}
	}
\subcaptionbox{\ref{prop:MT_smooth}\ref{item:MT-exc-1}, $\height=3$
	\label{fig:exception-1}}[.3\linewidth]
{
	\begin{tikzpicture}[scale=0.8]
		\path[use as bounding box] (-1.4,0.4) rectangle (3.2,4.6);
		\draw (-1,4.2) -- (3,4.2);
		\draw (-0.6,0.4) -- (-0.8,1.8);
		\node[left] at (-0.65,0.8) {\small{$-2$}};
		\draw (-0.8,1.6) -- (-0.6,3.2);
		\node[left] at (-0.65,2.5) {\small{$-2$}};
		\draw (-0.6,3) -- (-0.8,4.4);
		\node[left] at (-0.65,3.7) {\small{$-2$}};
		\draw[dashed] (-0.8,2.5) -- (0,2.5) to[out=0,in=90] (0.4,2); 
		\node at (0,2.7) {\small{$-1$}};
		\draw (1.4,0.4) -- (1.2,1.8);
		\node[left] at (1.35,0.8) {\small{$-2$}};
		\draw[dashed] (1.2,1.6) -- (1.4,3.2);
		\node[left] at (1.35,2.5) {\small{$-1$}};
		\draw (1.4,3) -- (1.2,4.4);
		\node[left] at (1.35,3.7) {\small{$-2$}};
		\draw (2.8,0.2) -- (3,1.4);
		\node[left] at (3,0.8) {\small{$-2$}};
		\draw (3,1.2) -- (2.8,2.4);
		\node[left] at (3,1.6) {\small{$-2$}}; 
		\draw[dashed] (2.8,2.2) -- (3,3.4);
		\node[left] at (3,2.8) {\small{$-1$}};
		\draw (3,3.2) -- (2.8,4.4);
		\node[left] at (3,3.7) {\small{$-3$}};
		\draw[very thick] (0.2,2.2) -- (2,2.2) to[out=0,in=180] (2.7,1.9) -- (3,1.9);
		\node at (2,2.4) {\small{$-3$}};
		\node at (0,4.4) {\small{$-1$}};
		%
	\end{tikzpicture}
}
	\subcaptionbox{\ref{prop:MT_smooth}\ref{item:MT-exc-2}, $\height=3$
	\label{fig:exception-2}}[.3\linewidth]
{
	\begin{tikzpicture}[scale=0.8]
		\path[use as bounding box] (-0.8,0.4) rectangle (3.2,4.6);
		\draw (-0.4,4.2) -- (3,4.2);
		\draw (0,0.2) -- (-0.2,1.2);
		\node[left] at (-0.05,0.6) {\small{$-2$}};
		\draw (-0.2,1) -- (0,2);
		\node[left] at (-0.05,1.4) {\small{$-3$}};
		\draw[dashed] (0,1.8) -- (-0.2,2.8);
		\node[left] at (-0.05,2.2) {\small{$-1$}};
		\draw (-0.2,2.6) -- (0,3.6);
		\node[left] at (-0.05,3) {\small{$-2$}};
		\draw (0,3.4) -- (-0.2,4.4);
		\node[left] at (-0.05,3.8) {\small{$-3$}};
		\draw (1.5,0.4) -- (1.2,1.8);
		\node[left] at (1.4,0.8) {\small{$-2$}};
		\draw[dashed] (1.2,1.6) -- (1.5,3.2);
		\node[left] at (1.4,2.5) {\small{$-1$}};
		\draw (1.5,3) -- (1.2,4.4);
		\node[left] at (1.4,3.7) {\small{$-2$}};
		\draw (2.8,0.2) -- (3,1.4);
		\node[left] at (3,0.8) {\small{$-2$}};
		\draw (3,1.2) -- (2.8,2.4);
		\node[left] at (3,1.6) {\small{$-2$}}; 
		\draw[dashed] (2.8,2.2) -- (3,3.4);
		\node[left] at (3,2.8) {\small{$-1$}};
		\draw (3,3.2) -- (2.8,4.4);
		\node[left] at (3,3.7) {\small{$-3$}};
		\draw[very thick] (-0.2,1.6) -- (0.2,1.6) to[out=0,in=180] (1,2.1) -- (2,2.1) to[out=0,in=180] (2.7,1.9) -- (3,1.9);
		\node at (2,2.3) {\small{$-2$}};
		\node at (0.5,4.4) {\small{$-1$}};
		%
	\end{tikzpicture}
}
%
\vspace{-0.5em}
	\caption{Exceptional cases in Proposition \ref{prop:MT_smooth}\ref{item:MT-ht=2-cha=2}, \ref{item:MT-exception}, occurring only if $\cha\kk=2$.}
	\label{fig:MT_exceptions}
\end{figure}

\begin{remark}[Adjusting minimal models in case \ref{prop:MT_smooth}\ref{item:MT-ht=2}: Fig.\ \ref{fig:MT-ht=2-2} $\rightarrow$ Fig.\  \ref{fig:MT-ht=2-1}]\label{rem:adjust}
	Given a log smooth surface of negative Kodaira dimension, its minimal model $(\bar{X},\bar{D})$ is as in Proposition \ref{prop:MT_smooth}. In case \ref{prop:MT_smooth}\ref{item:MT-ht=2}, we can choose $(\bar{X},\bar{D})$ so that its minimal log resolution $(X,D)$ is as in Figure \ref{fig:MT-ht=2-1}, that is, either as in   \ref{prop:MT_smooth}\ref{item:MT-columnar}--\ref{item:MT-q-columnar} with at most one degenerate fiber, which is columnar, or a Platonic $\A^{1}_{*}$-fiber space as in \ref{prop:MT_smooth}\ref{item:MT-Platonic}.
	
	Indeed, suppose $(\bar{X},\bar{D})$ is not of the above type, so $(X,D)$ is as in Figure \ref{fig:MT-ht=2-2}. Let 
	$(X,D)\to (X',D')$ be a contraction of $F_{1}$ to a $0$-curve $F'\subseteq X'$. Then $(X',D')$ is a minimal log resolution of another minimal model. We have $F'\subseteq D'$ in cases \ref{item:MT-columnar}, \ref{item:MT-q-columnar} and $F'\not\subseteq D$ in case \ref{item:MT-3-fibers}, so $(X',D')$ is as in Figure  \ref{fig:MT-ht=2-2}, as needed.
\end{remark}

	\begin{remark}[$\A^1$-ruled cases of Proposition \ref{prop:MT_smooth}]\label{rem:A1-fibration}
		Let $(\bar{X},\bar{D})$ be an open log del Pezzo surface as in Proposition \ref{prop:MT_smooth}\ref{item:MT-P2}, \ref{item:MT-ht=1}, or  \ref{item:MT-columnar}--\ref{item:MT-3-fibers}. Then $\bar{X}\setminus \bar{D}$ is $\A^{1}$-ruled. 
		Indeed, in case \ref{prop:MT_smooth}\ref{item:MT-ht=1} we have $\height(\bar{X},\bar{D})=1$, so $\bar{X}\setminus \bar{D}$ is $\A^1$-ruled by Lemma \ref{lem:Tsen}, otherwise $D$ contains a twig which is not negative definite, see Figure \ref{fig:MT-ht=2-1}, so $\bar{X}\setminus \bar{D}$ is $\A^1$-ruled by Lemma \ref{lem:twig}.
	\end{remark}

\begin{remark}[Deep base points of $\A^{1}$-fibrations in case $\height=2$] \label{rem:hidden_A1_fibrations}
	In cases \ref{prop:MT_smooth}\ref{item:MT-columnar}--\ref{item:MT-3-fibers}, the $\A^{1}$-fibration of $X\setminus D$  has a base point on $D$. In fact, for any $m\geq 2$ one easily constructs examples where any $\A^{1}$-fibration of $X\setminus D$ requires at least $m$ blowups on $D$ to have its base points resolved. To see one such example, take $X=\F_{m}$ and $D=\Sec_m+F+C=[m,0,-m]$, where $\Sec_m=[m]$ is the negative section, $F$ is a fiber, and $C=[-m]$ is a section disjoint from $\Sec_m$, see Figure \ref{fig:MT-q-columnar-nu=0}. Write $\{q\}=F\cap C$. Then $\height(X,D)=2$ and $X\setminus D$ is $\A^{1}$-fibered. Indeed, blow up $m$ times at $q$, each time on the proper transform of $C$. The latter becomes a $0$-curve whose linear system gives an $\A^{1}$-fibration of $X\setminus D$.
	
	Let now $V\subseteq X$ be a closure of a smooth fiber of some $\A^{1}$-fibration of $X\setminus D$. Then $V$ is a curve meeting $D$ in one point. Put $d=V\cdot F$. We have $V\cdot C=V\cdot (mF+\Sec_{m})>V\cdot \Sec_m$, so since $C$ and $\Sec_m$ are disjoint, we have $V\cdot \Sec_{m}=0$ and $V\cdot C=dm>d$. In particular, $V\cap D=\{q\}$ and $V$ has a cusp of multiplicity $d$ at $q$. Blow up at $q$, let $F'$ be the exceptional curve, and let $V',C'$ be the proper transforms of $V,C$. Then $V'\cdot F'=d$ and $V'\cdot C'=d(m-1)$; in particular, $V'$ still has a cusp of multiplicity $d$. Repeating this argument with $(V',F',C')$ instead of $(V,F,C)$ we see that it takes at least $m$ blowups to resolve the cusp of $V$, as claimed.
\end{remark}

\begin{corollary}[Proposition \ref{prop:MT_smooth} without minimality assumption]\label{cor:MT}
	Let $(X,D)$ be a log smooth surface such that $\kappa(K_{X}+D)=-\infty$. Assume that $(X,D)$ admits a minimal model with nonzero boundary (for instance $D$ is not negative definite). Then one of the following holds.
	\begin{enumerate}
		\item\label{item:cor-P2-id} We have $X=\P^2$ and $\deg D\leq 2$.
		\item\label{item:cor-ht} We have  $\height(X,D)\leq 2$.
		\item\label{item:cor-P2-map} We have $\height(X,D)=3$ and there is a birational morphism $\psi\colon (X,D)\to (\P^2,\pp)$ such that $\deg\pp=2$, $\psi$ is an isomorphism near $\psi^{-1}_{*}\pp$ and it contracts some connected component of $D$. 
		\item\label{item:cor-exception} We have  $\cha\kk=2$,  $\height(X,D)=3$, and $(X,D)$ admits a morphism onto a log surface as in Example \ref{ex:exception}.
		\item\label{item:cor-eliptic-relative} We have  $\cha\kk\in \{2,3,5\}$, $\height(X,D)\leq 6$ and $(X,D)$ admits a morphism onto a log surface $(\bar{Y},\bar{T})$ with $\bar{Y}$ canonical of rank one, and $\bar{T}\in |-K_{\bar{Y}}|$ cuspidal, contained in $\bar{Y}\reg$, see Proposition \ref{prop:MT_smooth}\ref{item:MT-eliptic-relative}.
	\end{enumerate}
\end{corollary}
\begin{proof}
	Let $\bar{\psi}\colon (X,D)\to(\bar{X},\bar{D})$ be a morphism onto a minimal model of $(X,D)$ with $\bar{D}\neq 0$. We have $\kappa(K_{\bar{X}}+\bar{D})=-\infty$, so $(\bar{X},\bar{D})$ is as in Proposition \ref{prop:MT_smooth}. Since $(X,D)$ is log smooth, the morphism $\bar{\psi}$ factors through a morphism $\psi\colon (X,D)\to (X\am,D\am)$ onto the minimal log resolution $(X\am,D\am)$ of $(\bar{X},\bar{D})$.
	
	Pulling back $\P^{1}$-fibrations from $X\am$ we get $\height(X,D)\leq \height(X\am,D\am)$, so in cases \ref{prop:MT_smooth}\ref{item:MT-ht=1}--\ref{prop:MT_smooth}\ref{item:MT-ht=2-cha=2} we get that \ref{item:cor-ht} holds.  Exceptional cases \ref{prop:MT_smooth}\ref{item:MT-exception} and  \ref{prop:MT_smooth}\ref{item:MT-eliptic-relative} yield \ref{item:cor-exception} and \ref{item:cor-eliptic-relative}, respectively. Consider case \ref{prop:MT_smooth}\ref{item:MT-P2}. If $\psi=\id$ then \ref{item:cor-P2-id} holds, so assume $\psi\neq \id$, let $p$ be a base point of $\psi^{-1}$ and let $L_{p}\subseteq X$ be the first exceptional curve over $p$. A pencil of lines through $p$ pulls back to a $\P^{1}$-fibration of $(X,D)$ of height at most $\deg D\am$ if $p\in D\am$ or $L_{p}\not\subseteq D$; and $\deg D\am+1$ otherwise. Since $\deg D\am\leq 2$ we get either $\height(X,D)\leq 2$, so \ref{item:cor-ht} holds; or $\height(X,D)=3$, $\deg D\am=2$, $p\not\in D\am$ and $L_{p}\subseteq D$, in which case $L_{p}$ is contained in a connected component of $D$ contracted by $\psi$, so \ref{item:cor-P2-map} holds.
\end{proof}

\begin{proof}[Proof of Theorem \ref{thm:MT-version}] 
	Let $\pi\colon (X,D)\to (\bar{X},\bar{D})$ be the minimal log resolution. Since $\rho(\bar{X})=1$, we have $\bar{D}^2>0$, so $D$ is not negative definite. Thus $(X,D)$ is as in Corollary \ref{cor:MT}. Since by assumption $\bar{X}\not\cong \P^{2}$, case \ref{cor:MT}\ref{item:cor-P2-id} does not hold. Case \ref{cor:MT}\ref{item:cor-P2-map} does not hold, either, since by the minimality of $\pi$ no connected component of $\bar{D}$ contracts to a smooth point.
	Case \ref{cor:MT}\ref{item:cor-ht} asserts that $\height(X,D)\leq 2$, as needed.
	
	In the remaining cases \ref{cor:MT}\ref{item:cor-exception},\ref{item:cor-eliptic-relative} we have a  birational morphism $\bar{\psi}\colon (X,D)\to (\bar{Y},\bar{T})$, where $(\bar{Y},\bar{T})$ is a log del Pezzo surface of rank one from Example \ref{ex:exception}; or a log surface as in \ref{cor:MT}\ref{item:cor-eliptic-relative}, consisting of a canonical surface of rank one and an irreducible member of $|-K_{\bar{Y}}|$. Since $(X,D)$ is log smooth, $\bar{\psi}$ factors through a morphism $\psi\colon (X,D)\to (Y,T)$ onto a minimal log resolution of $(\bar{Y},\bar{T})$. We have $(\rho(X)-\#D)-(\rho(Y)-\#T)=1-1=0$, so $\#D-\#T=\rho(X)-\rho(Y)$ and therefore $\Exc\psi\subseteq D$. Since $D$ is snc-minimal, it follows that $\psi$ is an isomorphism. Thus in case \ref{cor:MT}\ref{item:cor-eliptic-relative} we get that \ref{thm:MT-version}\ref{item:thm-relative} holds. In case \ref{cor:MT}\ref{item:cor-exception}, we check directly that, since $-(K_{\bar{X}}+\bar{D})$ is ample, $\bar{D}$ is the image of the unique $(-1)$-curve in $D$, see Figures \ref{fig:exception-1}, \ref{fig:exception-2}; hence \ref{thm:MT-version}\ref{item:thm-exception}  holds, as claimed.
\end{proof}

\begin{proof}[Proof of Corollary \ref{cor:Russell}]
As in the proof of Corollary \ref{cor:MT} above, we can assume that $(X,D)$ is a minimal log resolution of its minimal model $(\bar{X},\bar{D})$. Assume $\bar{D}\neq 0$, so $(\bar{X},\bar{D})$ is as in Proposition \ref{prop:MT_smooth}. Since $D$ is connected, the only possible cases are \ref{prop:MT_smooth}\ref{item:MT-P2}, \ref{item:MT-ht=1},  \ref{item:MT-columnar} and \ref{item:MT-q-columnar}. Thus \ref{cor:Russell}\ref{item:Rus-ht} holds, and \ref{cor:Russell}\ref{item:Rus-Cylinder} follows from Remark  \ref{rem:A1-fibration}.

Therefore, we can assume $\bar{D}=0$. If $\bar{X}\cong \P^2$ then the assertion is clear, so we can assume $\bar{X}\not\cong \P^2$. Then $\bar{X}$ is a log terminal del Pezzo surface of rank one, with exactly one singularity. The exceptional divisor $D$ of its minimal resolution is an admissible chain or fork. We claim that
\begin{equation}\label{eq:K+D}
	h^{0}(-K_{X}-D)\geq K_{X}\cdot (K_{X}+D)>0.
\end{equation}
We have $h^{2}(-K_{X}-D)=h^{0}(2K_{X}+D)=0$. By Lemma \ref{lem:rationality} $X$ is rational, so the Riemann-Roch theorem gives $h^{0}(-K_{X}-D)\geq p_{a}(D)+K_{X}\cdot(K_{X}+D)$. Since $p_a(D)=0$, this proves the first inequality of \eqref{eq:K+D}.  Since $D$ is admissible, every component $T$ of $D$ has $K_{X}\cdot T\geq 0$ and $\cf(T)\in [0,1)$, see Lemma \ref{lem:ld_formulas}. Thus $K_{X}\cdot (K_{X}+D)\geq K_{X}\cdot (K_{X}+\sum_{T}\cf(T)\, T)=K_{\bar{X}}^2>0$ because $-K_{\bar{X}}$ is ample. This ends the proof of \eqref{eq:K+D}.

The inequality \eqref{eq:K+D} implies that $K_{X}+D+G=0$ for some effective divisor $G$. Fix a $\P^{1}$-fibration of $X$. Its fiber $F$ satisfies $F\cdot (D+G)=-F\cdot K_{X}=2$, so $\height(X,D)\leq F\cdot D\leq 2$, which proves \ref{cor:Russell}\ref{item:Rus-ht}. To prove \ref{cor:Russell}\ref{item:Rus-Cylinder} we may assume that $\height(X,D)=2$, so $F\cdot D=2$ and $G$ is vertical.

If the chosen $\P^{1}$-fibration has no degenerate fibers then $X$ is a Hirzebruch surface and $D$ is the negative section, contrary to the fact that $\height(X,D)=2$. Thus we can assume that the support of $G$ is contained in the sum of  degenerate fibers. Since $K_{X}\cdot G=-K_{X}\cdot (K_{X}+D)<0$ by inequality \eqref{eq:K+D}, such $G$ contains a $(-1)$-curve $G_0$. If $|K_X+D+G_0|=\emptyset$ then by Lemma  \ref{lem:trick-trees} $D+G_0$ is a sum of disjoint rational trees, so $G_0$ meets the chain $D$ at most once; and $\kappa(K_{X}+D+G_0)=\kappa(K_{X}+D)=-\infty$. Thus  replacing $(X,D)$ by its image after the contraction of $G_0$, we can eventually assume that $|K_X+D+G_0|\neq \emptyset$. Since $|K_X+D+G_0|=|G_0-G|$, we get $G=G_0$. For every component $C$ of $D+G$ we have $\beta_{D+G}(C)=C\cdot (D+G-C)=-C\cdot (K+C)=2$, so $D+G$ is circular. It follows that the fiber $F_{G}$ containing $G$ is columnar.

If $L\neq G$ is another $(-1)$-curve then $0=L\cdot (K_{X}+D+G)\geq L\cdot D-1$, so $L\cdot D\leq 1$, and as before we replace $(X,D)$ by its image after the contraction of $L$. Thus we can assume that $F_{G}$ is a unique degenerate fiber. Let $\psi\colon X\to \F_{m}$ be a contraction of $F_{G}$ to a $0$-curve. Then $\psi_{*}D\hor$ is either a $2$-section, or a sum of two non-disjoint $1$-sections on a Hirzebruch surface. It follows that some component of $\psi_{*}D\hor$ has positive self-intersection number, so some component of the chain $D$ has non-negative self-intersection number. Lemma \ref{lem:twig} implies that 
$X\setminus D$ is $\A^{1}$-fibered, as needed.
\end{proof}

\begin{proof}[Proof of Corollary \ref{cor:uniruled}]
	As before, we can assume that $(X,D)$ is as in Proposition \ref{prop:MT_smooth}. In cases \ref{prop:MT_smooth}\ref{item:MT-P2}, \ref{item:MT-ht=1}, or  \ref{item:MT-columnar}--\ref{item:MT-3-fibers}, the result follows from Remark \ref{rem:A1-fibration}. Consider case \ref{item:MT-Platonic}, $\cha\kk\not\in \{2,3,5\}$. Then $X\setminus D$ admits a \emph{Platonic} $\A^{1}_{*}$-fibration over $\P^1$, see Figure \ref{fig:MT-Platonic}. The construction in \cite[\sec 7.5]{Miy_Tsu-opendP} shows that there is a finite morphism $\nu\colon \P^1\to \P^1$ such that the normalized pullback of $X\setminus D$ along $\nu$ is $\A^1$-ruled, as needed.
	
	Thus for the remaining part of the proof we can assume that $p\de \cha\kk\in \{2,3,5\}$, and $(X,D)$ is as in \ref{prop:MT_smooth}\ref{item:MT-Platonic} or \ref{item:MT-ht=2-cha=2}--\ref{item:MT-eliptic-relative}. We will argue as in \cite[\sec 7.5]{Miy_Tsu-opendP}: first, in Claim \ref{cl:good-fibration} we choose a suitable $\P^1$-fibration of $X$ (or, in one case, its blowup $\tilde{X}$), and next, in Claim \ref{cl:local-properties}, we pull it back by a totally ramified $p$-fold covering $\P^1\to \P^1$. We show that the normalized pullback of $X\setminus D$ is $\A^{1}$-ruled, hence $X\setminus D$ is $\A^{1}$-uniruled, as needed.
	\smallskip
	
	For a log smooth surface $(\tilde{X},\tilde{D})$ with a fixed $\P^1$-fibration $f\colon \tilde{X}\to \P^1$, we introduce the following terminology. For a vertical curve $C$ we denote by $\mu(C)$ its multiplicity  in a fiber. We call a fiber $F$  \emph{good} if it has exactly one component not contained in $D$, call it $C$, and the following holds:
	\begin{enumerate-alt}
		\item \label{item:good-C} every connected component of $F\redd-C$ meets at most one component of $\tilde{D}\hor$,
		\item \label{item:good-mult} $\mu(C)\in \{1,p\}$.
	\end{enumerate-alt}
	As we will see, good fibers will \enquote{become} nondegenerate after taking the covering, cf.\ Figure \ref{fig:covering-4A2}. For example, a columnar fiber 
	satisfies condition \ref{item:good-C}, so it is good if and only if its unique $(-1)$-curve has multiplicity $p$. 	
	
	We say that $F$ is \emph{bad} if it is not good. We denote by $v$ the number of bad fibers, and put $h=\#\tilde{D}\hor$.
	
	\begin{claim}\label{cl:good-fibration}
	There exists a birational morphism $(\tilde{X},\tilde{D})\to (X,D)$, with $\tilde{D}$ being the reduced total transform of $D$, and a $\P^1$-fibration $f\colon \tilde{X}\to \P^1$ such that, with the above notation, the following hold.
	\begin{enumerate}
		\item\label{item:H} Every horizontal component $H$ of $\tilde{D}$ is either a $1$-section, or a $p$-section such that $f|_{H}$ is totally ramified.
		\item\label{item:H-disjoint} The horizontal components are disjoint, and $h\leq 3$. 
		\item\label{item:V} We have either $h\leq 1$, or $v\leq 1$, or 
$(h,v)=(2,2)$ and one of the bad fibers satisfies condition \ref{item:good-C}.
		\end{enumerate}
	\end{claim}
	\begin{proof}
		We study each case of Proposition \ref{prop:MT_smooth} separately. In case \ref{prop:MT_smooth}\ref{item:MT-elipitic-relative-3} with $\bar{Y}$ of type $4\rA_2$ the log surface $(\tilde{X},\tilde{D})$ will be a blowup of $X$ at one point, see Figure \ref{fig:covering-4A2}; in the remaining cases we will take $(\tilde{X},\tilde{D})=(X,D)$. 
		
		In cases \ref{prop:MT_smooth}\ref{item:MT-Platonic}, \ref{prop:MT_smooth}\ref{item:MT-ht=2-cha=2} and \ref{prop:MT_smooth}\ref{item:MT-exception}, the required properties \ref{item:H}--\ref{item:V} are satisfied by the $\P^1$-fibrations described in the corresponding statements of \ref{prop:MT_smooth}, see Figures \ref{fig:MT-Platonic} and \ref{fig:MT_exceptions}. Indeed, in case \ref{prop:MT_smooth}\ref{item:MT-Platonic} $D\hor$ consists of two disjoint $1$-sections, so \ref{item:H} and \ref{item:H-disjoint} hold. The only degenerate fibers are the three columnar ones, and the multiplicities of their $(-1)$-curves are the Platonic triples $\{2,2,n\}$, $\{2,3,3\}$, $\{2,3,4\}$ or $\{2,3,5\}$. Since $p\in \{2,3,5\}$, at least one of those multiplicities equals $p$, hence at least one of those fibers is good. Thus we get \ref{item:V}, case $(h,v)=(2,2)$.
		
		In case \ref{prop:MT_smooth}\ref{item:MT-ht=2-cha=2} we have $p=2$ and $D\hor$ consists of one $2$-section $H$ such that $f|_{H}$ is totally ramified, so \ref{item:H}--\ref{item:V} hold. Eventually, in case \ref{prop:MT_smooth}\ref{item:MT-exception} $D\hor$ consists of a $1$ section and a $2$-section as above, so \ref{item:H} and \ref{item:H-disjoint} hold. Part \ref{item:V} holds, too: indeed, Figure \ref{fig:MT_exceptions} shows that we get $v=1$ in \ref{prop:MT_smooth}\ref{item:MT-exc-1} and case $(h,v)=(2,2)$ of \ref{item:V} in  \ref{prop:MT_smooth}\ref{item:MT-exc-2}. 
		\smallskip
		
		It remains to consider case \ref{prop:MT_smooth}\ref{item:MT-eliptic-relative}. Then $(\bar{X},\bar{D})$ is a minimal log resolution of a surface $(\bar{Y},\bar{T})$, where $\bar{Y}$ is a canonical del Pezzo surface of rank one, and $\bar{T}\in |-K_{\bar{Y}}|$ is an irreducible, cuspidal member, contained in $\bar{Y}\reg$. Let $(Y,D_Y)$ be the minimal log resolution of $\bar{Y}$: it is shown in Figures \ref{fig:can_ht=4}, \ref{fig:can_ht=2} or \ref{fig:can_ht=1}. Put $s=\#D_Y$. The proper transform $T_Y$ of $\bar{T}$ on $Y$ is again an irreducible, cuspidal member of $|-K_{Y}|$, which is disjoint from $D_Y$, and by adjunction meets every $(-1)$-curve on $Y$ once. Some of these $(-1)$-curves are drawn as dashed lines in Figures  \ref{fig:can_ht=4}, \ref{fig:can_ht=2} and  \ref{fig:can_ht=1}. Moreover, by Noether's formula we have $T_{Y}^2=K_{Y}^2=10-\rho(Y)=9-s$. The log surface $(X,D)$ is obtained by resolving the cusp of $T_{Y}$, which replaces $T_Y$ by a fork $\langle 1;[2],[3],[s-3]\rangle$. The last twig of this fork is the proper transform of $\bar{T}$, which we denote by $T$. The list in \ref{prop:MT_smooth}\ref{item:MT-eliptic-relative} shows that $s\geq 6$,  
		the singularity type of $\bar{Y}$ is listed in Table \ref{table:canonical}, excluding the first row, and the corresponding entry in the last column is $\rC$.
		
		We will use the following consequence of the Hurwitz formula \cite[Corollary IV.2.4]{Hartshorne_AG}.
		\begin{equation}\label{eq:Hurwitz}
			\parbox{.9\textwidth}{Let $g \colon \P^1\to \P^1$ be a morphism of degree $p$, ramified 
				with ramification indices $e_{1},\dots,e_{k}$. If $\sum_{i=1}^{k}(e_{i}-1)\geq 2p-1$ 
				then $g$ is totally ramified.}
		\end{equation}

		Assume first that $p=2$. Consider the $\P^1$-fibration of $Y$ shown in Figure \ref{fig:8A1}, \ref{fig:7A1} or \ref{fig:can_ht=1}. Since $s\geq 6$ and $\bar{Y}$ is not of type $\rE_{s}$, see Table \ref{table:canonical}, this  $\P^1$-fibration has $u\geq 2$ degenerate fibers. Let $f$ be its pullback to $X$. 
		Then $D\hor$ consists of a $2$-section $T$ and either: a $1$-section if $\height(\bar{Y})=1$, see Figure \ref{fig:can_ht=1}; or a $2$-section if $\bar{Y}$ is of type $7\rA_1$, see Figure \ref{fig:7A1}; or two $2$-sections  if $\bar{Y}$ is of type $8\rA_1$, see Figure \ref{fig:8A1}. In any case, we have $h\leq 3$ and the components of $D\hor$ are disjoint, so \ref{item:H-disjoint} holds. The degenerate fibers of $f$ are: the pullback of the fiber passing through the cusp of $T_Y$, and pullbacks of the degenerate fibers from $Y$. Let $F$ be such a fiber and, in case $8\rA_1$, assume that it is not of type $[1,1]$ (the rightmost fiber in Figure \ref{fig:8A1}). Then $F$ has exactly one component not contained in $D$ which meets all components of $D\hor-T$, hence $F$ is good. Thus \ref{item:V} holds, with $v\leq 1$. Moreover, $F$ meets each $2$-section $H$ in $D\hor$ in one point, which is therefore a ramification point of $f|_{H}$. Since there are at least three such fibers, by \eqref{eq:Hurwitz} $f|_{H}$ is totally ramified, so \ref{item:H} holds, too, as needed.
		\smallskip
		
		Assume now that $\bar{Y}$ is of type $3\rA_2$ or $2\rA_4$, see Figure \ref{fig:3A2} or \ref{fig:2A4}. Then $p=3$ or $5$, respectively, see Table \ref{table:canonical}. In each case $D_Y$ has a connected component $C_Y=[(2)_{p-1}]$, and there is a $(-1)$-curve $L_Y$ such that $L_Y\cdot D_Y=2$ and $L_Y$ meets $C_Y$ once, in a tip. Let $C,L$ be the preimages of $C_Y$, $L_Y$ of $D$. Then $T+L+C=[p,1,(2)_{p-1}]$ supports a good fiber of a $\P^1$-fibration $f\colon X\to \P^1$. The horizontal part of $D$ consists of a $p$-section $H_{p}=[2]$ in the preimage of $D_Y$, and a $1$-section $H=[1]$ in the preimage of the cusp of $\bar{T}$, so \ref{item:H-disjoint} holds. By Lemma \ref{lem:delPezzo_fibrations}, every fiber has exactly one component not in $D$. The connected components of $D\vert-T-C$ are: a $(-2)$- and a $(-3)$-curve meeting $H$, a $(-2)$-curve meeting $H_p$, and a chain $R=[2,2]$ which is disjoint from $D\hor$ if $\bar{Y}$ is of type $3\rA_2$, and meets $H_p$ once if $\bar{Y}$ is of type $2\rA_4$. It follows that the degenerate fibers other than $F$ are supported on chains $[2,1,2]$ and $[3,1,2,2]$, so we get condition \ref{item:V}, case $(h,v)=(2,2)$. Let $L_2$, $L_3$ be the $(-1)$-curves in those fibers, they have $\mu(L_2)=2$, $\mu(L_3)=3$. If $\bar{Y}$ is of type $3\rA_2$, so $H_3$ is a $3$-section, we have $L_2\cdot H_{3}=L_{3}\cdot H_3=1$. Thus $f|_{H_{3}}$ is ramified with index $3$ at $H_3\cap L$, $H_3\cap L_3$, and with index at least $2$ at $L_2\cap H_3$. By \eqref{eq:Hurwitz} $f|_{H_3}$ is totally ramified. Similarly, if $\bar{Y}$ is of type $2\rA_4$ then $f|_{H_5}$ is ramified  at $L\cap H_5$ with index $5$, at $L_3\cap H_5$ with index $\geq 3$, at $R\cap H_5$ with index $\geq 2$, and at each point of $L_2\cap H_5$, with index $\geq 2$ if $\#L_2\cap H_5=2$, and $\geq 3$ if $\#L_2\cap H_5=1$. By \eqref{eq:Hurwitz}, $f|_{H_5}$ is totally ramified, which proves 
		\ref{item:H}.
		\smallskip
		
		Consider the remaining cases with  $\height(\bar{Y})=2$, see Table \ref{table:canonical} and Figure \ref{fig:can_ht=2}. We reduce the proof to cases $7\rA_1$ and $3\rA_2$, settled above, as follows. There is a $(-1)$-curve $L_Y\subseteq Y$ meeting $D$ once in $T$ and once in the preimage of $D_Y$, such that contracting $L_Y$ and some $(-1)$-tips of the subsequent images of $D_Y$ we get a birational morphism $Y\to Y'$ onto the minimal resolution of a surface of type $7\rA_1$ or $3\rA_2$, see \cite[Example 5.6(c),(a)]{PaPe_ht_2}. Since $L_Y$ meets $T_Y$ once, this morphism is an isomorphism near the cusp of $T_Y$, so it lifts to a birational morphism $\phi\colon (X,D)\to (X',D'+L')$. Here $(X',D')$ is the log surface considered in cases $7\rA_1$ or $3\rA_2$ above, and $L'$ is a $(-1)$-curve containing the base point of $\phi^{-1}$. The description $\phi$ in each case shows that, after possibly composing  $\phi$ with an automorphism of $(X',D')$, the curve $L'\not\subseteq D'$ lies in some good fiber $F'$ of the $\P^1$-fibration $f'\colon X'\to \P^1$ constructed above. The pullback of $F'$ to $X$ is still a good fiber of $f\de f'\circ \phi$, and $\phi^{-1}$ is an isomorphism near each of the remaining fibers, so properties \ref{item:H}--\ref{item:V} of $f'$ hold for $f$, too.
		\smallskip
		
		It remains to consider the case when $p=3$ and $\bar{Y}$ is of type $4\rA_2$. We keep notation from Example \ref{ex:4A2_construction}, see Figure \ref{fig:4A2}. Note that, since the cubic $\qq$ is cuspidal, the curve $A_0$ passes through the point $Q\cap R$, and since the lines $\ell$, $\ell_1$, $\ell_2$ tangent to $\qq$ are concurrent, see Figure \ref{fig:4A2-constr-cha=3}, the curves $L$, $L_1$ and $L_2$ meet at one point. Let $\eta\colon \tilde{X}\to X$ be the blowup at the preimage of $Q\cap R$. Then $(\tilde{X},\tilde{D})$ is shown in the top-right part of Figure \ref{fig:covering-4A2}. We use the same letters for curves on $Y$ and their proper transforms on $\tilde{X}$. 
		
		The linear system $|A_0+2L+L_1|$ induces a $\P^1$-fibration $f\colon \tilde{X}\to \P^1$. The horizontal part of  $\tilde{D}$ consists of $3$-sections $T$, $L$ and a $1$-section $\Exc\eta$, so \ref{item:H-disjoint} holds. The chain
		$A_0+L+L_1=[1,2,1]$ supports a bad fiber, and $Q+A_2+G_2=[3,1,2,2]$, $R+L_1'+G_2=[3,1,2,2]$ support good fibers. Let $F$ be the fiber containing the remaining connected component $[2,1,3]$ of $D\vert$. By Lemma \ref{lem:fibrations-Sigma-chi} there are no more degenerate fibers, and $F$ has exactly one component not in $D$, call it $C$, so $F$ is good, too. Thus \ref{item:V} holds. The restrictions $f|_{T}$ and $f|_{L_2}$ are ramified with index $3$ on $L$, $A_2$ and $L_1'$, so by \eqref{eq:Hurwitz} they are totally ramified, hence \ref{item:H} holds. 
	\end{proof}

\begin{figure}
	\begin{tikzpicture}[scale=1.2]
		\begin{scope}[shift={(0,5)}]
			\draw (-1.2,4) -- (4.6,4);
			\node at (-1,3.8) {\small{$-1$}}; 
			\draw[thick] (-1.2,3.1) -- (0.2,3.1) to[out=0,in=180] (1,3.05) -- (2.9,3.05) to[out=0,in=180] (4,1.95) -- (4.6,1.95);
			\node at (-1,2.9) {\small{$-5$}}; \node at (-1,3.3) {\small{$T$}};
			\draw[thick] (-1.2,2.1) -- (0.4,2.1) to[out=0,in=180] (1.2,2.75) -- (3.25,2.75);
			\draw[thick] (3.45,2.75) to[out=0,in=180] (4.1,3.55) to[out=0,in=-99] (4.3,3.7) to[out=79,in=180] (4.6,3.8) -- (4.6,3.8);
			\node at (-1,1.9) {\small{$-2$}}; \node at (-1,2.3) {\small{$L_2$}};
			\draw (0,1) -- (0.2,2.2);		\node at (-0.15,1.5) {\small{$-2$}}; \node at (0.3,1.5) {\small{$L_1$}};
			\filldraw (0.02,3.1) circle (0.05); 
			\draw[dashed] (0.2,2) -- (0,3.2); 	\node at (-0.15,2.6) {\small{$-1$}};  \node at (0.3,2.6) {\small{$L$}};
			\filldraw (0.18,2.1) circle (0.05); 
			\draw[dashed] (0,3) -- (0.2,4.2);		\node at (-0.15,3.7) {\small{$-2$}}; \node at (0.32,3.7) {\small{$A_0$}};
			%
			\draw (1.6,0.8) -- (1.4,1.8); 	\node at (1.25,1.3) {\small{$-2$}};
			\node at (1.7,1.7) {\small{$G_1$}};
			\filldraw (1.42,1.7) circle (0.05); 
			\draw (1.4,1.6) -- (1.6,2.6);		\node at (1.25,2.1) {\small{$-2$}};
			\draw[dashed] (1.6,2.4) -- (1.4,3.4); 	\node at (1.25,2.88) {\small{$-1$}}; 
			\draw (1.4,3.2) -- (1.6,4.2);		\node at (1.25,3.7) {\small{$-3$}}; \node at (1.67,3.7) {\small{$Q$}};
			%
			\draw (3,0.8) -- (2.8,1.8); 	\node at (2.65,1.3) {\small{$-2$}}; 
			\node at (3.1,1.7) {\small{$G_2$}};
			\filldraw (2.82,1.7) circle (0.05); 
			\draw (2.8,1.6) -- (3,2.6);		\node at (2.65,2.1) {\small{$-2$}};
			\draw[dashed] (3,2.4) -- (2.8,3.4); 	\node at (2.65,2.88) {\small{$-1$}};
			\draw (2.8,3.2) -- (3,4.2);		\node at (2.65,3.7) {\small{$-3$}}; \node at (3.05,3.7) {\small{$R$}};
			%
			\draw (4.4,0.8) -- (4.2,1.8); 	\node at (4.05,1.3) {\small{$-3$}};
			\draw (4.2,1.6) -- (4.4,2.6);		\node at (4.1,2.2) {\small{$-1$}};
			\draw (4.4,2.4) -- (4.2,3.4); 	\node at (4.05,2.9) {\small{$-2$}};
			\filldraw (4.22,3.3) circle (0.05);
			\draw[dashed] (4.2,3.2) -- (4.4,4.2);		\node at (4.05,3.7) {\small{$-2$}};
			\draw[<-] (1.7,0.5) -- (1.7,-0.5);
			\node[left] at (1.7,0) {\footnotesize{blowup}};
			\node[right] at (1.7,0) {\footnotesize{at $\bullet$}};
		\end{scope}
		\begin{scope}[shift={(8,5)}]
			\draw (-1.2,4) -- (4.6,4); \node at (-1,3.8) {\small{$0$}};
			\draw (-1.2,2.55) -- (4.6,2.55); \node at (-1,2.35) {\small{$0$}};
			\draw (-1.2,1.1) -- (4.6,1.1); \node at (-1,0.9) {\small{$0$}};
			\draw (0.1,4.2) -- (0.1,0.8); \node at (0.25,1.7) {\small{$0$}};
			\draw[dashed] (1.5,4.2) -- (1.5,0.8); \node at (1.65,1.7) {\small{$0$}};
			\draw[dashed] (2.9,4.2) -- (2.9,0.8); \node at (3.05,1.7) {\small{$0$}};
			\draw[dashed] (4.3,4.2) -- (4.3,0.8); \node at (4.45,1.7) {\small{$0$}};
			\draw[<-] (1.7,0.5) -- (1.7,-0.5);
			\node[left] at (1.7,0) {\footnotesize{blowdown}};
			\node[right] at (1.7,0) {\footnotesize{to $\P^1\times \P^1$}};
		\end{scope} 
		\begin{scope}
		\draw (-1.2,4) -- (4.6,4);
		\node at (-1,3.8) {\small{$-1$}};
		\draw[thick] (-1.2,3.05) -- (2.9,3.05) to[out=0,in=180] (3.8,1.1) -- (4.6,1.1);
		\node at (-1,2.9) {\small{$-6$}};
		\draw[thick] (-1.2,1.1) -- (0,1.1) to[out=0,in=180] (1.2,2.75) -- (3.2,2.75);
		\draw[thick] (3.4,2.75) to[out=0,in=180] (4,3.55) to[out=0,in=-99] (4.3,3.7) to[out=79,in=180] (4.6,3.8) -- (4.6,3.8);
		\node at (-1,0.9) {\small{$-3$}};
		\draw[thick] (0,0) -- (0.2,1);		\node at (-0.15,0.5) {\small{$-3$}};
		\draw (0.2,0.8) -- (0,1.8); 	\node at (-0.15,1.3) {\small{$-1$}};
		\draw[dashed, thick] (0,1.6) -- (0.2,2.6);		\node at (-0.15,2.1) {\small{$-3$}};
		\draw (0.2,2.4) -- (0,3.4); 	\node at (-0.15,2.88) {\small{$-1$}};
		\draw[dashed, thick] (0,3.2) -- (0.2,4.2);		\node at (-0.15,3.7) {\small{$-3$}};
		\draw[thick] (1.4,0) -- (1.6,1);		\node at (1.25,0.5) {\small{$-3$}};
		\draw (1.6,0.8) -- (1.4,1.8); 	\node at (1.25,1.3) {\small{$-1$}};
		\draw[thick] (1.4,1.6) -- (1.6,2.6);		\node at (1.25,2.1) {\small{$-3$}};
		\draw[dashed] (1.6,2.4) -- (1.4,3.4); 	\node at (1.25,2.88) {\small{$-1$}};
		\draw[thick] (1.4,3.2) -- (1.6,4.2);		\node at (1.25,3.7) {\small{$-3$}};
		\draw[thick] (2.8,0) -- (3,1);		\node at (2.65,0.5) {\small{$-3$}};
		\draw (3,0.8) -- (2.8,1.8); 	\node at (2.65,1.3) {\small{$-1$}};
		\draw[thick] (2.8,1.6) -- (3,2.6);		\node at (2.65,2.1) {\small{$-3$}};
		\draw[dashed] (3,2.4) -- (2.8,3.4); 	\node at (2.65,2.88) {\small{$-1$}};
		\draw[thick] (2.8,3.2) -- (3,4.2);		\node at (2.65,3.7) {\small{$-3$}};
		\draw[thick] (4.2,0) -- (4.4,1);		\node at (4.05,0.5) {\small{$-3$}};
		\draw (4.4,0.8) -- (4.2,1.8); 	\node at (4.05,1.3) {\small{$-1$}};
		\draw[thick] (4.2,1.6) -- (4.4,2.6);		\node at (4.05,2.1) {\small{$-3$}};
		\draw (4.4,2.4) -- (4.2,3.4); 	\node at (4.05,2.9) {\small{$-1$}};
		\draw[dashed, thick] (4.2,3.2) -- (4.4,4.2);		\node at (4.05,3.7) {\small{$-3$}};
		\end{scope}
		\node at (5.8,2.55) {\footnotesize{$3$ to $1$}};
		\node at (5.8,2.3) {\footnotesize{covering}};
		\draw [<-] (4.8,2.1) -- (6.8,2.1);
		\node at (5.8,1.9) {\footnotesize{branched along}};
		\node at (5.8,1.65) {\footnotesize{thick curves}};
		\begin{scope}[shift={(8,0)}]
			\draw (-1.2,4) -- (4.6,4);
			\node at (-1,3.8) {\small{$-3$}};
			\draw (-1.2,3.05) -- (2.9,3.05) to[out=0,in=180] (3.8,1.1) -- (4.6,1.1);
			\node at (-1,2.9) {\small{$-2$}};
			\draw (-1.2,1.1) -- (0,1.1) to[out=0,in=180] (1.2,2.75) -- (3.2,2.75);
			\draw (3.4,2.75) to[out=0,in=180] (4,3.55) -- (4.6,3.55);
			\node at (-1,0.9) {\small{$-1$}};
			\draw (0,0) -- (0.2,1);		\node at (-0.15,0.5) {\small{$-1$}};
			\draw (0.2,0.8) -- (0,1.8); 	\node at (-0.15,1.3) {\small{$-3$}};
			\draw[dashed] (0,1.6) -- (0.2,2.6);		\node at (-0.15,2.1) {\small{$-1$}};
			\draw (0.2,2.4) -- (0,3.4); 	\node at (-0.15,2.88) {\small{$-3$}};
			\draw[dashed] (0,3.2) -- (0.2,4.2);		\node at (-0.15,3.7) {\small{$-1$}};
			\draw (1.4,0) -- (1.6,1);		\node at (1.25,0.5) {\small{$-1$}};
			\draw (1.6,0.8) -- (1.4,1.8); 	\node at (1.25,1.3) {\small{$-3$}};
			\draw (1.4,1.6) -- (1.6,2.6);		\node at (1.25,2.1) {\small{$-1$}};
			\draw[dashed] (1.6,2.4) -- (1.4,3.4); 	\node at (1.25,2.88) {\small{$-3$}};
			\draw (1.4,3.2) -- (1.6,4.2);		\node at (1.25,3.7) {\small{$-1$}};
			\draw (2.8,0) -- (3,1);		\node at (2.65,0.5) {\small{$-1$}};
			\draw (3,0.8) -- (2.8,1.8); 	\node at (2.65,1.3) {\small{$-3$}};
			\draw (2.8,1.6) -- (3,2.6);		\node at (2.65,2.1) {\small{$-1$}};
			\draw[dashed] (3,2.4) -- (2.8,3.4); 	\node at (2.65,2.88) {\small{$-3$}};
			\draw (2.8,3.2) -- (3,4.2);		\node at (2.65,3.7) {\small{$-1$}};
			\draw (4.2,0) -- (4.4,1);		\node at (4.05,0.5) {\small{$-1$}};
			\draw (4.4,0.8) -- (4.2,1.8); 	\node at (4.05,1.3) {\small{$-3$}};
			\draw (4.2,1.6) -- (4.4,2.6);		\node at (4.05,2.1) {\small{$-1$}};
			\draw (4.4,2.4) -- (4.2,3.4); 	\node at (4.05,2.9) {\small{$-3$}};
			\draw[dashed] (4.2,3.2) -- (4.4,4.2);		\node at (4.05,3.7) {\small{$-1$}};
		\end{scope}
	\end{tikzpicture}
	\caption{Proof of Corollary \ref{cor:uniruled} in case \ref{prop:MT_smooth}\ref{item:MT-eliptic-relative}, $\bar{Y}$ of type $4\rA_2$, $\cha\kk=3$; see Example \ref{ex:covering-4A2}.}
	\label{fig:covering-4A2}
\end{figure}

	Let $f\colon \tilde{X}\to \P^1$ be a $\P^1$-fibration chosen in Claim \ref{cl:good-fibration}, and let $\tilde{X}_{\nu}$ be the normalization of a fiber product of $f\colon \tilde{X}\to \P^1$ with a totally ramified covering $\P^1\to \P^1$ given by $x\mapsto x^p$.
	
	\begin{claim}\label{cl:local-properties}
		The following hold. 
	\begin{enumerate}
		\item\label{item:pullback-sing} The only singular points of $\tilde{X}_{\nu}$ are Hirzebruch-Jung singularities, isomorphic to the normalization of a hypersurface singularity $\{\tau^{p}=x^{a}y^{b}\}$, where $\{x=0\}$,  $\{y=0\}$ are local equations of some vertical curves.
		\item\label{item:pullback-V} If a vertical curve $C$ has multiplicity $\mu(C)\in \{1,p\}$ then its preimage $C_{\nu}$ is a vertical curve with $\mu(C_{\nu})=1$.
		\item\label{item:pullback-H} The horizontal part of the reduced preimage of $D$ on $\tilde{X}_{\nu}$ consists of $h$ disjoint $1$-sections.
	\end{enumerate}
	\end{claim}
	\begin{proof} 
	This claim follows from a direct local computation, analogous to the one in \cite[\sec 7.2]{Miy_Tsu-opendP}. Fix a point $z_{\nu}\in \tilde{X}_{\nu}$, and let $z$ be its image on $\tilde{X}$. If $z$ lies on the intersection of two vertical curves, we get a singularity at $z_{\nu}$ as described in \ref{item:pullback-sing}. Assume that $z$ lies on exactly one vertical curve, say $C$, and let $c=\mu(C)$. Let $(x,y)$ be local coordinates of $\tilde{X}$ at $z$, such that $C=\{x=0\}$. Let $t$ and $\tau$ be parameters of $\P^1$ at the images of $z$ and $z_{\nu}$. Then $t=ux^{c}$, where $u$ is a regular, nonvanishing function on some neighborhood of $z$. Denoting the pullbacks of functions on $\tilde{X}$ to $\tilde{X}_{\nu}$ by the same letters, we get $\tau^{p}=t=ux^{c}$. Write $c=p\bar{c}+r$ for some integer $r\in \{0,\dots,p-1\}$. Then the rational function $\xi\de \tau x^{-\bar{c}}$ satisfies $\xi^{p}=ux^{r}$, hence is regular as $\tilde{X}_{\nu}$ is normal. We get $\tau=\xi x^{\bar{c}}$, so $\tilde{X}_{\nu}$ is regular at $z_{\nu}$, and  if $c\in \{1,p\}$ then $\tau$ is a local equation of $C_{\nu}$, so $\mu(C_{\nu})=1$, as claimed in \ref{item:pullback-V}. 
	To prove \ref{item:pullback-H} we can assume that $z\in H$ and $C$ is a nondegenerate fiber. If $H$ is a $1$-section then we can choose $(x,y)$ so that $y$ is a local equation of $H$, and then its preimage $\{y=0\}$ is a $1$-section. Otherwise, by Claim \ref{cl:good-fibration}\ref{item:H} $H$ is a $p$-section and $(C\cdot H)_{z}=p$, so we can choose $(x,y)$ in such a way that $H=\{x=y^p\}$. Then the total transform of $H$ is given by $\{x^{p}=y^p\}$, so its (reduced) preimage $\{x=y\}$ is a $1$-section, as needed. By Claim \ref{cl:good-fibration}\ref{item:H-disjoint} the horizontal components of $D$ are disjoint, hence so are their preimages, as needed.
	\end{proof}

	Let $F_{1},\dots,F_{v}$ be the bad fibers of $f$. By Claim \ref{cl:good-fibration}\ref{item:V} we can assume that $F_2$ satisfies \ref{item:good-C}. Let $X''\to \tilde{X}_{\nu}$ be the minimal resolution.  Let $D''$ be the reduced total  transform of $X''$ of $(D+F_1)\redd$ if $v\geq 1$ and of $D$ otherwise. By Claim \ref{cl:local-properties}\ref{item:pullback-H} $D''\hor$ consists of $h$ disjoint $1$-sections. Let $(X'',D'')\to (X',D')$ be a maximal contraction of superfluous vertical $(-1)$-curves in $D''$ and its images, such that the components of $D'\hor$ remain disjoint. Since the rational map $X'\map X$ restricts to a dominant morphism $X'\setminus D'\to X\setminus D$,  it is enough to show that $X'\setminus D'$ is $\A^1$-ruled. If $h=1$ this follows from Lemma \ref{lem:Tsen}, so we can assume $h\geq 2$.
	
	\begin{claim}\label{cl:F}
		For a fiber $F$ of $f$, denote by $F'$ its reduced total transform on $X'$. The following hold.
		\begin{enumerate}
			\item\label{item:cl-good} If $F$ is a good fiber then $F'$ is a $0$-curve not contained in $D'$.
			\item\label{item:cl-bad} If $v\geq 1$ then $F_1'$ is a $0$-curve contained in $D'$.
			\item\label{item:cl-ugly} If $v=2$ then $F_2'$ is either columnar, see Definition \ref{def:fibers}\ref{item:F-columnar}, or a $0$-curve not contained in $D'$.
		\end{enumerate}
	\end{claim}
	\begin{proof}
		\ref{item:cl-good} Let $C$ be the unique component of $F\redd$ off $\tilde{D}$. Condition \ref{item:good-mult} and Claim \ref{cl:local-properties}\ref{item:pullback-V} imply that its preimage $C''$ on $X''$ is a curve of multiplicity $1$ in a fiber $F''$. Hence $F''-C''$ contracts to smooth points, and by condition \ref{item:good-C} the images of components of $D''\hor$ after this contraction remain disjoint, so $F'$ is the image of $C''$, as needed.
		
		\ref{item:cl-bad} We have $F_{1}'\subseteq D'$ by the definition of $D''$. If $F_1'$ contains a $(-1)$-curve then the latter meets at least two components of $D'\hor$. By Claims \ref{cl:good-fibration}\ref{item:H-disjoint} and \ref{cl:local-properties}\ref{item:pullback-H}, $D'\hor$ consists of at most three $1$-sections, so there is at most one $(-1)$-curve in $F_1'$, and it has multiplicity $1$, which is impossible. Thus $F_1'=[0]$, as required.
		
		\ref{item:cl-ugly} Since $F_2$ is bad but satisfies \ref{item:good-C}, it does not satisfy \ref{item:good-mult}. By Claim \ref{cl:good-fibration}\ref{item:H}, this implies that the component of $F_{2}$ not contained in $\tilde{D}$ does not meet $\tilde{D}\hor$. By Claim \ref{cl:good-fibration}\ref{item:V} we have $h=2$, so by \ref{item:good-C} the two horizontal components of $\tilde{D}$ meet $F_2$ in two connected components of $F_2\cap \tilde{D}\vert$. Let $F_{2}''$ be the total transform of $F_2$ on $D''$. Now $F_{2}''\cap D''$ has two connected components meeting $D''\hor$, so they both contain components of multiplicity one. It follows that its image $F_{2}'$ is columnar or nondegenerate, as needed.
	\end{proof}
	
	If $v\geq 2$ then $(h,v)=(2,2)$ by Claim \ref{cl:good-fibration}\ref{item:V}, and Claim \ref{cl:F} shows that $D'$ is a chain with a $0$-curve. Otherwise Claim \ref{cl:F} implies that $X'$ is a Hirzebruch surface, and $D'$ is a sum of $h\geq 2$ disjoint $1$-sections and at most $1$ fiber. In any case, $D'$ has a twig which is not negative definite, so $X'\setminus D'$ is $\A^1$-ruled by Lemma \ref{lem:twig}.
\end{proof}

\begin{example}[Proof of the Corollary \ref{cor:uniruled} in case \ref{prop:MT_smooth}\ref{item:MT-eliptic-relative} with $\bar{Y}$ of type $4\rA_2$, see Figure \ref{fig:covering-4A2}]\label{ex:covering-4A2}
	We now follow the above proof more closely in case studied at the end of Claim \ref{cl:good-fibration}, that is, when $(\bar{X},\bar{D})$ is as in \ref{prop:MT_smooth}\ref{item:MT-eliptic-relative}, with $\bar{Y}$ of type $4\rA_2$.  The $\P^1$-fibration $f$ from Claim \ref{cl:good-fibration} is shown in the top-left part of Figure \ref{fig:covering-4A2}. 
	The one in the bottom-right is constructed by taking the normalized pullback $\tilde{X}_{\nu}$ of $f$ along the morphism $\P^1\ni x\mapsto x^{3}\in \P^1$ and resolving the singularities of $\tilde{X}_{\nu}$. 
	Alternatively, one can first blow up the points of $\tilde{X}$ which will become singularities of $\tilde{X}_{\nu}$, and then take the normalized pullback. This blowup is described in the bottom-left. The key point is that, as shown in Claim \ref{cl:local-properties}\ref{item:pullback-H}, after taking the pullback the $3$-sections $T$, $L_2$ become $1$-sections. Now we can contract degenerate fibers to $0$-curves such that the images of the last three fibers (called \enquote{good} in the proof) are not contained in the boundary. This way we get $3$ horizontal and $1$ vertical line on $\P^1\times \P^1$. Restricted to the open part, our map becomes a dominant (non-separable) morphism from a cylinder $\A^{1}\times \P^1\setminus \{3 \mbox{ points}\}$ to $X\setminus D$, so the latter is $\A^1$-uniruled, as needed. 
	
	Similar description for Platonic $\A^{1}_{*}$-fiber spaces in non-special field characteristics is given in \cite[\sec 7.5]{Miy_Tsu-opendP}.
\end{example}

We end this section with the construction of exceptional log del Pezzo surfaces from Proposition \ref{prop:MT_smooth}\ref{item:MT-ht=2-cha=2}, \ref{item:MT-exception}, see Figure \ref{fig:MT_exceptions}. They occur only if $\cha\kk=2$.

\begin{example}[Exceptions of height $2$ in $\cha\kk=2$ from Proposition {\ref{prop:MT_smooth}\ref{item:MT-ht=2-cha=2}}, see Figure \ref{fig:MT-ht=2_cha=2}]\label{ex:cha=2_ht=2}\
	
	Assume $\cha\kk=2$. Let $\cc$ be a smooth conic, and let $\ll_{1},\dots, \ll_{\nu}$ for some $\nu\geq 3$ be lines tangent to $\cc$: since $\cha\kk=2$, they all meet at one point, say $p_0\not\in \cc$. Write $\{p_j\}=\ll_j\cap \cc$ Blow up at $p_0$ and, for each $j\in \{1,\dots,\nu\}$, at $p_j$ and its infinitely near point on the proper transform of $\cc$. Denote the resulting morphism by $\phi\colon Y\to \P^2$, let $C_Y=\phi^{-1}_{*}\cc$, let $A_j$ be the $(-1)$-curve in $\phi^{-1}(p_j)$, and let $D_Y=\phi^{*}(\cc+\sum_{j=1}^{\nu}\ll_j)\redd-A_0-\sum_{j=2}^{\nu}A_j$. Then the pencil of lines through $p_0$ induces a $\P^1$-fibration of height $2$ of $(Y,D_Y)$ as in Figure \ref{fig:MT-ht=2_cha=2}, i.e.\ such that $D_Y=\langle 1;[2],[2],[2\nu-4]\rangle+2\nu\cdot [2]$, $(D_Y)\hor=[2\nu-4]$ is a $2$-section, and all degenerate fibers are supported on chains $[2,1,2]$, where the $(-2)$-tips are contained in $D_Y$ and the $(-1)$-curve meets $(D_Y)\hor$ once. This way, we get one example of a log surface as in \ref{prop:MT_smooth}\ref{item:MT-ht=2-cha=2}. 
	
	The remaining ones are constructed by blowing up further within the fibers, as follows. Write $\{q_j\}=C_Y\cap A_j$. Let $\psi\colon X\to Y$ be a composition of blowups over $q_2,\dots, q_{\mu}$ for some $\mu\leq \nu$ such that each $\psi^{-1}(q_j)$ is a chain with a unique $(-1)$-curve $A_j'$, and for $j\geq 3$ this $(-1)$-curve meets $\psi^{-1}_{*}C_Y$. Let $D=\psi^{*}(D_Y+\sum_{j=2}^{\mu}A_j)\redd-\sum_{j=2}^{\mu}A_j'$. Then $\rho(X)=\#D$ and $(X,D)$ is as in Proposition \ref{prop:MT_smooth}\ref{item:MT-ht=2-cha=2}. In particular, $\height(X,D)\leq 2$. If the inequality was strict then by Lemma \ref{lem:fibrations-Sigma-chi} some fiber of a witnessing $\P^1$-fibration would be contained in $D$: this is impossible since $D$ is snc-minimal and contains no $0$-curves.
	
	Let $(X,D)\to (\bar{X},\bar{D})$ be the contraction of all components of $D$ except the $(-1)$-curve. Then $(\bar{X},\bar{D})$ is dlt, $\rho(\bar{X})=1$, and using Lemma \ref{lem:ld_formulas}\ref{item:ld_twig} we compute that $-(K_{\bar{X}}+\bar{D})$ is ample: indeed, we have $\bar{D}\cdot (K_{\bar{X}}+\bar{D})= A_{1}'\cdot (K_{X}+A_{1}'+\frac{1}{2}T_{1}'+\frac{1}{2}T_{2}'+(1-\frac{1}{d})D\hor)<0$, where $A_{1}'=[1]$ is the proper transform of $\bar{D}$, $T_i'=[2]$ are the vertical twigs of $D$ meeting $A_{1}'$, and $d\geq 1$ is the discriminant of the twig of $D$ containing $D\hor$. Thus $-(K_{\bar{X}}+\bar{D})$ is ample, as needed. We note that the underlying surface $\bar{X}$ is a del Pezzo surface of rank one and height 2, obtained in \cite[Lemma 5.17(3)]{PaPe_ht_2} for $k_2=2$.
\end{example}

\begin{example}[Exceptions of height $3$ in $\cha\kk=2$ from Proposition \ref{prop:MT_smooth}\ref{item:MT-exception}, see Figures \ref{fig:exception-1} and \ref{fig:exception-2}]\label{ex:exception}\
		
	Assume $\cha\kk=2$. Let $\cc\subseteq \P^{2}$ be a conic, and let $\ll_{1},\ll_{2},\ll_{3}$ be lines tangent to $\cc$: as in Example \ref{ex:cha=2_ht=2}, the assumption $\cha\kk=2$ implies that these lines meet at some point $p_0$. Write $\{p_j\}=\ll_{j}\cap \cc$. Blow up at $p_0$, twice at $p_1,p_2$ and three times at $p_3$, each time on the proper transform of $\ll_j$. Next, blow up further over $p_1$ as follows: once on the proper transform of $\cc$ in case  \ref{prop:MT_smooth}\ref{item:MT-exc-1}; and twice on the proper transform of $\ll_1$ in case \ref{prop:MT_smooth}\ref{item:MT-exc-2}. Denote the resulting morphism by $\phi\colon X\to \P^{2}$, let $A_j$ be the $(-1)$-curve in $\phi^{-1}(p_j)$, let $L_j=\phi^{-1}_{*}\ll_j$, and let $D=\phi^{*}(\cc+\ll_1+\ll_2+\ll_3)\redd-A_1-A_2-A_3$. Then $(X,D)$, together with the $\P^1$-fibration of height $3$ given by the pencil of lines through $p_0$, is as in Proposition \ref{prop:MT_smooth}\ref{item:MT-exception}. 
	
	Let $(X,D)\to (\bar{X},\bar{D})$ be the contraction of $D-A_0$. Then $(\bar{X},\bar{D})$ is dlt; $\rho(\bar{X})=1$; and $-(K_{\bar{X}}+\bar{D})$ is ample: indeed, by Lemma  \ref{lem:ld_formulas}\ref{item:ld_twig}  we have $\bar{D}\cdot (K_{\bar{X}}+\bar{D})=A_0\cdot(K_{X}+A_0+\cf(L_1)\, L_1+\frac{1}{2}L_2+\frac{2}{3}L_3)=\cf(L_1)-\frac{5}{6}<0$, as $\cf(L_1)=\frac{3}{4}$ or $\frac{4}{5}$ in case \ref{prop:MT_smooth}\ref{item:MT-exc-1} and \ref{item:MT-exc-2}, respectively. 
	We note that the underlying surfaces $\bar{X}$ are a del Pezzo surfaces of rank one and height 2, listed in \cite[Lemma 5.17]{PaPe_ht_2} as follows: case \ref{prop:MT_smooth}\ref{item:MT-exc-1} corresponds to (2) with parameters $\nu=k_{2}=3$, $k_{3}=l=3$, and case \ref{prop:MT_smooth}\ref{item:MT-exc-2} to (4) with $\nu=l_2=3$, $l_3=k_3=2$.
	
	We claim that $\height(X,D)=3$. Suppose the contrary, and fix a $\P^{1}$-fibration of height at most $2$. Since $\height(\bar{X})=2$, the curve $A_0$ lies in some fiber $F_{0}$. Since $A_0$ meets $L_{1},L_{2},L_{3}$, Lemma \ref{lem:degenerate_fibers} shows that $F_0\cdot (L_{1}+L_{2}+L_{3})\geq 2\geq F_{0}\cdot D_0$ by assumption, so $F_0\cdot (L_{1}+L_{2}+L_{3})=2$ and $D-(L_1+L_2+L_3)$ is vertical. Pulling back $-K_{\P^{2}}=\ll_1+\ll_2+\ll_3$ we compute that $-K_{X}=L_{1}+L_{2}+L_{3}+2A_0-\epsilon A_1$, where $\epsilon=3,1$ in case \ref{prop:MT_smooth}\ref{item:MT-exc-1} and \ref{item:MT-exc-2}, respectively. By adjunction $2=-K_{X}\cdot F_{0}=2-\epsilon A_{1}\cdot F_0$, so $A_1$ is vertical. Hence the vertical divisor $D-(L_1+L_2+L_3+A_0)+A_1$ contains a chain $[2,2,1,3,2,2]$ in case \ref{prop:MT_smooth}\ref{item:MT-exc-1}, or a fork $\langle 3;[2],[2,1],[2,2,2]\rangle$ in case  \ref{prop:MT_smooth}\ref{item:MT-exc-2}. This is a contradiction, since the latter are not negative semi-definite. 
\end{example}

\section{Proof of Proposition \ref{prop:MT_smooth}: description of almost minimal models}\label{sec:MT_proof}

In this section we prove Proposition \ref{prop:MT_smooth}. We keep notation from its statement, so $(\bar{X},\bar{D})$ is a minimal dlt surface of negative Kodaira dimension, and $\pi\colon(X,D)\to (\bar{X},\bar{D})$ is its minimal log resolution.

We can and do assume that $X\not\cong \P^2$ and $\height(X,D)\geq 2$, as otherwise cases \ref{prop:MT_smooth}\ref{item:MT-P2} or \ref{prop:MT_smooth}\ref{item:MT-ht=1} hold. 
We now list some consequences of this assumption and introduce notation for the remaining part of the proof. 

\begin{lemma}[Some elementary properties of $(\bar{X},\bar{D})$, cf.\ {\cite[II.3.14.5]{Miyan-OpenSurf}}]\label{lem:Miy_computation} The following hold.
	\begin{enumerate}
		\item \label{item:Miy-del-Pezzo} The dlt log surface $(\bar{X},\bar{D})$ is a log del Pezzo surface of rank one, i.e.\ $-(K_{\bar{X}}+\bar{D})$ is ample and $\rho(\bar{X})=1$. 
		\item \label{item:Miy-n} Let $n\de \#\bar{D}$. Then $n=\#D-\rho(X)+1$.
		\item \label{item:Miy-R-connected} The divisor $R\de \pi^{-1}_{*}\bar{D}$ is connected. We denote by $D_0$ the connected component of $D$ containing $R$. 
		\item \label{item:Miy-peeling} The exceptional divisor of $\pi$ is the sum of $D-D_0$ and of all maximal admissible twigs of $D_0$. 
		\item \label{item:Miy_DR} One of the following holds:
		\begin{enumerate}
			\item\label{item:Miy_chain} $n\in \{1,2\}$ and $D_0$ is a rational chain,
			\item\label{item:Miy_fork} $n=1$ and $D_0$ is a rational fork with admissible twigs such that $\delta_{D_{0}}>1$, $\beta_{D_{0}}(R)=3$ and $R^2\geq -1$.
		\end{enumerate}
	\end{enumerate}  
\end{lemma}
\begin{proof}
	\ref{item:Miy-del-Pezzo} Since $(\bar{X},\bar{D})$ is minimal and $\kappa(K_{\bar{X}}+\bar{D})=-\infty$, there is a morphism $f\colon \bar{X}\to B$ of Picard rank one, such that $\dim B\leq 1$ and $-(K_{\bar{X}}+\bar{D})$ is $f$-ample, see Section \ref{sec:singularities}. We need to prove that $\dim B=0$. Suppose $\dim B=1$. Then a fiber $F$ of $f$ satisfies
		$0>F \cdot (K_{\bar{X}}+\bar{D})=2p_{a}(F)-2+F \cdot \bar{D}$  
	so $p_{a}(F)=0$ and $F\cdot \bar{D}\leq 1$. Thus $f\circ\pi$ is a $\P^1$-fibration of height one with respect to $D$, so $\height(X,D)\leq 1$, a contradiction.
	
	\ref{item:Miy-n} We have $\bar{D}=\pi_{*}D$, so $\#D-\rho(X)=\#\bar{D}-\rho(\bar{X})=n-1$ by \ref{item:Miy-del-Pezzo} and by the  definition of $n$.
	
	\ref{item:Miy-R-connected} 
	We have $\rho(\bar{X})=1$, so all components of $\bar{D}$ intersect each other. In particular $\bar{D}$ is connected. Since $(\bar{X},\bar{D})$ is dlt, all singular points of $\bar{D}$ are snc, and lie in $\bar{X}\reg$. Thus $R\cong \bar{D}$ 
	is connected, too, as claimed.
	
	\ref{item:Miy-peeling} This follows from the description of peeling for reduced $D$, see  \cite[II.3.5.2]{Miyan-OpenSurf} or \cite[Lemma~4.7(1)]{Palka_almost_MMP}.
	
	\ref{item:Miy_DR} We have $n=\#\bar{D}>0$ by assumption. Let $C$ be a component of $R$, and let $T_1,\dots, T_v$ be the connected component of $\Exc\pi$ meeting $C$. By \ref{item:Miy-peeling} they are admissible twigs of $D$. By Lemma \ref{lem:ld_formulas}\ref{item:ld_twig} the coefficient in $\pi^{*}(K_{\bar{X}}+\bar{D})$ of the component of $T_{i}$ meeting $C$ equals $1-\frac{1}{d(T_i)}$. By \ref{item:Miy-del-Pezzo} the divisor  $-(K_{\bar{X}}+\bar{D})$ is ample, so 
	\begin{equation*}
		\begin{split}
			0&\ >\pi(C)\cdot (K_{\bar{X}}+\bar{D})= C\cdot(K_{X}+C)+C\cdot (R-C)+\sum_{i=1}^{v}\left(1-\frac{1}{d(T_i)}\right)=\\
			&=2p_{a}(C)-2+\beta_{R}(C)+v-\sum_{i=1}^{v}\frac{1}{d(T_i)} 
			\geq 2p_{a}(C)-2+\beta_{R}(C)+\frac{1}{2}v.
		\end{split}
	\end{equation*}
	It follows that $p_{a}(C)=0$ and either $\beta_{R}(C)=1$, $v\leq 1$ or $\beta_{R}(C)=0$, $v\leq 3$. In the first case, $n=2$ and $D_0$ is a rational chain, so \ref{item:Miy_chain} holds. In the second case, $n=1$ and either $D_0$ is a rational chain, so \ref{item:Miy_chain} holds, or $D_0$ is a fork, $\beta_{D}(R)=3$, $\delta_{D_{0}}=\sum_{i=1}^{3}\frac{1}{d(T_i)}>1$ and $R^2\geq -1$ since $D_0$ is not negative definite, so \ref{item:Miy_fork} holds.
\end{proof}

\subsection{Case when $D_0$ is a chain}\label{sec:chain}

In this section we consider case \ref{lem:Miy_computation}\ref{item:Miy_chain}, i.e.\ when the connected component $D_0$ of $D$ containing the proper transform $R$ of $\bar{D}$ is a rational chain. Our goal is to prove that $(X,D)$ is  as in Proposition \ref{prop:MT_smooth}\ref{item:MT-columnar}--\ref{item:MT-3-fibers}, so $\height(X,D)=2$ and $X\setminus D$ is $\A^{1}$-ruled by Remark \ref{rem:A1-fibration}. The arguments here are independent of $\cha\kk$. 

By Lemma \ref{lem:Miy_computation}\ref{item:Miy_DR} $\bar{D}$ has $n=1$ or $2$ components. We treat those cases separately in Lemmas \ref{lem:n=1} and \ref{lem:n=2}.

\begin{lemma}\label{lem:n=2}
	Assume $n=2$. Then $(X,D)$ is as in Proposition \ref{prop:MT_smooth}\ref{item:MT-q-columnar} and $D$ contains a non-degenerate fiber, see Figures \ref{fig:MT-q-columnar-nu=0}, \ref{fig:MT-q-columnar}.
\end{lemma}
\begin{proof}
	Let $R_1,R_2$ be the components of $R$. We have $\pi(R_{i})^2>0$ since $\rho(\bar{X})=1$. Denoting by $T^{(j)}$ the $j$-th component of the twig $T$ meeting $R_i$, we have $\pi^{*}\pi(R_i)=R_{i}+\sum_{j=1}^{\#T} c_{j}\, T^{(j)}$, where $c_{j}\de d(\sum_{i=1}^{j-1}T^{(i)})/d(T)<1$. Thus $\pi(R_{i})^2=R_{i}^{2}+c_{\#T}<R_{i}^2+1$. We conclude that $R_{i}^2>\pi(R_{i})^2-1> -1$, so $R_{i}^2\geq 0$.
	
	Suppose $R_i^2\geq 1$ for both $i\in \{1,2\}$. Then $\pi(R_i)^2\geq R_i^2\geq 1$, and $\pi(R_i)^2=R_i^2$ if and only if $R_i$ is a tip of $D$. Since $\pi(R_1)\cdot \pi(R_2)=R_1\cdot R_2=1$ and $\pi(R_1)$, $\pi(R_2)$ are numerically proportional, we get $\pi(R_1)^2=\pi(R_2)^2=1$. Hence $D_{0}=R$, so $\pi(R_i)\subseteq \bar{X}\reg$ for both $i\in \{1,2\}$. Lemma \ref{lem:T-in-smooth-locus}\ref{item:P1} gives $\bar{X}\cong \P^2$, contrary to our assumptions.
	
	Thus we can assume that $R_1^2=0$. Consider the $\P^1$-fibration induced by $|R_1|$. Since by assumption $\height(X,D)\geq 2$, we have $\beta_{D}(R_1)=2$, and $D\hor$ consists of two $1$-sections, namely $R_2$ and some $H_1\subseteq \Exc\pi$. 
	 Write $H_1=[m]$ for some $m\geq 2$. We have $\#D-\rho(X)=n-1=1$ by Lemma \ref{lem:Miy_computation}\ref{item:Miy-n}, so Lemma \ref{lem:delPezzo_fibrations}\ref{item:Sigma} shows that every fiber other than $R_1$ has exactly one component not in $D$. If there are no degenerate fibers then $X=\F_{m}$ and $D=[m,0,-m]$ as in \ref{prop:MT_smooth}\ref{item:MT-q-columnar}, see Figure \ref{fig:MT-q-columnar-nu=0}. Let $F$ be a degenerate fiber. Since $D\vert$ contains no $(-1)$-curves, $F$ has a unique $(-1)$-curve, which therefore has multiplicity at least $2$ in $F$, so it does not meet $D\hor$. Now because $D$ has no circular subdivisor, Lemma \ref{lem:degenerate_fibers}\ref{item:columnar} implies that $F$ is columnar. Since both $H_1$ and $R_2$ are non-branching in $D$, we infer that there is only one degenerate fiber. Let $\eta\colon X\to \F_{m}$ be the contraction of all components of $F$ not meeting $H_1$. Then $\eta(R_2)$ is a $1$-section disjoint from the negative section $\eta(H_1)$, so $\eta(R_2)^2=m$. We have $R_2^2=\eta(R_2)^2-1=1-m$, so again $D$ is as in \ref{prop:MT_smooth}\ref{item:MT-q-columnar}. 
\end{proof}

\begin{lemma}
	\label{lem:n=1}
	Assume $n=1$. 
	Then $(X,D)$ is as in Proposition \ref{prop:MT_smooth}\ref{item:MT-columnar}--\ref{item:MT-3-fibers}, see Figures \ref{fig:MT-columnar-nu=0}, \ref{fig:MT-q-columnar-nu=0}, \ref{fig:open_dP_non-min-3}. 
\end{lemma}
\begin{proof}
	If $R=D_{0}$ then, since $X\not\cong \P^2$, Lemma \ref{lem:T-in-smooth-locus}\ref{item:P1} implies that $X=\F_{m}$ and \ref{prop:MT_smooth}\ref{item:MT-columnar} holds. Thus we can assume $D_{0}\neq R$. Put $k\de R^2\geq 0$. If $R$ is a tip of $D_{0}$ then, since $\height(X,D)\geq 2$, we have $R^2>0$. In this case, we choose a point $p\in R\setminus \Exc\pi$; otherwise we choose $p\in R\cap \Exc\pi$. In any case, let $\eta\colon (Y,B)\to (X,D)$ be a composition of $k$ blowups at $p$ and its infinitely near points on the proper transforms of $R$, where $B=(\eta^{*}D)\redd$. Then the linear system of $F_R\de \eta^{-1}_{*}R$ induces a $\P^1$-fibration of $Y$ such that $B\hor$ consists of two $1$-sections, say $H_1,H_2$, such that $H_2\not\subseteq \Exc\eta$. Let $T_{i}$ be the twig of $B$ meeting $H_i$ and not containing $F_R$. 
	
	\begin{casesp}
	\litem{There is a vertical curve $A\not\subseteq B$, disjoint from $\Exc\eta$, meeting at most one chain $T_{i}+H_{i}$} 
	We have $A\cdot F_R=0$ because $A$ is vertical. Since $A$ meets at most one $T_{i}+H_{i}$, the fiber $F_R$ lies in a rational twig of $A+B$. By Lemma \ref{lem:twig} $Y\setminus (A+B)$ is $\A^1$-fibered, so 
	$\kappa(K_{Y}+A+B)=-\infty$ by  Iitaka's easy addition theorem  \cite[II.1.13.1]{Miyan-OpenSurf}, cf.\ \cite[Corollary 6.14]{Fujita-noncomplete_surfaces}. 
	
	Put $V=\eta(A)$. Since $A$ is disjoint from $\Exc\eta$, the sum $D+V$ is snc. Moreover, $\kappa(K_{X}+D+V)= \kappa(K_{Y}+B+A)=-\infty$, and $\height(X,D+V)\geq \height(X,D)\geq 2$. Let $(\bar{X}',\bar{D}')$ be a minimal model of $(X,D+V)$: it is a dlt log del Pezzo surface of rank one, see Lemma \ref{lem:Miy_computation}\ref{item:Miy-del-Pezzo}. Let $\psi\colon (X,D+V)\to (X',D')$ be the induced morphism onto its minimal log resolution. We have $\#\bar{D}'\geq \#(D+V)-(\rho(X)-\rho(\bar{X}'))=n+1=2$. By Lemma \ref{lem:Miy_computation}\ref{item:Miy_DR} we have $\#\bar{D}'=2$, so $\Exc\psi\subseteq D+V$. Lemma \ref{lem:n=2} shows that  $(X',D')$ is as in Figures \ref{fig:MT-q-columnar-nu=0}, \ref{fig:MT-q-columnar}.
	
	In particular, $D'$ has two components with a nonnegative self-intersection number, say $F'=[0]$ and $H'=[1-m]$ for some $m\geq 2$. Since $D+V$ has exactly one such, namely $R$, we infer that $F'$ or $H'$ contains a base point $p$ of $\psi^{-1}$. In particular, $\psi\neq \id$. Since $\Exc\psi\subseteq D+V$ and $D$ has no $(-1)$-curves, it follows that $V$ is a $(-1)$-curve, and $\{p\}=\psi(V)$ is the unique base point of $\psi^{-1}$. 
	
	We have $\psi^{-1}_{*}(F'+H')=R+C$ for some component $C$ of $D-R$. Then $C^2\leq 0$, so $p\in \psi(C)$,  and $C$ either lies in an admissible twig of $D_0$, or in a connected component of $D-D_0$ which is an admissible chain or fork. In the latter case, $\{p\}=F'\cap H'$ and $D$ is as in \ref{prop:MT_smooth}\ref{item:MT-columnar}, see Figures \ref{fig:MT-columnar-non-min-1}, \ref{fig:MT-columnar-non-min-2}. 
	Consider the former case. If $p\in F'$ then $(X,D)$ is as in \ref{prop:MT_smooth}\ref{item:MT-q-columnar}, see Figures \ref{fig:MT-q-columnar-non-min-1}, \ref{fig:MT-q-columnar-non-min-2}. Assume $p\in H'\setminus F'$. If $\rho(X')>2$ then $(X,D)$ is as in \ref{prop:MT_smooth}\ref{item:MT-3-fibers}, see Figure \ref{fig:open_dP_non-min-3}; and if $\rho(X')=2$ then $(X,D)$ is as in  \ref{prop:MT_smooth}\ref{item:MT-columnar} with $(\nu,m)=(2,2)$, see Figure \ref{fig:MT-columnar-non-min-2}.

\litem{Every vertical curve $A\not \subseteq B$ meets either $\Exc\eta$ or  $T_i+H_i$ for both $i\in \{1,2\}$}
	Since $n=1$, Lemma \ref{lem:Miy_computation}\ref{item:Miy-n} gives $\#B-\rho(Y)=\#D-\rho(X)=n-1=0$. By Lemma \ref{lem:fibrations-Sigma-chi}, there is exactly one fiber $F$ with two components not in $B$. Denote those components by $A_1,A_2$.
	
	If each $A_{j}$ meets $T_{i}+H_{i}$ for both $i\in \{1,2\}$ then $F$ satisfies $F\cdot B\hor>2$ or has a circular subdivisor, which is impossible. Therefore, say, $A_1$ meets $\Exc\eta$. In particular, $\eta\neq \id$.  Recall that $\Exc\eta=[1,(2)_{k-1}]\subseteq H_1+T_1$, and $H_1$ is the first tip of $\Exc\eta$. Our assumption implies that both $A_1$ and $A_2$ meet $H_1+T_1$. If one of them, say $A_1$, meets $H_1$ then $F\cdot H_1\geq (A_1+T_1)\cdot H_1 \geq 2$, which is false. Hence $0\neq T_1\subseteq F$ and both $A_1$ and $A_2$ meet $T_1$. 
	
	Suppose $F'\neq F$ is another degenerate fiber. Then $F'$ has only one component not in $D$, which is its unique $(-1)$-curve. Since $B$ has no circular subdivisors, Lemma \ref{lem:degenerate_fibers}\ref{item:adjoint_chain} implies that $F'$ is columnar. In particular, $F'$ contains $T_1+T_2$, a contradiction since $T_1\subseteq F$. Thus $F$ is the unique degenerate fiber.

	We can assume that $A_1$ is a $(-1)$-curve. Indeed, if $A_2$ is the unique $(-1)$-curve in $F$ then by Lemma \ref{lem:degenerate_fibers}\ref{item:adjoint_chain} $T_1+A_1$ meets the $1$-section $H_1$ in a tip, so $T_1\subseteq \Exc\eta$. In particular, $A_2$ meets $\Exc\eta$, too, so interchanging the roles of $A_1$ and $A_2$ we can guarantee that $A_1$ is a $(-1)$-curve, as needed. 
	
	Let $T$ be the component of $T_1$ meeting $A_1$, so $T$ lies in a $(-2)$-chain $T_1^{\eta}\de \Exc\eta-H_1$. Since $A_{1}+T_{1}^{\eta}$ is a proper subdivisor of $F\redd$, it is negative  definite, so $T$ is one of the tips of $T_{1}^{\eta}$. 
	
	Assume that $T$ meets $H_1$. Then after the contraction of $A_1$, $T$ becomes a $(-1)$-curve of multiplicity one in the image of $F$, hence a tip of that fiber. Thus $A_1$ is a tip of $F$. It follows that $\tilde{F}\de F-A_1+H_1$ is a fiber of another $\P^1$-fibration of $Y$, such that the vertical part of $B$ is contained in $\tilde{F}$; and the horizontal part of $B$ consists of two sections: $F_{R}$ and $H_2$. Lemma \ref{lem:fibrations-Sigma-chi} implies that $\tilde{F}$ is the only degenerate fiber. Since $\Exc\eta\subseteq \tilde{F}\redd$, this $\P^1$-fibration is a pullback of a $\P^1$-fibration of $X$. The fiber $\eta_{*}\tilde{F}$ has a unique $(-1)$-curve, namely $\eta(A_2)$, and meets $D\hor$ in different connected components of $D\cap \eta_{*}\tilde{F}$. Lemma \ref{lem:degenerate_fibers}\ref{item:adjoint_chain} implies that $\eta_{*}\tilde{F}$ is columnar. 
	Let $\upsilon \colon Y\to \F_{m}$ be the contraction of all vertical curves not meeting $H_2$. Then $\upsilon(H_2)$ is the negative section, and $\upsilon(F_R)$ is a $1$-section meeting $\upsilon(H_1)$ once, so $\upsilon(H_1)^2=m+2$ and thus \ref{prop:MT_smooth}\ref{item:MT-columnar} holds, see Figure \ref{fig:MT-columnar}.
	
	Assume now that $T$ does not meet $H_1$. We have $\eta(A_1)=[0]$, so $\eta(A_1)\cdot D\geq \height(X,D)\geq 2$ by assumption. Suppose $T_1\subseteq \Exc\eta$, so $T$ is a tip of $B$. Then the above inequality implies that $A_1$ meets a component $V$ of $F$ contained in $B-T_1$. Then after contraction of $A+T_1\subseteq F$, the image $H_1$ meets the images of two components of $F$, namely $V$ and $A_2$, so $H_1\cdot F\geq 2$, a contradiction. 
	
	Thus $T_1\not\subseteq \Exc\eta$, so $\beta_{B}(T)\geq 2$. After the contraction of $A_1$ the image of $T$ becomes a $(-1)$-curve, so it is not branching in the image of $F$. It follows that $\beta_{B}(T)=2$ and $A_1\cdot B=1$. Thus $\eta(A_1)=[0]$ meets $D$ in the common point of $R$ and some component $H$ of $\Exc\pi$. Write $H=[m]$. Consider the $\P^1$-fibration induced by $|\eta(A_1)|$. The horizontal part of $D$ consists of two $1$-sections: $R$ and $H$. Lemma \ref{lem:fibrations-Sigma-chi} implies that every degenerate fiber has only one $(-1)$-curve, hence  is columnar and meets $D\hor$ in a tip of $D\vert$. Since $H$ meets at most one component of $D\vert$, there is at most one such fiber. Let $\upsilon\colon X\to \F_{m}$ be the contraction of all vertical curves which are disjoint from $H$. As before, we infer that $\upsilon_{*}D\hor=[m,-2-m]$ and \ref{prop:MT_smooth}\ref{item:MT-columnar}  holds.
	\qedhere\end{casesp} 
\end{proof}

\subsection{Case when $D_0$ is a fork: elementary properties}\label{sec:fork-intro}

By Lemma \ref{lem:Miy_computation}\ref{item:Miy_DR} and the results of Section \ref{sec:chain}, to prove Proposition \ref{prop:MT_smooth} we can and do assume that the connected component $D_0$ of $D$ containing $R\de \pi^{-1}_{*}\bar{D}$ is a fork, with $R$ as a branching component, and admissible twigs $T_1,T_2,T_3$ satisfying $\delta_{D_{0}}\de \sum_{i=1}^{3}\frac{1}{d(T_i)}>1$. Out aim is to prove that $(X,D)$ is a Platonic $\A^{1}_{*}$-fiber space as in Proposition \ref{prop:MT_smooth}\ref{item:MT-Platonic}, or, if $\cha\kk\in \{2,3,5\}$, as in one of the exceptional cases \ref{prop:MT_smooth}\ref{item:MT-ht=2-cha=2}--\ref{item:MT-eliptic-relative}.

In case $\cha\kk=0$ this assertion is proved by Miyanishi and Tsunoda in \cite{Miy_Tsu-opendP}, see \cite[\sec II.5]{Miyan-OpenSurf} for a self-contained account. Below, we essentially follow the same path as in loc.\ cit., simplifying and completing some arguments to cover the case of arbitrary characteristic. We compare these approaches in  Remark \ref{rem:correcing-Miy}. 
Most of the proof is independent of $\cha\kk$, the dependence appears only at the very end of Section \ref{sec:fork:R1}.
\smallskip

An important ingredient of the proof is an application of a two-ray game, which can be summarized as follows. We extract some exceptional divisor $\check{T}$ over $\bar{X}$, see \cite[1.39]{Kollar_singularities_of_MMP}, getting a log terminal surface $\check{X}$ of rank two, such that one extremal ray of $\operatorname{NE}(\check{X})$ is spanned by $\check{T}$. It follows from the cone theorem, see Lemma \ref{lem:dual} below, that the other ray of  $\operatorname{NE}(\check{X})$ is spanned by a curve $\check{A}$ which can be contracted by a morphism of relative Picard rank one. The proof proceeds by analyzing this contraction.

We call the curve $\check{A}$ the \emph{dual of $\check{T}$}. We now list its basic properties.

\begin{lemma}[The dual curve]\label{lem:dual}
	Let $T$ be a component of $D-R$. Let $\check{A}$ be the curve dual to $T$, and let $A$ be its proper transform on $X$. Then one of the following holds. 
	\begin{enumerate}
		\item\label{item:A2=0} $\check{A}^2=0$, and for a general choice of $\check{A}$ we have $A=[0]$, $A$ is disjoint from $D-R-T$ and meets $T$.
		\item\label{item:A2<0} $\check{A}^2<0$, and $A$ is a $(-1)$-curve such that $A+D-R-T$ is negative definite.
	\end{enumerate}
\end{lemma}
\begin{proof}
	Write $\pi=\pi_{T}\circ\tau_{T}$, where $\pi_{T}\colon \check{X}\to \bar{X}$ is the extraction of $T$, so $\tau_{T}\colon X\to \check{X}$ is the contraction of $D-R-T$. Put $\check{R}=\tau_{T}(R)$, $\check{T}=\tau_{T}(T)$ and $\check{D}=\check{R}+\cf(T)\check{T}$, so $\pi_{T}^{*}(K_{\bar{X}}+\bar{D})=K_{\check{X}}+\check{D}$. We have $\check{R}\cdot (K_{\check{X}}+\check{D})=\bar{D}\cdot (K_{\bar{X}}+\bar{D})<0$, so $K_{\check{X}}+\check{D}$ is not nef; and $\check{T}\cdot (K_{\check{X}}+\check{D})=\check{T}\cdot \pi_{T}^{*}(K_{\bar{X}}+\bar{D})=0$, so $\check{T}$ is not $(K_{\check{X}}+\check{D})$-negative. By the cone theorem, see \cite[Theorem 5.5]{Fujino_MMP}, there is a $(K_{\check{X}}+\check{D})$-negative curve $\check{A}\neq \check{T}$ which can be contracted by a morphism $\phi\colon \check{X}\to B$ of relative Picard rank one. In particular, $\check{A}^2\leq 0$. 	Let $A\de (\tau_{T}^{-1})_{*}\check{A}$ be the proper transform of $\check{A}$ on $X$.
	
	Consider the case $\check{A}^2=0$, and take for $\check{A}$ a general fiber of $\phi$. Then $\check{A}\not\subseteq \check{D}$ and $\check{A}$ does not pass through $\Sing \check{X}$, so $A^2=\check{A}^2=0$ and $A\cdot T=\check{A}\cdot \check{T}$. Since $\rho(\check{X})=2$ and $\check{T}$ is not a fiber, we have $\check{A}\cdot\check{T}>0$, as needed.
	
	Consider the case $\check{A}^2<0$. Then $A+\Exc\tau_{T}=A+D-R-T$ is negative definite, in particular $A^2<0$. We have $0>\check{A}\cdot (K_{\check{X}}+\check{D})\geq A\cdot (K_{X}+R)$, so either $A\neq R$ and thus $A\cdot K_{X}<0$, so $A$ is a $(-1)$-curve; or $A=R$ and $A$ is a $(-1)$-curve by Lemma \ref{lem:Miy_computation}\ref{item:Miy_fork}, as needed.
\end{proof}

We now list some simple conditions which imply that $(X,D)$ is a Platonic $\A^{1}_{*}$-fiber space.

\begin{lemma}[Sufficient condition for a Platonic $\A^{1}_{*}$-fiber space] \label{lem:Sigma-0}
	Assume that $X$ admits a $\P^{1}$-fibration of $X$ such that $D\hor$ consists of two $1$-sections, one of which is $R$. Then the other section is a branching component of $D-D_0$; and $(X,D)$ is a Platonic $\A^{1}_{*}$-fiber space as in Proposition  \ref{prop:MT_smooth}\ref{item:MT-Platonic}, see Figure \ref{fig:MT-Platonic}.
\end{lemma}	
\begin{proof} 
	Put $H=D\hor-R$. By assumption, at least two twigs of $D_0$ are vertical, say $T_1,T_2\subseteq D\vert$. Let $F_i$ for $i\in \{1,2\}$ be the fiber containing $T_i$. By Lemma \ref{lem:delPezzo_fibrations}\ref{item:Sigma} $F_i$ has exactly one component not in $D\vert$, say $L_i$, which is its unique $(-1)$-curve. By Lemma \ref{lem:degenerate_fibers}\ref{item:unique_-1-curve} $F_i$ is a chain of type $[T_i^{*},1,T_i]$, where $T_{i}^{*}$ is a connected component of $D\vert$ meeting $H$ in the last tip, and not meeting $R$. If $H\subseteq T_3$ then $T_3$ contains two connected components as above, namely $T_{1}^{*}$ and $T_{2}^{*}$, which is impossible. Therefore,  $H\not\subseteq D_0$. In particular, $T_{3}$ is vertical, too, so $H$ meets three twigs of $D_0$, namely $T_1^{*}$, $T_{2}^{*}$ and $T_{3}^{*}$. The connected component of $D$ containing $H$ contracts to a log terminal singularity, so it is an admissible fork $H+\sum_{i}T_{i}^{*}$. If $F$ is a degenerate fiber other than the three fibers $[T_{i}^{*},1,T_{i}]$ found above, then by Lemma \ref{lem:delPezzo_fibrations}\ref{item:Sigma} $F$ has a unique $(-1)$-curve, meeting both $1$-sections $H$ and $R$: this is impossible as such $(-1)$-curve has multiplicity at least $2$ in $F$. We conclude that all the remaining fibers are smooth, so $(X,D)$ is a Platonic $\A^{1}_{*}$-fiber space, see Proposition \ref{prop:MT_smooth}\ref{item:MT-Platonic}, as needed.
\end{proof} 

The condition $\delta_{D_0}>1$ implies that at least one of the twigs $T_1,T_2,T_3$ is of type $[2]$, cf.\ Section \ref{sec:log_surfaces}. Say that $T_3=[2]$ and put $T\de T_3$. 

\begin{lemma}[Bounds on $\cf(C)$]\label{lem:cf}
	 Let $C$ be a component of $D$, and let $\cf(C)$ be the coefficient of $C$ in $\pi^{*}(K_{\bar{X}}+\bar{D})-K_{X}$, see formula \eqref{eq:discrepancy}. Then the following hold.
\begin{enumerate}
	\item\label{item:cf-lt} We have $\cf(R)=1$ and $\cf(C)\in [0,1)$ for every component $C$ of $D-R$. 
	\item\label{item:cf-T} We have $\cf(T)=\frac{1}{2}$.
	\item\label{item:cf-Ti} Let $C$ be a component of $T_i$. Then $\cf(C)\in [\frac{1}{d(T_i)},1-\frac{1}{d(T_i)}]$ and $\cf(C)=1-\frac{1}{d(T_i)}$ if and only if $C$ meets $R$. 
	\item\label{item:cf-A} Let $A$ be a $(-1)$-curve on $X$. Then $\sum_{C}\cf(C)\, C\cdot A<1$. In particular, if $A\neq R$ then $A$ is disjoint from $R$. 
\end{enumerate}	
\end{lemma}
\begin{proof}
	Part \ref{item:cf-lt} follows from the fact that $R=\pi^{-1}_{*}\bar{D}$, and $\pi$ is a minimal log resolution of a dlt log surface $(\bar{X},\bar{D})$. Parts \ref{item:cf-T}, \ref{item:cf-Ti} follow from Lemma \ref{lem:ld_formulas}\ref{item:ld_twig}. Part \ref{item:cf-A} follows from ampleness of $-(K_{\bar{X}}+\bar{D})$: indeed, since $A\not\subseteq \Exc\pi$ we have $0>\pi(A)\cdot (K_{\bar{X}}+\bar{D})=-1+\sum_{C}\cf(C)\, C\cdot A$ by adjunction.
\end{proof}

\subsection{Case when $D_0$ is a fork and $R^2\geq 0$}\label{sec:fork-R0}

Recall from Lemma \ref{lem:Miy_computation}\ref{item:Miy_DR} that $R^2\geq -1$. In this section we settle the case when the inequality is strict. 

\begin{lemma}\label{lem:R0}
	Assume that $R^2\geq 0$. Then $(X,D)$ is a  Platonic $\A^{1}_{*}$-fiber space, see Proposition \ref{prop:MT_smooth}\ref{item:MT-Platonic}. 
\end{lemma}
\begin{proof} 
Let $A$ be the proper transform of the curve dual to $T$, see Lemma \ref{lem:dual}. Suppose $A^2=0$, so case \ref{lem:dual}\ref{item:A2=0} holds. Consider the $\P^{1}$-fibration induced by $|A|$. Since $A$ is disjoint from $D-R-T$, the twigs $T_{1},T_2$ are vertical. Since by assumption $R^2\geq 0$, it follows that $R$ is horizontal. We have $0>\pi(A)\cdot K_{\bar{X}}=A\cdot (K_{X}+R+\frac{1}{2}T)$, so $A\cdot (R+\frac{1}{2}T)<2$. Thus $A\cdot R=A\cdot T=1$, so $D\hor$ consists of $1$-sections $R$ and $T$; contrary to Lemma \ref{lem:Sigma-0}. 

	Thus case  \ref{lem:dual}\ref{item:A2<0} holds, so $A=[1]$ and $A+D-R-T$ is negative definite. Since $A+D-R-T$ is not negative definite, $A$ meets $T$. Thus by Lemma \ref{lem:cf}\ref{item:cf-A} we have $A\cdot R=0$, $A\cdot T=1$ and $\sum_{C\neq T}\cf(C)\, C\cdot A<\frac{1}{2}$.
	
	By Lemma \ref{lem:cf}\ref{item:cf-Ti} the coefficient of each component of $T_i$ equals at least $\frac{1}{d(T_i)}$. The above inequality and the condition $\delta_{D_{0}}>1$ imply that $A$ does not meet both $T_1$ and $T_2$. Say that $A\cdot T_1=0$. We consider two cases. 
	
	\begin{casesp}
	\litem{$A$ is disjoint from $T_2$}\label{case:CT2=0} 			
	Since $A+(D-R)$ is not negative definite, $A$ meets $D-D_0$. 
	
	Suppose that $A$ meets no $(-2)$-curve in $D-T$. Then every component $C$ of $D-D_0$ meeting $A$ has self intersection number at most $-3$. Lemmas \ref{lem:Alexeev} and \ref{lem:ld_formulas}\ref{item:ld_chain} yield $\cf(C)\geq \cf_{C}(C)=1-\frac{2}{-C^2}\geq \frac{1}{3}$, so Lemma \ref{lem:cf}\ref{item:cf-A} implies that $A$ meets $D-T$ once, in a $(-3)$-curve $C$ such that $\cf(C)<\frac{1}{2}$. Applying Lemma \ref{lem:Alexeev} we infer from the latter inequality that the connected component of $D$ containing $C$ is a chain $[3,2,\dots,2]$. But then $A+(D-R)$ is negative definite, a contradiction.	
	
	Thus $A$ meets some $(-2)$-curve in $D-D_0$, call it $C$. Since $A+(D-R-T)$ is negative definite, $C$ and $T$ are the only $(-2)$-curves in $D$ meeting $A$; and if $C$ is not a tip of $D$ then $C$ meets two components of $D$ of self-intersection at most $-3$. In the latter case, by Lemma \ref{lem:Alexeev} the coefficient of $C$ equals at least the coefficient of the middle component in the chain $[3,2,3]$, which  equals $\frac{1}{2}$ by Lemma \ref{lem:ld_formulas}\ref{item:ld_chain}: thus $\cf(C)\geq \frac{1}{2}$, contrary to Lemma \ref{lem:cf}\ref{item:cf-A}. We conclude that $C$ is a $(-2)$-tip of $D-D_0$. 
	
	If $A\cdot D=2$ then the linear system $|C+2A+T|$ induces a $\P^{1}$-fibration like in Lemma \ref{lem:Sigma-0}, as needed. Suppose $A\cdot D\geq 3$, so $A$ meets some component $V$ of $D-D_0-C$. Recall that since $A+(D-R-T)$ is negative definite, we have $V^2\leq -3$. By Lemma \ref{lem:cf}\ref{item:cf-A} we have $\cf(V)<\frac{1}{2}$, so as before we see that such $V$ is unique, and it is a tip of a connected component $D_{V}=[3,2,\dots,2]$ of $D$. Lemma \ref{lem:cf}\ref{item:cf-A} implies that $\cf(C)<\frac{1}{2}-\cf(V)\leq \frac{1}{6}$. 
	
	Let $D_C$ be the connected component $D_C$ of $D$ containing $C$. Suppose $C\neq D_C$. If $D_C=D_V$ then, since $C$ is a tip of $D$, we have $\cf(C)+\cf(V)=1$, which is impossible. Now, since $D_C\neq D_V$ and $D_C+A+D_V$ is negative definite, the component of $D_C$ meeting $C$ is not a $(-2)$-curve, so by Lemma \ref{lem:Alexeev}  $\cf(C)\geq \frac{1}{5}>\frac{1}{6}$; a contradiction.
	
	Thus $C=D_C$. The linear system $|C+2A+T|$ induces a $\P^1$-fibration such that $D\hor$ consists of a $1$-section $R$ and a $2$-section $V$. By Lemma \ref{lem:delPezzo_fibrations}\ref{item:Sigma} each degenerate fiber contains exactly one component off $D$, which is its unique $(-1)$-curve. Since such a $(-1)$-curve cannot meet a $1$-section $R$, we conclude that there are exactly three such fibers, each containing exactly one twig of $D_0$.
	There is a morphism $\psi\colon X\to\F_{m}$ contracting those fibers to $0$-curves which is an isomorphism near $R$. Then $\psi(R)$ is a $1$-section with a nonnegative self-intersection number; and $\psi(V)$ is a $2$-section disjoint from $\psi(R)$. This contradicts numerical properties of the Hirzebruch surface $\F_{m}$. Indeed, denoting by $\Sec_m=[m]$ a negative section, we see that $\psi(R)-\Sec_m$ and $\psi(V)-2\Sec_m$ are numerically equivalent to some multiples of a fiber, hence $0=(\psi(R)-\Sec_m)\cdot (\psi(V)-2\Sec_m)=-\Sec_m\cdot (\psi(V)+2\psi(R))-2m\leq 0$, so $m=0$, i.e.\ $\F_m=\P^1\times \P^1$, and $\psi(V)\cdot \Sec_0=0$, so $\psi(V)$ consists of two horizontal lines, a contradiction.

	\litem{$A$ meets some component $C$ of $T_2$}\label{case:CT2-nonzero} 
	By Lemma \ref{lem:cf}\ref{item:cf-A},\ref{item:cf-T} we have $\cf(C)<1-\cf(T)=\frac{1}{2}$, so by Lemma \ref{lem:cf}\ref{item:cf-Ti} $C$ does not meet $R$, in particular $\#T_2\geq 2$. 
	
	Suppose $C$ is a $(-2)$-tip of $D$. 
	If $A$ meets a component $V$ of $D-D_0$ then the fact that $\#T_2\geq 2$ and $V+A+T_2$ is negative definite implies that $V^2\leq -4$, so $\cf(V)\geq \cf_{V}(V)=\frac{1}{2}$, which is impossible by Lemma \ref{lem:cf}\ref{item:cf-A}. Thus $A$ meets $D$ only in $C$ and $T$. The linear system $|C+2A+T|$ induces a $\P^1$-fibration of $X$ such that $D\hor$ consists of two $1$-sections, namely $R$ and the second component of $T_2$, a contradiction with Lemma \ref{lem:Sigma-0}. 
	
	Suppose $T_{1}\neq [2]$. Since $\delta_{D_0}>1$, we have $d(T_2)\leq 5$. Since $\#T_2\geq 2$ we get $T_2=[2,3]$, $[(2)_{k}]$ or $[3,2]$. In the first case $C$ is a $(-2)$-tip of $D$, and the same is true in the second case as $A+T_2$ is negative definite: thus the previous paragraph shows that these cases cannot occur. Eventually, if $T_2=[3,2]$ then $\cf(C)=\frac{3}{5}>\frac{1}{2}$, a contradiction with Lemma \ref{lem:cf}\ref{item:cf-A}.

	Thus $T_1=[2]$. Case \ref{case:CT2=0} above shows that, interchanging the roles of $T_1$ and $T$ if needed, we can assume that there is a $(-1)$-curve $A_{1}$ meeting $D$ in $T_1$ and some component $C_1$ of $T_2$ which is disjoint from $R$, and not a $(-2)$-tip of $D$. Let $\eta\colon X'\to X$ be a composition of $R^2$ blowups at $R\cap T_2$ and its infinitely near points on the proper transforms of $R$, let $D'=(\eta^{*}D)\redd$ and let $A,A_{1}',T,T_1'$ and $R'$ be the proper transforms of $A,A_1,T,T_1$ and $R$. Then $|R'|$ induces a $\P^{1}$-fibration of $X'$ such that $D'\hor$ consists of three $1$-sections,  namely $T',T_1'$ and the component $H$ of $T_2'\de (\eta^{*}T_2)\redd$ meeting $R'$. The $1$-section $H$ is not the proper transform of $C$ or $C_1$, since the latter are disjoint from $R$. Thus $T_2'-H+A'+A_1'$ is contained in some degenerate fiber $F$. Since the $(-1)$-curves $A,A_1'$ meet $1$-sections $T,T_1'$, they have multiplicity $1$ in $F$, and by Lemma \ref{lem:fibrations-Sigma-chi} $F-A'-A_1'$ is contained in $D$, so $F=A+(T_2'-H)+A_1'=[1,2,\dots, 2,1]$. Thus $C$ or $C_1$ is a $(-2)$-tip of $D$, a contradiction.
	\qedhere\end{casesp}
\end{proof}	

\subsection{Case when $D_0$ is a fork and $R^2=-1$}\label{sec:fork:R1}

By Lemmas \ref{lem:Miy_computation}\ref{item:Miy_DR} and \ref{lem:R0} above we can assume $R^2=-1$. We study several cases, depending on the shape of $D_0$. The core of the argument is still independent of $\cha\kk$, but once we restrict the possible shapes of $D$ it will be clear that some of them occur only for special values of $\cha\kk$. 

We have $D_0=\langle 1,T_1,T_2,[2]\rangle$. For $i\in \{1,2\}$ we denote by $C_{i}$ the component of $T_{i}$ meeting $R$.

\begin{lemma}[Case {$D_0=\langle 1;[\dots, 2],T_2,[2]\rangle$}]\label{lem:-2-tips}
	Assume $R=[1]$ and $C_{1}=[2]$. Then $(X,D)$ is either a Platonic $\A^{1}_{*}$-fiber space, see \ref{prop:MT_smooth}\ref{item:MT-Platonic}, or $\cha\kk=2$ and $(X,D)$ is as in \ref{prop:MT_smooth}\ref{item:MT-ht=2-cha=2} or \ref{item:MT-exc-1}, see Figures \ref{fig:MT-ht=2_cha=2}, \ref{fig:exception-1}.
\end{lemma}
\begin{proof}
	Put $F_{1}=C_1+2R+T$, so $(F_{1})\redd=[2,1,2]$. The linear system $|C_{1}+2R+T|$ induces a $\P^{1}$-fibration of $X$ such that $D\hor$ consists of a $2$-section $C_{2}$ and, if $T_{1}\neq [2]$, a $1$-section $H\subseteq T_{1}$ meeting $C_{1}$. If $T_1=[2]$ put $H=0$. Let $F_{2},\dots, F_{\nu}$ be the remaining degenerate fibers; they are not contained in $D$. 
	
	We contract these fibers to $0$-curves in such a way that the image of $C_2$ is smooth and disjoint from the image of $H$: to do this, we contract inductively $(-1)$-curves in the images of the fiber which meet $H+C_2$ at most once, and if there is no such then the fiber is smooth. Since the image of $C_2$ is a smooth and rational $2$-section on a Hirzebruch surface $\F_m$, numerical properties of $\F_{m}$ imply that $m\in \{0,1\}$, and we can choose the last contraction so that $m=1$. Now the negative section is disjoint from the image of $C_2$, hence equal to $H$ if $H\neq 0$. Contracting it we get a morphism $\psi\colon X\to \P^{2}$ such that $\cc\de \psi(C_{2})$ is a conic, and $\ll_{i}\de\psi_{*}F_{i}$ for $i\in \{1,\dots, \nu\}$ are lines meeting at some point $p_0\not\in \cc$. 

	\begin{claim*} For each $i\in \{1,\dots, \nu\}$, the line $\ll_i$ is tangent to $\cc$ at some point $p_i\neq p_0$. \end{claim*}
	\begin{proof}
	Suppose $\ll_{2}$ meets $\cc$ in two points, say $q_1$, $q_2$. If neither of them is a base point of $\psi^{-1}$ then $L_2\cdot C_2=2$, so $L_2$ cannot be a component of $D$, hence $L_2=[1]$ by Lemma \ref{lem:delPezzo_fibrations}\ref{item:-1_curves}, and $1>2\cf(C_2)\geq 1$ by Lemma \ref{lem:cf}\ref{item:cf-A},\ref{item:cf-Ti}; a contradiction.  Say that $q_2$ is a base point of $\psi^{-1}$, and let $A_2\subseteq F_2$ be a $(-1)$-curve in $\psi^{-1}(q_2)$. 
	
	Suppose $A_2$ is the unique $(-1)$-curve in $F_2$. Then $L_2$ is the component of $T_2$ meeting $C_2$, so $(F_2-L_2)\cdot C_2=1$. Since $T_2$ lies in a twig of $D$ and meets both $R$ and $L_2$, it follows that $C_2$ meets $F_2-L_2$ off $D$, in a component of multiplicity one, but there is no such, a contradiction.
	
	Thus by Lemma \ref{lem:delPezzo_fibrations}\ref{item:Sigma} we have $H\neq 0$, the fiber $F_2$ has exactly two $(-1)$-curves, and each fiber $F_j$ for $j\geq 3$ has exactly one. Let $A_0$ be a $(-1)$-curve in $\psi^{-1}(p_0)$, so since $H\neq 0$, we have $A_0\subseteq F_j$ for some $j\geq 2$. If, say, $A_0\subseteq F_3$ then $A_3$ is a unique $(-1)$-curve in $F_3$, so $L_3\subseteq D$ and $\ll_3\cap \cc$ is not a base point of $\psi^{-1}$, which means that  $L_3\cdot C_3=2$, which is impossible for two components of $D$. Hence $A_0\subseteq F_2$, so the $(-1)$-curves in $F_2$ are $A_0$ and $A_2$. In particular, $L_2$ is a component of $D$, and $q_1$ is not a base point of $\psi^{-1}$, so $L_2$ meets $C_2$. Now $A_2$ meets the twig $T_2$ in $C_2$ and in some component of $T_2-C_2\subseteq L_2$, say $C$. By Lemma \ref{lem:cf}\ref{item:cf-A},\ref{item:cf-Ti} we have $1>\cf(C_2)+\cf(C)\geq 1$, a contradiction.
	\end{proof}
	
	Assume $\nu\leq 2$. Let let $s_j$ be the number of blowups within $\psi$ which are not isomorphisms near the common point of the images of $F_j$ and $C_2$. Since $\sum_{j}s_{j}=\cc^2-C_2^2\geq 6$, we have $\nu=2$ and $s_2\geq 4$. Thus the pencil of conics tangent to $\cc$ at $p_1,p_2$ pulls back to a $\P^1$-fibration of $X$ as in Lemma \ref{lem:Sigma-0}, with $1$-sections being the proper transforms of second exceptional curves over $p_1$ and $p_2$. We conclude that $(X,D)$ is a Platonic $\A^{1}_{*}$-fiber space.
	
	Thus we can assume $\nu\geq 3$; note that in this case $\cha\kk=2$.	Say that the chain $T_2-C_2$ lies in a fiber $F_2$.
	
	Assume that $T_{1}=[2]$, so $H=0$. Then by Lemma \ref{lem:delPezzo_fibrations}\ref{item:Sigma} each $F_{i}$ has exactly one $(-1)$-curve, say $A_{i}$, and hence meets $C_{2}$ either in $A_i$ or, in case $i=2$, in the chain $T_{2}-C_{2}$. It follows from Lemma \ref{lem:degenerate_fibers}\ref{item:adjoint_chain} that each $F_{i}$ is as described in \ref{prop:MT_smooth}\ref{item:MT-ht=2-cha=2}, which ends the proof in case $T_2=[2]$. 
		
	Assume now that $T_{1}\neq [2]$, so $C_1$ meets a $1$-section $H$. Interchanging the roles of $T_1$ and $T_2$ if needed, we can assume that $T_2\neq [2]$, too. 
	Suppose that $T_2\neq C_2$, and the fiber $F_2$ containing the chain $T_2-C_2$ has only one $(-1)$-curve, say $A_2$. Then $F_2$ has no component of multiplicity $1$ off $D$, so it meets both $H$ and $C_2$ in components of $D$, of multiplicity $1$ and $2$, respectively. Thus $(F_2)\redd$ is a sum of $A_2$ and chains $(T_1-C_1-H)$, $(T_2-C_2)$, containing tips of multiplicity $1$ and $2$ in $F$, respectively. It follows that $(F_2)\redd=[3,1,2,2]$ or $\langle k;[2],[2],[(2)_{k-2},1]\rangle$. In either case we get $\delta_{D_0}<1$, a contradiction.
	
	Thus we can assume that $F_2\supseteq T_2-C_2$ has at least two $(-1)$-curves. By Lemma \ref{lem:delPezzo_fibrations}\ref{item:Sigma} $F_2$ has exactly two $(-1)$-curves, and each $F_j$ for $j\geq 3$ has exactly one. In particular, for $j\geq 3$ the component of multiplicity $1$ of $F_j$ meeting $H$ is contained in $D$. Since there is at most one such component, we conclude that $\nu=3$, and the chain $T_{1}-C_1-H$ is nonzero and contained in $F_3$. Since $F_3$ does not contain $T_2-C_2$, it  meets $C_2$ in a $(-1)$-curve. It follows that $(F_3)\redd=[2,1,2]$ or $\langle 2;[1],[2],[2]\rangle$, but in the latter case $\#T_1=5$, which together with the assumption $T_2\neq [2]$ leads to a contradiction with the condition $\delta_{D_{0}}>1$. Thus $T_1$ is a $(-2)$-chain $[2,2,2]$. The inequality $\delta_{D_0}>1$ implies that $T_2=[2,2]$ or $[3]$. In the first case, we conclude by interchanging the roles of $T_1$ and $T_2$. 
	
	Consider the second case. Then $D=\langle 1;[2\dec{3},\bs{2},2\dec{3}],[\ub{3}]\dec{3},[2]\rangle+[2]\dec{3}+V$, where $H$ is the middle $(-2)$-curve in $T_1$, $C_2=T_2$, and the $(-1)$-curve in $F_3$ meets $D$ in the components decorated by $\dec{3}$. Moreover, $V$ is a sum of admissible chains and forks such that $(F_2)\redd=U+V+U'$  for some $(-1)$-curves $U,U'$. One of those $(-1)$-curves, say $U$, meets the $1$-section $H$, so $U$ is a tip of $F_2$. Now $\psi$ is the contraction of $U+H$ and all vertical curves not meeting $H$. As before, let $s_j$ be the number of blowups within $\psi$ which are not isomorphisms near the common point of the images of $F_j$ and $C_2$. Then $s_1=s_2=2$, so $s_3=\cc^2-C_2^2-s_1-s_2=3$. It follows that $(F_3)\redd=\langle 2;[1,3],[1],[2]\rangle$, where the first $(-1)$-curve meets $H$, the second $(-1)$-curve meets $C_2$, and the $(-3)$-curve equals $L_3$. Thus $D$ is of the same type as in \ref{prop:MT_smooth}\ref{item:MT-exc-1}. The $\P^1$-fibration described there is induced by the pencil of conics tangent to $\cc$ at $p_1$ and $p_2$; its degenerate fibers are supported on $T_1+U=\langle 2;[2],[2],[1]\rangle$, $C_2+U'+V-L_3=[3,1,2,2]$, and  $G+L+T_3=[2,1,2]$, where $G$ is the $(-2)$-curve in $F_3$ and $L$ is the proper transform of the line joining $p_1$ with $p_2$. We have seen in Example \ref{ex:exception} that $\height(X,D)=3$, as needed.
\end{proof}

By Lemma \ref{lem:-2-tips} we can assume that the tips of $T_1$, $T_2$ meeting $R$ are not $(-2)$-curves. Together with the inequality $\delta_{D_{0}}>1$ this implies that, say, $T_2=[3]$ and $T_1=[2,3]$ or $[k]$ for some $k\in \{3,4,5\}$. We now consider each of these two cases separately. 

\begin{lemma}
	\label{lem:exc-2}
	Assume that $D_0=\langle 1;[2,3],[3],[2]\rangle$. Then $(X,D)$ is as in Proposition \ref{prop:MT_smooth}\ref{item:MT-exc-2}, see Figure \ref{fig:exception-2}.
\end{lemma}
\begin{proof}
	Let $G$ be the $(-2)$-tip of the twig $T_1=[2,3]$, so $\cf(G)=\frac{2}{5}$. Let $A$ be the proper transform of the dual curve of $G$, see Lemma \ref{lem:dual}. Suppose $A^2=0$, so \ref{lem:dual}\ref{item:A2=0} holds. Consider the $\P^{1}$-fibration induced by $|A|$. Since $D_0-G$ is not negative semi-definite, $R$ is horizontal, so $D\hor=R+G$. We have $0>\pi(A)\cdot K_{\bar{X}}=A\cdot (K_{X}+R+\frac{2}{5}G)$, so $A\cdot (R+\frac{2}{5}G)<2$, which implies that $R$ is a $1$-section and $G$ is a $2$-section, as it cannot be a $1$-section by Lemma \ref{lem:Sigma-0}. Hence the components of $D_0$ meeting $R$ lie in different fibers, and have multiplicity $1$ in each. Lemma \ref{lem:delPezzo_fibrations}\ref{item:Sigma} implies that at least two of these fibers have only one $(-1)$-curve each. In particular, there is a fiber $F$ with a unique $(-1)$-curve, say $A$, meeting $R$ in a $(-3)$-curve, say $C$. Such a $(-1)$-curve $A$ has multiplicity at least $3$ in $F$, so it does not meet the $2$-section $G$. Thus $2=F\cdot G=C\cdot G\leq 1$, a contradiction.
	
	Thus by Lemma \ref{lem:dual} $A$ is a $(-1)$-curve and $A+D-R-G$ is negative definite. In particular, $A$ meets $T_1$.
	
	 If $A$ meets the $(-3)$-curve $C_1=T_1-G$ then since $\cf(C_1)=\frac{4}{5}$, Lemma \ref{lem:cf}\ref{item:cf-A} implies that $A\cdot C_1=1$, $A\cdot (D_0-C_1)=0$, and $A$ meets $D-D_0$ in $(-2)$-curves only: indeed if a component of $D$ has self-intersection at most $-3$, then by Lemmas \ref{lem:Alexeev} and \ref{lem:ld_formulas}\ref{item:ld_chain} its coefficient is at least $\frac{1}{3}$, which is bigger than $1-\cf(C_1)$. Since $A+D-R-G$ is negative definite, $A$ meets at most one such $(-2)$-curve, say $V$, and $V$ does not meet any $(-2)$-curves in $D$. On the other hand, since $A+D-R$ is not negative definite, $V$ meets a component of $D$, so by Lemma \ref{lem:Alexeev} $\cf(V)$ equals tat least the coefficient of the $(-2)$-tip of the chain $[2,3]$, which equals $\frac{1}{5}$ by Lemma \ref{lem:ld_formulas}\ref{item:ld_chain}. Thus $\cf(C_1)+\cf(V)\geq 1$, contrary to Lemma \ref{lem:cf}\ref{item:cf-A}.
	
	Thus $A$ meets $D_0$ only in $G$, and $A\cdot G\leq 2$. Suppose $A\cdot G=2$. Let $\phi$ be the contraction of $A+R-T_{1}$, followed by a blowup at the image of $C_1\cap G$. Then the proper transforms of $G$ and $C_1$ are disjoint curves of self-intersection $1$ and $2$, contrary to the Hodge index theorem.
	
	Thus $A\cdot G=1$. Since $A+D-R$ is not negative definite, $A$ meets some component $V$ of $D-D_0$. If $V=[2]$ then $A\cdot V=1$ because $A+D-R-G$ is negative definite, so $|G+2A+V|$ induces a $\P^{1}$-fibration such that $C_1$ is a $1$-section and meets a vertical chain $D-T_1=[2,1,3]$ in the middle component: this is impossible. Thus $V^2\leq -3$. Lemmas \ref{lem:Alexeev} and \ref{lem:ld_formulas}\ref{item:ld_chain} give $\cf(V)\geq 1-\frac{2}{-V^2}\geq \frac{1}{3}$. Since by Lemma \ref{lem:cf}\ref{item:cf-A} we have $\sum_{V}\cf(V)V\cdot A_1<1-\cf(G)=\frac{3}{5}$, such $V$ is unique, meets $A$ normally and has self-intersection $-3$.
	
	Let $(X,D)\to (\hat{X}, \hat{D})$ be the contraction of $D-T_2$. Then $\hat{D}$, the image of $T_2$, satisfies $p_{a}(\hat{D})=2$, $\hat{D}^2=T_{2}^2+10=7$, and is contained in the smooth part of $\hat{X}$. Since $\rho(\hat{X})=1$, the equality $\hat{D}\cdot(K_{\hat{X}}+\hat{D})=2$ implies that $K_{\hat{X}}=-\frac{5}{7}\hat{D}$. The image $\hat{A}$ meets $\hat{D}$ only in the cusp, with multiplicity $2$, so $-\frac{10}{7}=\hat{A}\cdot K_{\hat{X}}=A\cdot (K_{X}-G+\cf(V)V)=-2+\cf(V)$, and therefore $\cf(V)=\frac{4}{7}$. 
	
	The chain $A+D_0-T_2=[1,2,3,1,2]$ supports a fiber $F_0$ of a $\P^{1}$-fibration of $X$ such that $D\hor$ consists of a $1$-section $V$ and and a $2$-section $T_2$. By Lemma \ref{lem:delPezzo_fibrations}\ref{item:Sigma} every degenerate fiber $F\neq F_0$ meets the $2$-section $T_2$ in its unique $(-1)$-curve, which therefore has multiplicity $2$ in $F$. It follows that $F$ is supported on a chain $[2,1,2]$ or $\langle 2;[2],[2],[1,2,\dots, 2]\rangle$, and meets $V$ in a $(-2)$-tip. As a consequence, the connected component $D_{V}$ of $D$ containing $V$ is a sum of $V$ and chains $[2]$, $[2,2,2]$ or forks $\langle 2;[2],[2],[(2)_{k}]\rangle$ meeting $V$ in a tip (in the latter case, this tip lies in the short twig). Using Lemmas \ref{lem:Alexeev} and \ref{lem:ld_formulas} we check directly that the equality $\cf(V)=\frac{4}{7}$ occurs only if $D_{V}=[2,3,2,2,2]$, so there are two degenerate fibers other than $F$, supported on $[2,1,2]$ and $\langle 2;[1,2],[2],[2]\rangle$. Thus the type of $D$ is as in \ref{prop:MT_smooth}\ref{item:MT-exc-2}, and the $(-1)$-curves corresponding to decorations $\dec{3}$, $\dec{1}$ are $A$ and the $(-1)$-curve in the last fiber, call it $A'$. To see the $\P^{1}$-fibration as in \ref{prop:MT_smooth}\ref{item:MT-exc-2}, let $U$, $H$ be the $(-2)$-curves meeting $V$, where $U$ is a tip of $D$. Then the chain $T_1+A+V+U=[3,2,1,3,2]$ supports a fiber of a $\P^1$-fibration such that $D\hor$ consists of a $1$-section $R$ and $2$-section $H$; and the chain $T_2+A'+(D_V-H-V-U)=[3,1,2,2]$  supports another fiber. By Lemma \ref{lem:delPezzo_fibrations}\ref{item:Sigma} the remaining part of $D$, which  consists of two disjoint $(-2)$-curves, lies in one fiber, supported on $[2,1,2]$, as needed. Note that $\cha\kk=2$, since the restriction of this $\P^{1}$-fibration to $H$ is a morphism $\P^1\to \P^1$ of degree $2$, ramified at three points. Example \ref{ex:exception} shows that $\height(X,D)=3$.
\end{proof}

\begin{lemma}
	\label{lem:relative}
	Assume $D_0=\langle 1;[k],[3],[2]\rangle$ for some $k\in \{3,4,5\}$. Then $(X,D)$ is either a Platonic $\A^{1}_{*}$-fiber space as in \ref{prop:MT_smooth}\ref{item:MT-Platonic}, or $\cha\kk\in \{2,3,5\}$ and $(X,D)$ is as in Proposition \ref{prop:MT_smooth}\ref{item:MT-eliptic-relative}.
\end{lemma}
\begin{proof}
	Let $(X,D)\to (\bar{Y},\bar{T})$ be the contraction of $D-T_1$. Then $\rho(\bar{Y})=1$ and $\bar{T}\subseteq \bar{Y}\reg$ is a cuspidal curve with $p_{a}(\bar{T})=1$, so by Lemma \ref{lem:T-in-smooth-locus}\ref{item:elliptic} $\bar{Y}$ is canonical del Pezzo surface of rank one. We have $K_{\bar{Y}}^2=\bar{T}^2=T_{1}^2+6\leq 3$, so by Noether's formula the minimal resolution of $\bar{Y}$ has Picard rank $10-K_{\bar{Y}}^2\geq 7$. Proposition \ref{prop:canonical} shows that $\bar{Y}$ is either of one of the types listed in \ref{prop:MT_smooth}\ref{item:MT-eliptic-relative}, or of type $\rE_{k+3}$. In the first case, using the computation of $\height(X,D)$ given in Proposition \ref{prop:height}, we conclude that $(X,D)$ is as in Proposition \ref{prop:MT_smooth}\ref{item:MT-eliptic-relative}, as needed.
	
	Consider the latter case. Then on the minimal resolution of $\bar{Y}$ there is a $(-1)$-curve $A_{Y}$ meeting the exceptional divisor only once, in the first tip of the longest twig $V_{Y}=[(2)_{k-1}]$, see Figures \ref{fig:E6}--\ref{fig:E8}: this easy  fact can be read off from the structure of a witnessing $\P^{1}$-fibration, see \cite[Remark 4.5]{PaPe_ht_2}. By adjunction, $A_{Y}$ meets the proper transform of $\bar{T}$ once, too. Let $A,V$ be the proper transforms of $A_Y$, $V_Y$ on $X$. Then $V+A+T_1=[(2)_{k-1},1,k]$ supports a fiber of a $\P^{1}$-fibration such that $D\hor$ consists of two $1$-sections, namely $R$ and the branching component of $D-D_0$. It follows from Lemma \ref{lem:delPezzo_fibrations} that all degenerate fibers are columnar, see Lemma \ref{lem:degenerate_fibers}\ref{item:columnar}; so $(X,D)$ is a Platonic $\A^{1}_{*}$-fiber space as in \ref{prop:MT_smooth}\ref{item:MT-Platonic}.
\end{proof}

\begin{remark}[Comparison with {\cite[\sec II.5]{Miyan-OpenSurf}}]\label{rem:correcing-Miy}
	The above proof of Proposition \ref{prop:MT_smooth} in case $D_0$ is a fork essentially follows the steps of original argument \cite{Miy_Tsu-opendP}, which assumes $\cha\kk=0$. It is explained in \cite[\sec II.5]{Miyan-OpenSurf}. We now sketch the comparison between loc.\ cit.\ and the proof above.
	
	Like the proof of Lemma \ref{lem:R0} above, loc.\ cit.\ begins with extracting the twig $T=[2]$ and playing a two-ray game. The argument then splits into two cases depending on whether the other ray is spanned by $R$, or not. 
	
	In each of these cases, Theorems 5.6.1  and 5.4.1 show that $X$ admits a \emph{permissible pencil}, that is, a $\P^{1}$-fibration of $X$ such that the three twigs of $D_0$ lie in different fibers, and the one containing $T$ is supported on a chain $[2,1,2]$, see \sec II.5.2 loc.\ cit. The assumption $\cha\kk=0$ is used exactly once, in the proof of Lemma 5.5.6 loc.\ cit. In this lemma, we have $D_{0}=\langle 1;[2],T_2,[2]\rangle$ for some admissible chain $T_2$. Like in Lemma \ref{lem:-2-tips} above, loc.\ cit.\ considers a $\P^1$-fibration $p\colon X\to \P^1$ with a degenerate fiber supported on $D_0-T_2=[2,1,2]$. Then $D\hor=T$ is a $2$-section. By Lemma \ref{lem:delPezzo_fibrations}\ref{item:Sigma} every degenerate fiber has a unique $(-1)$-curve, so it meets $T$ in a ramification point of $p|_{T}$. Now, the assumption $\cha\kk=0$ is used to deduce that there are at most two such fibers (inequality $t\leq 1$ at the bottom of p.\ 165 loc.\ cit). This is true whenever $\cha\kk\neq 2$, but if $\cha \kk=2$ one gets new possibilities having at least three degenerate fibers, see Proposition \ref{prop:MT_smooth}\ref{item:MT-ht=2-cha=2} and Lemma \ref{lem:-2-tips}.
	
	Next, \sec II.5.2 loc.\ cit.\ shows that, if $\cha\kk=0$, any permissible pencil $\Lambda$ satisfies the assumptions of Lemma \ref{lem:Sigma-0}, so $(X,D)$ is a Platonic $\A^{1}_{*}$-fiber space as in  \ref{prop:MT_smooth}\ref{item:MT-Platonic}. This is done by constructing morphisms $X\to\P^{2}$ which map $\Lambda$ to a pencil of lines, and a multi-section in $D$ to a curve $\cc$, tangent to the image of each  degenerate member of $\Lambda$. The existence of such a curve $\cc$ is then obstructed by Lemma 5.2.5 loc.\ cit., which holds whenever $\cha\kk\nmid \deg\cc$.  However, if $\cha\kk\in \{2,3,5\}$ then the resulting $\cc$ may exist: this leads to  exceptions \ref{prop:MT_smooth}\ref{item:MT-ht=2-cha=2}--\ref{item:MT-eliptic-relative}. 
\end{remark}

\section{Elliptic fibrations and planar cubics}

We now move on to the proof of Proposition \ref{prop:canonical}. As a preliminary step, we set up a tool to produce elliptic fibrations, which will be useful in our forthcoming articles, too. 

Abusing the terminology slightly, we say that a curve $C$ on a smooth projective surface $X$ is \emph{elliptic} if $p_{a}(C)=1$. A \emph{quasi-elliptic fibration} is a morphism $X\to B$ whose general fiber is elliptic; an \emph{elliptic fibration} is a quasi-elliptic fibration whose general fiber is smooth. We note that given an smooth curve $C$ on $X$ with $C^2=0$, the linear system $|C|$ induces a $\P^1$-fibration of $X$ in case $C$ is rational. In case $C$ is elliptic an analogous statement is false, as the linear system may consist only of $C$. We will use the following refinement to construct quasi-elliptic fibrations from certain arrangements of curves. An explicit example is given in Remark  \ref{rem:elliptic_reducible} below. 

\begin{lemma}[Producing elliptic fibrations]
	\label{lem:elliptic}
	Let $X$ be a smooth rational projective surface. Let $E$ be a connected effective divisor on $X$ such that $p_a(E)=1$. Assume that $E$ is irreducible, or that the linear system $|E|$ has no fixed components. Then the following hold.
	\begin{enumerate}
		\item\label{item:elliptic=1} Assume that $E^{2}=1$. Then there is a point $p\in E\reg$ such that after blowing up at $p$, the proper transform of $E$ becomes a fiber of an elliptic or a quasi-elliptic fibration over $\P^1$, which admits a section.
		\item\label{item:elliptic>1} Put $k=E^2$ and $k_{\epsilon}=k-\epsilon$, where $\epsilon=1$ if $\cha\kk=2$ and $E$ is either a rational cuspidal curve, or a sum of two smooth rational curves tangent to each other with multiplicity $2$; and $\epsilon=0$ otherwise.
		
		Assume that $k_{\epsilon}\geq 2$. Fix an open dense subset $U\subseteq E\reg$ and $k_{\epsilon}-2$ points $p_{1},\dots,p_{k_\epsilon-2}\in U$. Then there is a nonempty open subset $V\subseteq U$, which is dense unless $\cha\kk=2$, and has the following property. For any $1+\epsilon$ points $p_{k_\epsilon-1},\dots,p_{k-1}\in V$ there is a point $p_{k}\in U$ such that the points $p_1,\dots,p_{k}$ are pairwise distinct, and after blowing up at $p_1,\dots,p_k$, the proper transform of $E$ is a fiber as in \ref{item:elliptic=1}.
	\end{enumerate}
\end{lemma}
\begin{proof}
	\ref{item:elliptic=1} Since the surface $X$ is rational, the Riemann--Roch theorem gives $\dim |E|\geq \frac{1}{2}(E-K)\cdot E=E^2-p_{a}(E)+1=1$. In particular, if $E$ is irreducible then $E$ cannot be a fixed component of $|E|$. Thus in any case, the linear system $|E|$ has no fixed components, and $\dim |E|\geq 1$.
	
	It follows that $|E|$ has a member $E_0$ which contains no components of $E$. Because $E^2=E\cdot E_0=1$, this member meets $E$ exactly once. Let $\phi\colon X'\to X$ be a blowup at the point $E_0\cap E$. The pencil generated by the proper transforms $\phi^{-1}_{*}E$ and $\phi^{-1}_{*}E_{0}$ provides a morphism $\pi\colon X'\to \P^1$, whose fibers have arithmetic genus one. The fiber $\phi^{-1}_{*}E$ of $\pi$ is connected since $E$ is. It follows that $\pi$ is its own Stein factorization, i.e.\ all fibers of $\pi$ are connected. Thus $\pi$ is a (quasi-)elliptic fibration. The curve $\Exc\phi$ is its rational section.
	
	\ref{item:elliptic>1} 	We can assume that $E^{2}=2+\epsilon$. Indeed, in order to reduce to this case we blow up at $p_1,\dots, p_{k-r}$, and then replace $E$ by its proper transform, and $U$ by the preimage of $U\setminus \{p_1,\dots,p_{k-r}\}$.
	
	We consider first the case $\epsilon=0$; the other case is treated in an analogous way below. Then $E^2=2$. We define a (set-theoretic) mapping $s\colon E\reg\to E\reg$, as follows. Blowing up at $q$ and applying \ref{item:elliptic=1} to the proper transform of $E$, we see that there is a member $E_{q}\in |E|\setminus \{E\}$ passing through $q$. If all such members are tangent to $E$ at $q$, we put $s(q)=q$, otherwise we choose one which is not, and write $E_q\cap E=\{q,s(q)\}$. 
	
	We claim that $s$ is injective. Suppose that for different $q_1,q_2\in E\reg$ we have $s(q_1)=s(q_2)$. Let $X'\to X$ be a blowup at $s(q_1)$, let $E'$ be the proper transform of $E$, and let $H$ be the exceptional curve. Fix $p\in H$. For $i\in \{1,2\}$, the pencil generated by $E'$ and $E_{q_i}$ as above has a member passing through $p$, call it $E_{i}'$. After blowing up at $p$, the pencil generated by the proper transforms of $E_{1}'$ and $E_{2}'$ induces a fibration over $\P^1$, such that the proper transform of $E$ is a $1$-section. Thus $E\cong \P^1$, a contradiction.
	
	Put $Z=\{q\in E\reg: s(q)=q\}$. We claim that $U\setminus Z$ contains a nonempty, open subset $V'$ of $E\reg$, which is dense in $E\reg$ unless $\cha\kk= 2$. Suppose that for some component $C$ of multiplicity one in $E$, the set $Z\cap C$ is infinite. Choose two points $q_1,q_2\in Z\cap C$. Then there are members $E_{i}\in |E|$ such that $(E_{i}\cdot E)_{q_{i}}=2$, $i\in \{1,2\}$. Consider a pencil $\cE\subseteq |E|$ generated by $E_{1}$, $E_{2}$. Resolving base points of $\cE$, we get a fibration which restricts to a morphism $C\to \P^1$ of degree two. By definition of $s$, this morphism is ramified at every point of the infinite set $Z\cap C$. Hence it is totally ramified. It follows that $\cha\kk=2$, and by \cite[Proposition IV.2.5]{Hartshorne_AG} the curve $C$ is rational. Moreover, if $E$ has two local analytic branches at some point $p\in C$, then they are tangent to each other: indeed, otherwise a member $E_{p}\in \cE$ passing through $p$ would satisfy $(E_{p}\cdot E)_{p}\geq 3$, which is impossible. Since $p_{a}(E)=1$, it follows that either $E=C$ and $E$ is a rational, cuspidal curve, or $C$ meets $E-C$ in one point, with multiplicity two. Since $\epsilon=0$, we can always choose a component $C$ of $E$ which is not of this type. This proves the claim.
	
	Since $s\colon E\reg\to E\reg$ is injective, we can choose a set-theoretic mapping $r\colon E\reg \to E\reg$ such that $r\circ s=\id$. Since $U\subseteq E\reg$ is dense, the set $E\reg\setminus U$ is finite. Thus $V\de V'\setminus r(E\reg \setminus U)$ is an open dense subset of $V'$. Hence $V$ is a nonempty, open subset of $E\reg$, dense unless $\cha\kk=2$. For every $q\in V$ we have $s(q)\neq q$ because $q\in V'\subseteq U\setminus Z$, and $s(q)\in U$ since $q=r(s(q))\not\in r(E\reg \setminus U)$. Thus for every $p_{1}\in V$, there is a $p_{2}\de s(p_1)\in U$ such that after blowing up at $p_1,p_2$ the proper transform of $E$ induces a fibration as needed. 
	\smallskip
	
	In case $\epsilon=1$, we define a mapping $(s_1,s_2)\colon E\reg\to E\reg\times E\reg$ in an analogous way, so that after blowing up at $q$, $s_1(q)$ and $s_{2}(q)$, the proper transform of $|E|$ induces a (quasi)-elliptic fibration; and if possible, we have $\#\{q,s_1(q),s_2(q)\}>1$. The set  $\{q\in E: q=s_1(q)=s_2(q)\}$ is finite. Indeed, otherwise a component $C$ of $E$ containing infinitely many points of $Z$ would, like before, admit a totally ramified morphism $C\to \P^1$ of degree three, so $\cha\kk=3$, which is false. Hence there is a point $q\in U$ such that a member of $|E|$ passing through $q$ meets $U\setminus \{q\}$ normally. We blow up at $q$, replace $E$ by its proper transform and $U$ by the preimage of $U\setminus \{q\}$. Then we are back in case $E^2=2$, and the choice of $q$ guarantees that the set $Z$  defined above does not contain the proper transform of the component of $E$ containing $q$. Arguing as before we get the required fibration.
\end{proof}

\begin{remark}[Necessity of additional blowups in Lemma \ref{lem:elliptic}]
	In Lemma \ref{lem:elliptic}\ref{item:elliptic>1} one cannot take $V=U$. Indeed, let $\bar{E}\subseteq \P^2$ be a smooth cubic, and fix pairwise distinct points $q_{1},\dots, q_{7},q\in \bar{E}$ such that $\sum_{i=1}^{7}q_{7}=-2q$ in the group law on $\bar{E}$. Let $\pi\colon X\to \P^2$ be a blowup at $q_{1},\dots, q_{7}$, put $E=\pi^{-1}_{*}\bar{E}$, so $E^{2}=2$. Take $U=E$. Now if $\tau$ is a blowup at  $p_1\de \pi^{-1}(q)$, the unique base point of $|\tau^{-1}_{*}E|$ is infinitely near to $q$, so we cannot choose $p_2\neq p_1$. 
	
	The additional blowup if $\cha\kk=2$ is necessary, too. Indeed, assume $\cha\kk=2$ and let $\bar{E}\subseteq \P^2$ be a cuspidal cubic. It has an inflection point $p\in \bar{E}$; hence $\bar{E}\reg$ admits a structure of a group scheme isomorphic to $\G_{a}$, with $p$ as a neutral element. Blow up seven times over $p$, each time on the proper transform of $\bar{E}$, denote the resulting morphism by $\pi$ and put $E=\pi^{-1}_{*}\bar{E}$. As before, we have $E^2=2$. Take any $p_{1}\in E\reg$ and let $\tau$ be a blowup at $p_1$. Then as before, a unique base point of $|\tau^{-1}_{*}E|$ is infinitely near to $p_1$: indeed, if $E'\in |E|$ meets $E$ at $p_1,p_2\in E\reg$, then $\pi(p_1)+\pi(p_2)=0$ in the group law on $\bar{E}\reg\cong \G_{a}$, hence $p_1=p_2$ since $\cha\kk=2$. 
\end{remark}

\begin{remark}[A typical application of Lemma \ref{lem:elliptic}]\label{rem:elliptic_reducible}
	Let $X$ be a smooth projective surface, and let $E$ be an effective, connected divisor on $X$. Assume that there is a birational morphism $\phi\colon X\to \check{X}$ such that, putting $\check{E}\de \phi_{*}E\subseteq X$, we have $E=\phi^{*}\check{E}$ and $\check{E}$ satisfies the assumptions of Lemma \ref{lem:elliptic}; e.g.\ is a nodal curve with $p_{a}(\check{E})=1$ and $\check{E}^2\geq 3$. Then by Lemma \ref{lem:elliptic}\ref{item:elliptic>1}, after $\check{E}^2$ sufficiently general blowups on the proper transform of $\check{E}$, the proper transform of $|E|$ induces an elliptic fibration. 
	
	For example, take $E=T+2A$, where $T=[3]$, $A=[1]$, $A\cdot T=2$ and $A$ meets $T$ in two points. Then applying Lemma \ref{lem:elliptic}\ref{item:elliptic=1} to the image of $E$ after the contraction of $A$, we infer that there is a point $p\in T\setminus A$ such that after blowing up $p$, the proper transform of $|E|$ induces a (quasi)-elliptic fibration. This proper transform is $T'+2A'$, where $T'=[4]$, $A'=[1]$ and $T'\cdot A'=2$. For more general settings where such a divisor induces a (quasi-)elliptic fibration see \cite[Theorem 3.3]{Kumar-Murthy}. 
\end{remark}

In the proof of Proposition \ref{prop:canonical} we will construct a birational map $\bar{Y}\map \P^2$ such that the proper transform of $\bar{T}$ is a planar cubic, and study the latter using its group law or explicit equations, which we now recall.

\begin{lemma}[Group laws on planar cubics]\label{lem:group_law}
	Let $\qq\subseteq \P^{2}$ be a cubic curve with a chosen inflection point $p\in \qq\reg$. Then we have an isomorphism $\qq\reg\ni q\mapsto q-p\in \mathrm{Cl}^{0}(\qq)$, which endows $\qq\reg$ with a group scheme structure, with neutral element $p$. Moreover, we have $\qq\reg\cong \mathbb{G}_{m}$ if $\qq$ is nodal and $\qq\reg\cong\mathbb{G}_{a}$ if $\qq$ is cuspidal.
\end{lemma}
\begin{proof}
	This result is proved in \cite[Propositions III.2.2 and III.2.5]{Silverman}. See also \cite[Examples II.6.10.2, II.6.11.4 and Exercises II.6.6, II.6.7]{Hartshorne_AG} for the proof in case $\cha\kk\neq 2$.
\end{proof}

\begin{lemma}[Classification of rational planar cubics]\label{lem:cubic}
	Let $\qq\subseteq \P^2$ be a cubic parametrized by $\nu\colon \P^1\to \P^2$. Then for some coordinates $[s:t]$ on $\P^1$ and $[x:y:z]$ on $\P^2$, exactly one of the following holds.
			\begin{parlist}
				\item \label{item:nodal_cubic} $\nu[s:t]=[s^2t:st^2: s^3+t^3]$ and $\qq=\{x^3+y^3=xyz\}$. \\ 
				In this case $\qq$ has a node at $\nu[0:1]=\nu[1:0]=[0:0:1]$ with tangent lines $\{x=0\}$, $\{y=0\}$. Inflection points of $\qq$ are $\nu[\zeta:-1]=[\zeta:-1:0]$ for all $\zeta$ such that $\zeta^3=1$. The  corresponding inflectional tangents are $\{3\zeta x+3\zeta^{2}y+z=0\}$. In particular, $\qq$ has exactly one inflection point if $\cha\kk=3$ and three otherwise. 
				\item \label{item:cuspidal_cubic} $\nu[s:t]=[s^3:s^2 t:t^3]$ and $\qq=\{y^3=x^2 z\}$.\\ 
				In this case $\qq$ has a cusp at  $\nu[0:1]=[0:0:1]$ with tangent line $\{x=0\}$. If $\cha\kk\neq 3$ then $\nu[1:0]=[1:0:0]$ is the unique inflection point, with tangent $\{z=0\}$. If $\cha\kk=2$ or $3$ then $\qq$ is \emph{strange}, i.e.\ there is a point $p\in \P^2$ such that for every $[s:t]\in \P^1$, the line  $\ll_{[s:t]}$ tangent to $\qq$ at $\nu[s:t]$ passes through $p$. If $\cha\kk=2$ then $p=[1:0:0]$ is the inflection point of $\qq$, and $\ll_{[s:t]}=\{s^2z=t^2y\}$. If $\cha\kk=3$ then $p=[0:1:0]$, and each line $\ll_{[s:t]}=\{s^3 z = t^3 x\}$ is an inflectional tangent.
				\item \label{item:funny_cubic} $\cha\kk=3$, $\nu[s:t]=[s^2t:s(s+t)^2:-(s+t)t^2]$ and $\qq=\{x^2y+y^2z+z^2x=0\}$. \\
				In this case $\qq$ has a cusp at $\nu[1:1]=[1:1:1]$ with tangent line $\ll_0\de \{x+y+z=0\}$ and no inflection point. For every $p\in \ll_0$ there is a unique line passing through $p$ and  tangent to $\qq$; and  for every $p_1\in \qq\reg$ there are points $p_2,p_3\in \qq\reg$ and lines $\ll_{1},\ll_2,\ll_3$ such that $(\ll_{i}\cdot \qq)_{p_{i}}=2$, $(\ll_{i}\cdot \qq)_{p_{i+1}}=1$, $i\in \{1,2,3\}$, where $p_4=p_1$.
			\end{parlist}
		\end{lemma}
		\begin{proof} 
			This result is classical, see e.g.\ \cite{Mulay}. We sketch the argument for the convenience of the reader.
			
			 Assume fist that $\qq$ is nodal. We can choose coordinates so that the node is at $\nu[0:1]=\nu[1:0]=[0:0:1]$, with tangent lines $\{x=0\}$ and $\{y=0\}$. Then, after a diagonal change of coordinates on $\P^2$, we have $\nu[s:t]=[s^2t:st^2:p]$ for some cubic form $p\in\kk[s,t]$. After a further  triangular change of coordinates on $\P^2$, we can assume that $p=as^3+bt^3$ for some $a,b\in \kk^{*}$. After a diagonal change of coordinates on $\P^1$ and $\P^2$, we get a parametrization as in \ref{item:nodal_cubic}.
			
			Assume now that $\qq$ is cuspidal; and put that cusp at $\nu[0:1]=[0:0:1]$, with tangent line $\{x=0\}$.
			
			Consider the case when $\qq$ has an inflection point, say $p\in \qq\reg$. We can put $p$ at $\nu[1:0]=[1:0:0]$, with inflectional tangent $\{z=0\}$. Then after diagonal change of coordinates on $\P^2$, we get $\nu$ as in \ref{item:cuspidal_cubic}. 
			
			Consider the case when $\qq$ has no inflection points. Fix coordinates so that $\{z=0\}$ is tangent to $\qq$ at $\nu[1:0]=[1:0:0]$, and meets $\qq$ normally at $\nu[1:1]=[1:1:0]$. Then, after a diagonal change of coordinates, $\nu[s:t]=[s^3:s^2t:t^2(t-s)]$. If $\cha\kk\neq 3$ then the line $\{x=27z+9y\}$ meets $\qq$ only at $\nu[3:1]\in \qq\reg$, so the latter is an inflection point; contrary to our assumption. Thus $\cha\kk=3$. To get a parametrization in \ref{item:funny_cubic}, replace $\nu$ with $\psi\circ \nu \circ \phi$, where $\phi[s:t]=[t-s:s+t]$ and $\psi[x:y:z]=[x-y+z:z:y+z]$.
			\smallskip
			
			The properties of $\qq$ in cases \ref{item:nodal_cubic} and \ref{item:cuspidal_cubic} follow by a direct computation. We  now prove the last statement of \ref{item:funny_cubic}. Let $p_1=\nu[s:t]\in \qq\reg$ for some $[s:t]\in \P^1\setminus \{[1:1]\}$. The tangent line $\ll_1$ to $\qq$ at $p_1$ is given by  $\ll_{1}=\{t(s+t)^2x+st^2y=s^2(s+t)z\}$. We have $\ll_0\cap \ll_1=\{[s:-(s+t):t]\}$. This point uniquely determines $[s:t]\in \P^1$, which proves the first part. Define $\sigma\in \Aut(\P^2)$ by $\sigma[x:y:z]=[z:x:y]$, so $\sigma^3=\id$. The equation of $\qq$ shows that $\sigma$ preserves $\qq$, and the line $\ll_1$ meets $\qq$ at $p_2\de \sigma(p)$. Thus the line $\ell_2$ tangent to $\qq$ at $p_2$ meets $\qq$ at $p_3\de \sigma(p_2)$, and eventually the line $\ll_3$ tangent to $\qq$ at $p_3$ meets $\qq$ at $\sigma(p_3)=\sigma^3(p_1)=p_1$, as needed. 
		\end{proof}
		
		\begin{remark}[A cubic without a natural group law]
			Lemmas \ref{lem:group_law} and \ref{lem:cubic} show that for every irreducible planar cubic $\qq$, exactly one of the following holds.
			\begin{enumerate}
				\item The smooth locus $\qq\reg$ of $\qq$ admits a group law such that the sum of three points $p,q,r\in \qq\reg$ is zero if and only if there is a line $\ll$ meeting $\qq$ at $p,q,r$ (counted with multiplicities).
				\item We have $\cha\kk=3$, and $\qq$ is projectively equivalent to  the cuspidal curve $\{x^2y+y^2z+z^2x=0\}$ described in Lemma \ref{lem:cubic}\ref{item:funny_cubic}. 
			\end{enumerate}
		\end{remark}

\section{Classification of canonical del Pezzo surfaces of rank one}\label{sec:canonical}

We now recall the classification of canonical del Pezzo surfaces of rank one, summarized in Table \ref{table:canonical}. Current proofs build on the pioneering work of du Val \cite{duVal}, together with modern insights of, among others, Demazure \cite{Demazure}, Hidaka--Watanabe \cite{HW_canonical} and Bindschadler--Brenton--Drucker \cite{BBD_canonical}. The latter article classifies all canonical del Pezzo surfaces of rank one whose minimal resolutions have the maximal number of exceptional components (namely, $8$). A complete classification for $\kk=\C$ was given by Furushima \cite[Theorem 2]{Furushima}. We refer to \cite{Alexxev-Nikulin_delPezzo-index-2} for a full statement and a detailed  historical discussion. For a general algebraically closed field $\kk$ of characteristic other than $2,3$ or $5$, Ye \cite{Ye} proved that the classification is the same as in case $\kk=\C$. Eventually, case $\cha\kk\in \{2,3,5\}$ was recently settled by Kawakami--Nagaoka \cite{Kawakami_Nagaoka_canonical_dP-in-char>0}. Here new singularity types appear, and some already occurring in case $\cha\kk=0$ acquire additional moduli, cf.\ \cite[Table 15]{PaPe_ht_2}.

Below we provide a characteristic-independent argument leading to a classification based on the geometry of $\P^1$-fibrations of low height. We restrict our attention to canonical del Pezzo surfaces $\bar{X}$ of height at least $3$. Indeed, if $\height(\bar{X})=1$ or $2$ then the geometry of $\bar{X}$ is simple; and an elementary exercise, or an application of \cite[Theorem A]{PaPe_ht_2}, shows that $\bar{X}$ is constructed as in Remark 4.5 or Section 5A loc.\ cit. The minimal resolutions of these surfaces, together with witnessing $\P^1$-fibrations are shown in Figures \ref{fig:can_ht=2} and  \ref{fig:can_ht=1}; some of those fibrations were used e.g.\ in \cite{MZ_canonical}. In these figures solid lines represent all components of $D$ which are $(-2)$-curves, while dashed lines correspond to $(-1)$-curves. Thickening of some lines refers to a discussion in \sec \ref{sec:relatives}.

\begin{proposition}[Canonical del Pezzo surfaces of rank $1$, $\height\geq 3$, see Figure \ref{fig:can_ht=4}]\label{prop:canonical_ht>2} Let $\bar{X}$ be a canonical del Pezzo surface of Picard rank $1$. Assume $\bar{X}\not\cong \P^2$ and $\height(\bar{X})\neq 1,2$. Then $\height(\bar{X})=4$ and one of the following holds.
	\begin{enumerate}
		\item\label{item:4A2} The surface $\bar{X}$ is of type $4\rA_2$, isomorphic to the one constructed in Example \ref{ex:4A2_construction}, see Figure \ref{fig:4A2}.
		\item\label{item:8A1} We have $\cha\kk=2$. The surface $\bar{X}$ is of type $8\rA_1$, isomorphic to one of the surfaces constructed in Example \ref{ex:cha_kk=2}\ref{item:8A1-construction}, see Figure \ref{fig:8A1}. The class of such surfaces has moduli dimension $2$.
	\end{enumerate}
\end{proposition}

\begin{proof}
	Let $\pi\colon (X,D)\to (\bar{X},0)$ be a minimal log resolution. Because $\bar{X}$ is canonical, the exceptional divisor $D$ consists of $(-2)$-curves and $K_{X}=\pi^{*}K_{\bar{X}}$. By \cite[Proposition 2.12]{BT_del-Pezzo-char>0} the linear system $|-K_X|$ has no fixed component, and its general member $E$ is irreducible and reduced, cf.\ \cite[p.\ 41]{Demazure} or \cite[Lemma 3.1]{Ye}. We have $p_{a}(E)=1$ and $E^{2}=K_{\bar{X}}^{2}>0$, since $-K_{\bar{X}}$ is ample. By Lemma \ref{lem:rationality}, $X$ is rational. Noether's formula gives $E^2=10-\rho(X)=9-\#D$.
	\smallskip
	
	We claim that $\#D\geq 7$. Suppose the contrary. We have $\rho(X)>2$, since otherwise $X$ is either $\P^2$ or the Hirzebruch surface $\F_{2}$, with $D$ either zero or the negative section, contrary to our assumptions. Thus $X$ contains a $(-1)$-curve $L$. We have $E\cdot (K_{X}+D+L)=-K_{X}^2-K_{X}\cdot L=\#D-8<0$, so $|K_{X}+D+L|=\emptyset$. By Lemma \ref{lem:trick-trees} $D+L$ is a sum of disjoint rational trees.

		Suppose $L$ is the unique $(-1)$-curve in $X$. Since $X$ is rational and $\rho(X)>2$, $X$ admits a $\P^1$-fibration with a degenerate fiber $F$. Then $F$ contains $L$, so it is a unique degenerate fiber. By Lemma \ref{lem:delPezzo_fibrations}\ref{item:Sigma} we have $F\redd-L\subseteq D$ and $\#D\hor=1$. Since $D+L$ is a sum of trees, we have $F\redd\cdot D\hor=1$. Because $F\redd-L\subseteq D$ consists of $(-2)$-curves, Lemma \ref{lem:degenerate_fibers}\ref{item:F-2} implies that all components of $F$ have multiplicity at most two in $F$. Hence $2\geq F\cdot D\hor \geq \height(\bar{X})$, a contradiction.
		
		Therefore, $\bar{X}$ contains at least two $(-1)$-curves, say $L$ and $A$. As before, we have $E\cdot (K_{X}+D+L+A)=-K_{X}^{2}+2=\#D-7<0$, so $|K_{X}+D+L+A|=\emptyset$ and $D+L+A$ is a sum of rational trees. %
		
		Since $\pi(L)^2>0$, the intersection matrix of $D+L$ is not negative definite. Thus $L\cdot D\geq 1$, and if the equality holds then $D+L$ contains a fork of type $\langle 2;[2],[2],[1,(2)_{k}]\rangle$ for some $k\geq 0$. Such a fork supports a fiber of a $\P^1$-fibration of height at most $2$, which is impossible. Hence $L\cdot D\geq 2$ and $A\cdot D\geq 2$. 
		
		Let $G_1$, $G_2$ be components of $D$ meeting $L$. Suppose $L\cdot D\geq 3$ and $A\cdot D \geq 3$. The assumption $\#D\leq 6$ implies that, after renaming $L$ and $A$ if needed, we get that $L\cdot D=3$ and $G_1,G_2$ are connected components of $D$. Now $G_1+L+G_2=[2,1,2]$ supports a fiber of a $\P^1$-fibration  of height two, a contradiction.
		
		Therefore, say, $L\cdot D=2$. We have  $\beta_{D}(G_1+G_2)\geq 3$, since otherwise  $G_1+L+G_2=[2,1,2]$ supports a fiber of a $\P^1$-fibration of height at most two. It follows that the connected component $D_L$ of $D+L$ containing $L$ has at least six components. Put  $W=D+L-D_L$. Since $\#(D+L)\leq 7$, we have $\#W\leq 1$. Since $A\cdot D\geq 2$ and $D+A+L$ is a sum of trees, we have $W=[2]$, $A\cdot W=1$ and $A\cdot D_L=1$. As before, the components $W,W'$ of $D$ meeting $A$ satisfy $\beta_{D}(W+W')\geq 3$. We get $D=W+G_1+G$, where $G=\langle 2;[2],[2],[2]\rangle$ meets both $L$ and $A$ in a branching component $G_2$.  Consider a $\P^1$-fibration induced by $|G_1+2L+G_2|$. Then all twigs of $G$ are $1$-sections. By Lemma \ref{lem:delPezzo_fibrations}\ref{item:-1_curves} the fiber $F_W$ containing $W$ is a sum of $W$ and $(-1)$-curves, so $F_{W}=[1,2,1]$.  
		One of its tips, say $L'$, meets at least two twigs of $G$, so $L'+D$ is not a sum of trees; a contradiction.
	\smallskip
	
	Having shown  $\#D\geq 7$, we continue the proof of Proposition \ref{prop:canonical_ht>2}. Recall that the curve $E$ is a member of $|-K_{X}|$, so it satisfies $p_{a}(E)=1$ and $9-\#D=E^2>0$. By Lemma \ref{lem:elliptic}\ref{item:elliptic=1} there is a blowup $\gamma\colon X'\to X$ at $9-\#D$ smooth points of $E$ such that the linear system of $E'\de \gamma^{-1}_{*}E$ induces a (quasi-)elliptic fibration $f\colon X'\to \P^1$. Since $E'\in |-K_{X'}|$, by adjunction formula every $(-1)$-curve on $X'$ meets $E'$ once, and every $(-2)$-curve on $X'$ is disjoint from $E'$. In particular, the fibers of $f$ contain no $(-1)$-curves, i.e.\ $f$ is minimal; and all components of $D'\de \gamma^{-1}_{*}D$ are contained in the fibers of $f$.

	Let $U'$ be a component of a fiber $F'$ meeting $\Exc\gamma$. The components of $\Exc\gamma$ are sections of $f$, so $U'$ has multiplicity $1$ in $F'$. By the Kodaira classification of elliptic fibers, $U'$ is non-branching in $F'\redd$. In particular, $U'\cdot D'\leq 2$, so $U\cdot D\leq 2$, where  $U=\gamma(U')$. Since $U'$ meets $\Exc\gamma$, we have $U^2\geq (U')^2+1=-1$. On the other hand, $U^2\leq (U')^2+(\rho(X')-\rho(X))=-2+(9-\#D)=7-\#D\leq 0$ by the claim above. If $U^2=0$ then $|U|$ induces a $\P^1$-fibration of height at most two, which is impossible. Hence $U^2=-1$.
	
	Suppose $U\cdot D\leq 1$. Since $U+D$ is not negative definite, it contains a fork $\langle 2;[2],[2],[1,(2)_k]\rangle$ for some $k\geq 0$, which supports a fiber of a $\P^1$-fibration of height two, a contradiction. 
	
	Thus $U\cdot D=2$. Because $U'$ has multiplicity one in $F'$, the classification of elliptic fibers implies that the components of $F'-U'$ meeting $U'$ are tips of $F'\redd-U'$. Thus the components of $D$ meeting $U$ are tips of $D$. If $U$ meets two disjoint tips of $D$, then their sum with $U$ supports a fiber of a $\P^1$-fibration of height two, which is impossible. Hence $U$ meets $D$ in both tips of a connected component of type $[2,2]$ or $[2]$. In particular, $F'$ meets $\Exc\gamma$ once, so $\gamma$ has only one base point, which gives $\#D=8$.
	
	Applying this argument to every degenerate fiber of $f$, we get that every connected component of $D$ is of type $[2,2]$ or $[2]$, and meets a $(-1)$-curve in both tips. These $(-1)$-curves meet normally at the base point of $\gamma$. 
	
	Let $U_1$, $U_2$ be two such $(-1)$-curves, and let $V_i$ be the connected component of $D$ meeting $U_i$, for $i\in \{1,2\}$. Consider the $\P^{1}$-fibration of $X$ induced by $|U_{1}+U_{2}|$.  Then $D\hor=V_1+V_2$, and each $V_i$, $i\in \{1,2\}$ is either a $2$-section or a chain of two $1$-sections. In particular, $\height(\bar{X})\leq (U_1+U_2)\cdot D=4$. 
	
	By Lemma \ref{lem:delPezzo_fibrations}\ref{item:-1_curves}, degenerate fibers of this $\P^1$-fibration consist of $(-1)$-curves and connected components of $D$, so by Lemma \ref{lem:degenerate_fibers}\ref{item:F-2} they are supported on chains of type $[1,1]$, $[1,2,1]$, $[1,2,2,1]$ or $[2,1,2]$.
	\smallskip
	
	Assume that $V_1=[2]$. Since $\#D=8$ is even, we can take $V_2=[2]$, too. Then $\#D\hor=2$. By Lemma \ref{lem:delPezzo_fibrations}\ref{item:Sigma} every fiber other than $U_{1}+U_{2}$ has exactly one component not in $D\vert$, so it is supported on a chain $[2,1,2]$. It follows that $D$ is a disjoint union of $(-2)$-curves, i.e.\ $\bar{X}$ is of type $8\rA_{1}$, see Figure \ref{fig:8A1}. We need to show that $\bar{X}$ is constructed as in Example \ref{ex:cha_kk=2}\ref{item:8A1-construction}. 
	
	Let $\phi$ be a contraction of $U_{2}$, followed by some contraction of remaining fibers to $0$-curves, so $\phi(X)$ is some Hirzebruch surface $\F_{n}$. Put $H_i=\phi(V_i)$, $i\in \{1,2\}$. Then $H_1$ is a smooth, rational $2$-section. Numerical properties of Hirzebruch surfaces imply that $n\leq 2$. We can choose $\phi$ so that $n=1$. Let $\contS_1\colon \F_{1}\to\P^2$ be the contraction of the negative section $\Sec_1$. For $i\in \{1,2\}$ we have $H_i \cdot \Sec_{1}=p_a(H_i)=i-1$. Put $\pp= (\contS_1\circ \phi)_{*}D$. Then $\pp=\ll_1+\ll_2+\ll_3+\cc+\qq$, where $\ll_j$'s are lines meeting at some $p_0$, $\cc=\contS_1(H_1)$ is a conic such that $(\cc\cdot \ll_{j})_{p_{j}}=2$ for some $p_j\neq p_0$, $j\in \{1,2,3\}$. It follows that  $\cha\kk=2$ and $\qq=\contS_1(H_2)$ is a cubic passing through $p_0$ and tangent to $\ll_j$ at each $p_j$. After a standard quadratic transformation at $p_1,p_2,p_3$ we get a configuration as in Example \ref{ex:cha_kk=2}\ref{item:8A1-construction}. In Lemma \ref{lem:8A1}\ref{item:8A1_ht} we will see that  $\height(\bar{X})=4$, so \ref{item:8A1} holds.
	\smallskip
	
	\noindent
	Assume now that $V_j=[2,2]$ for all $j$, i.e.\ $\bar{X}$ is of type $4\rA_{2}$. We need to show that $\bar{X}$ is as in Example \ref{ex:4A2_construction}.
	
	Let $V_{j}^{+}$, $V_{j}^{-}$ be the components of $V_{j}$. Now, all degenerate fibers except $U_1+U_2$ are of type $[1,2,2,1]$. Let $V_3,V_4$ be the vertical connected components of $D$, and for $i\in \{3,4\}$ let $F_{i}=L_{i}^{+}+V_{i}^{+}+V_{i}^{-}+L_i^{-}=[1,2,2,1]$ be the fiber containing $V_i$, where $L_i^{+}$ is the $(-1)$-curve meeting $V_{i}^{+}$. Since the intersection matrix of $L_{i}^{\pm}+D$ is not negative definite, we have $L_i^{\pm}\cdot D\geq 2$, so $F\cdot D=4$ implies that $L_i^{\pm}\cdot D=2$. In Figure \ref{fig:4A2}, 
	the sections $V_{1}^{+}$, $V_{1}^{-}$, $V_{2}^{+}$, $V_{2}^{-}$ are denoted by $L_1$, $L_2$, $Q$, $R$; 
	the curves $L_{3}^{+},V_{3}^{+},V_{3}^{-},L_{3}^{-}$ are denoted by $A_1,G_1\cp{1},G_1\cp{2},L_{1}'$;
	the curves $L_{4}^{+},V_{4}^{+},V_{4}^{-},L_{4}^{-}$ are denoted by $L_2',G_2\cp{2},G_2\cp{1},A_2$;
	and the curves $U_1,U_2$ are denoted by $L$, $A_0$.
	
	Let $\phi$ be the contraction of $U_{1}+V_{1}^{-}+\sum_{i=3}^{4}(F_{i}-L_{i}^{-})$. Then $\phi(X)=\P^2$. If $V_2^{+}$ meets $L_{i}^{+}$, $i\in \{3,4\}$, then $\phi(V_2^{-})=[2]$, which is impossible. Hence, say, $V_2^{+}$ meets $L_{3}^{+}$ and $L_{4}^{-}$. Now $\phi_{*}D$ is as in Example \ref{ex:4A2_construction}, with $\ll_{j}=\phi(L_{j+2}^{+})$, $\qq=\phi(V_{1}^{+})$, and $\ll=\phi(U_2)$.  In Lemma \ref{lem:Aut4A2}\ref{item:4A2_ht} we will see that $\height(\bar{X})=4$, so \ref{item:4A2} holds.
\end{proof}

\begin{figure}[htbp]
	\subcaptionbox{{\centering $4\rA_2$, see Example \ref{ex:4A2_construction}.} \label{fig:4A2}}[.54\linewidth]{
		\begin{tikzpicture}[scale=1.2]
			\draw (-1,3) -- (2.6,3) to[out=0,in=180] (3.5,1.4);
			\node at (3.15,2.55) {\small{$Q$}};
			\draw[dashed] (-0.4,3.1) -- (-0.6,2.1);
			\node at (-0.7,2.55) {\small{$A_1$}};
			\draw (-0.6,2.3) -- (-0.4,1.4);
			\node at (-0.8,1.8) {\small{$G_1\cp{1}$}};
			\draw (-0.4,1.6) -- (-0.6,0.7);
			\node at (-0.8,1.2) {\small{$G_1\cp{2}$}};
			\draw[dashed] (-0.6,0.9) -- (-0.4,-0.1);
			\node at (-0.7,0.45) {\small{$L_1'$}};
			\draw[dashed] (1.4,3.1) -- (1.2,2.1);
			\node at (1.1,2.55) {\small{$A_2$}};
			\draw (1.2,2.3) -- (1.4,1.4);
			\node at (1,1.8) {\small{$G_2\cp{1}$}};
			\draw (1.4,1.6) -- (1.2,0.7);
			\node at (1,1.2) {\small{$G_2\cp{2}$}};
			\draw[dashed] (1.2,0.9) -- (1.4,-0.1);
			\node at (1.1,0.45) {\small{$L_2'$}};
			\draw[dashed] (1.8,1.6) -- (2.8,1.4);
			\node at (2.2,1.7) {\small{$L$}};
			\draw[dashed] (2.6,3.1) -- (2.6,-0.1);
			\node at (2.4,2.55) {\small{$A_0$}};
			\draw (-1,0.05) -- (2.6,0.05) to[out=0,in=180] (3.5,1.6);
			\node at (3.15,0.45) {\small{$R$}};
			\draw (-1,2.85) -- (-0.5,2.85)  to[out=0,in=180] (1.2,0.2) -- (1.4,0.2) to[out=0,in=-150] (2.4,1.6);
			\node at (2.1,0.6) {\small{$L_1$}};
			\draw (-1,0.2) -- (-0.5,0.2) to[out=0,in=180] (1.2,2.85)-- (1.4,2.85) to[out=0,in=145] (2.05,1.45);
			\node at (2,2.2) {\small{$L_2$}};
			\node at (4.6,0.7) {\parbox{3cm}{\centering\small{If $\cha\kk=3$ then:  $A_0\cap Q\cap R\neq \emptyset$, $L\cap L_1\cap L_2\neq \emptyset$.}}};
		\end{tikzpicture}
	}
	\subcaptionbox{$8\rA_{1}$, $\cha\kk=2$, see Example \ref{ex:cha_kk=2}\ref{item:8A1-construction} \label{fig:8A1}}[.45\linewidth]{
		\begin{tikzpicture}
			\draw[thick] (0,1.6) -- (2.3,1.6) to[out=0,in=-100] (3.05,2.3) to[out=80,in=180] (3.6,2.8)-- (4,2.8);
			\node at (3.8,2.55) {\small{$Q$}};
			\draw[thick] (0,1.4) -- (2.3,1.4) to[out=0,in=100] (3.05,0.7) to[out=-80,in=180] (3.6,0.2) -- (4,0.2);
			\node at (3.8,0.4) {\small{$L_7$}};
			\draw (0.2,3.1) -- (0,1.9);
			\node at (-0.1,2.6) {\small{$L_1$}};
			\draw[dashed] (0,2.1) -- (0.2,0.9);
			\draw (0.2,1.1) -- (0,-0.1);
			\node at (-0.15,0.4) {\small{$L_4$}};
			\draw (1.2,3.1) -- (1,1.9);
			\node at (0.9,2.6) {\small{$L_2$}};
			\draw[dashed] (1,2.1) -- (1.2,0.9);
			\draw (1.2,1.1) -- (1,-0.1);
			\node at (0.85,0.4) {\small{$L_5$}};
			\draw (2.2,3.1) -- (2,1.9);
			\node at (1.9,2.6) {\small{$L_3$}};
			\draw[dashed] (2,2.1) -- (2.2,0.9);
			\draw (2.2,1.1) -- (2,-0.1);
			\node at (1.85,0.4) {\small{$L_6$}};
			\draw[dashed] (3.2,3.1) -- (2.9,1.4);
			\draw[dashed] (2.9,1.6) -- (3.2,-0.1);			
		\end{tikzpicture}	
	}
	\caption{Canonical del Pezzo surfaces of rank one and height 4.}
\label{fig:can_ht=4}
\end{figure}

\section{Del Pezzo surface of rank one and type $4\rA_2$}\label{sec:4A2}

We now give a construction of the surface of type $4\rA_2$ from Proposition \ref{prop:canonical_ht>2}\ref{item:4A2}, and list some of its properties. Most importantly, in Lemma \ref{lem:Aut4A2}\ref{item:4A2_Aut} below we will see that  the automorphism group of $\bar{X}$ acts transitively on the set of singular points of $\bar{X}$. We infer those properties from the geometry of the $\P^1$-fibration of height $4$ found in the proof of Proposition \ref{prop:canonical_ht>2}, see Figure \ref{fig:4A2}. 

\begin{example}[Del Pezzo surface of type $4\rA_2$, see Figure \ref{fig:4A2}] \label{ex:4A2_construction} 
Let $\qq\subseteq \P^{2}$ be a rational cubic with inflection points $p_1,p_2$. Let $\ll_i$ be the line tangent to $\qq$ at $p_i$, and fix a line $\ll$ tangent to the singular point $r\in \qq$, see Figure \ref{fig:4A2-construction}. By Lemma \ref{lem:cubic}, the configuration $\qq+\ll_{1}+\ll_{2}+\ll$ is unique up to a projective equivalence; and since $\qq$ has at least two inflection points, it singular point $r\in \qq$ is a node if $\cha\kk\neq 3$ and a cusp if $\cha\kk=3$.

Blow up twice over $r$ and three times over $p_j$, $j\in \{1,2\}$; each time on the proper transform of $\ll$ and $\ll_{j}$, respectively. Denote this morphism by $\phi\colon X\to \P^{2}$, and put $L_{j}=\phi^{-1}_{*}\ll_{j}$, $Q=\phi^{-1}_{*}\qq$. Let $A_0,A_j$ be the $(-1)$-curves over $r,p_{j}$; and write $(\phi^{*}\qq)\redd=G+A_0=[2,2]+[1]$, $\phi^{-1}(p_{j})=G_{j}+A_j=[2,2,1]$. Since $L_{1}+L_2=[2,2]$, $D\de (L_{1}+L_{2})+G+G_1+G_2$ contracts to four $\rA_{2}$-points on a del Pezzo surface $\bar{X}$ of rank one, as needed. 

We now construct a $\P^1$-fibration of $(X,D)$ of height $4$. Put $L=\phi^{-1}_{*}\ll$ and let $L_{i}'$ be the proper transform of the line joining $r$ with $p_i$. Then $|A_{1}+G_{1}+L_{1}'|$ induces a $\P^1$-fibration. Its degenerate fibers are $A_{i}+G_{i}+L_{i}'=[1,2,2,1]$, $i\in \{1,2\}$, and $A_0+L=[1,1]$. The horizontal part of $D$ consists of four $1$-sections, namely $L_1$, $L_2$, $Q$ and the first exceptional curve over $r$, call it $R$. They meet $D$ in four components each, see Figure \ref{fig:4A2}.
\end{example}

\begin{figure}[htbp]
	\subcaptionbox{$\cha\kk\neq 3$ \label{fig:4A2-constr-cha-neq-3}}[0.4\textwidth]
	{
		\begin{tikzpicture}
			\draw (-1.7,-0.5) to (1.7,-0.5);
			\node at (1,-0.7) {\small{$\ll$}};
			\filldraw (0,-0.5) circle (0.06);
			\node at (0.15,-0.35) {\small{$r$}};	
			\draw[add= 0.1 and 0.8] (0,2) to (-0.866,0.5);
			\node at (-1.5,-0.1) {\small{$\ll_1$}};
			\filldraw (-0.866,0.5) circle (0.06);
			\node at (-1.1,0.6) {\small{$p_1$}};
			\draw[add= 0.1 and 0.8] (0,2) to (0.866,0.5);
			\node at (1.5,-0.1) {\small{$\ll_2$}};
			\filldraw (0.866,0.5) circle (0.06);
			\node at (1.1,0.6) {\small{$p_2$}};	
			\filldraw (0,2) circle (0.06);
			\node at (0.3,2) {\small{$p_0$}};
			\draw[thick] (-0.866,2) to[out=-60,in=60] (-0.866,0.5) to[out=-120,in=90] (0,-0.5) to[out=-90,in=0] (-0.3,-0.9) to[out=180,in=-90] (-0.5,-0.7) to[out=90,in=180] (0,-0.5) to[out=0,in=-60] (0.866,0.5) to[out=120,in=-120]  (0.866,2);
			\node at (-0.75,-0.1) {\small{$\qq$}};
		\end{tikzpicture}
	}
	\subcaptionbox{$\cha\kk=3$ \label{fig:4A2-constr-cha=3}}[0.4\textwidth]
	{
	\begin{tikzpicture}
			\draw[add=0.1 and 0.8] (0,2) to (0,0.5);
			\node at (0.2,0.5) {\small{$\ll$}};
			\filldraw (0,-0.5) circle (0.06);
			\node at (0.2,-0.5) {\small{$r$}};	
			\draw[add= 0.1 and 0.6] (0,2) to (-0.866,0.5);
			\node at (-1.5,-0.1) {\small{$\ll_1$}};
			\filldraw (-0.866,0.5) circle (0.06);
			\node at (-1.1,0.6) {\small{$p_1$}};
			\draw[add= 0.1 and 0.6] (0,2) to (0.866,0.5);
			\node at (1.5,-0.1) {\small{$\ll_2$}};
			\filldraw (0.866,0.5) circle (0.06);
			\node at (1.1,0.6) {\small{$p_2$}};	
			\filldraw (0,2) circle (0.06);
			\node at (0.3,2) {\small{$p_0$}};
			\draw[thick] (-0.866,2) to[out=-60,in=60] (-0.866,0.5) to[out=-120,in=90] (0,-0.5) to[out=90,in=-60] (0.866,0.5) to[out=120,in=-120]  (0.866,2);
			\node at (-0.75,-0.1) {\small{$\qq$}};
	\end{tikzpicture}
	}
\caption{Configuration $\qq+\ll_1+\ll_2+\ll$ used in Example \ref{ex:4A2_construction}.}
\label{fig:4A2-construction}
\end{figure}

\begin{lemma}
\label{lem:Aut4A2}
Let $\bar{X}$ be the del Pezzo surface of rank one and type $4\rA_2$. We keep notation from Example \ref{ex:4A2_construction}.
\begin{enumerate}
	\item\label{item:4A2_ht} We have $\height(\bar{X})=4$.
	\item\label{item:4A2_K} Let $\bar{T}$ be an irreducible member of $|-K_{\bar{X}}|$. Then $\bar{T}$ is smooth if $\cha\kk\neq 3$ and cuspidal if $\cha\kk=3$. 
	\item\label{item:4A2_moduli} If $\cha\kk=3$ then the set $\Pcusp(4\rA_2)$ of isomorphism classes of log surfaces $(\bar{X},\bar{T})$ as in \ref{item:4A2_K} has moduli dimension $1$, see Section \ref{sec:moduli}. In fact, it is represented by an almost universal family over $B=\P^1\setminus \P^1_{\F_3}$ with symmetry group $\Aut(B)\cong \PGL_{2}(\F_3)$, see Table \ref{table:symmetries}.
	\item\label{item:4A2_Aut} The group $\Aut(\bar{X})=\Aut(X,D)$ is isomorphic to $\Z/2\rtimes G$, where $G$ is the alternating group $A_{4}$ if $\cha \kk \neq 3$, and a symmetric group $S_{4}$ if $\cha\kk=3$. The subgroup $G$ acts transitively on the set of singular points of $\bar{X}$. Explicitly, the group $\Aut(X,D)$ is generated by the following automorphisms, see Figure \ref{fig:4A2}:
	\begin{enumerate}
		\item\label{item:Aut4A2_Z2} $
		(G_{1}\cp{1},G_{1}\cp{2},G_{2}\cp{1},G_{2}\cp{2},L_1,L_2,Q,R)\mapsto
		(G_{1}\cp{2},G_{1}\cp{1},G_{2}\cp{2},G_{2}\cp{1},L_2,L_1,R,Q)$
		\item\label{item:Aut4A2_(123)} $
		(G_{1}\cp{1},G_{1}\cp{2},G_{2}\cp{1},G_{2}\cp{2},L_1,L_2,Q,R)\mapsto
		(L_2,L_1,G_{1}\cp{2},G_{1}\cp{1},G_{2}\cp{1},G_{2}\cp{2},Q,R)$
		\item\label{item:Aut4A2_(12)(34)} $			(G_{1}\cp{1},G_{1}\cp{2},G_{2}\cp{1},G_{2}\cp{2},L_1,L_2,Q,R)\mapsto
		(G_{2}\cp{2},G_{2}\cp{1},G_{1}\cp{2},G_{1}\cp{1},R,Q,L_1,L_2)$
		\item\label{item:Aut4A2_(12)} $
		(G_{1}\cp{1},G_{1}\cp{2},G_{2}\cp{1},G_{2}\cp{2},L_1,L_2,Q,R)\mapsto
		(G_{2}\cp{1},G_{2}\cp{2},G_{1}\cp{1},G_{1}\cp{2},L_2,L_1,Q,R)$ if $\cha\kk=3$
	\end{enumerate}
\end{enumerate} 
\end{lemma}
\begin{proof}
\ref{item:4A2_ht} Suppose the contrary, i.e.\ $X$ admits a $\P^1$-fibration whose fiber $F$ satisfies $F\cdot D\leq 3$. By Lemma \ref{lem:delPezzo_fibrations}\ref{item:-1_curves} its degenerate fibers are sums of $(-1)$-curves and components of $D$, so by Lemma \ref{lem:degenerate_fibers}\ref{item:F-2} they are supported on chains $[2,1,2]$, $[1,2,1]$ or $[1,2,2,1]$. Since $\#D\hor\leq 3$, Lemma \ref{lem:delPezzo_fibrations}\ref{item:Sigma} implies that there are at most two fibers with two $(-1)$-curves each, so since $\#D\vert=\#D-\#D\hor\geq 5$, there is at least one fiber $F$ such that $F\redd=[2,1,2]$. Each tip of $F\redd$ meets $D\hor$ once, and since the $(-1)$-curve in $F$ has multiplicity $2$, inequality $F\cdot D\leq 3$ implies that it does not meet $D\hor$. Thus $D\hor$ consists of two $1$-sections, and the remaining degenerate fibers meet them in $(-1)$-curves of multiplicity $1$. Since $\bar{X}$ is of type $4\rA_2$, there are two such fibers, both of type $[1,2,2,1]$. This is a contradiction with Lemma \ref{lem:delPezzo_fibrations}\ref{item:Sigma}. 
\smallskip

\ref{item:4A2_K} Let $T$ be the proper transform of $\bar{T}$ on $X$, so $T\in |-K_{X}|$ since $X$ is canonical. Let $\phi\colon (X,D)\to (\P^2,\qq+\ll_1+\ll_2)$ be as in Example \ref{ex:4A2_construction}, see Figure \ref{fig:4A2}. Since by adjunction  $T$ is disjoint from $D$ and $T\cdot A=1$ for every $(-1)$-curve on $X$, we see that $\ttt\de \phi(T)$ is a cubic meeting $\qq$ in $p_{1},p_{2},r\in \ttt\reg$ with multiplicity $3$.

Assume first that $\cha\kk\neq 3$. Suppose $T$ is singular, so $\ttt$ is a rational cubic with inflection points $p_{1}$, $p_{2}$. By Lemma \ref{lem:group_law}, $\ttt\reg$ admits a group law. In this group law we have $3p_{1}=3p_{2}=0$ since $p_1,p_2$ are inflection points of $\ttt$. Intersecting $\ttt$ with $\qq$ we get $3p_{1}+3p_2+3r=0$, so $3r=0$, which means that $r$ is an inflection point of $\ttt$, too. Lemma \ref{lem:cubic} implies that $\ttt$ is nodal, and $p_{1},p_{2},\ttt$ lie on the same line, say $\ll$. Since $p_1,p_2\in \qq$ and $r\in \Sing \qq$ we get $3< \ll\cdot \qq=3$, a contradiction.

Assume now that $\cha\kk=3$. Since $\qq$ is a rational cubic with two inflection points $p_1,p_2$, by Lemma \ref{lem:cubic}\ref{item:cuspidal_cubic} we can fix coordinates on $\P^2$ so that $\qq=\{y^3=x^2 z\}$, $p_{1}=[1:0:0]$, $\ll_{1}=\{z=0\}$; $p_{2}=[1:1:1]$, $\ll_{2}=\{x=z\}$; and $\qq$ has a cusp at $r=[0:0:1]$, with tangent line $\{x=0\}$. Now the pencil of cubics meeting $\qq$ at $p_1,p_2,r$ with multiplicity at least $3$ is given by $\{\lambda (y^3-x^2z)=\mu(y^3-xz^2)\}$, $[\lambda:\mu]\in \P^{1}$, and every irreducible member of this pencil has a singular point at $[\mu(\lambda-\mu)^{1/3}:(\lambda\mu)^{2/3}:\lambda(\mu-\lambda)^{1/3}]$, which is a cusp by Lemma \ref{lem:cubic}. 
\smallskip

\ref{item:4A2_moduli} The above construction shows that $\P^1\setminus \{[0:1],[1:0],[1:1],[1:-1]\}$ is the base of the required almost universal family. We postpone a detailed proof for the moment, as it will be more clear from an equivalent construction of $\bar{X}$ given in  \cite[Proposition 5.1(5)]{KN_Pathologies}, which we recall in Remark \ref{rem:4A2-alternative} below.
\smallskip

\ref{item:4A2_Aut} Clearly, we have $\Aut(\bar{X})=\Aut(X,D)$. Recall from Example \ref{ex:4A2_construction} that the log surface $(X,D)$ is constructed via a  morphism $\phi\colon (X,D)\to (\P^2,\pp)$, where $\pp=\qq+\ll_1+\ll_2$, $\qq$ is a rational cubic, and for $i\in \{1,2\}$, $\ll_{i}$ is a line such that $(\ll_{i}\cdot \qq)_{p_i}=3$ for some $p_{i}\in \qq\reg$. The singular point $r\in \qq$ is a node if $\cha\kk\neq 3$ and a cusp if $\cha\kk=3$, cf.\  Lemma \ref{lem:cubic}. Given the configuration $\pp\subseteq \P^2$, the morphism $\phi$ constructed in Example \ref{ex:4A2_construction} is uniquely determined, up to an action of $\Aut(X,D)$, by the choice of a line $\ll$ tangent to $\qq$ at $r$.

To construct an automorphism $\sigma\in \Aut(X,D)$, we  define a morphism $\tilde{\phi}\colon (X,D)\to (\P^2,\pp)$ as the contraction of a divisor which is supposed to be the image $\sigma_{*}(\Exc\pi)$. Next, we put $\sigma= \tilde{\phi}^{-1}\circ \phi\colon X\map X$. The above uniqueness property of $\phi$ shows that $\sigma$ extends to an automorphism of $(X,D)$ if and only if
\begin{equation}\label{eq:condition_L}
	\parbox{.9\linewidth}{the second blowup at $r$ in the decomposition of $\tilde{\phi}$ is centered on the proper transform of $\ll$.}
\end{equation}
Note that if $\cha\kk=3$ then condition \eqref{eq:condition_L} is automatically satisfied because $r\in \qq$ is a cusp. 
\smallskip

We keep notation $L_i,G_{i}$ etc.\ from Example \ref{ex:4A2_construction}, see Figure \ref{fig:4A2}. Write $\{p\}=\ll_{1}\cap\ll_{2}$ and let $\ll_{12}$ be the line joining $p_1$ and $p_2$. For $i\in \{1,2\}$, let $\ll_{i}'$ be the line joining $p_i$ with $r$. Put $L_{12}=\phi^{-1}_{*}\ll_{12}$, $L_{i}'=\phi^{-1}_{*}\ll_{i}'$.

In case $\cha\kk\neq 3$, we introduce further notation. Fix a  primitive third root of unity $\zeta$. Let $\bar{\nu}\colon \P^1\to \P^2$ be a parametrization of $\qq$ such that $p_{1}=\nu[\zeta:-1]$, $r=\nu[0:1]=\nu[0:1]$ and $\ll$ is tangent to the branch of $\qq$ whose preimage contains $[0:1]$.  Then $p_2=\nu[\zeta^2:-1]$, and $\qq\cap \ll_{12}=\{p_1,p_2,\nu[1:-1]\}$, see Lemma \ref{lem:cubic}\ref{item:nodal_cubic}. The induced parametrization $\nu\colon \P^1\to X$ of $Q$ satisfies
\begin{equation}\label{eq:nu}
	Q\cap A_0= \nu[0:1],\ Q\cap R = \nu[1:0],\ Q\cap A_{1}=\nu[\zeta:-1],\ Q\cap A_{2}=\nu[\zeta^2:-1],\ Q\cap L_{12}=\nu[1:-1].
\end{equation}
Moreover we fix a parametrization $\bar{\nu}_1\colon \P^1\to \P^2$ of $\ll_{1}$ such that $p=\bar{\nu}_1[1:0]$, $\ll_1\cap \ll=\bar{\nu}_1[0:1]$ and $p_1=\bar{\nu}_1[\zeta^{2}:1]$. An explicit computation in coordinates of Lemma \ref{lem:cubic}\ref{item:nodal_cubic}, with $\ll=\{x=0\}$, shows that $\bar{\nu}_1[\zeta:1]=\ll_{1}\cap \ll_{2}'$. Indeed, we have $p=[1:1:3]$, $\ll_{1}\cap \ll=[0:1:-3\zeta^{2}]$, $p_{1}=[\zeta:-1:0]$, so $\bar{\nu}_1[s:t]=[s:s+t:3s-3\zeta^{2} t]$; and $\ll_{2}'=\{\zeta x+y=0\}$, so $\ll_{1}\cap \ll_{2}'=[1:-\zeta:3-3\zeta]=\bar{\nu}_1[\zeta:1]$, as claimed. By symmetry, we have a parametrization $\bar{\nu}_2$ of $\ll_2$ such that $p=\bar{\nu}_1[1:0]$, $\ll_2\cap \ll=\bar{\nu}_2[0:1]$, $p_2=\bar{\nu}_2[\zeta:1]$ and $\ll_{2}\cap \ll_{1}'=\bar{\nu}_2[\zeta^2:1]$. Thus for $i\in \{1,2\}$, the induced parametrization $\nu_i\colon \P^1\to X$ of $L_i$ satisfies
\begin{equation}\label{eq:nu_i}
	L_{i}\cap L=\nu_i[0:1],\ L_1\cap L_2=\nu_{i}[1:0],\ L_{i}\cap A_i=\nu_{i}[\zeta^{2i}:1],\ L_{i}\cap L_{3-i}'=\nu_{i}[\zeta^{i}:1].
\end{equation}
Having introduced this notation, we can proceed with the construction of each automorphism from \ref{item:4A2_Aut}.
\smallskip

\ref{item:Aut4A2_Z2} Let $\tilde{\phi}\colon X\to \P^2$ be the contraction of $(L_{1}'+G_{1})+(L_{2}'+G_{2})+(A_{0}+Q)$, such that $\tilde{\phi}(L_{1}')=p_{1}$, $\tilde{\phi}(L_{2}')=p_{2}$, $\tilde{\phi}(L_1\cap L_2)=p$, $\tilde{\phi}(A_0)=r$. Then for $i\in \{1,2\}$ we have  $\ll_{i}=\tilde{\phi}(L_{3-i})$, $\ll_{i}'=\tilde{\phi}(A_{i})$, and $\tilde{\phi}(R)$ is a rational cubic with a singular point at $r$, tangent with multiplicity $3$ to $\ll_{i}$ at $p_i$, so $\tilde{\phi}(R)=\qq$. In particular, $\tilde{\phi}_{*}D=\pp$. The second blowup over $r$ is centered on the image of $L$, so to verify condition \eqref{eq:condition_L}, we need to check that $\tilde{\phi}(L)=\ll$ in case $\cha\kk\neq 3$. Consider a parametrization of $\ll_1$ given by $\tilde{\nu}_1\de\tilde{\phi}\circ \nu_2\colon \P^1\to \P^2$. Formulas  \eqref{eq:nu_i} give $p=\tilde{\phi}(L_1\cap L_2)=\tilde{\nu}_1[1:0]$, $p_1=\tilde{\phi}(L_{2}\cap L_{1}')=\tilde{\nu}_1[\zeta^2:1]$ and $\ll_1\cap \ll_{2}'=\tilde{\phi}(L_2\cap A_2)=\tilde{\nu}_1[\zeta:1]$. Hence $\tilde{\nu}_1=\bar{\nu}_{1}$. It follows from \eqref{eq:nu_i} that $\ll_{1}\cap \ll=\tilde{\nu}_1[0:1]=\tilde{\phi}(L_2\cap L)$, so $\ll=\tilde{\phi}(L)$, as needed.

\ref {item:Aut4A2_(123)} Let $\tilde{\phi}\colon X\to \P^2$ be the contraction of $(A_2+L_2+L_1)+(L_{12}+G_1)+(A_0+R)$, such that $\tilde{\phi}(A_2)=p_1$, $\tilde{\phi}(L_{12})=p_2$, $\tilde{\phi}(G_{2}\cp{1}\cap G_{2}\cp{2})=p$, $\tilde{\phi}(A_0)=r$. Then $\ll_{1}=\tilde{\phi}(G_{2}\cp{2})$, $\ll_{2}=\tilde{\phi}(G_{2}\cp{1})$, and like before we see that $\tilde{\phi}(Q)=\qq$, so $\tilde{\phi}_{*}D=\pp$.  It remains to check condition \eqref{eq:condition_L} in case $\cha\kk\neq 3$. Consider a parametrization of $\qq$ given by $\tilde{\nu}=\tilde{\phi}\circ \nu\circ \delta$, where $\delta[s:t]=[\zeta s:t]$, so $\delta\in \Aut(\P^1)$ and  $\delta^{-1}[s:t]=[\zeta^2 s:t]$. Using formulas \eqref{eq:nu}, we compute that  $p_{1}=\tilde{\phi}(Q\cap A_{2})=\tilde{\phi}\circ \nu[\zeta^2:-1]=\tilde{\nu}[\zeta:-1]$, $p_{2}=\tilde{\phi}(Q\cap L_{12})=\tilde{\phi}\circ \nu[1:-1]=\tilde{\nu}[\zeta^2:-1]$; and $r=\tilde{\nu}[0:1]=\tilde{\nu}[1:0]$, so $\tilde{\nu}=\bar{\nu}$. The second blowup over $r$ in the decomposition of $\tilde{\phi}$ is centered at the image of $Q\cap A_{0}=\nu[0:1]$; as needed.

\ref{item:Aut4A2_(12)(34)} Let $\tilde{\phi}\colon X\to \P^2$ be the contraction of $(L_{2}'+G_{2})+(A_{1}+G_{1})+(L+L_2)$ such that $\tilde{\phi}(L_{2}')=p_1$, $\tilde{\phi}(A_1)=p_2$, $\tilde{\phi}(Q\cap R)=p$, $\tilde{\phi}(L)=r$. Then $\tilde{\phi}(R)=\ll_1$, $\tilde{\phi}(Q)=\ll_2$, $\tilde{\phi}(A_{2})=\ll_{1}'$; $\tilde{\phi}(L_1)=\qq$ and $\tilde{\phi}(A_0)$ is a line tangent to $\qq$ at $r$. The second blowup over $r$ is centered on the image of $A_0$, so to verify condition \eqref{eq:condition_L}, we need to check that $\tilde{\phi}(A_0)=\ll$ if $\cha\kk\neq 3$. Let $\tilde{\nu}_2\de \tilde{\phi}\circ \nu\circ \delta$ be a parametrization of $\ll_2$, where $\delta[s:t]=[s:-t]$. Using formulas \eqref{eq:nu}, we compute that  $p_{2}=\tilde{\phi}(Q\cap A_{1})=\tilde{\nu}_{2}[\zeta:1]$, $\ll_{2}\cap \ll_{1}'=\tilde{\phi}(Q\cap A_{2})=\tilde{\nu}_2[\zeta^2:1]$, $p=\tilde{\phi}(Q\cap R)=\tilde{\nu}_2[1:0]$, so $\tilde{\nu}_2=\bar{\nu}_2$. Now $\tilde{\phi}(Q\cap A_0)=\tilde{\nu}_2[0:1]=\ll_{2}\cap \ll$ by formula \eqref{eq:nu_i}, so $\tilde{\phi}(A_0)=\ll$, as claimed.

\ref{item:Aut4A2_(12)} Assume $\cha\kk=3$, and let $\tau\in \Aut(\P^2,\pp)$ be a automorphism given by $(p,p_1,p_2,r)\mapsto (p,p_2,p_1,r)$. Since $\cha\kk=3$, this automorphism lifts to the required automorphism of $(X,D)$.
\smallskip

It remains to show that the group $\Aut(X,D)$ is generated by the automorphisms in \ref{item:4A2_Aut}. First, we note that an element of $\Aut(X,D)$ is uniquely determined by its action on the set of components of $D$. Indeed, suppose that $\sigma\in \Aut(X,D)$ fixes all components of $D$. Since $\rho(\bar{X})=1$, the group $\NS_{\Q}(X)$ is generated by $\pi^{*}K_{\bar{X}}=K_{X}$ and the classes of components of $D$, so $\sigma$ acts trivially on $\NS_{\Q}(X)$. Because $\Exc\phi$ is negative definite, it follows that $\sigma$ fixes each component of $\Exc \phi$, and therefore descends to an automorphism of $\P^2$ fixing $\phi_{*}D=\pp$ componentwise. Thus $\sigma=\id$, as claimed.

Consider a group homomorphism $\alpha\colon \Aut(X,D)\to S_{4}$ given by the action of $\Aut(X,D)$ on the set of connected components of $D$. Clearly, the involution $\iota$ from \ref{item:Aut4A2_Z2} lies in $\ker \alpha$. We claim that $\ker\alpha=\langle \iota\rangle$. 

Let $\sigma\in \ker \alpha$. Since $\sigma(Q+R)=Q+R$ and $\sigma(L_1+L_2)=L_1+L_2$, we have $\sigma(L)=L$ and $\sigma(A_0)=A_0$, so $\sigma$ preserves the $\P^1$-fibration of $X$ induced by $|A_{0}+L|$, see Figure \ref{fig:4A2}. It follows that $\sigma(A_{i}+L_{i}')=A_{i}+L_{i}'$ for $i\in \{1,2\}$. Composing $\sigma$ with $\iota$ if needed, we can assume that $\sigma(A_1)=A_1$, so $\sigma(L_1')=L_1'$. Then $\sigma$ fixes components of $D$ meeting $A_1$, in particular $\sigma(Q)=Q$. Since $L_{2}'$ does not meet $Q$, we have $\sigma(A_{2})=A_{2}$, so $\sigma(L_{2}')=L_{2}'$. It follows that $\sigma$ fixes each component of $D$, so $\sigma=\id$, which proves the claim.

Let $G\subseteq \Aut(X,D)$ be the subgroup generated by the remaining automorphisms listed in \ref{item:4A2_Aut}. The restriction $\alpha|_{G}\colon G\to S_{4}$ is injective, its image equals $A_{4}$ if $\cha\kk\neq 3$ and $S_{4}$ if $\cha\kk= 3$. It remains to show that $\alpha$ is not surjective if $\cha\kk\neq 3$. Suppose the contrary. Then there is $\sigma\in \Aut(X,D)$ such that $\sigma(G_{1})=G_{2}$, $\sigma(L_1+L_2)=L_1+L_2$ and $\sigma(Q+R)=Q+R$. Composing $\sigma$ with $\iota$, we can assume that $\sigma(G_{1}\cp{i})=G_{2}\cp{i}$ for $i\in \{1,2\}$. If $\sigma(Q)=R$ then $\sigma(L_{12})$ meets $G_{1}\cp{2}$, $G_{2}\cp{2}$ and $R$, so its image $\cc\de \phi(\sigma(L_{12}))$ satisfies 
\begin{equation*}
	3\deg \cc=\cc\cdot \qq=(3\cdot (\sigma(L_{12})\cdot A_{1})+1)+(3\cdot (\sigma(L_{12})\cdot A_{2})+1)+(3\cdot \sigma(L_{12})\cdot A_0+2)\equiv 1\pmod{3},
\end{equation*}
which is false. Hence $\sigma(Q)=Q$. If $\sigma(L_1)=L_2$ then $\sigma(A_1)$ meets $G_{2}\cp{1}$, $L_2$ and $Q$, so  $\aa\de \phi(\sigma(A_{1}))$ satisfies 
\begin{equation*}
	0=(\ll_{1}-\ll_{2})\cdot \aa=3(\sigma(A_1)\cdot A_1)+2-3(\sigma(A_1)\cdot A_2)\equiv 2\pmod{3},	
\end{equation*}
which again is false. Thus $\sigma$ is as in \ref{item:Aut4A2_(12)}, so it induces an automorphism $\bar{\sigma}$ of $(\P^2,\pp+\ll)$ interchanging $\ll_{1}$ with $\ll_{2}$. Since $\cha\kk\neq 3$, $\bar{\sigma}(\ll)$ is the other line tangent to the node of $\qq$; a contradiction. 
\end{proof}

\begin{remark}[Alternative construction in case $\cha\kk=3$, see {\cite[Proposition 5.1(5)]{KN_Pathologies}}]\label{rem:4A2-alternative}
Assume $\cha\kk=3$. Then a del Pezzo surface $\bar{X}$ of rank one and type $4\rA_2$ can be constructed as follows. Fix an $\F_{3}$-rational point $p\in \P^2$, let $\ll,\ll_1,\dots, \ll_{8}$ be all $\F_{3}$-rational lines not passing through $p$, and let $p_1,\dots, p_{8}$ be all $\F_{3}$-rational points not lying on $\ll\cup \{p\}$. Let $\psi\colon X\to \P^{2}$ be the blowup at $p_1,\dots,p_{8}$, and let $X\to \bar{X}$ be the contraction of $\psi^{-1}_{*}\ll_i$ for $i\in \{1,\dots,8\}$. Clearly, $\bar{X}$ is a surface of type $4\rA_2$, hence by Proposition \ref{prop:canonical_ht>2} it is isomorphic to the one constructed in Example \ref{ex:4A2_construction}. More explicitly, choosing coordinates $[x:y:z]$ on $\P^2$ so that $p=[0:0:1]$ and $\ll=\{z=0\}$, the morphism $\phi\colon X\to \P^2$ from Example \ref{ex:4A2_construction} is given by the contraction of $(-1)$-curves $\psi^{-1}_{*}\{x=0\}$, $\psi^{-1}[1:0:1]$, $\psi^{-1}[1:1:1]$ and $(-2)$-curves $\psi^{-1}_{*}\{x+z=0\}$, $\psi^{-1}_{*}\{x=y+z\}$, $\psi^{-1}_{*}\{y=x+z\}$, $\psi^{-1}_{*}\{y=z\}$, $\psi^{-1}_{*}\{y+z=0\}$. Indeed, $\phi(\psi^{-1}_{*}\{x=z\})$ is a cuspidal cubic; $\phi(\psi^{-1}_{*}\{x=z\})$ is the line tangent to its cusp, and $\phi(\psi^{-1}_{*}\{x+y=z\})$, $\phi(\psi^{-1}_{*}\{x+y+z=0\})$ are its inflectional tangent lines.

Now it is clear that $\Aut(\bar{X})\cong \operatorname{GL}_{2}(\F_{3})$ acts transitively on $\Sing \bar{X}$, as claimed in Lemma \ref{lem:Aut4A2}\ref{item:4A2_Aut}. Moreover, every irreducible member $\bar{T}$ of $|-K_{\bar{X}}|$ is a proper transform of a cubic passing through $p_1,\dots,p_{8}$; and a direct computation shows that every such cubic is   cuspidal, as claimed in Lemma \ref{lem:Aut4A2}\ref{item:4A2_K}. Indeed, choosing coordinates $[x:y:z]$ as above, we see that every such cubic is of the form $\qq_{[a:b]}=\{ax(x^2-z^2)=by(y^2-z^2)\}$ for $[a:b]\in \tilde{\P}^{1}\de \P^1\setminus \P^1_{\F_{3}}$, and $\qq_{[a:b]}$ has a cusp at $r_{[a:b]}\de [b^{1/3}:a^{1/3}:0]$.

We can now complete the proof of Lemma \ref{lem:Aut4A2}\ref{item:4A2_moduli}, by constructing an almost universal family $f$ representing $\Pcusp(4\rA_2)$. Fix coordinates $[x:y:z]$ on $\P^2$ as above; and let $[a:b]$ be coordinates on $\P^1$. In the product $\P^2\times \tilde{\P}^1$, consider curves $\cP_{i}=\{p_i\}\times \tilde{\P}^1$, $\Delta=\{xa=yb,z=0\}$ and divisors $\cL_{i}=\ll_i\times \tilde{\P}^1$, $\cQ=\{a^3x(x^2-z^2)=b^3y(y^2-z^2)\}$. Let $\Psi\colon \cX\to \P^2\times \tilde{\P}^1$ be a composition of blowups at $\cP_{1},\dots, \cP_{8}$ and three blowups over $\Delta$, each time on the proper transform of $\cQ$. Let $\cD=\Psi^{-1}_{*}(\cQ+\sum_{i}\cL_i)$. Then $f\colon (\cX,\cD)\to \tilde{\P}^1$ represents $\Pcusp(4\rA_2)$. The fibers of $f$ over $z,w\in \tilde{\P}^1$ are isomorphic if and  only if there is an element of $\operatorname{GL}_{2}(\F_3)=\Aut(\P^2,p,\{p_1,\dots,p_{8}\})$ mapping $r_{z}$ to $r_{w}$. This happens if and only if $z,w$ lie in the same orbit of the induced action of $\operatorname{GL}_{2}(\F_3)$ on $\tilde{\P}^1$.

Thus $f$ is almost faithful, with symmetry group $\PGL_{2}(\F_3)$. We claim that the formal germ of $f$ at any $z\in \tilde{\P}^{1}$ is a versal deformation of $(X_z,D_z)$. To see this, consider any infinitesimal deformation $\check{f}\colon (\check{\cX},\check{\cD})\to T$ of $(X_z,D_z)$. Performing a sequence of $K_{\check{\cX}/T}$-negative contractions we get a blowdown $\check{\cX}\to \P^{2}_{T}$, mapping $\check{\cD}$ to a configuration of lines and a cubic as above. Choosing coordinates on $\P^2_{T}=\P^2\times T$ we can assume that the image of $\check{\cD}$ is $\sum_{i}(\ll_{i}\times T) + \cQ_T$ for some family of cuspidal cubics $\cQ_{T}$. The computation above shows that every such cubic has a cusp on $\tilde{\ll}\times T$, where  $\tilde{\ll}=\ll\setminus \sum_{i}\ll_i\cong \tilde{\P}^1$; in fact $\Sing \cQ_{T}=(\alpha_0(t),t)\in \tilde{\ll}\times T$ for some morphism $\alpha_0\colon T\to \tilde{\P}^1$. Putting $\alpha=(\id_{\P^2},\alpha_0)\colon \P^2\times T\to \P^2\times \tilde{\P}^1$ we get $\Delta_{T}=\alpha^{*}\Delta$ and thus $\cQ_{T}=\alpha^{*}\cQ$. By the universal property of blowing up  we conclude that $\check{f}$ is a pullback of $f$ along $\alpha_0$, as needed.

It remains to prove that the germ of $f_z$ at $z$ is a semiuniversal deformation of $(X_z,D_z)$. Suppose it is not.  Since it is versal, it follows from \cite[2.2.7(iii)]{Sernesi_deformations} that the semiuniversal deformation of $(X_z,D_z)$ is trivial; hence $f_z$ induces a trivial formal deformation. By \cite[Theorem 1.7]{Artin_algebraization_I} $f_z$ becomes a trivial deformation after an \'etale base change. This is impossible since the fibers of $f$ are pairwise non-isomorphic; a contradiction.
\end{remark}

\section{Del Pezzo surfaces of rank one and types $8\rA_1$, $4\rA_1+\rD_4$ and $7\rA_1$ ($\cha\kk=2$)}\label{sec:8A1}

Throughout this section we assume $\cha\kk=2$. We now construct surfaces of type $8\rA_1$ from Proposition \ref{prop:canonical_ht>2}\ref{item:8A1} and study their properties. In particular, we prove in Lemma \ref{lem:8A1}\ref{item:8A1_ht},\ref{item:8A1_Aut} that their height equals $4$ and their automorphism groups act transitively on the singular loci; as it was the case for type $4\rA_2$ studied in Section \ref{sec:4A2} above. This automorphism group is computed in \cite[Corollary 5.22]{Kawakami_Nagaoka_canonical_dP-in-char>0}. Nonetheless, we provide a self-contained argument based on the geometry of the witnessing $\P^1$-fibration found in the proof of Proposition \ref{prop:canonical_ht>2}, see Figure \ref{fig:8A1}, which makes it slightly more explicit.

It will be convenient for our future applications to gather here a construction and basic properties of other canonical del Pezzo surfaces of rank one which occur only if $\cha\kk=2$, namely those of types $7\rA_1$ and $4\rA_1+\rD_4$. They are of height $2$, and as such appeared in \cite[Examples 5.5 and 5.6(c)]{PaPe_ht_2}, see Figures \ref{fig:7A1} and \ref{fig:4A1+D4}. Here we present an alternative construction of these surfaces, used e.g.\ in \cite{Kawakami_Nagaoka_canonical_dP-in-char>0}.

\begin{example}[Del Pezzo surfaces of types $7\rA_1$, $4\rA_1+D_4$ and $8\rA_1$]
\label{ex:cha_kk=2}
Let $p_{1},\dots, p_{7}\in \P^2$ be all $\F_{2}$-rational points, and let $\ll_{1},\dots, \ll_{7}\subseteq \P^2$ be all $\F_{2}$-rational lines: each $\ll_{j}$ contains exactly three $p_{i}$'s, and each $p_{i}$ lies on exactly three $\ll_{j}$'s, see Figure \ref{fig:Fano} (note that such a line arrangement cannot be realized unless $\cha\kk=2$). 

Put $\pp=\sum_{j=1}^{7}\ll_{j}$, let $\phi\colon Y\to \P^2$ be a blowup at $p_{1},\dots, p_{7}$, and let $A_i=\phi^{-1}(p_i)$, $D_Y=(\phi^{*}\pp)\redd$. 

\begin{figure}[htbp]
	\begin{tikzpicture}[scale=1.5]
		\draw[add = 0.1 and 0.1] (1,0) to (-1,0);
		\draw[add = 0.1 and 0.1] (1,0) to (0,1.732);
		\draw[add = 0.1 and 0.1] (-1,0) to (0,1.732);
		\draw[add = 0.1 and 0.1] (0,0) to (0,1.732);
		\draw[densely dashed, add = 0.15 and 0.1] (-0.75,0) to (0,1.732);
		\node at (-0.7,-0.15) {\small{$\ll$}};
		\draw[add = 0.2 and 0.8] (1,0) to (0,0.577);
		\draw[add = 0.2 and 0.8] (-1,0) to (0,0.577);
		\filldraw (1,0) circle (0.04);
		\node at (1.1,0.2) {\small{$p_7$}};
		\filldraw (0,0) circle (0.04);
		\node at (-0.2,0.1) {\small{$p_6$}};
		\filldraw (-1,0) circle (0.04);
		\node at (-1.1,0.1) {\small{$p_5$}};
		\filldraw (0,1.732) circle (0.04);
		\node at (-0.2,1.8) {\small{$p_1$}};
		\filldraw (0,0.577) circle (0.04);
		\node at (-0.25,0.577) {\small{$p_4$}};
		\filldraw (0.5,0.866) circle (0.04);
		\node at (0.32,0.95) {\small{$p_3$}};
		\filldraw (-0.5,0.866) circle (0.04);
		\node at (-0.75,0.85) {\small{$p_2$}};
		\draw[add = 0.2 and -0.55] (0,0) to (0.5,0.866);
		\draw[add = -0.47 and 0] (0,0) to (0.6,1.04);
		\draw[add = 0.5 and -0.4] (-0.5,0.866) to (0.25,1.299);
		\draw[add = -0.75 and -0.07] (-0.5,0.866) to (0.25,1.299);
		\draw (0.6,1.04) to[out=60,in=30] (0.4,1.386) -- (0.3,1.328);
		\draw[thick] (-2,-0.2) to[out=0,in=-130] (-1,0) to[out=50,in=-90] (-0.4,0.65) to[out=90,in=-60] (-0.5,0.866) to[out=120,in=180] (0,1.732) to [out=0,in=180] (1.5,1.2) to[out=185,in=0] (0.5,0.866) to[out=180,in=60] (0,0.577) to[out=-120,in=135] (0,0) to[out=-45,in=-135] (1,0) to[out=45,in=180] (1.5,0.2);
		\node at (1.1,1.4) {\small{$\qq$}};
	\end{tikzpicture}
\caption{Example \ref{ex:cha_kk=2}: the Fano plane, a non $\F_{2}$-rational line $\ll$, and a cubic $\qq\in \cQ$.}
\label{fig:Fano}
\end{figure}
 
\begin{parlist}
	\litem{Type $7\rA_1$}\label{item:7A1-construction} 
	Let $Y\to \bar{Y}$ be the contraction of $D_Y$. Then $\bar{Y}$ is a del Pezzo surface of rank one and type $7\rA_{1}$, see Figure \ref{fig:7A1}. It is unique up to an isomorphism since $(\P^2,\pp)$ is; cf.\ \cite[Proposition 5.10]{Kawakami_Nagaoka_canonical_dP-in-char>0}. 
	
	This surface is isomorphic to the one constructed in \cite[Example 5.5]{PaPe_ht_2} with $\nu=3$. Indeed, let $\sigma\colon \P^2\map \P^2$ be a standard quadratic transformation centered at $p_1,p_2,p_3$. Then $\sigma^{-1}_{*}\pp$ is precisely the configuration from loc.\ cit., and $\sigma$ lifts to the required isomorphism between minimal log resolutions.
	
	\litem{Type $4\rA_1+\rD_4$}\label{item:4A1+D4-construction} Put $A_1^{\circ}\de  A_{1}\setminus D_Y$, $D_{Y}'\de D_Y+A_1$. Let $\tau\colon X\to Y$ be a blowup at a point of $A_{1}^{\circ}$, and let $X\to\bar{X}$ be the contraction of $\tau^{-1}_{*}D_Y'$.  Then $\bar{X}$ is of type $4\rA_{1}+\rD_{4}$, see Figure \ref{fig:4A1+D4}.  
	The set of isomorphism classes of such surfaces has moduli dimension $1$, see \cite[Proposition 5.17]{Kawakami_Nagaoka_canonical_dP-in-char>0} or \cite[Example 5.6(c)]{PaPe_ht_2}, with representing family over $A_1^{\circ}$ parametrizing non--$\F_{2}$-rational lines through $p_1$, see Figure \ref{fig:Fano}. The symmetry group of this family is 
	$\Aut(\P^2,\pp,p_1)\cong \Aut(A_{1}^{\circ})\cong S_3$. It has a non-trivial stabilizer corresponding to the non--$\F_2$- but $\F_{4}$-rational lines through $p_1$, see \cite[Corollary 5.20]{Kawakami_Nagaoka_canonical_dP-in-char>0}.
	
	\litem{Type $8\rA_1$}\label{item:8A1-construction} 	Let $\cQ$ be the net of cubics passing through $p_{1},\dots, p_{7}$, see Figure \ref{fig:Fano}. It is given by \begin{equation}\label{eq:Q}
		\{ayz(y+z)+bzx(z+x)+cxy(x+y)=0\},\quad [a:b:c]\in \P^{2}.
	\end{equation} 
	Its irreducible members correspond to $[a:b:c]\not\in \pp$. Any such member has a cusp at the point $[\sqrt{a}:\sqrt{b}:\sqrt{c}]$.
	
	Pick $p\in \P^2\setminus \pp$. Let $\qq\in \cQ$ be a cubic with cusp at $p$, let $\tau\colon X\to Y$ be a blowup at $\phi^{-1}(p)$ and let $X\to\bar{X}$ be the contraction of $(\phi\circ\tau)^{-1}_{*}(\pp+\qq)$. Then $\bar{X}$ is of type $8\rA_{1}$, see Figure \ref{fig:8A1}. The set of isomorphism classes of such surfaces is represented by an almost faithful family over $\P^2\setminus \pp$, with symmetry group 
	$\Aut(\P^2,\pp)=\PGL_{3}(\F_2)$, see 
	 \cite[Proposition 5.21]{Kawakami_Nagaoka_canonical_dP-in-char>0}. Its action has non-trivial stabilizers at $\F_{8}$-rational points of $\P^2\setminus \pp$, see \cite[Corollary 5.22]{Kawakami_Nagaoka_canonical_dP-in-char>0}. 
\end{parlist}
\end{example}

\begin{lemma}[Type $7\rA_1$, cf.\ {\cite[5.10, 5.11]{Kawakami_Nagaoka_canonical_dP-in-char>0}  or \cite[5.9 -- 5.11]{KN_Pathologies}}]\label{lem:7A1}
Let $\bar{X}$ be the del Pezzo surface of rank one and type $7\rA_1$. Let $\pi\colon (X,D)\to (\bar{X},0)$ be its minimal log resolution, and let $\phi\colon (X,D)\to (\P^2,\pp)$ be as in Example \ref{ex:cha_kk=2}\ref{item:7A1-construction}, i.e.\ $\pp$ is the sum of all $\F_2$-rational lines $\ll_1,\dots,\ll_7$, and $\phi$ is a blowup at all $\F_2$-rational points $p_1,\dots,p_7$. Then the following hold.
\begin{enumerate}
	\item \label{item:7A1_cuspidal} All irreducible members of $|-K_{\bar{X}}|$ are cuspidal. The set $\Pcusp(7\rA_1)$ has moduli dimension $2$. In fact, it is represented by an almost universal family over $\P^2\setminus \pp$ with symmetry group $\PGL_3(\F_2)$, see Table \ref{table:symmetries}.
	\item \label{item:7A1_K} We have $K_{X}=-\tfrac{3}{7}D-\tfrac{2}{7}U$, where $U\de\Exc\phi$
	\item \label{item:7A1_-1_curves} The only $(-1)$-curves on $X$ are the preimages of $p_1,\dots, p_7$.
	\item \label{item:7A1_ht} We have $\height{(\bar{X})}=2$. 
	\item \label{item:7A1_Aut} We have natural isomorphisms $\Aut(\bar{X})=\Aut(X,D)=\Aut(\P^2,\pp)=\PGL_{3}(\F_2)$. In particular, the group $\Aut(\bar{X})$ acts transitively on the set $\Sing \bar{X}$.
\end{enumerate}
\end{lemma}
\begin{proof}
\ref{item:7A1_cuspidal} Since $\bar{X}$ is canonical, we have $K_{X}=\pi^{*}K_{\bar{X}}$. Let $Q$ be an irreducible member of $|-K_{X}|$. By adjunction $Q\cdot A_{i}=1$ for all $i$, so $Q$ is isomorphic to its direct image $\phi_{*}Q$. The latter is an irreducible member of the net of cubics $\cQ$ given by formula \eqref{eq:Q}, so it is cuspidal, as needed. 

To construct an almost universal family representing $\Pcusp(7\rA_1)$, we argue as in Remark \ref{rem:4A2-alternative}. Put $\tilde{\P}^2=\P^2\setminus \pp$, where as above $\pp$ is the sum of all $\F_{2}$-rational lines. Let $\Delta\subseteq \P^2\times \tilde{\P}^2$ be the diagonal, let $\bar{\pp}=\pp\times \tilde{\P}^2$ and let $\bar{\cQ}\subseteq \P^2\times \tilde{\P}^2$ be a hypersurface $\{a^2yz(y+z)+b^2zx(z+x)+c^2xy(x+y)=0\}$, where $[x:y:z]$ and $[a:b:c]$ are coordinates of $\P^2$ and $\tilde{\P}^2$, respectively, cf.\ formula \eqref{eq:Q}. Blow up at $\P^2_{\F_2}\times \tilde{\P}^2$ and three times over $\Delta$, each time at the proper transform of $\bar{\cQ}$. Denote the resulting family by $f\colon \cX\to \tilde{\P}^2$, and let $\cD$ be the proper transform of $\bar{\pp}+\bar{\cQ}$. Then $f\colon (\cX,\cD)\to \tilde{\P}^2$ is a family representing $\Pcusp(7\rA_1)$. It is almost faithful, with symmetry group $\PGL_3(\F_2)$: indeed, the  diagonal action of $\PGL_{3}(\F_2)$ on $\P^2\times \tilde{\P}^2$ lifts to an action on $\cX$ such that two fibers of $f$ are isomorphic if and only if they lie in the same orbit. Moreover, the germ of $f$ at every $z\in \tilde{\P}^2$ is versal. Indeed, every  infinitesimal deformation $\check{f}\colon (\check{\cX},\check{\cD})\to T$ of $(X_z,D_z)$ blows down to a configuration of all seven $\F_2$-rational lines and a cuspidal cubic $\qq_{T}$ on $\P^2_{T}$. Such a configuration is a pullback of $\bar{\pp}+ \bar{\cQ}$ along the morphism $\alpha\colon T\to \P^2$ whose graph, viewed as a point of $\P^2_{T}$, is the singular locus of $\qq_{T}$. By the universal property of blowing up we get $\check{f}=\alpha^{*}f$, which proves that the germ of $f$ at $z$ is versal. Exactly as in Remark \ref{rem:4A2-alternative}, we conclude that it is semiuniversal. Indeed, if it is not then by \cite[2.2.7(iii)]{Sernesi_deformations} the morphism from the germ of $\tilde{\P}^2$ at $z$ to the base of a semiuniversal deformation of $(X_z,D_z)$ contracts some curve, hence the restriction of $f$ to this curve, call it $f_0$, is a trivial formal deformation. By \cite[Theorem 1.7]{Artin_algebraization_I} $f_0$  becomes trivial after an \'etale base change, which is impossible, since it has pairwise non-isomorphic fibers.

\ref{item:7A1_K} We have $K_{X}=\phi^{*}K_{\P^2}+U=-\tfrac{3}{7}\phi^{*}\pp+U=-\tfrac{3}{7}D-\tfrac{2}{7}U$, as claimed.

\ref{item:7A1_-1_curves} Let $A\subseteq X$ be a $(-1)$-curve, and suppose $A\neq \phi^{-1}(p_i)$, $i\in \{1,\dots, 7\}$. We have $\pi(A)^2>0>A^2$, so $A\cdot D\geq 1$. Part \ref{item:7A1_K} and adjunction formula give $1=-K_{X}\cdot A=\frac{3}{7}D\cdot A+\frac{2}{7}U\cdot A$, so $U\cdot A=\frac{7}{2}-\frac{3}{2}D\cdot A\leq 2$. If $A$ meets only one component of $U$ then $\phi(A)^2=-1+(U\cdot A)^2\in \{0,3\}$, which is impossible. Hence $A$ meets exactly two components of $U$, so $\phi(A)$ is a line joining two base points of $\phi$, a contradiction since $A\not\subseteq D$. 

\ref{item:7A1_ht} Let $F$ be a fiber of a $\P^{1}$-fibration of $X$. Part \ref{item:7A1_K} and adjunction give $F\cdot D=-\frac{7}{3}F\cdot K_{X}-\frac{2}{3}F\cdot U=\frac{2}{3}(7-F\cdot U)\in \{4,2\}$, since $F\cdot D\geq 1$. The pencil of conics passing through $p_1,\dots, p_4$ induces a $\P^1$-fibration of height $2$ on $(X,D)$, see Figure \ref{fig:7A1}.

\ref{item:7A1_Aut} Clearly, $\Aut(\bar{X})=\Aut(X,D)$. Part \ref{item:7A1_-1_curves} implies that every automorphism of $(X,D)$ fixes $\Exc\phi$, so $\Aut(X,D)=\Aut(\P^{2},\pp)=\PGL_{3}(\F_{2})$.
\end{proof}

Similar description for type $4\rA_1+\rD_4$ is not strictly needed now, but will be convenient as a future reference.

\begin{lemma}[Type $4\rA_1+\rD_4$, cf.\ {\cite[5.19, 5.20]{Kawakami_Nagaoka_canonical_dP-in-char>0} or \cite[5.20, 5.21]{KN_Pathologies}}]\label{lem:4A1+D4}
Let $\bar{X}$ be a del Pezzo surface of rank one and type $4\rA_1+\rD_4$. Let $\hat{\phi}\de \phi\circ \tau\colon (X,D)\to (\P^2,\pp)$ be as in Example \ref{ex:cha_kk=2}\ref{item:4A1+D4-construction}, that is, $\pp$ is the sum of all $\F_{2}$-rational lines $\ll_1,\dots,\ll_7$; $\phi$ is a blowup at all $\F_{2}$-rational points $p_1,\dots, p_7$; and $\tau$ is a blowup at a point $p\in \phi^{-1}(p_1)\setminus D_Y$. Say that $\ll_1,\dots,\ll_4$ do not pass through $p_1$. 
Then the following hold. 
\begin{enumerate}
	\item \label{item:4A1+D4_cuspidal} All irreducible members of $|-K_{\bar{X}}|$ are cuspidal. The set $\Pcusp(4\rA_1+\rD_4)$ has moduli dimension $2$. In fact, it is represented by an almost universal family over $\P^2\setminus \pp$, whose symmetry group is the stabilizer of $p_1\in \P^2_{\F_2}$ in $\PGL_{3}(\F_2)$, which is isomorphic to $S_4$, see Table \ref{table:symmetries}.
	\item \label{item:4A1+D4_-1_curves} The $(-1)$-curves on $X$ are: $\Exc\tau$, $\hat{\phi}^{-1}(p_i)$, $i\in \{2,\dots, 7\}$, $\hat{\phi}^{-1}_{*}\ll$ and $\hat{\phi}^{-1}\cc_{i}$, $i\in \{1,\dots,4\}$;  where $\ll$ is a line through $p_1$ such that $p\in \phi^{-1}_{*}\ll$, and $\cc_i$ is the conic passing through all four $\F_{2}$-rational points not lying on $\ll_i$ and tangent to $\ll$ at $p_1$.
	\item \label{item:4A1+D4_ht} We have $\height{(\bar{X})}=2$. 
	\item \label{item:4A1+D4_Aut} The group $\Aut(\bar{X})=\Aut(X,D)$ is isomorphic to $(\Z/2)^{2}\rtimes G$, where $G=\Z/3$ if the point $p\in \phi^{-1}(p_1)\setminus D_Y$ is $\F_4$-rational and $G=\{1\}$ otherwise. Explicitly, $\Aut(X,D)$ is generated by the following automorphisms:
	\begin{enumerate}
		\item\label{item:4A1+D4_(12)(34)} $(D_1,D_2,D_3,D_4,T_1,T_2,T_3)\mapsto (D_2,D_1,D_4,D_3,T_1,T_2,T_3)$
		\item\label{item:4A1+D4_(13)(24)} $(D_1,D_2,D_3,D_4,T_1,T_2,T_3)\mapsto (D_3,D_4,D_1,D_2,T_1,T_2,T_3)$
		\item\label{item:4A1+D4_Z3} $(D_1,D_2,D_3,D_4,T_1,T_2,T_3)\mapsto (D_2,D_3,D_1,D_4,T_2,T_3,T_1)$ if $p$ is $\F_4$-rational
	\end{enumerate}
	where $D_i=\hat{\phi}^{-1}_{*}\ll_{i}$, $i\in \{1,\dots, 4\}$; and for $i\in \{1,2,3\}$, $T_{i}$ is the twig of the $(-2)$-fork in $D$ such that the lines $\hat{\phi}(T_i), \ll_{i}$ and $\ll_{4}$ meet at one point. 
	
	In particular, the subgroup $(\Z/2)^{2}$ acts transitively on the set of $\rA_1$-type singularities of $\bar{X}$. 
\end{enumerate}
\end{lemma}
\begin{proof}
Put $A=\Exc\tau$. Let $\pi_{Y}\colon Y\to \bar{Y}$ be the contraction of $D_Y$, so $\bar{Y}$ is the surface from Example \ref{ex:cha_kk=2}\ref{item:7A1-construction}.

\ref{item:4A1+D4_cuspidal} Let $Q\subseteq X$ be an irreducible member of $|-K_{X}|$. We have $Q\cdot A=1$ by adjunction, hence $Q\cong \tau(Q)\in |-K_{Y}|$ is cuspidal by Lemma \ref{lem:7A1}\ref{item:7A1_cuspidal}. As before, the assertion about $|-K_{\bar{X}}|$ follows since $\bar{X}$ is canonical. 

The second statement follows from Lemma \ref{lem:7A1}\ref{item:7A1_cuspidal} and \cite[Lemma 2.19(1)]{PaPe_ht_2}. Indeed, let $f_{Y}\colon (\cY,\cD_Y)\to \P^2\setminus \pp$ be an almost universal family representing the set of minimal log resolutions of log surfaces in $\Pcusp(7\rA_1)$, constructed in Lemma \ref{lem:7A1}\ref{item:7A1_cuspidal}. Let $\cQ_{Y}$ be the component of $\cD_{Y}$ corresponding to a member of $|-K_{Y}|$, and let $\cA$ be a divisor on $\cY$ whose restriction to each fiber is a $(-1)$-curve meeting three components of the boundary. Now \cite[Lemma 2.19(1)]{PaPe_ht_2} implies that  a composition of $f_Y$ with a blowup at $\cA\cap \cQ$ is an almost universal family representing the set of minimal log resolutions of log surfaces in $\Pcusp(4\rA_1+\rD_4)$, as needed. The base of this family is the same as in \ref{lem:7A1}\ref{item:7A1_cuspidal}, namely $\P^2\setminus \pp$, and the symmetry group is the stabilizer of the image of $\cA$, i.e.\ $\Aut(\P^2,\pp,p_0)$. We note that this group is isomorphic to $S_4$ permuting $\{\ll_1,\dots,\ll_4\}$.

\ref{item:4A1+D4_-1_curves} Let $C$ be a $(-1)$-curve on $X$. Put $C_Y=\tau(C)$. By Lemma \ref{lem:7A1}\ref{item:7A1_-1_curves}, the $(-1)$-curves on $Y$ are $A_{i}\de \phi^{-1}(p_i)$, $i\in \{1,\dots, 7\}$. If $C\cdot A=0$ then $C_Y=[1]$, so $C_Y=A_i$ for some $i\in \{2,\dots, 7\}$. Thus we can assume $C\cdot A\geq 1$. 

Denote by $T$ the fork in $D$, and let $B$ be its branching component, so $\tau(B)=A_1$. The class $K_{X}\in \NS(X)$ is uniquely determined by equalities $K_{X}^2=10-\rho(X)=1$ and $K_{X}\cdot L=0$ for every component $L$ of $D$, so $K_{X}= -T-B-A$. By adjunction $1=-C\cdot K_{X}=C\cdot (T+B)+C\cdot A$, so  $C\cdot A=1$. In particular, $C_Y=[0]$. 

Since $\pi_{Y}(C_Y)^2>0$, we have $C_Y\cdot D_Y>0$. By Lemma \ref{lem:7A1}\ref{item:7A1_K}, we have $K_{Y}=-\frac{3}{7}D_{Y}-\frac{2}{7}U_{Y}$, where $U_{Y}=\sum_{i=1}^{7} A_i$. By adjunction $2=-K_{Y}\cdot C_Y=\frac{3}{7}D_Y\cdot C_Y+\frac{2}{7}U_{Y}\cdot C_Y$, so $U_{Y}\cdot C_Y\in \{1,4\}$. If $U_Y\cdot C_Y=1$ then $\phi(C)=\ll$. Assume that $U_{Y}\cdot C_Y=4$. Suppose $C_Y\cdot A_i>1$ for some $i\in \{2,\dots, 7\}$. We have $C_Y\cdot A_i\leq C_Y\cdot (U_Y-A_1)=3$, so such $i$ is unique, and $\phi(C)^2=C_Y^2+C_Y\cdot(U_Y-A_i)+(C_Y\cdot A_i)^2\in \{6,10\}$, which is impossible. Hence $C_Y$ meets four components of $U_Y$, once each, so $\hat{\phi}(C)=\cc_i$ for some $i$, as needed. 

\ref{item:4A1+D4_ht} Pulling back a $\P^1$-fibration from $Y$, we see that $\height\bar{X} \leq \height\bar{Y}= 2$ by Lemma \ref{lem:7A1}\ref{item:7A1_ht}. To see that $\height\bar{X}=2$, consider a $\P^1$-fibration which is not a pullback from $Y$, so $A$ is horizontal. Clearly, some degenerate fiber $F$ meets $A$ in a component $C\not \subseteq D$. By Lemma \ref{lem:delPezzo_fibrations}\ref{item:-1_curves} $C$ is a $(-1)$-curve. If $\hat{\phi}(C)=\cc_i$ then $C$ meets $D_i$ twice, so $D_i$ is horizontal and $F\cdot D\geq 2$, as needed. If $\hat{\phi}(C)=p_i$ or $\ll$ then $C$ meets three or four components of $D$, so at least one of them is horizontal; and if exactly one is horizontal then $C$ has multiplicity $2$ in $F$. In any case, we get $F\cdot D\geq 2$, as needed.

\ref{item:4A1+D4_Aut} Part \ref{item:4A1+D4_-1_curves} implies that $A$ is a unique $(-1)$-curve on $X$ meeting $D$ once, so any automorphism of $(X,D)$ fixes $A$. Thus we have a natural isomorphism $\Aut(X,D)\ni \sigma\mapsto \phi \circ \sigma \circ \phi^{-1}\in \Aut(\P^2,\pp+\ll)$.

Consider the group homomorphism $\alpha\colon \Aut(X,D)=\Aut(\P^2,\pp+\ll)\to \Aut(\P^1)$ given by the action of $\Aut(\P^2,\pp+\ll)$ on the pencil of lines passing through $p_1$. Its kernel is the subgroup of $\PGL_{3}(\F_2)$ consisting of elements which fix all three $\F_{2}$-rational lines passing through $p_1$; so $\ker\alpha$ is generated by involutions \ref{item:4A1+D4_(12)(34)} and \ref{item:4A1+D4_(13)(24)}. The image of $\alpha$ is the subgroup of $\PGL_{2}(\F_2)$ consisting of matrices whose one eigenvector is $p\in \phi^{-1}(p)\setminus D_{Y}=\P^{1}\setminus \P^{1}_{\F_2}$. Thus if $p$ is not $\F_{4}$-rational then $\im\alpha$ is trivial; otherwise $\im \alpha\cong \Z/3$ is generated by the image of the automorphism \ref{item:4A1+D4_Z3}. 
\end{proof}

We conclude with the main result of this section.

\begin{lemma}[Type $8\rA_1$, cf.\ {\cite[5.21, 5.22]{Kawakami_Nagaoka_canonical_dP-in-char>0} or \cite[5.22, 5.23]{KN_Pathologies}}]\label{lem:8A1}
Let $\bar{X}$ be a del Pezzo surface of rank one and type $8\rA_1$. Let $\phi\colon (X,D)\to (\P^2,\pp+\qq)$ be as in Example \ref{ex:cha_kk=2}\ref{item:8A1-construction}, i.e.\ $\pp$ is the sum of all $\F_2$-rational lines $\ll_1,\dots,\ll_7$; $\qq$ is a cubic passing through $p_1,\dots,p_7$, with a cusp at $p\not\in \pp$, and $\phi$ is a blowup at $p$ and at all $\F_2$-rational points $p_1,\dots,p_7$ of $\P^2$. Put $D_i=\phi^{-1}_{*}\ll_i$, $D_0=\phi^{-1}_{*}\qq$. Then the following hold. 
\begin{enumerate}
	\item \label{item:8A1_cuspidal} All irreducible members of $|-K_{\bar{X}}|$ are cuspidal. The set $\Pcusp(8\rA_1)$ has moduli dimension $3$; in fact, it is represented by an almost universal family over the threefold \eqref{eq:3fold}, with symmetry group $\PGL_2(\F_3)$.
	\item \label{item:8A1_-1_curves} The only $(-1)$-curves on $X$ are: $\phi^{-1}(p)$ and $\phi^{-1}(p_i)$, $\phi^{-1}_{*}\ll_{i}'$, $\phi^{-1}_{*}\cc_{i}$ for $i\in \{1,\dots, 7\}$, where $\ll_{i}'$ is the line joining $p$ with $p_i$, and $\cc_{i}$ is the conic passing through $p$ and all $p_j\not\in \ll_{i}$.
	\item \label{item:8A1_ht} We have $\height(\bar{X})=4$. 
	\item \label{item:8A1_Aut} The group $\Aut(\bar{X})=\Aut(X,D)$ is isomorphic to $(\Z/2)^{3}\rtimes G$, where $G=\Z/7$ if the point $p$ is $\F_8$-rational, and $G=\{1\}$ otherwise. Explicitly, $\Aut(X,D)$ is generated by the following automorphisms.
	\begin{enumerate}
		\item\label{item:8A1_j} For $j\in \{1,\dots,7\}$: $D_{0}\mapsto D_{j}$, $D_{i}\mapsto D_{k}$ whenever $\ll_i,\ll_j,\ll_k$ meet at one point
		\item\label{item:8A1_Z7} If $p$ is $\F_{8}$-rational: $D_0\mapsto D_0$, $D_{i}\mapsto D_{i+1}$, where $D_{8}=D_{1}$.
	\end{enumerate}
	In particular, the subgroup $(\Z/2)^{3}$ acts transitively on $\Sing\bar{X}$.
\end{enumerate}
\end{lemma}
\begin{proof}
\ref{item:8A1_cuspidal} Like in Lemma \ref{lem:7A1}\ref{item:7A1_cuspidal}, the proper transform of the pencil $|-K_{\bar{X}}|$ on $\P^2$ is contained in the net of cubics $\cQ$ given by formula \eqref{eq:Q}, whose all irreducible members are cuspidal.

For the second assertion, we upgrade the family constructed in the proof of Lemma \ref{lem:7A1}\ref{item:7A1_cuspidal} to keep track of the additional cubic. Let $\tilde{\P}^2=\P^2\setminus \pp$. Consider the product $\P^2\times \tilde{\P}^2\times \tilde{\P}^2$, with coordinates $[x:y:z]$, $[a:b:c]$,  $[\alpha:\beta:\gamma]$, and its subvarieties $\bar{\pp}'\de \pp\times \tilde{\P}^2\times \tilde{\P^2}$, $\Delta_{1}'\de \{(u,u,v): u,v\in \tilde{\P}^2\}$, $\bar{\cQ}_1'\de \{a^2yz(y+z)+b^2zx(z+x)+c^2xy(x+y)=0\}$, $\Delta_2'\de \sigma(\Delta_1')$, $\bar{\cQ}_2'\de \sigma(\bar{\cQ}_1')$, where $\sigma(u,v,w)=(u,w,v)$. Let $\Delta\subseteq \tilde{\P}^2\times \tilde{\P^2}$ be the diagonal, and let $\bar{\pp},\Delta_i$, $\bar{\cQ}_i$ be the restrictions of $\bar{\pp}'$, $\Delta_{i}'$, $\bar{\cQ}_i'$ to $\P^1\times B$, where 
\begin{equation}\label{eq:3fold}
B\de \{\alpha^2bc(b+c)+\beta^2ca(c+a)+\gamma^2ab(a+b)=0\}\subseteq (\tilde{\P}^2\times \tilde{\P}^2)\setminus \Delta.
\end{equation} 
Each fiber of the projection $(\P^2\times B,\bar{\pp}+\bar{\cQ}_1+\bar{\cQ}_2)\to B$ is $(\P^2,\pp+\qq_1+\qq_2)$, where $\qq_1,\qq_2$ are members of the pencil \eqref{eq:Q}, such that the cusp of $\qq_1$ is a smooth point of $\qq_2\setminus \pp$. Moreover, we have $\Sing \bar{\cQ}_i=\Delta_i$. Blowing up once at $\P^2_{\F_2}\times B$, once at $\Delta_1$, and three times over $\Delta_2$, each time on the proper transform of $\bar{\cQ}_2$, we get a family representing $\Pcusp(8\rA_1)$. Exactly as in the proof of Lemma \ref{lem:7A1}\ref{item:7A1_cuspidal} we conclude that this family is almost universal and its symmetry group is $\PGL_{3}(\F_2)$.
\smallskip

\ref{item:8A1_-1_curves} 
Put $C_0=\phi^{-1}(p)$. Let $L\subseteq X$ be a $(-1)$-curve, $L\neq C_0$. As in Example \ref{ex:cha_kk=2}\ref{item:8A1-construction}, let $\tau\colon (X,D-D_0)\to (Y,D_Y)$ be the contraction of $C_0$, so $\pi_{Y}\colon (Y,D_{Y})\to (\bar{Y},0)$ is a minimal log resolution of a del Pezzo surface of type $7\rA_1$. We have $\pi_{Y}\circ \tau(D_0)\in |-K_{\bar{Y}}|$, so since $\bar{Y}$ is canonical, $\tau(D_0)\in |-K_{Y}|$. Since $\tau(C_0)$ is a point of multiplicity $2$ on $\tau(D_0)$, we get $D_0+C_0\in |-K_{X}|$. By adjunction $1=-L\cdot K_{X}=L\cdot D_0+L\cdot C_0$, so $L\cdot C_{0}\in \{0,1\}$. If $L\cdot C_0=0$ then $\tau(L)$ is a $(-1)$-curve on $Y$, so $\phi(L)=p_i$ by Lemma \ref{lem:7A1}\ref{item:7A1_-1_curves}. Assume $L\cdot C_0=1$. Then $L_{Y}\de \tau(L)$ is a $0$-curve on $Y$. Since $\pi_{Y}(L_{Y})^2>0$, we have $L_{Y}\cdot D_{Y}\geq 1$. Put $U_Y=\tau_{*}(\Exc\phi)$. Then Lemma \ref{lem:7A1}\ref{item:7A1_K} and adjunction give $2=-L_{Y}\cdot K_{Y}=\frac{3}{7}L_{Y}\cdot D_{Y}+\frac{2}{7}L_{Y}\cdot U_{Y}$, so $L_{Y}\cdot U_{Y}\in \{1,4\}$. It follows that $\hat{\phi}(L)$ is a line or a conic, so $L=V_{i}$ or $C_{i}$ for some $i\in \{1,\dots, 7\}$, as claimed. 
\smallskip

\ref{item:8A1_ht} To see that $\height\bar{X}\geq 4$, consider a fiber $F$ of a $\P^1$-fibration of $Y$. By Lemma \ref{lem:7A1}\ref{item:7A1_ht}, we have $F\cdot D_{Y}=2$, so by adjunction $2=-F\cdot K_{Y}=F\cdot \tau (D_0)$, and therefore $\tau^{*}F\cdot D=F\cdot D_Y+F\cdot D_0=4$. Thus $|\tau^{*}F|$ induces a $\P^1$-fibration of height $4$, as needed. For the converse inequality, let $F$ be a degenerate fiber of some $\P^1$-fibration of $X$, and let $L$ be a $(-1)$-curve of multiplicity $\mu\geq 1$ in $F$. If $\phi(L)=p_i$ or $\ll_{i}'$ then $L$ meets four $(-2)$-curves in $D$, so combining Lemma \ref{lem:delPezzo_fibrations}\ref{item:-1_curves} with Lemma \ref{lem:degenerate_fibers}\ref{item:F-2} we see that $4\leq \mu L\cdot D\hor\leq F\cdot D\hor$, as needed. If $\phi(L)=\cc_i$ then $L\cdot D_i=2$, so the curve $D_i$ is horizontal, and since $L\cdot (D-D_i)=0$, the fiber $F$ contains another $(-1)$-curve $L'$, which again satisfies $L'\cdot D\hor\geq 2$. Thus $F\cdot D\hor\geq 4$, as claimed.
\smallskip

\ref{item:8A1_Aut} Let $C_0\de\hat{\phi}^{-1}(p)$, $A_{i}\de\hat{\phi}^{-1}(p_i)$, $V_{i}\de\hat{\phi}^{-1}_{*}\ll_{i}'$ and  $C_{i}\de\hat{\phi}^{-1}_{*}\cc_{i}$, $i\in \{1,\dots, 7\}$ be all $(-1)$-curves on $X$. An element $\sigma\in \Aut(X,D)$ is uniquely determined by the images $\sigma(C_0)$ and $\sigma(A_1),\dots,\sigma(A_7)$. Indeed, if $\sigma$ fixes $C_0,A_1,\dots,A_7$ then $\sigma$ descends to automorphism of $\P^2$ fixing $\pp+\qq$ componentwise, so $\sigma=\id$.

It follows from \ref{item:8A1_-1_curves} that if a $(-1)$-curve $L\subseteq X$ satisfies $L\cdot D=4$ then $L=A_i$ or $V_i$ for some $i\in \{1,\dots,7\}$. Hence every $\sigma\in \Aut(X,D)$ fixes $\sum_{i=1}^{7}A_i+V_i$, which is a sum of disjoint chains $A_i+V_i=[1,1]$. Let $\alpha\colon \Aut(X,D)\to S_7$ be the group homomorphism given by the action of $\Aut(X,D)$ on the set of those chains.

Part \ref{item:8A1_-1_curves} implies also that if a $(-1)$-curve $L$ satisfies $L\cdot D=2$ then $L=C_i$ for some $i\in \{0,\dots, 7\}$, in which case $L\cdot D_i=2$ and $L\cdot (D-D_i)=0$. It follows that every $\sigma\in\Aut(X,D)$ satisfies $\sigma(C_0)=C_j$ for some $j\in \{0,\dots, 7\}$, and $\sigma(C_i)=C_k$ if and only if $\sigma(D_i)=D_k$. 

Let $\sigma_{j}$ be an element of $\ker\alpha$ such that $\sigma_{j}(C_0)=C_j$. We have $\sigma_j(A_i)\cdot C_j=\sigma_j(A_i)\cdot \sigma_j(C_0)=A_i\cdot C_0=0$ for all $i\in \{1,\dots, 7\}$. Thus $\sigma_j(A_i)=V_i$ if $A_i$ meets $C_j$ and $\sigma_j(A_i)=A_i$ otherwise. We conclude that 
\begin{equation}\label{eq:sigma_j}
	\sigma_j(A_i)=A_i\mbox{ if }p_i\in \ll_j, \quad \sigma_j(A_i)=V_i \mbox{ if }p_i\not\in \ll_j, \quad \mbox{and } \sigma_j(C_0)=C_j.
\end{equation}
In particular, $\sigma_{j}$ is uniquely determined by $j\in \{0,\dots,7\}$. We have $\sigma_{0}=\id$, and $\sigma_{j}(C_j)=C_0$, so $\sigma_{j}^2=\id$. Fix $k\in \{1,\dots,7\}\setminus \{j\}$. The curve $\sigma_j(C_k)$ meets $\sigma_{j}(A_i+V_i)$ once, and it meets $\sigma_j(A_i)$ if $p_{i}\not\in \ll_{k}$ and $\sigma_{j}(V_i)$ otherwise. Condition \eqref{eq:sigma_j} shows that $\sigma_j(C_k)$ meets $A_{i}$ if and only if $p_{i}\in \ll_{j}\setminus \ll_{k}$ or $p_{i}\in \ll_{k}\setminus \ll_{j}$; so the conic $\phi(\sigma_{j}(C_k))$ contains all $\F_2$ rational points away from the third $\F_2$-rational line passing through $\ll_j\cap \ll_k$, say $\ll_r$. It follows that $\phi(\sigma_{j}(D_k))=\ll_r$, so $\sigma_j(D_{k})=D_r$. Thus $\sigma_j$ is as in \ref{item:8A1_j}. 

Assume for now that all automorphisms $\sigma_0,\dots,\sigma_7$ satisfying \eqref{eq:sigma_j} exist. Then $\ker\alpha=\{\sigma_{0},\dots,\sigma_{7}\}$ is generated by $\sigma_{i},\sigma_{j},\sigma_{k}$ for any triple $\{i,j,k\}$ such that $\ll_i\cap\ll_j\cap\ll_k=\emptyset$, so $\ker\alpha$ is isomorphic to $(\Z/2)^{3}$ and acts transitively on the set of components of $D$. Let $G\subseteq \Aut(X,D)$ be the subgroup consisting of elements fixing $C_0$. Then for every $\sigma\in \Aut(X,D)$ there is $\sigma_j\in \ker \alpha$ such that $\sigma_j\circ\sigma\in G$, so the restriction $\alpha|_{G}\colon G\to \im \alpha$ is surjective. Every element of $G$ descends to an automorphism of $(\P^2,\pp+\qq)$, i.e.\ an element of $\PGL_{3}(\F_2)$ with eigenvector $p$. If $p$ is not $\F_{8}$-rational then $G=\{\id\}$; otherwise $G\cong \Z/7$ permutes $p_1,\dots,p_7$ cyclically, as needed.

The existence of an automorphism $\sigma_{j}\in \Aut(X,D)$ satisfying the condition \eqref{eq:sigma_j} can be inferred from \cite[Theorem 2.1]{Hosoh_automorphisms}, see \cite[p.\ 24]{Kawakami_Nagaoka_canonical_dP-in-char>0}. We provide an explicit construction. 
Say that $j=1$ and $p_5,p_6,p_7\in \ll_1$. Let $\tilde{\phi}\colon X\to \P^2$ be the contraction of $C_1+\sum_{i=1}^{4}V_{i}+\sum_{i=5}^{7}A_{i}$, such that $\tilde{\phi}(V_i)=p_i$, $i\in \{1,\dots,4\}$. Then $\tilde{\phi}_{*}D=\pp+\qq'$, where $\qq'\in \cQ$ is a cubic passing through $p_1,\dots,p_7$, with a cusp at $p'\de \tilde{\phi}(C_1)$. We put $\sigma_j\de \tilde{\phi}^{-1}\circ \phi\colon X\map X$. It remains to prove that $\sigma_{j}$ is regular. This happens if and only if $\qq'=\qq$, i.e.\ $p'=p$.

We prove the last equality by a direct computation. Write $\tilde{\phi}\circ \phi^{-1}\colon \P^2\map \P^2$ as a composition $\tau_{2}\circ \tau_{1}$, where $\tau_{1},\tau_2\colon \P^2\map \P^2$ are quadratic Cremona maps; $\tau_1$ is centered at $p,p_1,p_2$; and $\tau_{2}$ is centered at $q_1\de \tau_1(p_3)$, $q_2\de \tau_1(p_4)$ and $q\de \tau_1(\ll_{12})$, where $\ll_{12}$ is the line joining $p_1$ with $p_2$. Moreover, the coordinates on the target $\P^2$ are chosen in such a way that $\tau_2(p_{i}')=p_i$, where $p_{i}'=\tau_{1}(\ll_{i}')$, $i\in \{1,2\}$; and $\tau_{2}(\ll_{i}'')=p_{i+2}$, $i\in \{1,2\}$, where $\ll_{i}''$ is the line joining $q_i$ with $q$. To find the point $p'$, we note that  $\tau_{1}(\cc_{1})$ is a line joining $q_1$ with $q_2$, so the centers of $\tau_{2}^{-1}$ are: $p_1$, $p_2$, and the point $p'$ we are looking for.

Fix coordinates on $\P^2$ so that $p_1=[1:0:0]$, $p_2=[0:0:1]$, $p_3=[1:1:0]$, $p_4=[0:1:1]$. Then $p=[a:1:b]$ for some $a\neq b$, $a,b\not\in \F_2$. Let $\tau\colon \P^2\map \P^2$ be the standard Cremona transformation $\tau[x:y:z]=[yz:zx:xy]$. Fix $\delta_{1}\in \Aut(\P^2)$ such that $\delta_{1}(p_1)=[1:0:0]$, $\delta_{1}(p_2)=[0:0:1]$ and $\delta_{1}(p)=[0:1:0]$, say 
\begin{equation*}
	\delta_{1}[x:y:z]=[x+ay:y:z+by],
\end{equation*}
and take $\tau_{1}=\tau\circ \delta_{1}$. We have $\delta_{1}(p_3)=[1+a:1:b]$, $\delta_{1}(p_4)=[a:1:1+b]$, so $q_1=\tau_1(p_3)=[b:b(1+a):1+a]$, $q_2=\tau_{1}(p_4)=[1+b:a(1+b):a]$. By definition,  $q=\tau_1(\ll_{12})=[0:1:0]$, $p_{1}'=\tau_1(\ll_{1}')=[0:0:1]$, $p_{2}'=\tau_{1}(\ll_{2}')=[1:0:0]$. Fix $\delta_{2}\in \Aut(\P^2)$ such that $\delta_{2}(q_1)=[1:0:0]$, $\delta_{2}(q_2)=[0:0:1]$, $\delta_{2}(q)=[0:1:0]$, say 
\begin{equation*}
	\delta_{2}[x:y:z]=[ax+(1+b)z:a(1+a)x+(1+a+b)y+b(1+b)z:(1+a)x+bz].
\end{equation*}
and put $\tau_{2}'=\tau\circ\delta_{2}$. Then $\tau_{2}'(\ll_{1}'')=[0:0:1]$, $\tau_{2}'(\ll_{2}'')=[1:0:0]$. Since  $\delta_{2}(p_{1}')=[1+b:b(1+b):b]$, $\delta_{2}(p_{2}')=[a:a(1+a):1+a]$, we have $\tau_{2}'(p_{1}')=[b:1:1+b]$, $\tau_{2}'(p_{2}')=[1+a:1:a]$. Now $\tau_{2}=\delta_{3}\circ \tau_{2}'$, where $\delta_{3}\in \Aut(\P^2)$ is given by $\delta_{3}[b:1:1+b]=[1:0:0]$, $\delta_{3}[1+a:1:a]=[0:0:1]$, $\delta_{3}[0:0:1]=[1:1:0]$, $\delta_{3}[1:0:0]=[0:1:1]$. We compute that  
\begin{equation*}
	\delta_{3}[x:y:z]=[ay+z:x+y+z:x+by]	
\end{equation*}
The centers of $\tau_{2}^{-1}$ are $p_1$, $p_2$ and $p'$, so $p'=\delta_{3}[0:1:0]=[a:1:b]=p$, as claimed.
\end{proof}

\section{Proof of Proposition \ref{prop:canonical}: rational anti-canonical curves on canonical del Pezzo surfaces}\label{sec:relatives}

In this section we prove Proposition \ref{prop:canonical}, classifying log surfaces $(\bar{Y},\bar{T})$, where $\bar{Y}$ is a canonical del Pezzo surface of rank one, and $\bar{T}\subseteq \bar{Y}\reg$ is a rational curve of arithmetic genus one. This result, together with the height computation in Proposition \ref{prop:height} below, completes the proof of Lemma \ref{lem:relative} and thus of Proposition \ref{prop:MT_smooth}. 
\begin{proof}[Proof of Proposition \ref{prop:canonical}]
Let $Y\to \bar{Y}$ be the minimal resolution of $\bar{Y}$, let $D_{Y}$ be its exceptional divisor, and let $T$ be the proper transform of $\bar{T}$ on $Y$. Since $\rho(\bar{Y})=1$, the equality $p_{a}(\bar{T})=1$ implies that $\bar{T}\in |-K_{\bar{Y}}|$, so $T\in |-K_{Y}|$ as $\bar{Y}$ is canonical. In particular, $\rho(Y)=10-T^2=10-\bar{T}^2$ by Noether's formula, which proves \ref{prop:canonical}\ref{item:canonical-Noether}. We now prove the remaining parts of Proposition \ref{prop:canonical} by analyzing each type separately. In each case, $D_Y$ consists of $(-2)$-curves, drawn as solid lines in Figures \ref{fig:can_ht=4}--\ref{fig:can_ht=1}. Some $(-1)$-curves on $Y$ are shown as dashed curves. By adjunction, $T$ meets each such $(-1)$-curve once; and of course is disjoint from $D_Y$. 

\begin{casesp*}
	\litem{$\height(\bar{Y})\geq 3$, see Figure \ref{fig:can_ht=4}} By Proposition \ref{prop:canonical_ht>2}, $\bar{Y}$ is the surface of type $4\rA_2$ from Example \ref{ex:4A2_construction}, or one of the surfaces of type $8\rA_1$ from Example  \ref{ex:cha_kk=2}\ref{item:8A1-construction}. By Lemmas \ref{lem:Aut4A2}\ref{item:4A2_K} and \ref{lem:8A1}\ref{item:8A1_cuspidal} we have $\cha\kk=3$ or $2$, respectively, all irreducible members of $|-K_{\bar{Y}}|$ are cuspidal, as claimed in \ref{prop:canonical}\ref{item:canonical-types}, and the corresponding sets $\Pcusp(4\rA_2)$, $\Pcusp(8\rA_1)$ have moduli dimensions $1$ and $3$, respectively, as claimed in  \ref{prop:canonical}\ref{item:canonical-uniqueness}.
	
	\litem{$\height(\bar{Y})=2$, see Figure \ref{fig:can_ht=2}} Now $\bar{Y}$ is one of the surfaces constructed in \cite[\sec 5A]{PaPe_ht_2}. Note that loc.\ cit.\ gives two constructions of the same surface of type $\rA_3+\rD_5$, here we use the one from Remark 5.7(c). Let $E\subseteq Y$ be the sum of $(-1)$-curves corresponding to dashed lines in Figure \ref{fig:can_ht=2}: we note that the horizontal ones in Figure \ref{fig:2A4} are not explicitly constructed in loc.\ cit, we will construct them in the proof below.
	\smallskip
	
	Assume first that $\bar{Y}$ is of one of the types in \cite[Example 5.6]{PaPe_ht_2}, see Figures \ref{fig:A5+A2}--\ref{fig:4A1+D4}. As in loc.\ cit.\ let $Y\to Y'$ be the contraction of $(-1)$-tips of $D_{Y}+E$ and its images. Let $D_{Y}'$ and $E'$ be the sum of $(-2)$- and $(-1)$-curves in the image of $D_Y+E$. Then $(Y,D_{Y})$ is a minimal resolution of a del Pezzo surface shown in one of the Figures \ref{fig:3A2}--\ref{fig:7A1}, call it $\bar{Y}'$. The image $\bar{T}'$ of $\bar{T}$ on $\bar{Y}'$ is a member of $|-K_{\bar{Y}'}|$ contained in the smooth part of $\bar{Y}'$ and isomorphic to $\bar{T}$. Since the center of each blowup in the decomposition of $Y\to Y'$ is uniquely determined as the common point of the image of $T$ and a specific $(-1)$-curve, universal property of blowing up gives a natural one-to-one correspondence between log surfaces $(\bar{Y},\bar{T})$ and $(\bar{Y}',\bar{T}')$, cf.\  \cite[Lemma 2.18]{PaPe_ht_2}. Thus it is enough to prove Proposition \ref{prop:canonical} for surfaces shown in Figures \ref{fig:3A2}--\ref{fig:7A1}.  
	
	\begin{figure}[htbp]
		\subcaptionbox{$3\rA_{2}$ \label{fig:3A2}}[.1\linewidth]{\centering
			\begin{tikzpicture}[scale=0.9]
				\draw (0.1,3) -- (1.3,3);
				\draw[very thick, dashed] (0.2,3.1) -- (0,2.1);
				\draw (0,2.3) -- (0.2,1.4);
				\draw (0.2,1.6) -- (0,0.7);
				\draw[dashed] (0,0.9) -- (0.2,-0.1);
				%
				\draw (1.2,3.1) -- (1,1.9);
				\draw[dashed] (1,2.1) -- (1.2,0.9);
				\draw[very thick] (1.2,1.1) -- (1,-0.1);
				%
				\draw[very thick] (0.1,0.05) -- (1.1,0.05);
			\end{tikzpicture}
		}
		\subcaptionbox{$\rA_{1}+2\rA_{3}$\label{fig:2A3+A1}}[.2\linewidth]{\centering
			\begin{tikzpicture}[scale=0.9]
				\draw (0.1,3) -- (2.3,3);
				\draw[very thick] (-0.1,0) -- (2.1,0);
				%
				\draw[dashed] (0.2,3.1) -- (0,1.9);
				\draw[very thick] (0,2.1) -- (0.2,0.9);
				\draw[dashed] (0.2,1.1) -- (0,-0.1);
				%
				\draw[very thick] (1.2,3.1) -- (1,1.9);
				\draw[dashed] (1,2.1) -- (1.2,0.9);
				\draw (1.2,1.1) -- (1,-0.1);
				\draw[very thick] (2.2,3.1) -- (2,1.9);
				\draw[dashed] (2,2.1) -- (2.2,0.9);
				\draw (2.2,1.1) -- (2,-0.1);			
			\end{tikzpicture}
		}
		\subcaptionbox{$2\rD_4$ \label{fig:2D4}}[.22\linewidth]{\centering
			\begin{tikzpicture}[scale=0.9]
				\draw (0.1,3) -- (2.9,3);
				\draw[very thick] (-0.1,0) -- (2.9,0);
				%
				\draw (0.2,3.1) -- (0,1.9);
				\draw[dashed] (0,2.1) -- (0.2,0.9);
				\draw[very thick] (0.2,1.1) -- (0,-0.1);
				%
				\draw[very thick] (1.2,3.1) -- (1,1.9);
				\draw[dashed] (1,2.1) -- (1.2,0.9);
				\draw (1.2,1.1) -- (1,-0.1);
				\draw[very thick] (2.2,3.1) -- (2,1.9);
				\draw[dashed] (2,2.1) -- (2.2,0.9);
				\draw (2.2,1.1) -- (2,-0.1);
				\draw[dashed, very thick] (2.8,3.1) -- (3.1,1.4);
				\draw[dashed] (2.8,-0.1) -- (3.1,1.6);
			\end{tikzpicture}			
		}
		\subcaptionbox{$2\rA_1+2\rA_3$, $\cha\kk\neq 2$ \label{fig:2A3+2A1}}[.25\linewidth]{
			\begin{tikzpicture}[scale=0.9]
				\draw[dashed] (0,2.9) -- (3,2.9);
				\draw[very thick] (0.2,3.1) -- (0,1.9);
				\draw[dashed] (0,2.1) -- (0.2,0.9);
				\draw (0.2,1.1) -- (0,-0.1);
				\draw[very thick] (1.2,3.1) -- (1,1.9);
				\draw[dashed] (1,2.1) -- (1.2,0.9);
				\draw (1.2,1.1) -- (1,-0.1);
				\draw[dashed] (1.7,1.6) -- (2.8,1.4);
				%
				\draw[very thick] (2.8,3.1) -- (2.6,1.9);
				\draw (2.6,2.1) -- (2.8,0.9);
				\draw (2.8,1.1) -- (2.6,-0.1);
				\draw[very thick] (-0.05,1.55) to[out=0,in=180] (0.05,1.5) to[out=0,in=180] (1.1,2.7) -- (1.3,2.7) 
				to[out=0,in=170] (1.9,1.56) to[out=-10,in=-135] (2,1.61);
				\draw[very thick] (-0.05,1.45) to[out=0,in=180] (0.05,1.5) to[out=0,in=180] (0.9,0.3) -- (1.1,0.3) 
				to[out=0,in=170] (1.9,1.56) to[out=-10,in=135] (2,1.51);
			\end{tikzpicture}
		}
	
		\subcaptionbox{$\rA_1+\rA_2+\rA_5$ \label{fig:A5+A2+A1}}[.21\linewidth]{\centering
			\begin{tikzpicture}[scale=0.9]
				\draw (0.1,3) -- (2.4,3);
				\draw (0,0) -- (1,0) to[out=0,in=180] (2.3,2.8) -- (2.4,2.8);
				%
				\draw[dashed] (0.2,3.1) -- (0,2.1);
				\draw[very thick] (0,2.3) -- (0.2,1.4);
				\draw[very thick] (0.2,1.6) -- (0,0.7);
				\draw[dashed] (0,0.9) -- (0.2,-0.1);
				%
				\draw (1.2,3.1) -- (1,1.9);
				\draw[dashed, very thick] (1,2.1) -- (1.2,0.9);
				\draw (1.2,1.1) -- (1,-0.1);
				\draw[very thick] (2.3,3.1) -- (2.1,1.9);
				\draw[dashed] (2.1,2.1) -- (2.3,0.9);
				\draw (2.3,1.1) -- (2.1,-0.1);			
			\end{tikzpicture}
		}		
		\subcaptionbox{$2\rA_4$ \label{fig:2A4}}[.3\linewidth]{\centering
			\begin{tikzpicture}[scale=0.9]
				\draw[very thick] (-0.6,3) -- (1.2,3) to[out=0,in=120] (2,1.4);
				\draw[very thick] (-0.6,0) -- (1,0) to[out=0,in=-120] (2,1.6);
				%
				\draw[dashed] (-0.4,3.1) -- (-0.6,2.4);
				\draw (-0.6,2.55) -- (-0.4,1.95);
				\draw[very thick] (-0.4,2.05) -- (-0.6,1.45);
				\draw[very thick] (-0.6,1.55) -- (-0.4,0.95);
				\draw (-0.4,1.05) -- (-0.6,0.45);
				\draw[dashed] (-0.6,0.55) -- (-0.4,-0.1);
				%
				\draw (1.1,3.1) -- (0.9,1.9);
				\draw[dashed, very thick] (0.9,2.1) -- (1.1,0.9);
				\draw (1.1,1.1) -- (0.9,-0.1);	
				\draw[densely dashed] (-0.6,1.75) -- (-0.4,1.75) to[out=0,in=135] (-0.05,1.65);
				\draw[densely dashed] (0.25,1.35) to[out=-45,in=180] (1.2,0.5);
				\draw[densely dashed] (-0.6,1.2) -- (-0.4,1.2) to[out=0,in=-135] (0,1.4) -- (0.2,1.6) to[out=45,in=180]  (1.2,2.5);
			\end{tikzpicture}
		}	
		\subcaptionbox{$7\rA_1$, $\cha\kk=2$ \label{fig:7A1}}[.2\linewidth]{
			\begin{tikzpicture}[scale=0.9]
				%
				\draw (0,1.4) -- (2.4,1.4);
				%
				\draw (0.2,3.1) -- (0,1.9);
				\draw[dashed] (0,2.1) -- (0.2,0.9);
				\draw (0.2,1.1) -- (0,-0.1);
				\draw (1.2,3.1) -- (1,1.9);
				\draw[dashed] (1,2.1) -- (1.2,0.9);
				\draw (1.2,1.1) -- (1,-0.1);
				\draw (2.2,3.1) -- (2,1.9);
				\draw[dashed] (2,2.1) -- (2.2,0.9);
				\draw (2.2,1.1) -- (2,-0.1);			
			\end{tikzpicture}
		}

		\subcaptionbox{$\rA_{2}+\rA_{5}$ \label{fig:A5+A2}}[.12\linewidth]{
			\begin{tikzpicture}[scale=0.9]
				\draw (0,2.6) to[out=0,in=180] (1.3,3);
				\draw[dashed] (0.2,2.9) -- (0,3.9); 
				\draw (0.2,3.1) -- (0,2.1);
				\draw (0,2.3) -- (0.2,1.4);
				\draw (0.2,1.6) -- (0,0.7);
				\draw[dashed] (0,0.9) -- (0.2,-0.1);
				\draw (1.2,3.1) -- (1,1.9);
				\draw[dashed] (1,2.1) -- (1.2,0.9);
				\draw (1.2,1.1) -- (1,-0.1);
				\draw (0.1,0.05) -- (1.1,0.05);
			\end{tikzpicture}
		}
		\subcaptionbox{$\rA_8$ \label{fig:A8}}[.12\linewidth]{
			\begin{tikzpicture}[scale=0.9]
				\draw (0,2.6) to[out=0,in=180] (1.3,3);
				\draw[dashed] (0.2,2.9) -- (0,3.9); 
				\draw (0.2,3.1) -- (0,2.1);
				\draw (0,2.3) -- (0.2,1.4);
				\draw (0.2,1.6) -- (0,0.7);
				\draw[dashed] (0,0.9) -- (0.2,-0.1);
				\draw (1.2,3.1) -- (1,1.9);
				\draw (1,2.1) -- (1.2,0.9);
				\draw[dashed] (1,1.5) -- (2,1.7);
				\draw (1.2,1.1) -- (1,-0.1);
				\draw (0.1,0.05) -- (1.1,0.05);
			\end{tikzpicture}
		}	
		\subcaptionbox{$\rA_2+\rE_6$ \label{fig:E6+A2}}[.12\linewidth]{
			\begin{tikzpicture}[scale=0.9]
				\draw (0,2.6) to[out=0,in=180] (1.3,3);
				\draw[dashed] (-0.1,3.8) -- (0.9,4);
				\draw (0.2,2.9) -- (0,3.9); 
				\draw (0.2,3.1) -- (0,2.1);
				\draw (0,2.3) -- (0.2,1.4);
				\draw (0.2,1.6) -- (0,0.7);
				\draw[dashed] (0,0.9) -- (0.2,-0.1);
				\draw (1.2,3.1) -- (1,1.9);
				\draw[dashed] (1,2.1) -- (1.2,0.9);
				\draw (1.2,1.1) -- (1,-0.1);
				\draw (0.1,0.05) -- (1.1,0.05);
			\end{tikzpicture}	
		}
		\subcaptionbox{$\rA_{3}+\rD_{5}$ \label{fig:D5+A3}}[.18\linewidth]{
			\begin{tikzpicture}[scale=0.9]
				\draw (0,2.6) to[out=0,in=180] (1,3) -- (2.3,3);
				\draw (-0.1,0) -- (2.1,0);
				\draw[dashed] (0.2,2.9) -- (0,4); 
				\draw (0.2,3.1) -- (0,1.9);
				\draw (0,2.1) -- (0.2,0.9);
				\draw[dashed] (0.2,1.1) -- (0,-0.1);
				\draw (1.2,3.1) -- (1,1.9);
				\draw[dashed] (1,2.1) -- (1.2,0.9);
				\draw (1.2,1.1) -- (1,-0.1);
				\draw (2.2,3.1) -- (2,1.9);
				\draw[dashed] (2,2.1) -- (2.2,0.9);
				\draw (2.2,1.1) -- (2,-0.1);			
			\end{tikzpicture}	
		}
		\subcaptionbox{$\rA_1+\rA_7$ \label{fig:A7+A1}}[.2\linewidth]{
			\begin{tikzpicture}[scale=0.9]
				\draw (0.1,3) -- (2.3,3);
				\draw (-0.1,0) -- (2.1,0);
				\draw[dashed] (0.2,3.1) -- (0,1.9);
				\draw (0,2.1) -- (0.2,0.9);
				\draw[dashed] (0.2,1.1) -- (0,-0.1);
				\draw (1.2,3.1) -- (1,1.9);
				\draw[dashed] (1,2.1) -- (1.2,0.9);
				\draw (1.2,1.1) -- (1,-0.1);
				\draw (2.2,3.1) -- (2,1.9);
				\draw (2,2.1) -- (2.2,0.9);
				\draw[dashed] (2,1.5) -- (3,1.7);
				\draw (2.2,1.1) -- (2,-0.1);			
			\end{tikzpicture}
		}	
		\subcaptionbox{$4\rA_{1}+\rD_{4}$, $\cha\kk=2$ \label{fig:4A1+D4}}[.22\linewidth]{
			\begin{tikzpicture}[scale=0.9]
				\draw (0,1.4) -- (2.3,1.4);
				\draw (0.2,3.1) -- (0,1.9);
				\draw[dashed] (0,2.1) -- (0.2,0.9);
				\draw (0.2,1.1) -- (0,-0.1);
				\draw (1.2,3.1) -- (1,1.9);
				\draw[dashed] (1,2.1) -- (1.2,0.9);
				\draw (1.2,1.1) -- (1,-0.1);
				\draw (2.2,3.1) -- (2,1.9);
				\draw (2,2.1) -- (2.2,0.9);
				\draw[dashed] (2,1.6) -- (3,1.8);
				\draw (2.2,1.1) -- (2,-0.1);			
			\end{tikzpicture}
		}	\vspace{-1.5em}
		\caption{Canonical del Pezzo surfaces of rank one and height $2$, see \cite[\sec 5A]{PaPe_ht_2}.}
		\vspace{-0.5em}
		\label{fig:can_ht=2}	
	\end{figure}

Let $\psi\colon (Y,D_Y+E)\to (\P^2,\pp)$ be the morphism contracting all curves which are not thick in the corresponding figure. 
The image of $T$ after each blowdown in $\psi$ is a member of the anti-canonical system, hence by adjunction it meets every $(-1)$-curve exactly once. It follows that $\qq\cong T$ and $\qq\in |-K_{\P^{2}}|$ is a rational cubic. 

We remark that in cases depicted in Figures \ref{fig:3A2}--\ref{fig:2A3+2A1} the morphism $\psi$ is the same as the one used in \cite[\sec 5A]{PaPe_ht_2} to construct $Y$. Moreover, in case $\#D_Y=8$, which is most complicated from our viewpoint, all morphisms $Y\to \P^2$ are classified in \cite{BBD_canonical}. Our argument does not use this classification. 
	
	\begin{types}
		\litem{$3\rA_{2}$, see Figure \ref{fig:3A2}} The image $\pp$ of $D_{Y}+E$ is a triangle, say $\ll_1+\ll_2+\ll$. Write $\{p_0\}=\ll_1\cap \ll_2$, $\{p_j\}=\ll\cap\ll_j$, $j\in \{1,2\}$. We have $(\qq\cdot\ll)_{p_{1}}=3$, $(\qq\cdot \ll_2)_{p_{0}}=3$; so by Lemma \ref{lem:group_law} $\qq\reg$ admits a group law with $p_0=0$, and $\qq\reg\cong \G_{a}$ or $\G_{m}$. We have $3p_1=0$ and $p_1\neq 0$, so $\qq\reg\cong \G_{m}$ if $\cha\kk\neq 3$ and $\qq\reg\cong \G_{a}$ if $\cha\kk=3$; i.e.\ $\qq$ is nodal or cuspidal, respectively. The configuration $\qq+\pp$ exists and is unique up to a projective equivalence, see Lemma \ref{lem:cubic}. Moreover, if $\cha\kk\neq 3$ then the automorphism of $(\P^2,\pp+\qq)$ given by $(p_0,p_1,p_2,r)\mapsto (p_1,p_0,p_2,r)$, where $r$ is the node of $\qq$, switches the analytic branches of $\qq$ at $r$, as needed.
	
	\litem{$\rA_{1}+2\rA_{3}$, see Figure \ref{fig:2A3+A1}} As before, $\pp$ is a sum of lines $\ll_1,\ll_2,\ll_3$ passing through a point $p_0$, and a line $\ll$ which does not. Write $\{p_j\}=\ll_j\cap\ll_0$. We have $(\qq\cdot \ll_{1})_{p_0}=(\qq\cdot \ll_{2})_{p_{2}}=(\qq\cdot\ll_3)_{p_3}=2$. It follows that $\qq$ has an inflection point: indeed, otherwise by Lemma \ref{lem:cubic} we have $\cha\kk=3$ and $\qq$ is cuspidal, so the projection from $p_0$ induces a morphism from the normalization of $\qq$ to $\P^1$, of degree $2$, ramified at the preimages of $p_2$, $p_3$ and the cusp, which is impossible by the Hurwitz formula. 
	
	Thus $\qq\reg$ admits a group law as in Lemma \ref{lem:group_law}. Looking at the intersection of $\qq$ with lines $\ll_1,\ll_2,\ll_3$ and $\cc$, we get equalities $2p_{0}+p_{1}=2p_{2}+p_0=2p_3+p_0=p_{1}+p_{2}+p_{3}=0$ in this group law. Hence $p_{1}=-2p_0$ and $p_{1}=-(p_2+p_3)$, so $3p_{1}=-2p_0-2(p_{2}+p_{3})=-(2p_{2}+p_0)-(2p_{3}+p_0)=0$ and thus we can take $p_{1}=0$. If $\qq\reg\cong \G_{a}$ then the conditions $2p_{0}=0$, $p_{0}\neq 0$ imply that $\cha\kk=2$, so $p_{0}=-2p_{2}=0$; which is false. Hence $\qq\reg \cong \G_{m}$, i.e.\ $\qq$ is nodal. Since $p_0\neq p_1$, the equality $2p_0=0$ implies that $\cha\kk\neq 2$ and $p_0=-1\in \G_m$. Now since $2p_j+p_0=0$ for $j\in \{2,3\}$, we get $\{p_2,p_3\}=\{\imath,-\imath\}$. Up to the action of $\tau\in \Aut(\P^{2},\pp+\qq)$ given by $(p_{0},p_{1},p_{2},p_{3})\mapsto (p_0,p_1,p_3,p_2)$, we get a unique solution $(p_0,p_1,p_2,p_3)=(1,-1,\imath,-\imath)$. 
	
	It follows that the configuration $\pp+\qq$ exists and is unique up to a projective equivalence. Indeed, to construct it, fix a nodal cubic $\qq$ with an inflection point $p_0$, hence with a group law on $\qq\reg\cong \G_m$. Next, let $p_1,p_2,p_3$ be the points of $\qq\reg$ corresponding in the group law to $-1,\imath,-\imath$. They line on one line $\cc$. Letting $\ll_{j}$ be the line joining $p_0$ with $p_j$, $j\in \{1,2,3\}$, we get the required configuration $\pp+\qq$. Moreover, $\tau$ interchanges the branches of $\qq$ at its node. It follows that log surface $(\bar{Y},\bar{T})$, where $\bar{T}\subseteq \bar{Y}\reg$ is a nodal member of $|-K_{\bar{Y}}|$ exists, is unique up to an isomorphism, and admits an automorphism switching branches of $\bar{T}$ at the node, as claimed.

	\litem{$2\rD_{4}$, see Figure \ref{fig:2D4}} Now $\pp$ is a sum of lines $\ll_{1},\dots, \ll_{4}$ passing through a point $p_0$, and one line $\ll$ which does not. Write $\{p_j\}=\ll\cap \ll_{j}$. We have $(\ll_{1}\cdot\qq)_{p_0}=3$; $(\ll_{j}\cdot \qq)_{p_{j}}=2$; $j\in \{2,3\}$ and $(\ll_{4}\cdot\qq)_{p_{4}}=1$. By Lemma \ref{lem:group_law}, $\qq\reg$ admits a group law with $p_0=0$ and $\qq\reg\cong \G_a$ or $\G_m$. Looking at the intersection of $\qq$ with $\ll_2$ and $\ll_3$ we get $2p_{2}=2p_{3}=0$. Since $p_{0}\neq p_{2},p_{3}$, we get $\qq\reg\cong\G_{a}$ and $\cha\kk=2$. But then $2p_{4}=0$, so $(\ll_{4}\cdot\qq)_{p_{4}}=2$; a contradiction. Hence all members of $|-K_{\bar{Y}}|$ contained in $\bar{Y}\reg$ are smooth, as claimed.

	\litem{$2\rA_{1}+2\rA_{3}$, see Figure \ref{fig:2A3+2A1}} We have $\pp=\ll_{1}+\ll_{2}+\ll_{3}+\cc$, where $\ll_{1},\ll_{2},\ll_{3}$ are lines meeting at a point $p_0$, and $\cc$ is a conic tangent to $\ll_{1}$, $\ll_{2}$ at some points $p_{1},p_{2}$, and meeting $\ll_{3}$ at two points, say $p_{3}$, $p_{3}'$. The cubic $\qq$ passes through $p_0$ and is tangent to each $\ll_{j}$ at $p_{j}\in \qq\reg$, $p_j\neq p_0$. The pencil of lines passing through $p_0$ induces a morphism from the normalization of $\qq$ to $\P^1$ of degree two, ramified at three points. Since by assumption in this case we have $\cha\kk\neq 2$, we get a contradiction with the Hurwitz formula \cite[Corollary IV.2.4]{Hartshorne_AG}.  Thus the linear system $|-K_{\bar{Y}}|$ has no singular member in $\bar{Y}\reg$, as claimed.

\litem{$\rA_{1}+\rA_{2}+\rA_{5}$, see Figure \ref{fig:A5+A2+A1}} Now $\pp$ is a sum of a smooth conic $\cc$; two lines $\ll_1$, $\ll_{2}$ tangent to $\cc$ at, say, $p_1$, $p_2$; and a line joining $p_1$ with $p_2$. The cubic $\qq$ meets $\cc$ at three points  $p_1,p_2,q\in \qq\reg$, and we have $(\qq\cdot \ll_j)_{p_j}=3$, $(\qq\cdot \ll_{j})_{p_j}=(\qq\cdot \cc)_{q}=2$ for $j\in \{1,2\}$. Hence $\qq\reg$ admits a group law such that $3p_{j}=0$, $2p_{1}+2p_{2}+2q=0$. Since $3p_1=0$, we can take $p_{1}=0$. Then the equality  $3p_{2}=0$ implies that $\qq\reg\cong \G_{m}$ if $\cha\kk\neq 3$ and $\qq\reg\cong\G_{a}$ if $\cha\kk=3$. If $\cha\kk\in \{2,3\}$ then $2p_{2}+2q=0$ implies that $p_{2}+q=0$, so $p_1$, $p_2$, $q$ are collinear: this is false since these three points lie on a smooth conic $\cc$. Thus $\cha\kk\neq 2,3$. Changing the origin of $\qq\reg\cong \G_{m}$ to a point off $\{p_1,p_2\}$ we get $(p_1,p_2,q)=(\zeta,\zeta^{2},-1)$ for some primitive third root of unity $\zeta$. The two choices of $\zeta$ are equivalent by an automorphism of $\pp+\qq$ given by $(p_1,p_2,q,\ll_1\cap\ll_2)\mapsto (p_2,p_1,q,\ll_1\cap \ll_2)$. 

As before, we conclude that the configuration $\pp+\qq\subseteq \P^2$ exists, is unique up to a projective equivalence, and admits an automorphism switching the branches of $\qq$ at its node; hence the same is true for $(\bar{Y},\bar{T})$.

\litem{$2\rA_{4}$, see Figure \ref{fig:2A4}} First, we construct the horizontal $(-1)$-curves shown in Figure \ref{fig:2A4}. Let $\bar{\psi}\colon Y\to \P^{2}$ be the morphism from \cite[Example 5.3]{PaPe_ht_2}, defined as follows. Choose a section $H\subseteq D_Y$ and order the chains $F_1=[1,2,2,1]$, $F_2=[2,1,2]$ supporting the degenerate fibers in such a way that the first tip of $F_{j}$ meets $H$. Now let $\bar{\psi}$ be the contraction of $F_{1}-F_{1}\cp{2}+F_{2}-F_{2}\cp{1}+H$. Then each  $\bar{\ll}_{j}\de \bar{\psi}_{*}F_j$ is a line, and $\bar{\cc}\de \bar{\psi}_{*}(D_Y)\hor$ is a smooth conic passing through $\{p_0\}\de \bar{\ll}_1\cap\bar{\ll_2}$. Write $\bar{\ll}_j\cap \bar{\cc}=\{\bar{p}_j,\bar{p}_0\}$ for $j\in \{1,2\}$. The required $(-1)$-curves are proper transforms of the line joining $\bar{p}_1$ with $\bar{p}_2$, and of the line tangent to $\bar{\cc}$ at $\bar{p}_1$.

Now we can consider the morphism $\psi\colon Y\to \P^{2}$ defined like before, as the contraction of all thin lines in Figure \ref{fig:2A4}. Let $\ll_{1},\ll_{2}$ and $\ll_{1}',\ll_{2}'$ be images of vertical and horizontal components of $D_Y$, respectively. Then $\ll_{1}$, $\ll_{2}$, $\ll_{1}'$, $\ll_{2}'$ are lines in general position. Write $\{q_1\}=\ll_{1}\cap \ll_{2}'$, $\{q_2\}=\ll_{2}\cap \ll_{1}'$ and $\{q_{j}'\}=\ll_{j}\cap\ll_{j}'$ for $i\in \{1,2\}$. The points $q_1,q_2,q_1',q_2'$ are smooth points of the cubic $\qq$, such that 
$(\qq\cdot\ll_{j})_{q_{j}}=(\qq\cdot \ll_{j}')_{q_{j}'}=2$. The remaining component of $\pp$ is a line meeting $\qq$ at $q_1$, $q_2$ and some $q\in \qq\reg$. 

We claim that $\qq$ has an inflection point. Suppose the contrary, so the cubic $\qq$ is as in Lemma \ref{lem:cubic}\ref{item:funny_cubic}. Applying this result to the point $p_1=q_1$, we get, in the notation of \ref{lem:cubic}\ref{item:funny_cubic} $p_2=q_1'$, $p_3=q_2$ and $p_1=q_{2}'$, so $q_{2}'=q_{1}$, a contradiction. Thus $\qq$ has an inflection point, so $\qq\reg$ admits a group law as in Lemma \ref{lem:group_law}.

In this group law, we have $2q_{1}+q_{1}'=2q_{2}+q_{2}'=2q_{1}'+q_{2}=2q_{2}'+q_{1}=0$, so $(q_{1},q_{2},q_{1}',q_{2}')=(q_{1},4q_{1},13q_{1},7q_{1})$ and $15q_1=0$. Moreover, we have $q=-q_1-q_2=-5q_1$. Hence $3q=0$, so $q$ is an inflection point of $\qq$, and we can take $q=0$. Then $5q_1=0$, so $(q_1,q_2,q_1',q_2')=(q_1,4q_1,3q_1,2q_1)$. If $\cha\kk=5$ then $\qq\reg\cong \G_a$, so $\qq$ is cuspidal, the configuration $\pp+\qq$ is unique up to a projective equivalence, and the result follows.

Assume $\cha\kk\neq 5$. Then $\qq\reg\cong \G_{m}$, so $\qq$ is nodal. In the group law, $q_1$ is a primitive fifth root of unity. The automorphism $\tau$ of $(\P^2,\qq,q)$ switching branches of $\qq$ at the node induces a map $t\mapsto t^{-1}$ in the group law of $\qq\reg$, so $\tau$ maps $(q_1,q_2,q_1',q_2')$ to $(q_2,q_1,q_2',q_1')$, and therefore preserves $\pp$. Hence up to the action of $\tau$, we have two choices of $q_{1}\in \qq\reg\cong \G_{m}$, namely $\zeta$ and $\zeta^2$, where $\zeta$ is a fixed primitive fifth root of unity. This way, we get exactly two cubics $\qq_1$, $\qq_2$ as above; and for each $i\in \{1,2\}$, we have an automorphism $\tau\in \Aut(\P^2,\pp+\qq_i)$ switching the branches of $\qq_i$ at the node. 

It remains to prove that the morphisms $\psi_{i}\colon (Y_{i},D_{Y_i}+T_{i})\to (\P^2,\pp+\qq_{i})$ constructed as above for $i\in \{1,2\}$ yield isomorphic log surfaces $ (Y_{i},D_{Y_i}+T_{i})$. The above uniqueness result for the cubics $\qq_1$, $\qq_2$ implies that the automorphism $\iota\in \Aut(\P^2,\pp)$ given by $\iota\colon (q_1,q_2,q_1',q_2')\mapsto (q_1',q_2',q_{1},q_{2})$ satisfies $\iota(\qq_1)=\qq_2$. Definition of $\psi_{i}$ implies that the birational map $\psi_{2}^{-1}\circ \iota \circ \psi_{1}\colon Y_{1}\map Y_{2}$ is regular: indeed, when constructing $\psi$, we blow up at $(q_1,q_2)$ and $(q_{1}',q_{2}')$ in the same way. Thus we get an isomorphism $(Y_{1},D_{Y_{1}}+T_{1})\to (Y_{2},D_{Y_{2}}+T_{Y_2})$, as needed. Moreover, the same argument shows that the map $\psi_{i}^{-1}\circ\tau\circ \psi_{i}$ is regular, and it yields an automorphism of $(\bar{Y},\bar{T})$ switching the branches of $\bar{T}$ at the node, as needed.

\litem{$7\rA_1$, see Figure \ref{fig:7A1}} Here the result follows from Lemma \ref{lem:7A1}\ref{item:7A1_cuspidal}.
\end{types}

\litem{$\height(\bar{Y})=1$, see Figure \ref{fig:can_ht=1}} The log surface $(Y,D_Y)$ is as in \cite[Remark 4.5]{PaPe_ht_2}. As in case $\height=2$, contracting $(-1)$-curves meeting the boundary once we can assume that all degenerate fibers  are supported on chains of type $[2,1,2]$, call them $F_1,\dots, F_{\nu}$. By assumption $Y\not\cong \F_{2}$, so $\nu\in \{1,2,3\}$ and $\bar{Y}$ is of type $\rA_{1}+\rA_{2}$, $2\rA_{1}+\rA_{3}$, or $3\rA_{1}+\rD_{4}$, see Figures \ref{fig:A1+A2}--\ref{fig:3A1+D4}. Let $H$ be the section in $D_Y$, say that it meets each $F_{j}\cp{1}$. Let $\phi\colon (Y,D_Y)\to (\P^2,\pp)$ be the contraction of $H+F_{1}\cp{1}+F_{1}\cp{2}+\sum_{j=2}^{\nu}(F_{j}\cp{2}+F_{j}\cp{3})$. Then $\ll_{j}\de\phi_{*}F_{j}$ are lines meeting at $\{p_1\}\de \phi(H)$; and $\qq\de\phi(T)$ is a rational cubic. Moreover, $(\qq\cdot \ll_{1})_{p_{1}}=3$, and for $j>1$ we have $\ll_{j}\cap \qq=\{p_1,p_j\}$ for some $p_j\in \qq\reg\setminus \{p_1\}$ such that $(\qq\cdot \ll_{j})_{p_{j}}=2$. 

\begin{figure}[htbp]
	\subcaptionbox{$\rA_1$ \label{fig:A1}}
	[.12\linewidth]
	{
		\raisebox{2.8cm}{
			\begin{tikzpicture}[scale=0.9]
				\draw (0,3) -- (1.2,3);
			\end{tikzpicture}
		}
	}
	\subcaptionbox{$\rA_1+\rA_2$ \label{fig:A1+A2}}
	[.17\linewidth]
	{
		\begin{tikzpicture}[scale=0.9]
			\draw (0,3) -- (1.2,3);
			\draw (0.2,3.1) -- (0,1.9);
			\draw[dashed] (0,2.1) -- (0.2,0.9);
			\draw (0.2,1.1) -- (0,-0.1);
		\end{tikzpicture}
	}
	\subcaptionbox{$2\rA_1+\rA_3$ \label{fig:2A1+A3}}
	[.17\linewidth]{
		\begin{tikzpicture}[scale=0.9]
			\draw (0,3) -- (1.2,3);
			\draw (0.2,3.1) -- (0,1.9);
			\draw[dashed] (0,2.1) -- (0.2,0.9);
			\draw (0.2,1.1) -- (0,-0.1);
			\draw (1.1,3.1) -- (0.9,1.9);
			\draw[dashed] (0.9,2.1) -- (1.1,0.9);
			\draw (1.1,1.1) -- (0.9,-0.1);		
		\end{tikzpicture}	
	}
	\subcaptionbox{$3\rA_1+\rD_4$ \label{fig:3A1+D4}}
	[.2\linewidth]{
		\begin{tikzpicture}[scale=0.9]
			\draw (0,3) -- (2.1,3);
			\draw (0.2,3.1) -- (0,1.9);
			\draw[dashed] (0,2.1) -- (0.2,0.9);
			\draw (0.2,1.1) -- (0,-0.1);
			\draw (1.1,3.1) -- (0.9,1.9);
			\draw[dashed] (0.9,2.1) -- (1.1,0.9);
			\draw (1.1,1.1) -- (0.9,-0.1);
			\draw (2,3.1) -- (1.8,1.9);
			\draw[dashed] (1.8,2.1) -- (2,0.9);
			\draw (2,1.1) -- (1.8,-0.1);		
		\end{tikzpicture}	
	}
\smallskip

	\subcaptionbox{$\rA_4$ \label{fig:A4}}
	[.12\linewidth]
	{
		\begin{tikzpicture}[scale=0.9]
			\draw (0,3) -- (1.2,3);
			\draw (0.2,3.1) -- (0,1.9);
			\draw (0,2.1) -- (0.2,0.9);
			\draw (0.2,1.1) -- (0,-0.1);
			\draw[dashed] (0,1.5) -- (1.2,1.7);
		\end{tikzpicture}
	}	
	\subcaptionbox{$\rD_5$ \label{fig:D5}}
	[.12\linewidth]
	{
		\begin{tikzpicture}[scale=0.9]
			\draw (0,3) -- (1.2,3);
			\draw (0.2,3.1) -- (0,1.9);
			\draw (0,2.1) -- (0.2,0.9);
			\draw (0.2,1.1) -- (0,-0.1);
			\draw (-0.05,1.7) -- (1.15,1.9);
			\draw[dashed] (0.9,2.1) -- (1.1,0.9);
		\end{tikzpicture}
	}
	\subcaptionbox{$\rE_6$ \label{fig:E6}}
	[.15\linewidth]
	{
		\begin{tikzpicture}[scale=0.9]
			\draw (0,3) -- (1.2,3);
			\draw (0.2,3.1) -- (0,1.9);
			\draw (0,2.1) -- (0.2,0.9);
			\draw (0.2,1.1) -- (0,-0.1);
			\draw (-0.05,1.7) -- (1.15,1.9);
			\draw (0.9,2.1) -- (1.1,0.9);
			\draw[dashed] (1.1,1.1) -- (0.9,-0.1);
		\end{tikzpicture}
	}
	\subcaptionbox{$\rE_7$ \label{fig:E7}}
	[.15\linewidth]
	{
		\begin{tikzpicture}[scale=0.9]
			\draw (0,3) -- (1.8,3);
			\draw (0.2,3.1) -- (0,1.9);
			\draw (0,2.1) -- (0.2,0.9);
			\draw (0.2,1.1) -- (0,-0.1);
			\draw (-0.05,1.7) -- (1.15,1.9);
			\draw (0.9,2.1) -- (1.1,0.9);
			\draw (1.1,1.1) -- (0.9,-0.1);
			\draw[dashed] (0.8,0.2) -- (2,0);
		\end{tikzpicture}
	}	
	\subcaptionbox{$\rE_8$ \label{fig:E8}}
	[.15\linewidth]
	{
		\begin{tikzpicture}[scale=0.9]
			\draw (0,3) -- (2.1,3);
			\draw (0.2,3.1) -- (0,1.9);
			\draw (0,2.1) -- (0.2,0.9);
			\draw (0.2,1.1) -- (0,-0.1);
			\draw (-0.05,1.7) -- (1.15,1.9);
			\draw (0.9,2.1) -- (1.1,0.9);
			\draw (1.1,1.1) -- (0.9,-0.1);
			\draw (0.8,0.2) -- (2,0);
			\draw[dashed] (2,1.1) -- (1.8,-0.1);
		\end{tikzpicture}
	}
	\subcaptionbox{$\rA_1+\rA_5$ \label{fig:A1+A5}}
	[.17\linewidth]{
		\begin{tikzpicture}[scale=0.9]
			\draw (0,3) -- (2,3);
			\draw (0.2,3.1) -- (0,1.9);
			\draw[dashed] (0,2.1) -- (0.2,0.9);
			\draw (0.2,1.1) -- (0,-0.1);
			\draw (1.1,3.1) -- (0.9,1.9);
			\draw (0.9,2.1) -- (1.1,0.9);
			\draw (1.1,1.1) -- (0.9,-0.1);
			\draw[dashed] (0.9,1.5) -- (2.1,1.7);		
		\end{tikzpicture}	
	}
\smallskip

	\subcaptionbox{$\rA_7$ \label{fig:A7}}
	[.17\linewidth]{
		\begin{tikzpicture}[scale=0.9]
			\draw (-1,3) -- (2,3);
			\draw (0.2,3.1) -- (0,1.9);
			\draw (0,2.1) -- (0.2,0.9);
			\draw (0.2,1.1) -- (0,-0.1);
			\draw (1.1,3.1) -- (0.9,1.9);
			\draw (0.9,2.1) -- (1.1,0.9);
			\draw (1.1,1.1) -- (0.9,-0.1);
			\draw[dashed] (0.9,1.5) -- (2.1,1.7);	
			\draw[dashed] (0.2,1.5) -- (-1,1.3);	
		\end{tikzpicture}	
	}
	\subcaptionbox{$\rA_1+\rD_6$ \label{fig:A1+D6}}
	[.2\linewidth]
	{
		\begin{tikzpicture}[scale=0.9]
			\draw (0,3) -- (2.1,3);
			\draw (0.2,3.1) -- (0,1.9);
			\draw[dashed] (0,2.1) -- (0.2,0.9);
			\draw (0.2,1.1) -- (0,-0.1);
			\draw (1.1,3.1) -- (0.9,1.9);
			\draw (0.9,2.1) -- (1.1,0.9);
			\draw (1.1,1.1) -- (0.9,-0.1);
			\draw (0.85,1.7) -- (2.05,1.9);
			\draw[dashed] (1.8,2.1) -- (2,0.9);
		\end{tikzpicture}
	}	
	\subcaptionbox{$\rD_8$ \label{fig:D8}}
	[.2\linewidth]
	{
		\begin{tikzpicture}[scale=0.9]
			\draw (-1,3) -- (2.1,3);
			\draw (0.2,3.1) -- (0,1.9);
			\draw (0,2.1) -- (0.2,0.9);
			\draw[dashed] (0.2,1.5) -- (-1,1.3);
			\draw (0.2,1.1) -- (0,-0.1);
			\draw (1.1,3.1) -- (0.9,1.9);
			\draw (0.9,2.1) -- (1.1,0.9);
			\draw (1.1,1.1) -- (0.9,-0.1);
			\draw (0.85,1.7) -- (2.05,1.9);
			\draw[dashed] (1.8,2.1) -- (2,0.9);
		\end{tikzpicture}
	}
	\subcaptionbox{$\rA_1+\rE_7$ \label{fig:A1+E7}}
	[.2\linewidth]
	{
		\begin{tikzpicture}[scale=0.9]
			\draw (0,3) -- (2.1,3);
			\draw (0.2,3.1) -- (0,1.9);
			\draw[dashed] (0,2.1) -- (0.2,0.9);
			\draw (0.2,1.1) -- (0,-0.1);
			\draw (1.1,3.1) -- (0.9,1.9);
			\draw (0.9,2.1) -- (1.1,0.9);
			\draw (1.1,1.1) -- (0.9,-0.1);
			\draw (0.85,1.7) -- (2.05,1.9);
			\draw (1.8,2.1) -- (2,0.9);
			\draw[dashed] (2,1.1) -- (1.8,-0.1);
		\end{tikzpicture}
	}
	\subcaptionbox{$2\rA_1+\rD_6$ \label{fig:2A1+D6}}
	[.2\linewidth]{
		\begin{tikzpicture}[scale=0.9]
			\draw (0,3) -- (3,3);
			\draw (0.2,3.1) -- (0,1.9);
			\draw[dashed] (0,2.1) -- (0.2,0.9);
			\draw (0.2,1.1) -- (0,-0.1);
			\draw (1.1,3.1) -- (0.9,1.9);
			\draw[dashed] (0.9,2.1) -- (1.1,0.9);
			\draw (1.1,1.1) -- (0.9,-0.1);
			\draw (2,3.1) -- (1.8,1.9);
			\draw (1.8,2.1) -- (2,0.9);
			\draw (2,1.1) -- (1.8,-0.1);
			\draw[dashed] (1.8,1.5) -- (3,1.7);		
		\end{tikzpicture}	
	}	
	\caption{Canonical del Pezzo surfaces of rank one and height one, see \cite[Remark 4.5]{PaPe_ht_2}.}
	\label{fig:can_ht=1}
\end{figure}

If $\nu=1$ then $\pp+\qq\subseteq \P^2$ consists of a rational cubic and an inflectional tangent line. By Lemma \ref{lem:cubic} there are exactly two such configurations, one with $\qq$ nodal and the other with $\qq$ cuspidal. The former admits an automorphism switching the branches of $\qq$ at the node, as needed.

Assume $\nu\geq 2$. Since $p_1\in \qq$ is an inflection point, we can choose a group law on $\qq\reg$ with $p_1=0$. Then $2p_{j}=0$ for $j\in \{2,\dots,\nu\}$. Assume that $\qq$ is nodal. Then $p_j=\pm 1$ in $\qq\reg\cong \G_{m}$, so the fact that $p_1,\dots,p_{\nu}$ are distinct implies that $\nu=2$, $\cha\kk\neq 2$; and $(\P^2,\pp)$ is unique. Moreover, the automorphism interchanging branches of $\qq$ at the node, given by $t\mapsto t^{-1}$ on $\qq\reg$, preserves $\pp$, as needed.

Assume that $\qq$ is cuspidal. Then the equality $2p_{2}=0$ in $\qq\reg\cong \G_{a}$ implies that $\cha\kk=2$, and we can take for $p_{2},\dots,p_{\nu}$ any $\nu-1$ distinct points of $\qq\reg\setminus \{p_1\}$. To prove the uniqueness result \ref{prop:canonical}\ref{item:canonical-uniqueness}, fix a cuspidal cubic $\qq$, which is unique up to a projective equivalence, and let $q$ be its cusp. Since $\cha\kk=2$ by Lemma \ref{lem:cubic}\ref{item:cuspidal_cubic} $\qq$ has a unique inflection point $p_1$. The group $\Aut(\P^2,\qq)\cong \G_{m}$ fixes $p_1$ and acts transitively on $\qq\setminus \{q,p_1\}\cong \A^{1}_{*}$. Thus a point $p_2\in \qq\reg\setminus \{p_1\}$ is unique up to the action of $\Aut(\P^2,\qq)$. 

If $\nu=2$ it follows that $(\bar{Y},\bar{T})$ is unique up to an isomorphism. Assume $\nu=3$. Since $\Aut(\P^2,\qq,\{p_1,p_2\})$ is finite, the choice of $p_3\in \qq\reg\setminus \{p_1,p_2\}$ parametrizes an almost universal family representing $\Pcusp(3\rA_1+\rD_4)$. To be more precise, we argue as follows. Let $(Y',D')$ be a minimal log resolution of $(\bar{Y},\bar{T})\in \Pcusp(3\rA_1+\rD_4)$, let $H'=[2]$ be the $1$-section in $D'$, and let $(Y',D'-H')\to(Z,D_Z)$ be the contraction of the proper transform of the line joining $p_1$ with the cusp of $\qq$. Then $D_Z$ is a sum of four disjoint chains $[2,1,2]$ and a $(-4)$-curve $Q_Z$ meeting each of those chains once, in a $(-1)$-curve, see \cite[Figure 22]{PaPe_ht_2}. By Lemma 5.16 loc.\ cit., the set of isomorphism classes of such log surfaces $(Z,D_Z)$ has moduli dimension $1$. By Lemma 2.20 loc.\ cit.\  the same holds for $(Z,D_Z+H_Z)$, where $H_Z=[1]$ is the proper transform of $H'$. Applying Lemma 2.18 loc.\ cit. to the inner blowup $(Y',D')\to (Z,D_Z+H_Z)$ we get the required statement for $\Pcusp(3\rA_1+\rD_4)$.

\phantomsection\label{moduli}
	We have constructed an almost universal family representing $\Pcusp(3\rA_1+\rD_4)$ over the same base as in \cite[Lemma 5.16]{PaPe_ht_2}, i.e.\ over $B\cong \P^1\setminus \{{0,1,\infty}\}$. As noted at the beginning of this section, \cite[Lemma 2.18]{PaPe_ht_2} applied to a blowup at one of the points of $D_Y\cap T$ shows that this family lifts to one representing $\Pcusp(2\rA_1+\rD_6)$, over the same base. We claim that their symmetry groups are $S_3$ and $\Z/2$, respectively, see Table \ref{table:symmetries}. 
	
The point $b\in B$ corresponds to the cross-ratio of the four points  $Q_Z\cap (D_Z-Q_Z)$ on $Q_Z\cong \P^1$, call them $q_0,q_1,q_2,q_3$. Say that $q_0$ lies on the fiber containing the base point of the morphism $Y'\to Z$; and for the type $2\rA_1+\rD_6$ say that $q_1$ is the image of the center of $\tau$. The fibers corresponding to $(q_0,q_1,q_2,q_3)$ and $(q_0',q_1',q_2',q_3')$ are isomorphic if and only if there is an automorphism of $Q_Z$ mapping $\{q_0,q_1,q_2,q_3\}$ to $\{q_0',q_1',q_2',q_3'\}$ and mapping $q_0$ (resp.\ $q_0,q_1$) to $q_0'$ (resp.\ $q_0',q_1'$). The image in $\Aut(B)=S_3$ of the group of such automorphisms is $S_3$ and $\Z/2$, respectively, as needed.
\qedhere\end{casesp*}
\end{proof}

Our last result completes the proof of Proposition \ref{prop:MT_smooth} by computing the height of log surfaces in \ref{prop:MT_smooth}\ref{item:MT-eliptic-relative}. This result follows from \cite[Lemma 6.9(a)]{PaPe_ht_2}, but for completeness we give a direct argument here.

\begin{proposition}\label{prop:height}
	Let $(\bar{Y},\bar{T})$ be one of the log surfaces listed in Proposition \ref{prop:MT_smooth}\ref{item:MT-eliptic-relative}, that is, $\bar{Y}$ is a canonical del Pezzo surface of rank one, not of type $\rE_6$, $\rE_7$ or $\rE_8$, whose minimal resolution has at least $6$ exceptional curves; and $\bar{T}\subseteq \bar{Y}\reg$ is a cuspidal member of $|-K_{\bar{Y}}|$; see Table \ref{table:canonical}. Then $\height(\bar{Y},\bar{T})=\height(\bar{Y})+2$.
\end{proposition}
\begin{proof}
	Let $(X,D)\to (\bar{Y},\bar{T})$ be the minimal log resolution. It factors through the minimal log resolution $(Y,D_{Y}')\to (\bar{Y},0)$. Let $T_{Y}$ be the proper transform of $\bar{T}$ on $Y$, and let $T_{Y}^2$. Since $T_{Y}\in |-K_{Y}|$, by  Noether's formula $T_{Y}^{2}=9-\#D_{Y}'$, and by adjunction $T_{Y}$ is a $2$-section for any $\P^1$-fibration of $Y$. Pulling back such fibrations to $X$, we get $\height(X,D)\leq \height(\bar{Y})+2$. It remains to prove the opposite inequality.
	
	Let $T$, $D_{Y}$ be the proper transforms of $T_Y$ and $D_{Y}'$ on $X$, and let $D_0$ be the connected component of $D$ containing $T$, so $D_0$ is a fork with branching $(-1)$-curve $R$ and twigs $T_2=[2]$, $T_3=[3]$, $T=[k]$, where $k=6-T_{Y}^2=\#D_Y-3\geq 3$ by assumption; and $D_Y=D-D_0$. Then the morphism $X\to Y$ contracts $R+T_2+T_3$, and is an isomorphism near $D_Y$. Since $T_{Y}\in |-K_{Y}|$, we have $-K_{X}=T+T_2+T_3+2R$.

	\begin{claim*} Let $L\neq R$ be a $(-1)$-curve meeting $T_3$. Then $L\cdot T_3=L\cdot D_0=1$ and $L\cdot D_{Y}\geq \height(\bar{Y})$. Moreover, if $L\cdot D_{Y}=1$ then the component of $D_{Y}$ meeting $L$ is not a tip of $D_Y$.
	\end{claim*}
	\begin{proof}
		By adjunction, we have $L\cdot (T+T_2+T_3+2R)=1$, so $L\cdot T_3=L\cdot D_0=1$. The image of $L$ on $Y$ is a $0$-curve, so it meets $D_{Y}'$ at least $\height(\bar{Y})$ times, as needed.
		
		Suppose that $L\cdot D_{Y}=1$, and the component $H$ of $D_Y$ meeting $L$ is a tip of $D_Y$. The image of $H$ on $Y$, say $H'$, is a $1$-section, and $(D_{Y}')\hor=H'$. By Lemma \ref{lem:delPezzo_fibrations} every degenerate fiber contains exactly one $(-1)$-curve, so it meets $H'$ in a component of $D_{Y}'$, see Figure \ref{fig:can_ht=1}. Since $H'$ is a tip of $D_Y'$, it follows that there is only one degenerate fiber, and Lemma \ref{lem:degenerate_fibers}\ref{item:F-2} shows that it is supported on a fork $\langle 2;[2],[2],[1,(2)_{k-4}]\rangle$, where $k=\#D_{Y}\geq 6$. Thus $\bar{Y}$ is of type $\rE_{k}$, contrary to our assumptions.
	\end{proof}
	
	Fix a $\P^{1}$-fibration of $X$ with a fiber $F$. By adjunction we have $2=F\cdot (T+T_2+T_3+2R)$, so either $R$ is a $1$-section and $T$, $T_{2}$, $T_{3}$ are vertical; or $R$ is vertical, and $(D_{0})\hor$ consists of two $1$-sections or one $2$-section.
	
	Assume that $T_{3}$ lies in some fiber, say $F_{3}$. Then $T_{3}$ meets a vertical $(-1)$-curve $L\neq R$: indeed, otherwise  $T_{3}$ is a tip of $(F_3)\redd$ by Lemma \ref{lem:delPezzo_fibrations}, so the $(-1)$-curve $R$ meeting it has multiplicity at least $3$ in $F_{3}$, which is impossible. Assume that $L$ meets a vertical component $V$ of $D_{Y}$. If $R$ is vertical then $F_{3}=R+T_{3}+2L+V$, so $F_{3}\cdot D=2+2(L\cdot D_{Y}-1)+\beta_{D_Y}(V)\geq \height(\bar{Y}) +L\cdot D_{Y}+\beta_{D_{Y}}(V)\geq \height(\bar{Y})+2$ by the Claim above. If $R$ is horizontal then $(F_{3})\redd=[3,1,2,2]$ or $[1,3,1,2]$. In the first case, $F_{3}\cdot D\geq 1+3(L\cdot D_{Y}-1)+2(\beta_{D_Y}(V)-1)=3L\cdot D_{Y}+2\beta_{D_Y}(V)-4\geq \height(\bar{Y})+2$ by the Claim. In the second case, denoting the other $(-1)$-curve in $F_{3}$ by $L'$, we have $F_{3}\cdot D\hor\geq 1+2(L\cdot D_{Y}-1)+\beta_{D_{Y}}(V)+L'\cdot D_{Y}\geq \height(\bar{Y})+2$ as before. Eventually, assume that every $(-1)$-curve $L\neq R$ meeting $T_{3}$ is disjoint from $(D_{Y})\vert$. Then either $R$ is horizontal and $F_{3}=\langle 3;[1],[1],[1]\rangle$, so $F_{3}\cdot D\hor \geq 1+3\height(\bar{Y})\geq \height(\bar{Y})+2$; or $R$ is vertical and $F_{3}\cdot D\hor \geq 2+\height(\bar{Y})$; as needed.
	
	Assume now that $T_{3}$ is horizontal, so $R$ lies in a fiber $F_{R}$, which contains all vertical components of $D_0$. Since $(D_{0})\vert=[1,k]$, $[1,2]$ or $[2,1,k]$ the fiber $F_{R}$ contains at most $k-2$ vertical $(-2)$-curves. If $F_{R}$ is a unique degenerate fiber then $\#D\hor\geq 1+\#(D_Y)\hor\geq 1+\#D_Y-(k-2)=6\geq \height(\bar{Y})+2$, as needed. Thus we can assume that there is a degenerate fiber $F\neq F_{R}$. Such a fiber meets $T_{3}$ in a vertical $(-1)$-curve $L\neq R$. 
	
	If $(D_{Y})\hor=0$ then, since $F_{R}$ has at least two $(-1)$-curves, Lemma \ref{lem:delPezzo_fibrations}\ref{item:Sigma} implies that $\#(D_0)\hor=2$, so $T_3$ is a $1$-section; and $L$ is the unique $(-1)$-curve in $F$, so it has multiplicity at least $2$ in $F$, which is impossible. Thus $(D_{Y})\hor\neq 0$. In particular $F\cdot D=2+F\cdot D_{Y}\geq 3$, so we can assume  $\height(\bar{Y})\geq 2$.
	
	If $T_{3}$ is a $2$-section then either $L$ has multiplicity $2$ in $F$; or $T_{3}$ meets two $(-1)$-curves in $F$. In any case, the Claim gives $F\cdot D\geq 2+2\cdot(\height(\bar{Y})-L\cdot D\vert)\geq \height(\bar{Y})+2$ since $L\cdot D\vert\leq 2$ and $\height(\bar{Y})\geq 2$ by assumption.
	
	Assume now that $T_{3}$ is a $1$-section. Then $L$ has multiplicity one in $F$, so $L\cdot D\vert\leq 1$. Like before, we compute $F\cdot D\geq 2+\height(\bar{Y})-L\cdot D\vert$. The lemma follows unless $L\cdot D\vert=1$, $F-L$ is disjoint from $(D_Y)\hor$, and $F\cdot D_{Y}=\height(\bar{Y})-1$. Thus we can assume that all degenerate fibers are: $F_{R},F_{1},\dots, F_{s}$, where $F_{i}=[1,(2)_{s_{i}},1]$ for some $s_{i}\geq 1$, and $F_{i}$ meets $(D_{Y})\hor+T_{3}$ only in the first component, with multiplicity $\height(\bar{Y})$. 
	
	Let $T'\in \{T,T_2\}$ be the vertical component of $D_0-R$. Write $T'=[k']$, so $k'\in \{k,2\}$. Since $F_{R}$ is supported on the sum of $T'$, some $(-2)$-curves in $D_{Y}$, and $(-1)$-curves, we have $F_{R}=R+T'+\sum_{j=1}^{v}V_{j}$, where $V_{j}$ is a chain $[1,(2)_{v_{j}}]$, meeting $T'$ in a $(-1)$-curve; for some $v_{j}\geq 0$ such that $k'-1=\sum_{j}(v_{j}+1)=v+\sum_{j}v_j$.
		
	Suppose that some component $H$ of $(D_{Y})\hor$ meets a tip of $F_{R}$. Contracting all vertical curves which are disjoint from $H$ we obtain a morphism onto a Hirzebruch surface such that the images of $H$ and $T_{3}$ have self-intersection $-2$ and $-1$, which is impossible. It follows that $D_Y$ does not contain any $1$-sections; and $v_{j}\geq 1$ for all $j$. The latter implies that $k'\geq 3$, so $k'=k$, i.e.\ $T$ is vertical. Moreover, $F\cdot (D_Y)\hor\geq 2$, so $\height(\bar{Y})=F\cdot ((D_Y)\hor+T_3)\geq 3$. By Proposition \ref{prop:canonical_ht>2}, $\bar{Y}$ is of type $4\rA_2$ or $8\rA_1$. It follows that $v_{j}\in \{1,2\}$ for all $j\in \{1,\dots, v\}$. Since $8=\#D_{Y}=3+k=4+v+\sum_{j}v_{j}$, we get $v=2$ and $v_{1}=v_{2}=1$. Thus $\bar{Y}$ is of type $8\rA_1$, so  $s_i=1$ for all $i\in \{1,\dots, s\}$. Since $\#(D_Y)\vert=\sum_{i}s_i+\sum_{j}v_j=s+2$ and $\#D_{Y}=8$, we get $\#(D_{Y})\hor=8-s-2=6-s$, so $\#D\hor=8-s$. By Lemma \ref{lem:delPezzo_fibrations}\ref{item:Sigma} we have $s+v=\#D\hor-1$, so $s+2=8-s-1$, and eventually $s=\frac{5}{2}$, a contradiction.
\end{proof}

\bibliographystyle{amsalpha}
\bibliography{bibl}

\providecommand{\bysame}{\leavevmode\hbox to3em{\hrulefill}\thinspace}
\providecommand{\MR}{\relax\ifhmode\unskip\space\fi MR }
\providecommand{\MRhref}[2]{%
  \href{http://www.ams.org/mathscinet-getitem?mr=#1}{#2}
}
\providecommand{\href}[2]{#2}
\begin{thebibliography}{MKM83}

\bibitem[AN06]{Alexxev-Nikulin_delPezzo-index-2}
V.~Alexeev and V.~Nikulin, \emph{Del {P}ezzo and {$K3$} surfaces}, MSJ Memoirs,
  vol.~15, Math. Soc. Japan, Tokyo, 2006.

\bibitem[Art69]{Artin_algebraization_I}
M.~Artin, \emph{Algebraization of formal moduli. {I}}, Global {A}nalysis
  ({P}apers in {H}onor of {K}. {K}odaira), Univ. Tokyo Press, Tokyo, 1969,
  pp.~21--71.

\bibitem[BBD84]{BBD_canonical}
D.~Bindschadler, L.~Brenton, and D.~Drucker, \emph{Rational mappings of del
  {P}ezzo surfaces, and singular compactifications of two-dimensional affine
  varieties}, Tohoku Math. J. (2) \textbf{36} (1984), no.~4, 591--609.

\bibitem[BT22]{BT_del-Pezzo-char>0}
F.~Bernasconi and H.~Tanaka, \emph{On del {P}ezzo fibrations in positive
  characteristic}, J. Inst. Math. Jussieu \textbf{21} (2022), no.~1, 197--239.

\bibitem[Che97]{Cheltsov_non-rational}
I.~Chel'tsov, \emph{Del {P}ezzo surfaces with nonrational singularities}, Mat.
  Zametki \textbf{62} (1997), no.~3, 451--467.

\bibitem[Dem80]{Demazure}
M.~Demazure, \emph{Surfaces de del {P}ezzo — {I}}, Séminaire sur les
  Singularités des Surfaces, Lecture Notes in Mathematics, vol. 777, Springer,
  Berlin, Heidelberg, 1980.

\bibitem[DR01]{Daigle_Russell}
D.~Daigle and P.~Russell, \emph{Affine rulings of normal rational surfaces},
  Osaka J. Math. \textbf{38} (2001), no.~1, 37--100.

\bibitem[dV34]{duVal}
P.~du~Val, \emph{On isolated singularities of surfaces which do not affect the
  conditions of adjunction {II}, {III}}, Proc. Cambridge Phil. Soc. \textbf{30}
  (1934), 483--491.

\bibitem[Fuj79]{Fujita-Zariski}
T.~Fujita, \emph{On {Z}ariski problem}, Proc. Japan Acad. Ser. A Math. Sci.
  \textbf{55} (1979), no.~3, 106--110.

\bibitem[Fuj82]{Fujita-noncomplete_surfaces}
\bysame, \emph{On the topology of noncomplete algebraic surfaces}, J. Fac. Sci.
  Univ. Tokyo Sect. IA Math. \textbf{29} (1982), no.~3, 503--566.

\bibitem[Fuj95]{Fujisawa}
T.~Fujisawa, \emph{On non-rational numerical del {P}ezzo surfaces}, Osaka J.
  Math. \textbf{32} (1995), no.~3, 613--636.

\bibitem[Fuj21]{Fujino_MMP}
O.~Fujino, \emph{On minimal model theory for algebraic log surfaces}, Taiwanese
  J. Math. \textbf{25} (2021), no.~3, 477--489.

\bibitem[Fur86]{Furushima}
M.~Furushima, \emph{Singular del {P}ezzo surfaces and analytic
  compactifications of {$3$}-dimensional complex affine space {${\bf C}^3$}},
  Nagoya Math. J. \textbf{104} (1986), 1--28.

\bibitem[GZ94]{GurZha_1}
R.~V. Gurjar and D.-Q. Zhang, \emph{{$\pi_1$} of smooth points of a log del
  {P}ezzo surface is finite. {I}}, J. Math. Sci. Univ. Tokyo \textbf{1} (1994),
  no.~1, 137--180.

\bibitem[Har77]{Hartshorne_AG}
R.~Hartshorne, \emph{Algebraic geometry}, Springer-Verlag, New York, 1977,
  Graduate Texts in Mathematics, 52.

\bibitem[Hos96]{Hosoh_automorphisms}
T.~Hosoh, \emph{Automorphism groups of quartic del {P}ezzo surfaces}, J.
  Algebra \textbf{185} (1996), no.~2, 374--389.

\bibitem[HW81]{HW_canonical}
F.~Hidaka and K.~Watanabe, \emph{Normal {G}orenstein surfaces with ample
  anti-canonical divisor}, Tokyo J. Math. \textbf{4} (1981), no.~2, 319--330.

\bibitem[KM98]{KollarMori-bir_geom}
J.~Koll{\'a}r and S.~Mori, \emph{Birational geometry of algebraic varieties},
  Cambridge Tracts in Mathematics, vol. 134, Cambridge University Press,
  Cambridge, 1998, with the collaboration of C. H. Clemens and A. Corti.

\bibitem[KM99]{Keel-McKernan_rational_curves}
S.~Keel and J.~McKernan, \emph{Rational curves on quasi-projective surfaces},
  Mem. Amer. Math. Soc. \textbf{140} (1999), no.~669.

\bibitem[KN20]{Kawakami_Nagaoka_canonical_dP-in-char>0}
T.~Kawakami and M.~Nagaoka, \emph{Du {V}al del {P}ezzo surfaces in positive
  characteristic}, \arxiv{2008.07700}, 2020.

\bibitem[KN22]{KN_Pathologies}
\bysame, \emph{Pathologies and liftability of {D}u {V}al del {P}ezzo surfaces
  in positive characteristic}, Math. Z. \textbf{301} (2022), no.~3, 2975--3017.

\bibitem[Koj14]{Kojima_Sing-1}
H.~Kojima, \emph{Normal log canonical del {P}ezzo surfaces of rank one with
  unique singular points}, Nihonkai Math. J. \textbf{25} (2014), no.~2,
  105--118.

\bibitem[Kol92]{Flips_and_abundance}
J.~Koll\'ar (ed.), \emph{Flips and abundance for algebraic threefolds - {A}
  summer seminar at the university of {U}tah ({S}alt {L}ake {C}ity, 1991)},
  Ast\'erisque, no. 211, Soci\'et\'e math\'ematique de France, 1992.

\bibitem[Kol13]{Kollar_singularities_of_MMP}
J.~Koll\'{a}r, \emph{Singularities of the minimal model program}, Cambridge
  Tracts in Mathematics, vol. 200, Cambridge University Press, Cambridge, 2013,
  with a collaboration of S. Kov\'{a}cs.

\bibitem[KR07]{KR_contractible}
M.~Koras and P.~Russell, \emph{Contractible affine surfaces with quotient
  singularities}, Transform. Groups \textbf{12} (2007), no.~2, 293--340.

\bibitem[Miy01]{Miyan-OpenSurf}
M.~Miyanishi, \emph{Open algebraic surfaces}, CRM Monograph Series, vol.~12,
  American Mathematical Society, Providence, RI, 2001.

\bibitem[MKM83]{Kumar-Murthy}
N.~Mohan~Kumar and M.~P. Murthy, \emph{Curves with negative self-intersection
  on rational surfaces}, J. Math. Kyoto Univ. \textbf{22} (1982/83), no.~4,
  767--777.

\bibitem[MS80]{Miy_Su}
M.~Miyanishi and T.~Sugie, \emph{Affine surfaces containing cylinderlike open
  sets}, Journal of Mathematics of Kyoto University \textbf{20} (1980), 11--42.

\bibitem[MT84a]{Miy_Tsu-opendP}
M.~Miyanishi and S.~Tsunoda, \emph{Logarithmic del {P}ezzo surfaces of rank one
  with noncontractible boundaries}, Japan. J. Math. (N.S.) \textbf{10} (1984),
  no.~2, 271--319.

\bibitem[MT84b]{Miy_Tsu-nonconnected}
\bysame, \emph{Noncomplete algebraic surfaces with logarithmic {K}odaira
  dimension {$-\infty$} and with nonconnected boundaries at infinity}, Japan.
  J. Math. (N.S.) \textbf{10} (1984), no.~2, 195--242.

\bibitem[Mul99]{Mulay}
S.~Mulay, \emph{Classification of plane cubic curves}, Advances in commutative
  ring theory ({F}ez, 1997), Lecture Notes in Pure and Appl. Math., vol. 205,
  Dekker, New York, 1999, pp.~461--482.

\bibitem[MZ88]{MZ_canonical}
M.~Miyanishi and D.-Q. Zhang, \emph{Gorenstein log del {P}ezzo surfaces of rank
  one}, J. Algebra \textbf{118} (1988), no.~1, 63--84.

\bibitem[Né22]{Nemethi_book}
A.~Némethi, \emph{Normal surface singularities}, Ergebnisse der Mathematik und
  ihrer Grenzgebiete. 3. Folge. A Series of Modern Surveys in Mathematics,
  vol.~74, Springer, Cham, 2022.

\bibitem[Pal15]{Palka-AMS_LZ}
K.~Palka, \emph{A new proof of the theorems of {L}in-{Z}aidenberg and
  {A}bhyankar-{M}oh-{S}uzuki}, J. Alg. Appl. \textbf{14} (2015), no. 9.

\bibitem[Pal24]{Palka_almost_MMP}
\bysame, \emph{Almost minimal models of log surfaces}, \arxiv{2402.07187},
  2024.

\bibitem[PK13]{PK_QHP}
K.~Palka and M.~Koras, \emph{Singular {$\Bbb Q$}-homology planes of negative
  {K}odaira dimension have smooth locus of non-general type}, Osaka J. Math.
  \textbf{50} (2013), no.~1, 61--114.

\bibitem[PP24]{PaPe_ht_2}
K.~Palka and T.~Pełka, \emph{Classification of del {P}ezzo surfaces of rank
  one. {I}. {H}eight 1 and 2, {II}. {D}escendants with elliptic boundaries},
  \arxiv{2412.21174}, 2024.

\bibitem[Rei97]{Reid-chapters}
M.~Reid, \emph{Chapters on algebraic surfaces}, Complex algebraic geometry
  ({P}ark {C}ity, {UT}, 1993), IAS/Park City Math. Ser., vol.~3, Amer. Math.
  Soc., Providence, RI, 1997, pp.~3--159.

\bibitem[Rus81]{Russell-ruled}
P.~Russell, \emph{On affine-ruled rational surfaces}, Math. Ann. \textbf{255}
  (1981), no.~3, 287--302.

\bibitem[Sch01]{Schroer_non-rational}
S.~Schr\"{o}er, \emph{Normal del {P}ezzo surfaces containing a nonrational
  singularity}, Manuscripta Math. \textbf{104} (2001), no.~2, 257--274.

\bibitem[Ser06]{Sernesi_deformations}
E.~Sernesi, \emph{Deformations of algebraic schemes}, Grundlehren der
  mathematischen Wissenschaften [Fundamental Principles of Mathematical
  Sciences], vol. 334, Springer-Verlag, Berlin, 2006.

\bibitem[Sil09]{Silverman}
J.~Silverman, \emph{The arithmetic of elliptic curves}, second ed., Graduate
  Texts in Mathematics, vol. 106, Springer, Dordrecht, 2009.

\bibitem[Sug80]{Sugie}
T.~Sugie, \emph{On a characterization of surfaces containing cylinderlike open
  sets}, Osaka Math. J. \textbf{17} (1980), no.~2, 363--376.

\bibitem[Ye02]{Ye}
Q.~Ye, \emph{On {G}orenstein log del {P}ezzo surfaces}, Japan. J. Math. (N.S.)
  \textbf{28} (2002), no.~1, 87--136.

\bibitem[Zha02]{Zhang_canonical}
D.-Q. Zhang, \emph{Automorphisms of finite order on {G}orenstein del {P}ezzo
  surfaces}, Trans. Amer. Math. Soc. \textbf{354} (2002), no.~12, 4831--4845.

\end{thebibliography}

\end{document}